\documentclass[11pt]{article}
\usepackage[margin=3cm]{geometry}
\RequirePackage[colorlinks,citecolor=blue,urlcolor=blue]{hyperref}

\usepackage{amssymb}
\usepackage{amsmath}
\usepackage{amsthm}
\usepackage{verbatim}
\usepackage{graphicx}
\usepackage{amsmath}
\usepackage{bm}
\usepackage{setspace}
\usepackage[inline,shortlabels]{enumitem}
\usepackage[font=small]{caption}
\usepackage{chngcntr}
\usepackage{apptools}
\usepackage{booktabs}
\usepackage{microtype}

\RequirePackage[colorlinks,citecolor=blue,urlcolor=blue]{hyperref}

\usepackage[style=alphabetic,
			backend=bibtex,
			doi=false,
			isbn=false,
			url=false,
			firstinits=true,
			maxnames=4
			]{biblatex}

\bibliography{refs}
\renewbibmacro{in:}{}			
\DeclareFieldFormat[article,inbook,incollection,inproceedings,patent,thesis,unpublished]{citetitle}{#1}
\DeclareFieldFormat[article,inbook,incollection,inproceedings,patent,thesis,unpublished]{title}{#1} 

\usepackage{algorithm}
\usepackage{algpseudocode}

\algnewcommand{\IIf}[1]{\State\algorithmicif\ #1\ \algorithmicthen}

\usepackage[matrixnorms,opnorm=op]{arash_macros}

\newcommand{\rhoh}{\widehat \rho}

\newcommand{\Ac}{\mathcal A}

\newcommand\psim{\mathcal P}
\newcommand\wconv{\rightsquigarrow}
\newcommand\Cc{\mathcal C}
\newcommand\law{\mathcal L}

\DeclareMathOperator{\mult}{Mult}
\DeclareMathOperator{\ber}{Ber}
\newcommand\Ut{\widetilde U}
\newcommand{\Xt}{\widetilde X}

\DeclareMathOperator{\bin}{Bin}
\DeclareMathOperator{\cov}{cov}

\newcommand\approxd{\stackrel{d}{\approx}}

\newcommand\Yh{\widehat Y}
\newcommand\zh{\widehat z}
\newcommand\yh{\widehat y}

\newcommand\ac{{AC}\xspace}
\newcommand\cac{{FNAC}\xspace}

\usepackage{xspace}
\newcommand{\cacp}{{FNAC+}\xspace}
\newcommand{\scacp}{{SNAC+}\xspace}
\newcommand{\scac}{{SNAC}\xspace}
\newcommand{\cacf}{{FNAC(+)}\xspace}
\newcommand{\scacf}{{SNAC(+)}\xspace}
\newcommand{\nac}{{NAC}\xspace}

\newcommand\Ch{\hat{\mathcal C}}

\newcommand\alh{\hat \alpha}
\newcommand\Tch{\hat{\mathcal T}}

\newcommand\nh{\hat n}

\newcommand\beh{\hat{\beta}}
\newcommand\amax{a_{\max}}
\newcommand\thetamax{\theta_{\max}}
\newcommand\Fc{\mathcal F}

\newcommand\dmin{d_{\min}}

\newcommand\ph{\widehat p}
\newcommand\Delh{\widehat \Delta}
\newcommand\Th{\widehat T}
\newcommand{\dkol}{d_{\text{K}}}
\newcommand\pul{\underline{p}}
\newcommand\dav{d_{\text{av}}}
\DeclareMathOperator{\poi}{Poi}
\newcommand\dt{\widetilde d}
\newcommand\lamax{\lambda_{\max}}
\newcommand\Bc{\mathcal B}
\newcommand{\Dc}{\mathcal{D}}
\newcommand\ds{d^*}

\newcommand{\nt}{\tilde{n}}

\newcommand{\rhou}{\underline{\rho}}
\newcommand\kapt{\widetilde\kappa}

\newcommand\Gc{\mathcal G}
\newcommand\rhot{\widetilde \rho}
\newcommand\Tt{\widetilde T}

\newcommand\Gh{\widehat \Gc}
\newcommand\grad{\nabla}
\newcommand\hess{\nabla^2}
\newcommand\Ec{\mathcal E} 
\newcommand\At{\tilde A}
\newcommand\Yt{\widetilde Y}
\newcommand{\pt}{\widetilde p}
\newcommand\gh{\hat g}
\newcommand\dmax{d_{\max}}
\newcommand\alphan{\alpha_n}
\newcommand{\rhob}{\bar{\rho}}

\newcommand\omone{\omega_{n,1}}
\newcommand\omhalf{\omega_{n, 1/2}}
\newcommand\Delt{\widetilde \Delta}
\newcommand\delh{\hat \delta}

\newcommand\Hc{\mathcal H}
\newcommand\kh{{\hat k}}    

\newcommand\given{\,|\,}

\newcommand\pit{\widetilde{\pi}}

\newcommand{\Mc}{\mathcal{M}}
\newcommand\Qb{\mathbb Q}
\newcommand\Wc{\mathcal{W}}
\newcommand\Rc{\mathcal{R}}

\newcommand{\xip}{\xi'}

\newcommand{\epsnp}{\eps_n'}
\newcommand{\mt}{\widetilde{m}}
\newcommand\fvarmin{\underline{\fvar}}
\newcommand\taur{\tau_\rho}
\newcommand\taut{\tau_\theta}
\newcommand\tauc{\tau_\Cc}
\newcommand\tauh{\tau_h}
\newcommand{\Rt}{\widetilde{R}}
\newcommand\fvar{\vartheta} 
\newcommand\dpk{d_+^k}
\newcommand\Nc{\mathcal N}
\newcommand\evar{\mathbb V}

\DeclareMathOperator{\impu}{imp}
\DeclareMathOperator\miss{Mis}

\newcommand\atn{\tilde \alpha_n}
\newcommand\sigb{\bar \sigma}
\newcommand\Kn{K}

\title{Adjusted  chi-square test for degree-corrected block models} 
\author{Linfan Zhang and Arash A. Amini}

\begin{document}

\maketitle

\begin{abstract}
	We propose a goodness-of-fit test for degree-corrected stochastic block models (DCSBM). The test is based on an adjusted chi-square statistic for measuring equality of means among groups of $n$ multinomial distributions with $d_1,\dots,d_n$ observations. In the context of network models, the number of multinomials, $n$, grows much faster than the number of observations, $d_i$, corresponding to the degree of node $i$, hence the setting deviates from classical asymptotics. We show that a simple adjustment allows the statistic to converge in distribution, under null, as long as the harmonic mean of $\{d_i\}$ grows to infinity. 
	When applied sequentially, the test can also be used to determine the number of communities. The test operates on a 
	compressed version of the adjacency matrix, conditional on the degrees, and as a result is highly scalable to large sparse networks. We incorporate a novel idea of compressing the rows based on a $(K+1)$-community assignment when testing for $K$ communities. This approach increases the power in sequential applications without sacrificing computational efficiency, and we prove its consistency in recovering the number of communities. Since the test statistic does not rely on a specific alternative, its utility goes beyond sequential testing and can be used to simultaneously test against a wide range of alternatives outside the DCSBM family. In particular, we prove that the test is consistent against a general family of latent-variable network models with community structure.
	The test can also be easily applied to Poisson count arrays in clustering or biclustering applications, as well as bipartite and directed networks. We show the effectiveness of the approach by extensive numerical experiments with simulated and real data. In particular, applying the test to the Facebook-100 dataset, a collection of one hundred social networks, we find that a DCSBM with a small number of communities (say $ < 25$) is far from a good fit in almost all cases. Despite the lack of fit, we show that the statistic itself can be used as an effective tool for exploring community structure, allowing us to construct a community profile for each network.
\end{abstract}

\section{Introduction}

Network analysis has become an increasingly prominent part of data analysis as the developments in the age of the internet and in various sciences, especially life and social sciences, have produced a substantial collection of network data. Given a network, it is of interest to understand its structure, which is often done by finding communities or clusters. Probabilistic network models such as the Stochastic Block Model (SBM) \cite{Holland1983} and its variant the Degree-Corrected Stochastic Block Model (DCSBM)~\cite{Karrer2011} are commonly used to recover the community structure from network data. Both  models use a latent variable, the node label, to categorize nodes in a network into different communities. In the SBM, the probability of an edge formation between two nodes depends on the communities  they belong to. The DCSBM incorporates an additional propensity parameter to determine the edge probability, allowing heterogeneous node degrees within a community. 

The SBM and its degree-corrected variant have been the subject of intense study in recent years and numerous methods have been developed for fitting them.  A very incomplete list includes modularity maximization~\cite{Newman2004, Bickel21068}, likelihood-based approaches such as the profile likelihood~\cite{Bickel21068, zhao2012}, the pseudo-likelihood~\cite{amini2013} and the variational likelihood~\cite{daudin2008mixture, bickel2013, zhang2020variantional}, spectral methods based on the adjacency matrix \cite{rohe2011, chaudhuri2012spectral, qin2013regularized, Donniell2013, yun2014accurate, lei2015, chin2015spectral, joseph2016, abbe2016exact, zhou2019analysis}, the non-backtracking matrix \cite{Krzakala20935} and the Bethe-Hessian matrix~\cite{NIPS2014Saade}, 
semidefinite relaxations \cite{abbe2016exact, amini2018, li2018convex, fei2019SDP}, local refinements \cite{mossel2014consistency, gao2017achieving, gao2018, Lei2017generic, zhou2020optimal}, message-passing algorithms \cite{Decelle2011BP,  Zhang2014, abbe2015detection, mossel2016BP} and Bayesian approaches \cite{snijders1997estimation, hofman2008bayesian, Morup2012bayesian, suwan2016, vanderpas2018, paez2019hierarchical}. Many of these methods are based on the assumption that the number of communities $K$ is given and most come with consistency guarantees, when the data is generated from the corresponding model with $K$ communities. We refer to~\cite{Abbe2018} for a review of the theoretical limits of community detection in SBMs.

On the other hand, how well these network models fit the data, the so-called goodness-of-fit question, is studied comparatively much less. Prominent work in this area include the graphical approach of~\cite{hunter2008goodness} for general network models, and the recent work of Bickel and Sarkar~\cite{bickel2016} and its extension by Lei~\cite{Lei2016}, on a spectral goodness-of-fit test for the SBM. Developing goodness-of-fit tests specifically for the DCSBM is more challenging and to the best of our knowledge has not been considered so far, except for the work of Karwa~et~al.~\cite{karwa2016} on the related $\beta$-SBM. A related problem is that of model selection, that is, determining the number of communities assuming that the network is generated from some SBM (or DCSBM). 
An application of model selection is 
designing the stopping rule in hierarchical clustering~\cite{li2020hierarchical}.
The model selection problem has been studied more extensively and we provide an overview of the literature in Section~\ref{sec:related}. 

Compared to model selection, goodness-of-fit testing is a more general problem. When applied sequentially, such tests can also be used for model selection. However, their utility goes beyond model selection and they can be used to test against a wide range of alternatives. They also provide a quantitative and baseline-normalized measure of how well the model fits in various situations. On the other hand, the ability to simultaneously test against many alternatives can  be considered a weakness. To quote L. Breiman~\cite{breiman2001statistical}: 

\begin{quote}
    ``Work by Bickel, Ritov and Stoker (2001)~\cite{bickel2006tailor} shows that goodness-of-fit tests have very little power unless the direction of the alternative is precisely specified. The implication is that omnibus goodness-of-fit tests, which test in many directions simultaneously, have little power, and will not reject until the lack of fit is extreme.''
\end{quote}

In our experiments, we have found the opposite to be true for current network models. It is possible to construct powerful tests, without specifying the direction of the alternative, for one of the most established families of network models. For example, we demonstrate both theoretically and empirically
that the tests we develop for DCSBM are extremely powerful against a latent-variable community-structured model outside the DCSBM family (cf. Sections~\ref{sec:consist:DCLVM} and~\ref{sec:roc}). Moreover, for the majority of the real networks that we tested, the null hypothesis of a DCSBM with a small number of communities is strongly rejected (cf.~Section~\ref{sec:fb100:got:testing}). This is all the more surprising given that the DCSBM is considered the state-of-the-art in modeling real community-structured networks.

\subsection{Our contributions}
In this paper, we propose the adjusted chi-square test for measuring the goodness-of-fit of a DCSBM. The idea is as follows: Given a set of column labels, we compress the adjacency matrix by summing each row over the communities specified by the labels, a process we will refer to as column aggregation. Under a DCSBM, the rows of the compressed matrix will have a multinomial distribution, conditional on the node degrees $d_i$ (i.e., the row sums). Rows in the same (row) community will have the same multinomial parameter. Thus, the problem reduces to test whether groups of multinomials have equal means. The challenge is that the number of multinomials in each group is proportional to $n$, the total number of nodes, which grows to infinity fast, while the number of observations in each multinomial, $d_i$, grows much slower. We study this general multi-group testing problem in Section~\ref{sec:ac:test:0} and show that under mild conditions, as long as the harmonic mean $h(d_1,\dots,d_n)$ goes to infinity, a modified version of the classical chi-square statistic, which we refer to as Adjusted Chi-square (\textbf{AC}), has the standard normal distribution under the null hypothesis. 

We then extend these ideas to the analysis of networks, leading to the Network Adjusted Chi-square (\textbf{NAC})
family of tests. 
The family includes many variants depending on which subsets of the adjacency matrix are used and how the columns are aggregated. Assume that we want to test a $K$-community DCSBM. One variant of the test uses a subsampling scheme and aggregates using $K$ communities for the columns. We refer to this version as \textbf{\scac}, for subsampled NAC. We show that given a consistent set of labels, \textbf{\scac} has the standard normal distribution under null. Another variant of the test uses subsampling but aggregates using $(K+1)$-community column labels, while still using $K$-community row labels when testing for the equality of multinomials. We refer to this variant as \textbf{\scacp}. We show that \textbf{\scacp} has the same null distribution as \textbf{\scac}, but is more powerful against DCSBM alternatives in sequential applications (Section~\ref{sec:consist}).

We also develop bootstrapped versions of the tests which are more robust in practice and can be applied even when the null distribution of the test statistic is difficult to compute. Moreover, we introduce a smoothing idea that can further increase the robustness of sequential model selection.  

Our theoretical results are nonasymptotic, controlling the Kolomogrov distance of the distribution of the test statistic to the target, with explicit constants. The results are valid in the regime where the expected average degree of the network, $\lambda$, scales as $\gtrsim \log n$, hence applicable in the same sparsity regime where strong consistency (i.e., exact label recovery) is possible for DCSBMs. Our results, however, only require weakly consistent labels subject to bounds on the rate of convergence that are more relaxed than that of strong consistency. From a computational standpoint, evaluating the statistic is highly scalable, with an expected  computational overhead of $O(n(\lambda + K))$ over the cost of applying the community detection algorithm. To test a sequence of DCSBMs with $K=K_1,\dots,K_2$, the test requires an application of a community detection algorithm at most $K_2-K_1 + 2$ times.

We show the effectiveness of these ideas with extensive experiments on simulated and real networks. The code for these experiments is available at~\cite{zhang2020adjusted}. In particular, we apply the test to the Facebook-100 dataset~\cite{traud2011comparing, TRAUD2012fb}, a collection of one hundred social networks, and find that a DCSBM (or SBM) with a small number of communities (say $ < 25$) is far from a good fit in almost all cases. Despite the lack of fit, we show that the statistic itself can be used as an effective tool for exploring communities, due to its high sensitivity to block structure. Coupled with the smoothing idea,  \scacp allows us to construct a \emph{community profile} for each network, regardless of whether DCSBM is a good fit.

\subsection{Related work}\label{sec:related}
Various methods have been developed to address the model selection problem in the SBM and DCSBM. The popular Bayesian information criterion (BIC) has been adapted to the network setting in~\cite{Yan_2016_BIC, wang2017, Hu_2019_CBIC}. 
Likelihood ratio tests have been developed for comparing two block models in~\cite{yan2014model, wang2017, yuan2018likelihoodratio, ma2018}. Bayesian approaches, though computationally intensive, can estimate the structure and the number of communities simultaneously. Ideas include the use of Dirichlet process prior~\cite{paez2019hierarchical} and mixture of priors~\cite{newman2016estimating, riolo2017efficient, Geng2019baysian}.
Cross-validation, another widely used idea for model selection, has too been adapted to network settings~\cite{Kawamoto_2017, Chen2018, Li2020_ECV}. A leave-one-out scheme has been used in~\cite{Kawamoto_2017}  with the posterior predictive density of an edge, under the SBM, as the loss function. Chen and Lei~\cite{Chen2018} use a node-pair splitting idea, while~\cite{Li2020_ECV} uses edge sampling followed by low-rank matrix completion, an approach that can be applied to any low-rank network model.
A spectral approach to determining the number of communities in the SBM is explored in~\cite{le2015estimating}, based on counting the nonnegative eigenvalues of the non-backtracking and Bethe Hessian matrices.
The approach of \cite{le2015estimating} can be extended to other low-rank structured models such as DCSBM. Semidefinite programming have been shown in~\cite{yan18a} to be capable of performing label recovery and model selection in one shot. Modularity maximization can also perform the two tasks simultaneously~\cite{Newman2004}. 
 
Comparatively, the goodness-of-fit problem has been explored much less. The pioneering work of \cite{hunter2008goodness} graphically compares certain network statistics (such as degree distribution) between the observed network and a collection of networks simulated from the fitted model. This approach is quite general and can be applied to any network model, though its graphical nature makes it somewhat qualitative. The Monte Carlo simulation procedures in \cite{hunter2008goodness} have also been further exploited in other works \cite{li2013assessing, ospina2019assessment} to test the goodness-of-fit of graph models. Among them, we note that Karwa et al.~\cite{karwa2016} develops a chi-square test for SBM and uses Markov Chain Monte Carlo sampling to approximate its exact $p$-value. We make a detailed comparison with~\cite{karwa2016} in Section~\ref{sec:comparison}.  For SBMs, a spectral goodness-of-fit test was developed in~\cite{bickel2016} for the case of $K=2$ communities and subsequently extended to general $K$ in~\cite{Lei2016}. The test is based on the largest eigenvalue of a standardized residual adjacency matrix (cf. Section~\ref{sec:methods} for more details). Using results from random matrix theory~\cite{erdHos2012rigidity, lee2014necessary}, this eigenvalue has an asymptotic Tracy-Widom distribution under the null, a result that can be used to set the critical threshold. Although, we can apply the same ideas in the DCSBM setting, the null distribution result does not hold, due to the uncertainty in estimating the node connection propensity parameters. Whether a rigorous spectral goodness-of-fit test of this form exists for DCSBM is not clear. 

\medskip

The rest of the paper is organized as follows: Section~\ref{sec:ac:test:0}  introduces the adjusted chi-square test and its multi-group extension and establishes its null limiting distribution. Section~\ref{sec:nac:test} introduces NAC family of tests. Section~\ref{sec:analysis} establishes the null limiting distribution of \scac and \scacp and Section~\ref{sec:cons} shows their consistency against underfitted DCSBM and a latent-variable network block model. 
Section~\ref{sec:simulations} demonstrates the competitive performance of NAC tests, as a model selection method, compared to other state-of-the-art approaches. 
In Section~\ref{sec:fb100:got:testing}, we illustrate how  \scacp can be used to assess the goodness-of-fit for an ensemble of real networks, namely the Facebook-100 dataset. Section~\ref{sec:comm:profile} discusses how smoothed \scacp can be used to build community profiles of real networks. 

\section{Adjusted chi-square test}\label{sec:ac:test:0}

We start by developing a general test for the equality of the parameters among groups of multinomial observations. To set the ideas, we first consider the case of a single group and show how the classical chi-square test can be adjusted to accommodate a growing number of multinomials. We then discuss the multi-group extension and provide quantitative bounds for the null distribution of the test statistic in this general setting.

\subsection{Single-group case}
\label{sec:ac:test}
Let $\psim_L$ be the probability simplex in $\reals^L$, and 
consider the following problem: We have 
\begin{align}\label{eq:mult:model}
X_i \sim \text{Mult}(d_i,p^{(i)}), \quad i=1,\dots,n,
\end{align}
independently, where $X_i = (X_{i\ell}) \in \mathbb{N}^L$ and $p^{(i)} \in \psim_L$, and we would like to test the null hypothesis 
\begin{align}\label{eq:mult:null}
H_0: \; p^{(1)} = p^{(2)} = \dots = p^{(n)}= p.
\end{align}
Let $\psi(x,y) := (x-y)^2/y$.
The chi-square statistic for testing this hypothesis is 
\begin{align*}
\Yt^*_{(n,d)} := \sum_{i=1}^n \sum_{\ell=1}^L 
\psi\big(X_{i\ell}, d_i \pt_{\ell}\big), \quad 
\text{where}\quad \pt_\ell = \frac{\sum_{i=1}^n X_{i\ell}}{ \sum_{i=1}^n d_i}, \; \ell \in [L].
\end{align*}
Here, 
$\pt = (\pt_\ell) \in \psim_L$ is the pooled estimate of $p$ under the null, and $d = (d_1,\dots,d_n)$. We are also using the shorthand notation $[L] := \{1,\dots,L\}$.

Standard asymptotic theory gives the following %
(cf. Chapter 17 in~\cite{vaart_1998}): 
If $n$ is fixed and $\dmin := \min_i d_i \to \infty$, then,
\begin{align}\label{eq:Y:chisq:conv}
\Yt^{*}_{(n,d)} \wconv \chi^2_{(n-1)(L-1)}, \;\text{under $H_0$}.
\end{align}
A heuristic for the degrees of freedom of the limiting $\chi^2$ distribution can be given by counting parameters.
In the unrestricted model, we have a total of $n(L-1)$ free parameters among $p^{(1)},\dots,p^{(n)}$, while under the restricted null model, we only have $L-1$ free parameters. The difference gives the degrees of freedom of the limit.

The setting we are interested in, however, is the opposite of the classical setting. We would like to use the statistic when $n \to \infty$, while $\dmin$ is fixed or grows slowly with $n$. Assuming that $n$ is large enough so that $(n-1)(L-1) \approx n(L-1)$, \eqref{eq:Y:chisq:conv} suggests that we can approximate $\Yt^*_{(n,d)}$ in distribution by the sum of $n$ independent $\chi^2_{L-1}$ variables, that is,
\[
	\Yt^*_{(n,d)} \approxd 
	\sum_{i=1}^n \xi_i
\]
for some i.i.d. random variables $\xi_i \sim \chi^2_{L-1}$. The approximate inequality above is only in distribution and $\{\xi_i\}$ are not necessarily related to $\Yt^*_{(n,d)}$. Moreover, the central limit theorem suggests that the standardized version of $\sum_{i} \xi_i$ has a distribution close to a standard normal.

Based on the above heuristic argument, we propose the following adjusted test statistic:
\begin{align}
\Tt^*_n = \frac{1}{\sqrt{2}}\Big(\frac{\Yt^*_{(n,d)}}{\gamma_n} - \gamma_n\Big), \quad \text{where} \; \gamma_n = \sqrt{n(L-1)}.
\label{equ:CCstat}
\end{align}
Note that $\gamma_n^2$ is the expectation of $\sum_i \xi_i$ and $\sqrt{2}\gamma_n$ is its standard deviation. We refer to~\eqref{equ:CCstat} as the \emph{adjusted chi-square} (\textbf{\ac}) statistic. 
\begin{rem}\label{rem:ac:paper}
	The name adjusted chi-square has appeared in the literature in contexts completely different from our work. For example, adjustments to the chi-square statistic to account for the dependence of individuals have been proposed by Reed~\cite{reed2004adjusted} in randomized cluster trials, and by Jung et al.~\cite{jung2001evaluation} and Ahn et al.~\cite{ahn2002application} in  observational studies. 
\end{rem}

\subsection{Multi-group extension}  \label{sec:multi:group}
Before proceeding, let us introduce an extension of the testing problem~\eqref{eq:mult:null} to groups of observations. This extension is needed for the network applications. Consider model~\eqref{eq:mult:model} and assume that each observation is assigned to one of the $K$ known groups, denoted as $[K] = \{1,\dots,K\}$. Let $g_i \in [K]$ be the group assignment of observation $i$ and let $\Gc_k = \{i \in [n]:\; g_i = k\}$ be the $k$th group. We would like to test the null hypothesis that all the observations in the same group have the same parameter vector, that is,
\begin{align}\label{eq:group:null}
	H_0:\; p^{(i)} = p_{k*}, \; \forall i \in \Gc_k, \; k \in [K],
\end{align}
where for each $k \in [K]$, $p_{k*} = (p_{k\ell})_{\ell \in [L]} \in \psim_L$. 

In some problems, it is reasonable to assume that the groups $\Gc_k$ are known. However, in our network applications, the groups themselves are not known. In such	settings, we first estimate the label vector $g$ from data, to obtain $\gh$, and then form the test statistic based on the estimated groups $\Gh_k = \{i : \gh_i = k\}$. The resulting test is based on the extended chi-square statistic
\begin{align}
    \Yh_{(n,d)} = \sum_{k=1}^K \sum_{i \in \Gh_k} \sum_{\ell = 1}^L \psi( X_{i\ell}, d_i \ph_{k\ell})
    \quad\text{where}\quad
    \ph_{k\ell} = \frac{\sum_{i \in \Gh_k} X_{i\ell}}{ \sum_{i \in \Gh_k} d_i}, \; \ell \in [L].
\end{align}
Alternatively, we have $\Yh_{(n,d)} = \sum_{i=1}^n \sum_{\ell=1}^L \psi( X_{i\ell}, d_i \ph_{\gh_i\ell})$. We also let $Y_{(n,d)}$ be the idealized version of $\Yh_{(n,d)}$ with $\ph_{k\ell}$ replaced with $p_{k\ell}$ and $\Gh_k$ replaced with $\Gc_k$.
Let $\Th_n$ and $T_n$ be the adjusted chi-square statistics based on $\Yh_{(n,d)}$ and $Y^{(n,d)}$, respectively, that is,
\begin{align}\label{eq:Tn}
	\Th_n = \frac{1}{\sqrt{2}} \Big( \frac{\Yh_{(n,d)}}{\gamma_n} - \gamma_n\Big), \quad 
    T_n = \frac{1}{\sqrt{2}} \Big( \frac{Y_{(n,d)}}{\gamma_n} - \gamma_n\Big).
\end{align}
We are interested in understanding under what conditions $\Th_n$ has an approximately normal null distribution. This question is nontrivial, since we would like to allow $\{d_i\}$ as well as groups sizes $|\Gc_k|, k \in [K]$ to 
vary with $n$. Moreover, we would like to allow the groups to be estimated based on the same data we use for testing, in which case, $\gh$ and $\Th_n$ are most likely statistically dependent. 

We give a precise answer to the above question by quantifying the Kolomogorv distance between the distribution of $\Th_n$ and that of a standard normal variable $Z$, for any choice of $\{d_i\}$ and $\{|\Gc_k|\}$ that satisfy a mild set of conditions, and for consistent label estimates of a certain quality. We measure the quality of label estimation in terms of misclassification rate:
\begin{defn}
    The misclassification rate between two label vectors $g \in [K]^n$ and $\gh \in [K]^n$ is 
    \begin{equation*}
        \miss(g, \gh) = \min_{\omega} \frac{1}{n}\sum_{i = 1}^n 1\{g_i \neq \omega(\gh_i)\}
    \end{equation*}
    where the minimization ranges over all bijective maps $\omega: [K] \rightarrow [K]$.
\end{defn}

Recall that for two random variables $X$ and $Y$, the Kolomogrov distance between their distributions is defined as
\begin{align}
	\dkol(X, Y) := \sup_{t \in \reals} \big|\pr(X \le t) - \pr(Y \le t)\big|.
\end{align}
For a vector $d = (d_1,\dots,d_n)$, we write
$
h(d) = \big( n^{-1} \sum_{i=1}^n d_i^{-1} \big)^{-1}
$
for the harmonic mean of its elements, and $\dav = n^{-1} \sum_{i=1}^n d_i$ for the arithmetic mean. Since $d$ has positive elements, $\dav \ge h(d) \ge \dmin := \min_i d_i$. Let $\pi_k = |\Gc_k|/ n$ and write $\dav^{(k)} = \frac1{|\Gc_k|} \sum_{i \in \Gc_k} d_i$ for the arithmatic average of  $\{d_i\}$ within group $\Gc_k$, and define
\begin{align}\label{eq:omegan:dmax:taud}
    \omega_n := \min_k \pi_k \dav^{(k)}, \quad d_{\max} := \max_{i} d_i, \quad \tau_d := \omega_n/\dmax
\end{align}

The following result formalizes the heuristic argument of Section~\ref{sec:ac:test}, by providing a quantitative finite-sample bound on the Kolomogrov distances of $T_n$ and $\Th_n$ to a standard normal variable:
\begin{thm}\label{thm:CCtest:null}
	Let $X_i \sim \mult(d_i, p_{k*}),\, i \in \Gc_k, k \in[K]$ be $n$ independent $L$-dimensional multinomial variables, with probability vectors $p_{k*} = (p_{k\ell})$ and group labels $g = (g_i) \in [K]^n$ so that $\Gc_k = \{i: g_i = k\}$. 
	Let $\gh$ be some (estimated) group labels, potentially dependent on $\{X_i\}$ and consider $\Th_n$, based on $\gh$, and $T_n$ as in~\eqref{eq:Tn}. Let  $Z \sim N(0,1)$ and $\pul = \min_{k,\ell} p_{k\ell}$. Assume that $\min\{h(d),L\} \ge 2$.

\begin{enumerate}[(a)]
	\item \label{part:a:thm:CCtest:null} Then, under the null hypothesis~\eqref{eq:group:null}, for all $n \ge 1$,
	\begin{align}
		\dkol(T_n, Z) \le\frac{ C_{1,p}}{\sqrt{L n}}+  \frac{C_{2,p}}{h(d)}
    	\label{eq:Tn:Z:bound}
	\end{align}
	where $C_{1,p} = 55 / \pul^{4}$ and $C_{2,p} = (\pi e)^{-1/2} \max\{1, \pul^{-1} - L-1 \}$.

	\smallskip
	\item \label{part:b:thm:CCtest:null}
	Let $C_{3,p} = 6 /(\pul \, \tau_d)$ and pick a sequence $\{\alpha_n\}$ such that 
	\begin{align}\label{eq:alphan:up:bound:thm1}
		\alphan \le \min\Bigl\{\frac{\pul}{8C_{3,p}},\  \frac{2}{C_{3,p}^2L} \Bigr\}, \quad \text{for all}\; n \ge 1.
	\end{align}
	Assume that $\sqrt{2} \dmax \ge L C_{3,p}$,
	$\omega_n \ge L$ and $\log (K \omega_n) / \omega_n \le (\pul/8)^2 n$. Then, under the null hypothesis~\eqref{eq:group:null}, for all $n \ge 1$,
	\begin{align}
		\dkol(\Th_n, Z) &\le \dkol(T_n, Z) + 12 \frac{\sqrt{L}}{\pul}\Bigg( \sqrt{\frac{ \log (K\omega_n)  }{\omega_n}} + \frac{ K \log (K\omega_n) }{\sqrt{n}} \notag \\
		&\quad + \frac{C_{3,p}}{3L}\dmax \sqrt{K n}\, \alphan \Bigg) 
		+ 2 \pr \bigl(\miss(\gh, g) >  \alphan \bigr).\label{eq:Thn:Tn:bound}
	\end{align}
\end{enumerate}
	
\end{thm}
Note that we always have $\pul^{-1} \ge L$ since the elements of $p_{k*}$ are nonnegative and sum to one. In the proof of Theorem~\ref{thm:CCtest:null}, we will show that $\ex[Y_{(n,d)}] = \gamma_n^2$. But the standard deviation $v_n(p) := \sqrt{\var[Y_{(n,d)}]}$ has a more complicated form and is not equal to $\sqrt 2 \gamma_n$ in general. The proof gives an explicit expression for this variance, and we could have alternatively defined $\Th_n$ by dividing by $v_n(\ph)$ instead of $\sqrt{2} \gamma_n$. Nevertheless, Theorem~\ref{thm:CCtest:null} shows that we do not lose much by using the simpler standardization by $\sqrt{2} \gamma_n$.

In general, for $T_n$ to converge  in distribution  to the standard normal, we need $n \to \infty$ and $h(d) \to \infty$. For $\Th_n$ to converge to the normal distribution, we further need $\omega_n  \to \infty$, $K\log(K\omega_n) = o(\sqrt{n})$, 
 \begin{align}\label{eq:alphan:req:thm1}
	\alphan = o((\dmax \sqrt{n})^{-1}),  \quad 
\pr \bigl(\miss(\gh, g) > \alphan \bigr) = o(1).
 \end{align}
 Note that  $\log (K \omega_n) / \omega_n \le (\pul/8)^2  n$ and 
\eqref{eq:alphan:up:bound:thm1} are satisfied for large $n$, as long as $\pul$ is bounded away from zero. The assumption $\sqrt{2} \dmax \ge  C_{3,p}L$ also holds since $\dmax \ge h(d)$ and we require $h(d) \to \infty$. 
 
 As we will see, in network applications, typically $K$, $L$ and $\pul$ are  of constant order. Then, the requirements reduce to~\eqref{eq:alphan:req:thm1}, $h(d) \to \infty$, $\omega_n \to \infty$ and $\log(\omega_n) = o(\sqrt n)$. The condition $h(d) \to \infty$ is fairly mild in network applications, since $d_i$ will be the degree of node $i$, and one often assumes that the network degrees grow to infinity as $n \to \infty$ (a necessary condition for weak label consistency). See also the empirical evidence in Appendix~\ref{app:deg:grow}. Even if one does not want to assume $h(d) \to \infty$ over the whole network, the condition can still be reasonably achieved by manually filtering out nodes with small $d_i$, as will be discussed in detail in Section~\ref{sec:nac:test}.

 Since, in networks, $\dmax$ grows much slower than $\sqrt n$ (closer to $\log n$ in fact), Condition~\eqref{eq:alphan:req:thm1} on the misclassification rate $\alpha_n$ is, in general, much milder than strong consistency which is equivalent to $\alpha_n = o(n^{-1})$. In the network setting, it is typical to assume that all the degrees grow at the same rate, in which case, $h(d) \asymp \dmax \asymp \omega_n$, and we obtain the convergence rate 
\begin{align*}
    \dkol(\Th_n, Z) \lesssim \sqrt{\log \omega_n / \omega_n} + 
	 \alphan \omega_n \sqrt{n} 
	+ \pr \bigl(\miss(\gh, g) > \alphan \bigr).
\end{align*}

\section{Network AC test} \label{sec:nac:test}
We are now ready to apply the \ac test to DCSBMs. Let $A_{n\times n}$ be the adjacency matrix of a random network on $n$ nodes. A DCSBM with connectivity matrix $B \in [0,1]^{K \times K}$, node label vector $z = (z_i) \in [K]^n$ and connection propensity vector $\theta = (\theta_i)\in \reals_+^n$, assumes the following structure for the mean of $A$,
\begin{align}\label{eq:dcsbm:mean:matrix}
	\ex[ A_{ij} \mid z] = \theta_i \theta_j B_{z_i z_j}, \quad \forall \; i\neq j.
\end{align}
One further assumes that $A$ is symmetric and the entries $A_{ij}, i < j$ are drawn independently, while $A_{ii} = 0$ for all $i$. %
Common choices for the distribution of each element, $A_{ij}$, are Bernoulli and Poisson. In this paper, unless otherwise stated, we assume the Poisson distribution for derivations, following the original DCSBM paper~\cite{Karrer2011}. The Poisson assumption simplifies the arguments and  provides computational advantages. We show in simulations that the tests so-derived work well in the Bernoulli case when the network is sparse. The  SBM is a special case of~\eqref{eq:dcsbm:mean:matrix} with $\theta_i = 1$ for all $i$.

\subsection{\nac family of tests}\label{sec:nac:family}
 The network AC test can be performed on a general submatrix $A_{S_2 S_1} = (A_{ij}: i \in S_2, j \in S_1)$ of the adjacency matrix, for $S_1, S_2 \subseteq [n]$.
 We first present this general form, though one can assume $S_1 = S_2 = [n]$ on the first reading. Consider another label vector on $S_1$, say $\yh = (\yh_j)_{j \in S_1} \in [L]^{S_1}$---for some $L$ that can be different from $K$. Let $R = (R_{k\ell}) \in \reals_+^{K \times L}$ be  the weighted confusion matrix between $z_{S_1}$ and $\yh$, given by
\begin{align}\label{eq:R:def}
	R_{k\ell} = \frac1{|S_1|} \sum_{\ j \in S_1} \theta_j 1\{z_j = k, \yh_j = \ell \}.
\end{align}
Consider the column aggregation of $A_{S_2 S_1}$ w.r.t. $\yh$, defined as $X = (X_{i\ell}) \in \reals_+^{|S_2| \times L}$, with
\begin{align}\label{eq:X:def}
	X_{i\ell}(\yh) = \sum_{j \in S_1} A_{ij} 1\{\yh_j = \ell \}.
\end{align}
Assuming that $\yh$ is deterministic, we have
\begin{align*}
	\ex[X_{i\ell}(\yh)] = \sum_{j \in S_1} B_{z_i z_j} \theta_i \theta_j 1\{\yh_j = \ell \} &=  \theta_i \sum_{k=1}^{K} B_{z_i k}  \sum_{j \in S_1}  \theta_j 1\{z_j = k, \yh_j = \ell \} \\
	&= |S_1| \, \theta_i (B R)_{z_i \ell}.
\end{align*}
Let $d_i = \sum_{j \in S_1} A_{ij}$ be the degree of node $i$ in $S_2$. Under the Poisson model, $(A_{ij}, j \in S_1)$ is a vector of independent Poisson cooridnates. It is well-known that such a vector has a multinomial distribution conditional on the sum of its entries. That is,
\begin{align}\label{eq:Xis:distn}
	X_{i *}(\yh) \mid d_i \sim \mult(d_i, \rho_{z_i *}),
\end{align}
where $\rho_{z_i*}$ denotes the $z_i$th row of $\rho = (\rho_{k\ell})\in [0,1]^{K \times L}$, defined as
\begin{align}\label{eq:rho:def}
	\rho_{k \ell} = \frac{(B R)_{k\ell}}{ \sum_{\ell'} (B R)_{k\ell'}}.
\end{align}
In other words, 
conditioned on the degree sequence $d=(d_i, i \in S_2)$, all the rows of $X$ corresponding to $z$-community $k$, have multinomial distributions with probability vector $\rho_{k *}$. This observation allows us to apply the \ac test developed in Section~\ref{sec:multi:group}, to test whether all the rows with $z_i = k$, have the same multinomial distribution. 

Now, consider two estimated label vectors $\zh = (\zh_i) \in [K]^{n}$ and $\yh = (\yh_i) \in [L]^{S_1}$. Let $\Ch_k = \{i \in [n]: \zh_i = k\}$, $\Gh_k = \Ch_k \cap S_2$ 
 and $\nt = |S_2|$. Consider the multi-group version of the AC statistic based on $\zh$ and $\yh$:
 \begin{align}\label{eq:ncac:Tn:def}
 	 \Th_n = \frac1{\sqrt{2}} \Big( \frac1{\gamma_{\nt}}
 	 \sum_{k=1}^K \sum_{i \in \Gh_k} \sum_{\ell = 1}^L 
 	 \psi\big( X_{i\ell}(\yh), d_i \rhoh_{k\ell}\big)
 	- \gamma_{\nt}\Big)
 \end{align}
 where $\gamma_{\nt} = \sqrt{\nt(L-1)}$ and
 \begin{align}\label{eq:rhoh:def}
 	\rhoh_{k\ell} = \frac{\sum_{i \in \Gh_k} X_{i\ell}(\yh)}{ \sum_{i \in \Gh_k} d_i}, \;\; k \in [K],\, \ell \in [L].
 \end{align}
 
The above construction specifies a family of test statistics, depending on the choices of label vectors $\zh$ and $\yh$, and subsets $S_1$ and $S_2$. We refer to this family, as the NAC family of tests. The acronym \nac stands for Network Adjusted Chi-square, since the test is the natural extension of the adjusted chi-square test, introduced earlier, to networks.

\subsection{Full version}\label{sec:FNAC}
We now single out two specifc members of the NAC family. Let $S_1 = S_2 = [n]$ and consider the following choices for $\zh$ and $\yh$:
\begin{enumerate}
	\item \textbf{\cac}: $\yh = \zh$ and $\zh$ is an estimated label vector with $K$ communities,
	\item \textbf{\cacp}: $\zh$ and $\yh$ are estimated label vectors with $K$ and $L = K+1$ communities.
\end{enumerate}
The acronym \cac stands for Full NAC, where ``full'' refers to the choice $S_1 = S_2 = [n]$. 
There are two main reasons for introducing the \cacp version with $L= K+1$ column communities. First, \cac only works when $K \ge 2$; when $K = L = 1$, \eqref{eq:Xis:distn} leads to a noninformative statistic for \cac, because, then, $X_{i*} = d_i$ almost surely, conditioned on $d_i$. \cacp on the other hand still produces an informative statistic when $K=1$.
Second, the choice $L = K+ 1$ makes \cacp especially powerful in determining the number of communities by sequential testing from below, as we discuss extensively in Section~\ref{sec:consist}.

\begin{rem} \label{rem:biclust}
The NAC family of tests are easily applicable to non-square and nonsymmetric adjacency matrices, 
with potentially unequal number of communities or clusters for the rows and columns. In particular, they can be used to test directed or bipartite DCSBMs or SBMs. In addition, they can be easily applied if the cluster structure of one side is known but not the other. For example, they can be used for model selection and goodness-of-fit testing in problems involving clustering and biclustering of  (Poisson) count arrays, a common task in contemporary bioinformatics~\cite{anders2010differential}. More specifically, the biclustering problem on a Poisson count array corresponds to having an array $A_{n \times m} = (A_{ij})$, where 
\[
    A_{ij} \sim \poi(B_{z_i y_j}), 
\]
independently across $i \in [n]$ and $j \in [m]$. Here $z =(z_i) \in [K]^n$ and $y = (y_j) \in [L]^m$ are the unknown clusters of rows and columns, respestively. The goal of biclustering is to recover estimates of $z$ and $y$, hence simultaneously clustering rows and columns of $A = (A_{ij})$, given only an instance of $A$. It is clear from Section~\ref{sec:nac:family}, that an \nac test with $K$ and $L$ matching the number of row and column communities, resepectively, 
is immediately applicable in this case. In this paper, we focus on the symmetric DCSBM for simplicity. All the results hold in the general nonsymmetric case as well, with suitable modifications.
\end{rem}

\subsection{Subsampled version}
\label{sec:sub:cac}
The asymptotic null distribution of the full version statistics, \cac and \cacp, can be complicated. There are two main obstacles in applying Theorem~\ref{thm:CCtest:null} to these statistics. First, although the theorem allows for the dependence of $\zh$ on the entire adjacency matrix $A$, as long as it converges to the true label vector $z$, it cannot directly handle the dependence of $\yh$ on the entire $A$. Because then, $X_{i*}(\yh)$ will be formed by summing elements of $A_{i*}$ (the $i$th row of $A$) over subsets of the columns that depend on $A_{i*}$ itself. This dependence between $\yh$ and $A$ is algorithm-specific, that is, itself depends on the particular community detection algorithm used, leading to an unknown 
deviation of the distribution 
of individual $X_{i\ell}(\yh)$ from a Piosson. Moreover, the joint dependence of $\yh$ and $A$ induces an algorithm-specific joint distribution on $(X_{i*}(\yh), i \in [n])$ which is hard to characterize for interesting algorithms such as spectral clustering.   These issues are resolved if we assume $\yh = z$ w.h.p., which holds if the algorithm is strongly consistent, but this can only happen for \cac; in the case of \cacp, we always estimate with one more community relative to the truth, and the breaking of at least one true community causes an unknown skewness in the distribution of the resulting partitions; imagine bisecting an Erd\"{o}s-R\'{e}nyi (ER) network, resulting in two subnetworks that are more clustered  than a typical ER network.


The second obstacle is the symmetry of $A$ which makes $X_{i*}(\yh)$ and $X_{j*}(\yh)$ (mildly) dependent through the shared element $A_{ij} = A_{ji}$, even when $\yh = z$, and hence applies to both \cac and \cacp.

To circumvent the above obstacles, we introduce a particular subsampling scheme which provides several advantages. It takes care of the dependence issues, making the results independent of the community detection algorithm used. It also allows us to state unified results that apply regardless of the choice of $L$, hence the same results will be applicable to both \scac and \scacp.  Moreover, as we will show, by using the scheme, we avoid the assumption $\yh =z$. In fact, we no longer even need $\yh$ to be consistent for $z$ for the results to go through. Finally, it allows us to implement a further degree filtering step which potentially improves the growth rate of the harmonic mean of the remaining degrees, $h(d)$, making the assumption $h(d) \to \infty$ easier to satisfy in practice.

The scheme is detailed in Algorithm~\ref{alg:snac}. It involves a sampling step so that: a) $\yh$ no longer depends on the entries of $A$ needed to be summed; b) the symmetry is broken. It also has a filtering step to leave out nodes with small degrees, so that $h(d)$ is large.

\begin{algorithm}[t]
    \linespread{1.25}\selectfont 
    \begin{algorithmic}[1]
      
        \Require  Adjacency matrix $A$, number of row clusters $K$, number of column clusters $L \in \{K, K+1\}$, degree-filtering threshold $\sigma \in [0,1)$. Critical threshold $\tau > 0$.
        
        \Ensure Test statistic $\Th_n$ and whether null is rejected.
        
        \smallskip
        \State Fit $K$ clusters to the whole network 
        to get labels $\zh \in [K]^n$ and clusters $\Ch_k = \{i: \zh_i = k\}$.
        
        \State (\textbf{Sampling}) Choose a subset $S_1 \subset [n]$ by including each index $i \in [n]$, independently, with probability $1/2$. Let $S_2 = [n] \setminus S_1$ be the complement of $S_1$.
        
        \State Fit $L$ clusters to $A_{S_1 S_1} = (A_{ij}: i,j \in S_1)$, to learn the label vector $\yh$ on $S_1$.\label{step:yh}
        
        \State Form partial degrees $d_i := \sum_{j \in S_1} A_{ij}$ for all $i \in S_2$.
        
        \State (\textbf{Quantile filtering}) Within each $\Gh_k = \Ch_k \cap S_2$, keep nodes with $d_i$ at least the $\sigma$-th quantile of all $d_i$ in $\Gh_k$ to form $\Gh'_k$. Let $S'_2 = \bigcup_{k=1}^K \Gh'_k$. \label{step:filt}
        
	    \State Perform the test on $A_{S'_2 S_1^{}}$ using row labels $\zh_{S'_2}$ and column labels $\yh$ from Step~\ref{step:yh} to form $\Th_n$ as in~\eqref{eq:ncac:Tn:def} and reject the null if $\Th_n > \tau$.
    \end{algorithmic}
    \caption{\scacf}
    \label{alg:snac}
\end{algorithm}


We refer to Algorithm~\ref{alg:snac} as  subsampled NAC, or \scac for short, when $L = K$ and as \scacp when $L = K+1$.
 Note that step 5, the quantile filtering, can be skipped if the degrees are mostly large or we do not insist that the normal approximation to the null distribution hold. In the latter case, we can use the bootstrap debiasing of Section~\ref{sec:boot} to determine the critical region. In such cases, we set $S'_2 = S_2$ (equivalently $\sigma = 0$) and perform the test on $A_{S_2 S_1}$.

 \begin{rem}[On notation]\label{rem:notation:plus}
	In the sequel, we often state results that apply to either of \scac or \scacp. We will use the notation \textbf{\scacf} to mean the statement holds for either version. Similarly \textbf{\cacf} refers to either of \cac or \cacp.
\end{rem}

In Section~\ref{sec:analysis}, we show that, under the null model, the distributions of the test statistics of \scacf are close to a standard normal. Furthermore, we show that they are large when the model is underfitted, i.e., the presumed number of communities is smaller than that of the true model, with \scacp often being much larger than \scac. 
We also show that under DCLVM, a latent variable network block model, \scacf values are large. Such properties allow us to use \scacp for assessing the goodness-of-fit of DCSBM or SBM to an observed network and to determine the number of clusters in community detection.

\subsection{Bootstrap debiasing}
\label{sec:boot}
Per our discussion above, without subsampling, the full version statistics, \cacf, do not have a standard normal null distribution in general. However, they are expected 
to produce more powerful tests since
they utilize 
all the information in the network. As a result, they are great choices in practice if we can approximate their null distribution. 
The remedy is to use bootstrap simulation to determine their critical regions. 
In addition, bootstrap can correct deviations of the null distribution of~\scacf from the standard normal when some of the underlying assumptions fail to hold; see Remark~\ref{rem:bern}.

 Given adjacency matrix $A$, the null hypothesis that the number of communities is $K$, and the test statistic $\hat{T} = \hat{T}(A)$, the bootstrap debiasing is performed as follows:
\begin{enumerate}
	\item Fit a $K$-community SBM to $A$ and get label estimates $\hat{z}$ and connectivity matrix $\hat{B}$.
	\item For $j=1,\dots,J$, sample $A^{(j)} \sim \text{SBM}(\hat{z}, \hat{B})$ and evaluate the test statistic $\hat{T}^{(j)}$ based on $A^{(j)}$.
	\item Construct the debiased statistic $\hat{T}^{(\text{boot})} = (\hat{T} - \hat{\mu})/\hat{\sigma}$ where $\hat{\mu}$ and $\hat{\sigma}$ are the sample mean and the standard deviation of $\{\hat{T}^{(j)}\}_{j = 1}^J$.
\end{enumerate}
Note that we sample from SBM instead of DCSBM. To simulate from DCSBM, one has to estimate $(\theta_i)_{i=1}^n$, which cannot be done consistently, and whose estimates are highly variable. As a result, generating from an estimated DCSBM adds extra variance and produces samples that are actually further from the original network than those produced by the SBM fit. We also note that the distribution of our statistics are invariant to degrees, making SBM generation further justified. 

The test rejects for large values of $\hat{T}^{(\text{boot})} $ (or $|\hat{T}^{(\text{boot})}|$), with the threshold set, assuming that $\hat{T}^{(\text{boot})}$ has (approximately) a standard normal distribution under null. A similar idea is used in~\cite{Lei2016} for the spectral test. An alternative to debiasing is to use the empirical quantiles of $\{T^{(j}\}$ to set the critical threshold. We, however, found that the debiasing approach performs better in practice. See Appendix \ref{sec:boot:comp} for a detailed discussion and comparison of all the bootstrap methods in a simulation setting.

\subsection{Model selection}
A goodness-of-fit test can also be used as a model selection method, through a process of sequential testing. In particular, we can use \cacf (with bootstrap debiasing) and \scacf statistics to determine the number of communities when fitting DCSBM models.

The idea is to test the null hypothesis of $K$ communities, starting with $K =K_{\min}$, which is usually taken to be 1, and increasing $K$ to $K +1$ if the null is rejected. The process is repeated until we can no longer reject the null or a preset maximum number of communities, $K_{\max}$, is reached. The value of $K$ on which we stop is selected as the optimal number of communities. We refer to this procedure as \emph{sequential testing from below}. There is also the possibility of starting at $K=K_{\max}$ and working backwards. Testing from below is, however, more advantageous, especially if one expects a small number of communities a priori.

The rejection thresholds for  
\scacf can be determined based on the standard normal distribution. For 
\cacf, we need to apply the bootstrap debiasing of Section~\ref{sec:boot} before comparing the statistic with the threshold. Theorem~\ref{thm:consist} provides a theoretical guarantee for the consistency of the sequential testing from below, when 
\scacf is used. An empirical comparison of the model selection performance of this approach, with existing methods, is provided in Section~\ref{sec:model_select}.

\section{Null distribution}\label{sec:analysis}
We now derive the null distribution of~\scacf. We consider a DCSBM with $K_0$ true community, and the edge probability matrix $B = (\nu_n / n) B^0$ where $\nu_n$ is a scaling factor and $B^0$ satisfies
\begin{align}\label{assum:B}
	\min_{k,\ell} B_{k\ell}^0 \ge \tau_B \cdot \max_{k,\ell} B_{k\ell}^0.
\end{align}
Let $\Cc_k = \{i \in [n] : z_i = k\}$ be the true community $k$. We assume that
\begin{align}\label{assum:nk:theta}
	n_k := |\Cc_k| \ge \tau_\Cc \,n, \quad \theta_i \ge \tau_\theta \cdot \max_i \theta_i
\end{align}
for all $k \in [K_0]$ and $i \in [n]$. Here, $\tau_B, \tau_\Cc$ and $\tau_\theta$ are in $(0,1]$ and measure the deviation of the corresponding parameters from being balanced. 
To make $\nu_n$ identifiable, we further assume without loss of generality that $\infnorm{B^0} := \max_{k,\ell} B_{k\ell}^0 = 1$ and $\infnorm{\theta} := \max_i \theta_i = 1$. We require the following on the community detection algorithm:
\begin{assum}\label{assum:DCSBM}
    The community detection algorithm %
    applied with $K$ communities to the 
     DCSBM described above, producing labels $\{\zh_i\}$, satisfies:
    \begin{enumerate}[label=(\alph*)]
        \item \label{assume:weak:consist} Weak consistency: When $K = K_0$, there is a sequence $\alphan = o(1)$ such that $\pr \bigl(\miss(\zh, z) \le  \alphan \bigr) = 1- o(1)$.
        \item \label{assume:stability} Stability:
        For 
        $K \in [K_0+1]$,
        we have $|\{i: \zh_i= k\}| \ge \tau_0 n$ for all $k \in [K]$.
    \end{enumerate}
\end{assum}

Assumption~\ref{assum:DCSBM}\ref{assume:weak:consist}, known as the weak consistency or partial recovery, allows us to focus on the event where $\zh$ is close to $z$, the true label vector. As long as $\nu_n \to \infty$, there are algorithms that can achieve this~\cite{Abbe2018}. We, in fact, need $\alpha_n$ in Assumption~\ref{assum:DCSBM}\ref{assume:weak:consist} to go down faster than $o(1)$, but still much slower than what is needed for exact recovery (or strong consistency); see the discussion after Theorem~\ref{thm:null:dist:ncac}.  The growth rate of $\nu_n$ is roughly that of the  expected average degree (EAD) of the network, assuming that $B^0$, $\{n_k / n\}_k$ and the distribution of $\{\theta_i\}$ are roughly constant.


Assumption~\ref{assum:DCSBM}\ref{assume:stability} is even milder, and ensures that the algorithm does not produce extremely small communities when applied with $K \neq K_0$. It can be guaranteed by explicitly enforcing it in the algorithm: If the size of a recovered community is too small relative to $n$, we merge it with another community. Whether a specific community detection algorithm satisfies this condition automatically without explicit enforcement is an interesting research question. 

Recall $\sigma$, the threshold in step~\ref{step:filt} of Algorithm~\ref{alg:snac}, and let $\sigb := 1-\sigma$.
To state further assumptions, we define the following constants:
\begin{align}
	\label{eq:const:defs}
	c_1 &:=  
	\frac{\sigb \tau_\Cc}{5K_0},\quad  C_1:= \tau_\theta^2 \tau_\Cc \min_h \norm{B^0_{h*}}_1,
	\\
	\tau_a &:= \tau_\theta \tau_B \tau_\Cc, \quad 
	 \tau_\rho := \tau_\theta \tau_B\tau_{0}. 
	 \label{eq:c1:taua:taurho:def}
\end{align}
where $\tau_0$ is the constant in Assumption~\ref{assum:DCSBM}\ref{assume:stability}.
Let $\beta_n = \log [ (3/4) K_0^2\nu_n]$. We make the following assumptions:
\begin{align}
	\frac{\log n}{n} &\le 
	\frac{C_1}{300},
	\quad L \ge 2,
	\label{assum:scaling:a} \\
	\nu_n &\ge \frac{1}{C_1} \max\Bigl\{
		2\sqrt{2}C_2L, \
		10^3 \log n, \
		\frac{154}{\taur^2 c_1 K_0}
		\frac{\beta_n}{n}\Bigr\}, 
		\label{assum:scaling:b} \\
	\alphan &\le \min\Bigl\{
		 \frac{2}{LC_2^2},\  \frac{\tau_\Cc}{5} \frac{1-\sigma}{1+\sigma} \Bigr\},
		 \label{assum:scaling:c}
\end{align}
where $C_1$ is as defined in~\eqref{eq:const:defs} and $C_2 = 11/(c_1C_1\tau_\rho)$. 

\begin{thm}[Null distribution]\label{thm:null:dist:ncac}
	Consider an $n\times n$ adjacency matrix $A$ that is generated from a Poisson DCSBM with $K_0$ blocks, satisfying~\eqref{assum:B} and~\eqref{assum:nk:theta}.
    Let $\zh \in [K_0]^n$ be an estimated label vector based on $A$ 
	and  $\yh \in [L]^{|S_1|}$ an estimated label vector based on $A_{S_1S_1}$ satisfying Assumption~\ref{assum:DCSBM}\ref{assume:stability}. Let $\Th_{n}$ be the test statistic of \scacf. Assume that~\eqref{assum:scaling:a}--\eqref{assum:scaling:c} hold. Then,
\begin{align}
\begin{split}
    \dkol(\Th_n, Z) &\le  \frac{C_3}{\sqrt{\sigb L n}} + \frac{ C_4 }{C_1 \nu_n}\\
    &+ \frac{19 \sqrt{L}}{\tau_\rho}\left(\frac{1}{\sqrt{c_1C_1}} \sqrt{\frac{\beta_n }{K_0 \nu_n}} + \frac{K_0 \beta_n}{\sqrt{\sigb n}} + C_2\frac{K_0^{3/2}}{\sigb L}\nu_n \sqrt{n}\, \alphan\right) \\
    &+ 3\,\pr(\miss(\zh, z) > \alphan),
\end{split}
\end{align}
    where $C_3 = 94\taur^{-4}$
	and $C_4 = 4 (\pi e)^{-1/2} \max\{1,\tau_\rho^{-1} - L -1\}$.

   
\end{thm}
The bound in Theorem~\ref{thm:null:dist:ncac} applies to both \scac and \scacp. Assuming the common scaling $\log n \lesssim \nu_n \lesssim \sqrt{n}$ and $\alpha_n = o(1)$, the conditions on $\nu_n$ and $\alpha_n$ are satisfied as $n \to \infty$ 
and the bound simplifies to
\begin{align*}
	\dkol(\Th_n, Z) \lesssim \sqrt{\frac{\log \nu_n}{\nu_n}} +  \nu_n\sqrt{n}\, \alphan + \pr(\miss(\zh, z) > \alphan). 
\end{align*}
To have a null distribution close to the standard normal, we need to have 
\begin{align}\label{eq:req:on:alphan}
    \alphan = o((\nu_n \sqrt{n})^{-1}) 
    \quad \text{with} \quad \pr(\miss(\zh, z) > \alphan) = o(1).
\end{align}
There are community detection algorithms that can achieve this as long as $\nu_n \gtrsim \log n$~\cite{qin2013regularized, lei2015, chen2018convexified}. 
In fact, if $\nu_n \ge C \log n$ for a sufficiently large constant $C$, there are algorithms that achieve exact recovery, that is, we can take $\alpha_n = 0$ and still have $\pr(\miss(\zh, z) > \alphan) = o(1)$. It is also possible to satisfy~\eqref{eq:req:on:alphan} below the $\log n$ threshold on $\nu_n$---see for example~\cite{zhang2016minimax, gao2017achieving, zhou2020optimal}. However, for the distribution to converge we still need $\nu_n \gtrsim \log n$ from~\eqref{assum:scaling:b}. This is needed to to guarantee the concentration of degrees $d_i$ uniformly over all nodes $i \in S'_2$. Whether this requirement can be lifted and still achieve convergence in distribution is open.

\begin{rem}[Bernoulli vs. Poisson]\label{rem:bern}
Theorem~\ref{thm:null:dist:ncac} assumes Poisson generation for the DCSBM, and it is not clear if the result holds under the Bernoulli version. The main challenge is the conditional distribution of $X_{i*}(\yh)$ which is no longer a multinomial---that is,  \eqref{eq:Xis:distn} no longer holds---under the Bernoulli model.
To prove Theorem~\ref{thm:CCtest:null}, we use the Esseen's bound and control the moments of the conditional distribution of $X_{i*}(\yh)$. 
Under the Bernoulli model, these moments 
do not have a closed form~\cite{chen2000general} and are also hard to approximate. Another approach is to show that the conditional distribution is close to a multinomial. For example, using results in~\cite{loh1992stein}, one can show that, for any $i$, the Kolmogorov distance between the distribution of $X_{i*}(\yh)$ and a multinomial is of the order $\frac{\nu_n^2}n$, which goes to zero fast under the typical sparse scaling of $\nu_n \sim \log n$. However, since \scacf are roughly sums of $n$ chi-square statistics divided by $\sqrt{n}$, the small distances of their individual terms to the desired distribution may not carry over to the distribution of their sum. In general, it is not clear if the Kolmogorov distance for sums of this form can be controlled based solely on the distances of their individual terms.
Despite the above theoretical challenges, the null distribution under the Bernoulli setting is close enough to a standard normal in practice to make these results useful, especially if the bootstrap debiasing is also applied.
As we show in the simulations, 
which are all based on Bernoulli DCSBM, \scacp can consistently select the correct number of communities when applied sequentially, and the performances are similar with or without bootstrap debaising.

\end{rem}

\section{Consistency}\label{sec:cons}

We show the consistency of \scacf 
against alternative models by deriving lower bounds on the statistic that go to infinity, under the alternatives, as $n \to \infty$. We consider two alternative models: 1) DCSBM with the number of communities less than that of the null; 2) DCLVM, a general class of degree-corrected latent variable models discussed in more details in Section~\ref{sec:consist:DCLVM}. Combined with the null distribution in Theorem~\ref{thm:null:dist:ncac}, the first case above shows that \scacf can be applied in sequential testing from below to determine the number of communities consistently. In addition, its power against DCLVM shows its utility as a very general goodness-of-fit test beyond the DCSBM family.

\subsection{Consistency against underfitted DCSBM}\label{sec:consist}
We analyze the power of \scacf in distinguishing the null hypothesis $H_0: K = K_0$ from the alternative $H_1: K < K_0$. 
Theorem~\ref{thm:consist} provides a lower bound on the growth rate of the test statistic $\Th_n$ under the alternative. Recall that $\yh$ are labels derived for nodes $S_1$ based on $A_{S_1 S_1}$. Let parameters $\rho_{k\ell}$ be defined as in~\eqref{eq:rho:def},
and let
\begin{align}
	\omega_2 &= \frac1{18} \tau_\theta^2 \tau_a^2 c_1^2 \min_{k,h \in [K_0]:\;  k \neq h} \frac1{L}\norm{\rho_{k*} - \rho_{h*}}_2^2.
	\label{eq:omega2:def}
\end{align}
See~\eqref{eq:const:defs} and~\eqref{eq:c1:taua:taurho:def} for the definitions of $c_1$ and $\tau_a$.

\begin{thm}
\label{thm:consist}
	Let $A$ be an $n\times n$ adjacency matrix
    generated from a Poisson DCSBM with $K_0 \ge 2$ blocks %
    that satisfies~\eqref{assum:B} and~\eqref{assum:nk:theta}. Let $\Th_n$ be the 
    \scacf
    test statistic~\eqref{eq:ncac:Tn:def} formed as detailed in Algorithm~\ref{alg:snac},
    with $K < K_0$, 
    estimated by a community detection algorithm satisfying stability Assumption~\ref{assum:DCSBM}\ref{assume:stability}.
    Let $C_5 := c_1 C_1/9$, assume that $(\log n) / \nu_n \le C_1\tau_\rho^2/64$ and consider the event
    \begin{align}
         \Omega_n := \left\{ \max \left(\frac1{C_5 \nu_n}, \;\frac{768}{\tau_\rho^3}\sqrt{\frac{\log n}{C_1\nu_n}}\, \right) \le \omega_2 \right\}. 
    \end{align}
    Then, with probability at least $ 1- 9Ln^{-1} - \pr(\Omega_n^c) - \pr(\miss(\zh, z) > \alphan)$,
    \begin{align*}
    	 \Th_n \ge 	C_5\, \omega_2 \,\nu_n \sqrt{L n}.
    \end{align*}
\end{thm}

Quantity $\omega_2$ that appears in Theorem~\ref{thm:consist} is random (via $\{\rho_{k\ell}\}$), due to the randomness in $\yh$, and depends on the specific community detection algorithm used to form the test statistic. As discussed below, for any reasonable algorithm, under mild conditions on the connectivity matrix, we expect $\omega_2$ to be of constant order as $n \to \infty$, i.e., $\omega_2 \asymp 1$. In particular, we expect to have $\pr(\omega_2 \ge c_0) \to 1$ for some constant $c_0 > 0$,  as $n\to \infty$. Then, we have $\pr(\Omega_n^c) \to 0$, as long as $(\log n) / \nu_n \le c_0$.

Under these assumptions, Theorem~\ref{thm:null:dist:ncac} shows that for a given significance level $\alpha > 0$, \scacf statistic $\Th_n \asymp 1$ with probability approaching $1-\alpha$ when $K=K_0$, while Theorem~\ref{thm:consist} guarantees that $\Th_n \gtrsim \nu_n \sqrt{n}$, w.h.p., when $K < K_0$. This shows that  \scacf with a constant threshold or one that grows slower than $\nu_n \sqrt{n}$, leads to consistent model selection when applied sequentially from below (i.e., with $K < K_0$). In short, model selection consistency of 
\scacf only requires two assumptions: (a) $(\log n) / \nu_n = O(1)$, that is, the expected degree should grow no slower than $\log n$, and (b) $\omega_2$ should remain bounded below in probability.

In addition to consistency, Theorem~\ref{thm:consist} suggests that \scacp 
is more powerful than \scac in sequential testing from below,
due to using $L = K+1$ clusters for column compression. The difference between the two algorithms is manifested in their corresponding values of $\omega_2$.
Let us consider the hardest case in Theorem~\ref{thm:consist}, that is, testing the null hypothesis $K = K_0-1$ against the alternative $K = K_0$.
To simplify the discussion, assume that $\nu_n \gtrsim \log n$ and the community detection algorithm is strongly consistent (achieves exact recovery). First, consider the \scacp. Since $L=K+1 = K_0$ in this case, the estimated column labels $\yh$ match the true labels $z$ when computing the  \scacp statistic.
Recalling the definition of the confusion matrix from~\eqref{eq:R:def}, we 
obtain $R = \diag (\pit_k)$, where $\pit_k = \frac{1}{|S_1|}\sum_{j \in S_1} \theta_j 1\{z_j = k\}$ for all $k \in [K_0]$. Then, $\rho_{k \ell} = B^0_{k\ell} \pit_\ell / (\sum_{\ell'} B^0_{k\ell'} \pit_{\ell'})$. Note that both $B^0$ and $\{\pit_k\}$ are stable as $n \to \infty$. In particular, although the entries of $B$ vanish under the scaling $\nu_n / n \to 0$, the entries of $(\rho_{k\ell})$ do not. To guarantee that $\omega_2 > 0$, it is enough that the $K_0 \times K_0$ matrix $(B^0_{k\ell} \pit_\ell)$ has no two colinear rows, a mild identifiability condition.

On the other hand, for \scac we have $L = K_0-1$, causing the multinomial parameter matrix $\rho \in \mathbb{R}^{K_0\times(K_0-1)}$ to have rows 
that are weighted averages of its counterpart when $L = K_0$. We refer to~\cite{wang2017} for an example of how the weighted mixture of the rows of the connectivity matrix $B$ emerges in the underfitted case, and 
$\rho$ is mixed in the 
same way. Due to this averaging, the pairwise distances among the rows of $\rho$ will be smaller compared to when $L = K_0$ and thus $\omega_2$ will smaller for \scac,  suggesting a lower power relative to \scacp.

The $\rho$-mixtures in the case of \scac still lead to an $\omega_2$ that is bounded away from zero---hence preserving consistency---provided that the mixture weights do not converge to specific values that make the rows of $\rho$ identical. This implausible situation, however, can occur in some corner cases. Consider the extreme case of the SBM with a planted partition pattern for $B$ (equal to $p$ on the diagonal and $q$ off the diagonal) and equal community sizes. If the community detection algorithm 
recovers a superset of the true communities when underfiting, as shown, for example, for the spectral clustering in~\cite{ma2018}, $\rho$ will have identical rows in the limit and thus $\omega_2\to 0$
as $n\to \infty$,
making \scac powerless. More details on this example are included in Appendix~\ref{sec:comp:plus}. 

\begin{rem}\label{rem:overfit}
In sequential testing, one may want to know the growth rate of the test statistic $\Th_n$ in the overfitted case where $K > K_0$. The same argument as in Theorem~\ref{thm:null:dist:ncac} shows that under $K > K_0$, if the community detection algorithm is \emph{refinement consistent}---that is, recovers a refinement of the true clusters---then, $\Th_n$ has asymptotically a standard normal distribution, hence $\Th_n \sim 1$ as $n \to \infty$.  Some algorithms, such as spectral clustering, exhibit refinement consistency in practice; for an example see Appendix~\ref{sec:comp:plus}. Recent theoretical discussions of the phenomenon appear in~\cite{ma2018, zhang2021label}. 
\end{rem}

\subsection{Consistency against DCLVM}\label{sec:consist:DCLVM}
We consider a $K$-community DCLVM, with degree parameter $\theta$, label vector $z^* \in [K^*]^n$, mixture components $\{\Qb_k^*\}_{k=1}^K$ and latent variables $\{x_i\}_{i=1}^n \subset \mathcal X$, to be a network model specified as follows: Given $\{x_i\}$, each $(i,j)$ is drawn independently (of other edges) from a  Poisson distribution with mean 
\begin{align*}
	p_{ij} &:= \ex[\,A_{ij} \given x_i,\, x_j] = \frac{\nu_n}{n} \theta_i \theta_j g(x_i, x_j)
\end{align*}
and $x_i \sim \Qb^*_{z^*_i}$ independently across $i$. The mixture components $\{\Qb_k^*\}$ are distributions on the space $\mathcal X$, and when they are different they impose some latent community structure. An example, with specific forms for $g(\cdot,\cdot)$ and $\{\Qb_k\}$ is given in Section~\ref{sec:roc}. Here, we consider the general case, with minimal assumptions on $g(\cdot,\cdot)$ and $\{\Qb_k\}$.  We use similar assumptions on $\theta$ as in the DCSBM, namely,
\begin{align*}
 \max_{i}\theta_{i} = 1, \quad \theta_i \ge \tau_\theta, \; \forall i \in [n].
\end{align*}
By rescaling $\nu_n$ if need be,  we assume that $g$ has range $[0,1]$. 

Without strong assumptions on $\{\Qb_k^*\}$, the distribution of $x_i$ is a nonparametric mixture model which, in general, is not identifiable. One can shift mass from one of $\{\Qb_k^*\}$ to the other ones or create a new component, and redefine the label vector to get the same distribution. For example, suppose that we start with a two-community model with components $\Qb^*_1$ and $\Qb^*_2$. We relabel each $x_i$ by assigning it the new label $z_i \in [K]$ (rather than $z^*_i$). 
The same model for $x_i$ can be stated as $x_i \sim \Qb_{z_i}$ for new mixture components $\Qb_k= \pi_{k1} \Qb^*_1 + \pi_{k2} \Qb^*_2$ which are convex combinations of the original ones. We refer to $\{\Qb_k\}$ as the mixture components induced by $z$. The result that we present here applies to any of these parameterizations.




Assume that we perform the \scacf with $\Kn$ row communities and $L$ column communities.
Let $\zh \in [\Kn]^n$ be the estimated label vector based on the entire adjacency matrix $A$ and $\yh \in [L]^{|S_1|}$ that based on $A_{S_1 S_1}$. We assume that there are deterministic vectors $z \in [\Kn]^n$ and $y \in [L]^n$, and sequences $\{\alphan\}$ and $\{\kappa_n\}$ such that the following event 
\begin{align}\label{eq:M:event}
    \Mc_n := \{\miss(\zh, z) \le  \alphan \ \text{and} \ \miss(\yh,y_{S_1})  \le \kappa_n\},
\end{align}
has probability converging to 1, as $n \to \infty$. Here, $y_{S_1} = (y_i: i \in S_1)$ is the subvector of $y$ on $S_1$. Note that we do not require $z$ (or $y$) to be the original $z^*$.
Letting $n_k = |\{i: z_i = k \}|$, 
we assume
\begin{align}
    n_k \ge \tau_\Cc n, \quad 
    \forall k \in [\Kn].
\end{align}

Let $\{\Qb_k, k \in [\Kn]\}$ be the mixture components induced by label vector $z$ that appears in~\eqref{eq:M:event}. Define
\[
h_k(x) := \ex[g(x,\xi)], \quad \xi \sim \Qb_k, \ k \in [\Kn].
\]
We assume that there is an almost sure event $\Gamma$ with the following property: There exists a constant $\tau_h > 0$ and $r_1, r_2\dots,r_{\Kn} \in [K]$ such that on $\Gamma$, 
\begin{align}\label{eq:r:k:def}
    \forall k \in [\Kn], \; \forall i \in \Cc_k, \quad h_{r_k}(x_i) \ge \tau_h.
\end{align}
Note that~\eqref{eq:r:k:def} can be equivalently stated as $h_{r_{z_i}}(x_i) \ge \tau_h$ for all $i$. Condition~\eqref{eq:r:k:def} is mild and is satisfied 
if for any $k \in[\Kn]$, one of $h_r(\cdot), r \in [K]$ is uniformly bounded below over the support of $\Qb_k$. 
We also define
\begin{align}\label{eq:H:def}
    H_{\ell}(x) := \frac{\sum_k h_k(x) R_{k\ell}}{\sum_{\ell'} \sum_k h_k(x) R_{k\ell'} }, \quad 
    R_{k\ell} := \frac12 \sum_{j = 1}^n\theta_j 1\{z_j = k, \, y_j =\ell\}.
\end{align}
Noe that there exists a sequence $\{\ell_k\}_{k=1}^K$ such that 
\begin{align}\label{eq:ell:k:def}
    R_{k \ell_k} \ge \frac1L \sum_{\ell = 1}^L R_{k \ell}, \quad \forall k \in [K].
\end{align}
Fix one such sequence and consider the following quantities:
\begin{align}
\begin{split}\label{eq:fvar:def}
    \fvar_{k\ell} &:= \var(H_\ell(x)), \quad \text{where}\;x \sim \Qb_k, \\
    \fvarmin &:= \min_{k} \fvar_{k\ell_k}.
\end{split}
\end{align}
Let $\zeta_n = \max\{1, \, L \sqrt{\nu_n / n}, \, L /\sqrt{\nu_n \log n}\}$ and $c_2 = \tauc \tauh \taut^2/100$ and $\taur = \tauc \tauh \taut/(2L)$. We need the following assumptions:
\begin{align}
   \sqrt{\frac{\log n}{n}} &\le 
    \frac{2}{9}  \frac{\taur^2}{K},
    \quad  n \ge 2,
     \quad \label{assume:final:logn:over:n} \\
    \alpha_n &\le \sqrt{\frac{\log n}{\nu_n}} \le 
    \frac{21 \tauc^2 \tauh c_2^2}{L^2}, \quad 
     \frac{ n \kappa_n} {\nu_n} \le 
     4 c_2 \taur,
     \label{assume:mis:rate:bound} \\
    \fvarmin &\ge \frac{L^3}{c_2^3 \taur^3 }\max \left\{ \frac{2 \zeta_n }{\taur \tau_\Cc } \sqrt{\frac{\log n}{\nu_n}}, \ 
    \frac{1}{5 c_2} \frac{n \kappa_n}{\nu_n}
    \right\}. 
    \label{assume:final:fvarmin:bound}
\end{align}

\begin{thm}
\label{thm:consist:others}
	Let $A$ be an $n\times n$ adjacency matrix generated from a Poisson DCLVM with $K$ blocks that satisfies~\eqref{assum:B} and~\eqref{assum:nk:theta}. Let $\Th_n$ be the \scacf statistic~\eqref{eq:ncac:Tn:def} formed as detailed in Algorithm~\ref{alg:snac}. 
	Moreover, assume~\eqref{eq:r:k:def} and~\eqref{assume:final:logn:over:n}--\eqref{assume:final:fvarmin:bound}.
    Then, with probability at least $1- 12KLn^{-1} - Kn^{-c}- \pr(\Mc_n^c)$,
    \begin{align*}
    	 \Th_n \ge \frac{49 \, c_2^3}{\sqrt{L}} \, \fvarmin\sqrt{n}\nu_n,
    \end{align*}
    where $c > 0$ is a universal constant.
\end{thm}

The theorem roughly states the following: As long as the community detection algorithm produces  row and columns labels that converge to some deterministic labels $z$ and $y$ at the rates $\alpha_n \sim \sqrt{(\log n) / \nu_n}$ and $\kappa_n \sim \nu_n / n$ respectively, and the resulting induced mixture components $\{\Qb_k\}$ lead to a positive minimum variance $\fvarmin$, as defined in~\eqref{eq:fvar:def}, then \scacf are consistent in rejecting the underlying DCLVM model, with $\Th_n \gtrsim \sqrt{n} \nu_n \to \infty$.    
Note that  $\fvarmin > 0$, unless there exists a sequence of constants $a_1,\dots,a_K$ such that $\sum_{r} a_r h_r(x) = 0$ for $\Qb_k$-almost all $x$. That is, unless $\{h_r\}_{r=1}^K$ satisfy a non-trivial linear constraint under $\Qb_k$, the condition $\fvarmin > 0$ is guaranteed. An example where the condition $\fvarmin > 0$ is violated is when all $h_r(\cdot)$ are constant functions, as is the case for a DCSBM, consistent with the fact that we should not be able to reject a DCSBM.

\begin{rem}
The constant $\frac12$ in the defintion of $R_{k\ell}$ in~\eqref{eq:H:def} is for the convenience in the proof. It can be changed to any other prefactor (including $\frac1n$) since $H_\ell(x)$ is invariant to a rescaling of $R_{k\ell}$.
\end{rem}

\begin{rem}
One identifiable example of DCLVM is when $\{\Qb_k^*\}$ are Gaussian and their means are far apart. Gao et al.~\cite{gao2020community} consider a variant of such DCLVM  and show that consistent detection of communities is possible in the sparse setting.
Thus, under this setting, 
we can take the deterministic labels $z$ and $y$ in~\eqref{eq:M:event} to be equal to the underlying generating labels $z^*$, when applying \scac, and the missclassification rate requirements in~\eqref{assum:scaling:b} are satisfied. 
Theorem~\ref{thm:consist:others} then applies,
showing that \scac statistic is large and rejects the null hypothesis of  DCSBM generation.
\end{rem}

\begin{rem}
    The only unspecified constant in Theorem~\ref{thm:consist:others} is $c$ in $Kn^{-c}$ in the probability bound. This constant is related to the universal constants in the Hansen--Wright inequality for sub-Gaussian variables and can be specified, if one chooses the constants in that inequality. See the proof of Lemma~\ref{lem:varpi} in the Supplement for details.
\end{rem}

\subsection{Comparison with the existing literature}\label{sec:comparison}

The closest work in the literature to ours is the spectral goodness-of-fit test for SBMs~\cite{bickel2016, Lei2016}. Roughly speaking, Lei \cite{Lei2016} shows that, under a $K$-SBM, $n^{2/3}(\sigma_1(\At) - 2)$ has a \mbox{type-1} Tracy-Widom distribution asymptotically, where  $\sigma_1(\cdot)$ denotes the largest singular value, and $\At$ is a standardized version of the adjacency matrix, calculated based on fitting a $K$-SBM  (see Section~\ref{sec:methods}). This result requires the entries of the connectivity matrix $B$ to be bounded away from zero which excludes the sparse regime $\nu_n / n \to 0$ we consider here. Moreover, Lei's Theorem~3.3 provides an asymptotic power guarantee. Translating the results to our notation, assuming that the true model has more communities than the fitted model, the result shows that $n^{2/3} \sigma_1(\At) \gtrsim \nu_n n^{1/6}$ w.h.p. Since under the true model $n^{2/3}\sigma_1(\At) \approx 2 n^{2/3}$, one obtains a consistent test as long as $\nu_n n^{1/6} \gg n^{2/3}$, that is, $\nu_n \gg n^{1/2}$. This required scaling is in fact better than what is stated in~\cite{Lei2016}. Nevertheless, $\nu_n \gg n^{1/2}$ is far from the sparse regime $\nu_n \asymp \log n$ that our results allow. More importantly, it is not clear how to extend the spectral test to the degree-corrected setting. In Section~\ref{sec:methods}, we discuss the natural extension of the spectral test to the DCSBM and study its performance empirically. Due to the difficulty of estimating the $\theta$ parameter of DCSBM, theoretical guarantees for this (naive) extension are not easy to obtain. Our \scacp test avoids explicitly estimating $\theta$, by conditioning on the degrees which leads to the cancellation of individual $\theta_i$ in the resulting multinomial distributions. In practice, convergence to the Tracy-Widom distributions is known to be slow, whereas convergence to the normal distribution for \scacp happens quite fast (at a rate at most $\approx \nu_n^{-1/2}$ as we showed).

Another work with connections to ours is that of Karwa~et~al.~\cite{karwa2016} where chi-square statistics for the goodness-of-fit testing of SBM and $\beta$-SBM  are proposed. They introduce a block-corrected chi-square statistic for the SBM that uses the idea of block compression and has resemblance to our NAC family statistics. The similarity is, however, superficial, since we work conditional on the degrees, the parameters we consider are not the connectivity parameters $B$ but their normalized versions $\rho$ (compare equation~(5) in~\cite{karwa2016} with our equation~\eqref{eq:ncac:Tn:def}). The $\rho$ parameters have many desirable features; for example, they do not vanish in the sparse regime ($\nu_n / n \to 0$) while the connectivity parameters $B$ do, making the corresponding chi-square statistic numerically very unstable due to the division by these vanishing parameters. The cancellation of the degree-propensity parameters $\theta_i$ in $\rho$ is another key advantage, allowing us to use the same statistic in the degree-corrected case. In contrast, Karwa~et~al.~\cite{karwa2016} devise another test 
for the $\beta$-SBM (a close cousin of DCSBM, in the sparse regime) which requires $O(n^2)$ operations to compute. Another novelty of our approach relative to~\cite{karwa2016} is the idea of block compression with $K+1$ communities instead of $K$ which leads to a dramatic increase in the power of the test.

Another major difference with~\cite{karwa2016} is their interest in computing exact $p$-values which requires enumerating all graphs with a given sufficient statistic as the observed one. For example, for an SBM with known community structure, this translates to enumerating all graphs that have the exact same number of edges between communities as that of the observed network. Although, Karwa~et~al. develop clever sampling schemes to traverse this space, to get an accurate $p$-value, one has to sample a prohibitively large number of graphs in general, rendering the approach infeasible beyond small networks. In addition, their main arguments are for block models with a given community structure, and to get around the unknown nature of the communities in practice, they propose sampling the community labels and applying the known-community test on each. The space of all labels is again exponentially large of size $K^n$, and one requires a a very large sample to get any reasonable estimate, making the approach infeasible for large networks. The authors acknowledge this difficulty and suggest using labels obtained by spectral clustering in practice. One then has to worry about the dependence of these labels on the same data the test is computed from, a point where we carefully address in this paper. The asymptotic distributions we obtain for the adjusted statistics are very good approximations for large networks and allow us to apply the tests with minimal computational overhead to even networks of millions of nodes.

Compared with likelihood ratio (LR) tests~\cite{Yan_2014, wang2017}, our approach is more general since LR tests require a specific alternative model to compare with (often another SBM or DCSBM), while in goodness-of-fit testing, only the null has to be specified. In addition, rigorous results on LR tests, such as~\cite{wang2017}, often work with a computationally intractable version of the test where the label parameter $z$ is marginalized by summing over $K^n$ possibilities. In practice, these tests are often implemented by approximating the sum via variational inference or  
plugging-in parameter estimates obtained by a community detection algorithm into the complete likelihood, 
as we discuss in details in Section~\ref{sec:methods} and compare with in simulations. 
Although, the theory in \cite{wang2017} extends, in the case of SBM, to these approximations if the community detection algorithm is consistent, it is unclear whether the guarantees further extend to the DCSBM.


A pseudo-LR approach with rigorous guarantees is developed in~\cite{ma2018}. As in~\cite{wang2017}, the focus there, too, is on model selection and comparing DCSBM models, specifying both the null and alternative models, in contrast to NAC tests. Our approach is comparable to that of~\cite{ma2018} when applied sequentially for model selection, but NAC family of tests are computationally more efficient: (1) Computing the test statistic of~\cite{ma2018} has $O(n^2)$ computational complexity, whereas due to the column compression, we require only $O(M)$ where $M$ is the number of edges. (2) \cite{ma2018} creates new labels by binary segmentation, but we save time by reusing the labels estimated by the community detection algorithm.
In addition, their consistency results are based on the assumption that the community detection algorithm merges the true communities when it underfits and splits them when it overfits. However, our test only imposes the mild assumption that connectivity parameters are distinguishable among communities, allowing it to be compatible with many community detection algorithms.

As for the degree requirement, our method only requires $\nu_n \gtrsim \log n$, similar to model selection approaches in~\cite{ma2018, le2015estimating, Chen2018}, and slightly better than those of~\cite{Li2020_ECV, wang2017} that require $\nu_n/\log n \to \infty$. In contrast, the spectral goodness-of-fit test~\cite{bickel2016, Lei2016} has a much more severe requirement ($\nu_n \gg n^{1/2}$) as discussed earlier.


\section{Numerical Experiments}
\label{sec:simulations}

We now illustrate the performance of \cacp and \scacp on simulated and real networks.  We  use regularized spectral clustering~\cite{amini2013} as the community detection algorithm, since it is widely used, computationally efficient and conjectured to satisfy Assumption~\ref{assum:DCSBM} \cite{abbe2020entrywise}.
Given the number of communities $K$, the spectral clustering estimates the community labels by applying $k$-means clustering to the rows of the matrix formed by the $K$ leading eigenvectors of the normalized Laplacian. Regularization is attained by adding $\tau \dav / n$ (where $\dav$ is the network average degree) to every entry of the adjacency matrix before forming the Laplacian.  This regularization is known to improve the performance in the sparse regime ($\dav \ll n$) \cite{le2017concentration, zhang2018understanding}.

\subsection{Other methods}\label{sec:methods}

Along with our NAC tests, we consider the following approaches for comparison: Likelihood ratio test (LR)~\cite{wang2017}, Bayesian information criteria (BIC)~\cite{wang2017}, adjusted spectral test (AS)~\cite{Lei2016}, Bethe-Hessian spectral approach (BH)~\cite{le2015estimating}, network cross-validation (NCV)~\cite{Chen2018} and  edge cross-validation (ECV)~\cite{Li2020_ECV}. In each case, we add the suffix ``boot'' to the name, if bootstrap debiasing of Section~\ref{sec:boot} is applied to further adjust the statistic. Our implementation of LR, BIC and AS is slightly different from the corresponding references, as discussed below, with the code available at~\cite{amini2020nett}.

Let $\ell(B,\theta, \pi,z \given A) = \sum_i \log \pi_{z_i} + \sum_{i < j} \phi(A_{ij}; \theta_i \theta_j B_{z_i z_j} )$ be the log-likelihood of a DCSBM where $\pi$ is the class prior. We consider a Poisson likelihood rather than a Bernoulli one, mainly due to its computational efficiency for large sparse networks, hence $\phi(x; \lambda) = x \log \lambda - \lambda$. 
As suggested by~\cite{wang2017},
given some estimated labels $\zh$, we can plug-in the MLE of the remaining parameters ($B, \theta$ and $\pi$) into $\ell(B,\theta, \pi,\zh \given A)$, to get an estimate of the complete log-likelihood for the test. 
That is, we compute $\ell(\hat B,\hat \theta, \hat \pi, \hat z \given A)$ where $\hat z$ is the label estimate from the spectral clustering and $\hat B$, $\hat \theta$ and $\hat \pi$ are the natural estimates based on $\hat z$, that is,
\begin{align}\label{eq:MLE}
	\hat B_{k\ell} = \frac{N_{k\ell}(\hat z)}{m_{k\ell} (\hat z)}, \quad \hat \theta_i =  \frac{n_{\hat z_i}(\zh) d_i}{\sum_{j : \hat z_j = \hat z_i} d_i}, \quad \hat \pi_k = n_k(\zh) / n
\end{align}
where $N_{k\ell}(\zh)$ is the sum of the elements of $A$ in block $(k,\ell)$ specified by labels $\zh$, $n_k(\zh)$ is the number of nodes in community $k$ according to $\zh$ and $m_{k\ell}(\zh) = n_k(\zh) (n_\ell(\zh) - 1\{k = \ell\})$. We note that $\hat B$ is the same as the natural estimate of $B$ in the SBM. 

The LR test computes $\ell(\hat B,\hat \theta, \hat \pi,\zh \given A)$ for two DCSBMs with different number of communities and compares the difference to a threshold. The BIC score is $\ell(\hat B,\hat \theta, \hat \pi,\zh \mid A) - K(K+1) \log n / 2$ which is maximized to select the optimal $K$. For the AS, we consider the adjusted matrix $\At = (\At_{ij})$ where
\begin{align}\label{eq:At:Ph}
 	\At_{ij} = (A_{ij} - \hat P_{ij}) / ( n \hat P_{ij})^{1/2}, \quad \hat P_{ij} = \hat \theta_i \hat \theta_j \hat B_{\zh_i, \zh_j} \cdot 1\{i \neq j\}
 \end{align}
and compute its largest singular value $\sigma_1(\At)$. The difference with~\cite{Lei2016} is that we are using the (estimated) Poisson variance $ \hat P_{ij}$ rather than the Bernoulli variance $ \hat P_{ij}(1-\hat P_{ij})$. Moreover, we use the DCSBM estimate of mean matrix $(P_{ij})$. Using the Poisson variance  significantly improves the computational performance for sparse matrices, since then $\At$ can be written as the sum of a sparse matrix and a term involving the product of diagonal and low-rank matrices. This allows fast computation of $ \At x$ for any vector $x$, hence allows the singular value computation to scale to very large networks. 

In some simulations, we also consider AS-SBM, where we use the SBM estimate for $\hat P_{ij}$, which is obtained by setting $\hat \theta_i = 1$ in~\eqref{eq:At:Ph}. Using the same arguments as in~\cite{Lei2016}, one can show that in the case of AS (SBM), under a Poisson-SBM null, the distribution of $\sigma_1(\At)$ will be close to the Tracy-Widom distribution with index 1. However, the same cannot be said about AS which uses the DCSBM estimate of $(P_{ij})$. Nevertheless AS is the natural version to consider when fitting DCSBMs.

\subsection{Simulations}\label{subsec:sims}

\begin{figure}[t]
    \includegraphics[width=0.49\linewidth]{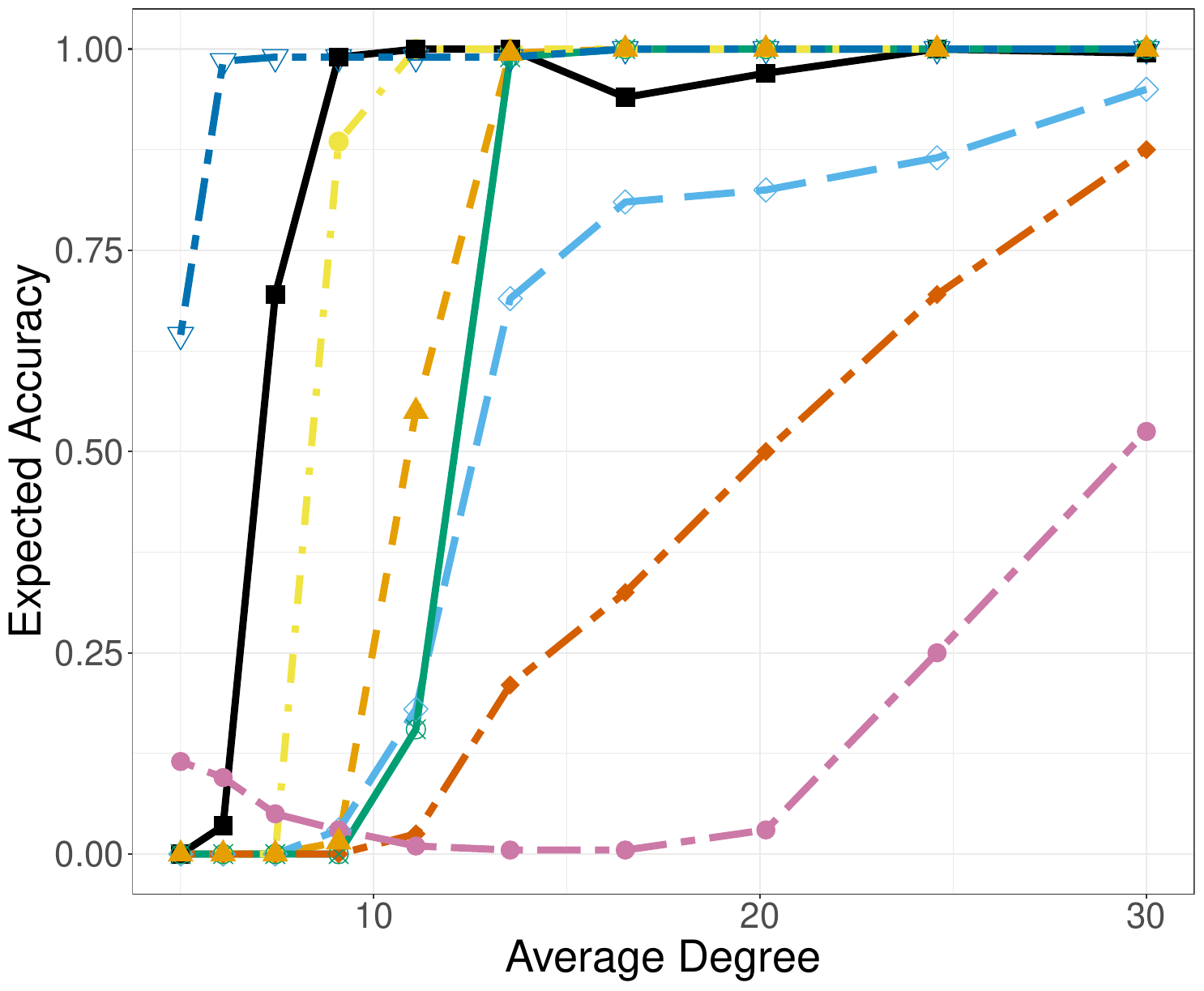}
    \includegraphics[width=0.49\linewidth]{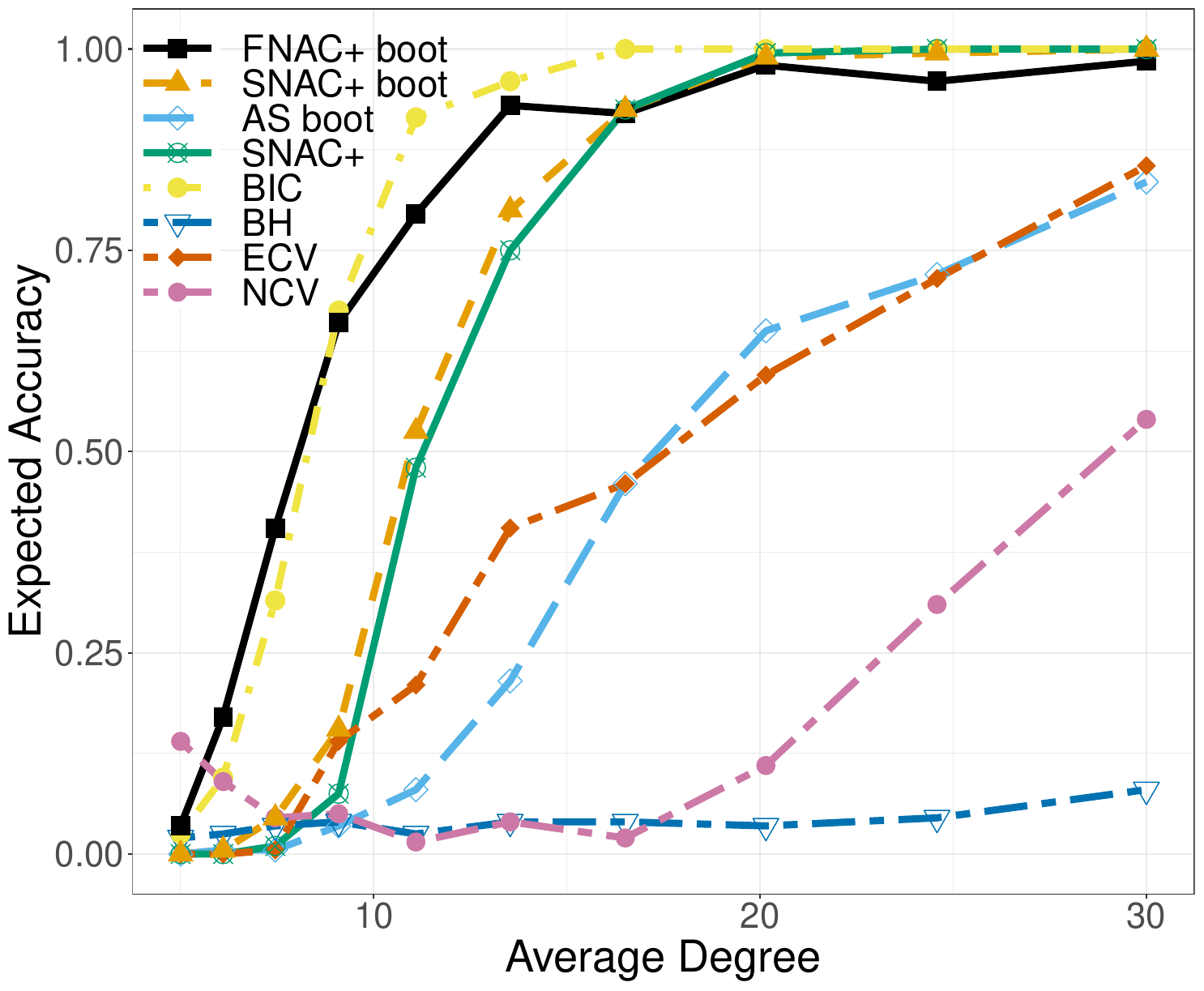}
	\caption{Expected accuracy of selecting the true number of  communities  versus expected average degree of the network. The data follows a DCSBM with $n = 5000$, $K = 4$, $\theta_i \sim \text{Pareto}(3/4, 4)$ and balanced community sizes. The connectivity matrices are $B_1$ (left) and $B_2$ (right), as defined in the text.}
	\label{fig:model_select}
\end{figure}

As discussed earlier, a goodness-of-fit test can be used in a sequential manner to perform model selection. We now provide simulations showing that, when applied sequentially, \cacp and \scacp are consistent, and competitive with other model selection approaches. For \scacp, we set $\sigma = 0$ in Algorithm~\ref{alg:snac}, i.e., quantile filtering is skipped. Here, we report results for samples from Bernoulli DCSBMs. Since we work with sparse networks, the Bernoulli model will be very close to its Poisson version. This was empirically confirmed, as we did not see a significant drop in performance for the \cacp and \scacp tests in our simulations, under a Bernoulli model relative to the Poisson.

\subsubsection{Model selection performance} 
\label{sec:model_select}
Let Pareto$(x_0, \alpha)$ denote a Pareto distribution with scale parameter $x_0$ and shape parameter $\alpha$, so that its mean is $\alpha x_0 / (\alpha-1)$.
We simulate data from a $K$-block DCSBM with connection propensity $\theta_i \sim$ Pareto$(3/4, 4)$, and a connectivity matrix which is one of the following:
\begin{enumerate}
	\item $B_1 \propto (1-\beta)I_{K}+ \beta\mathbf{1}\mathbf{1}^{T}$, that is, a simple planted partition model with out-in-ratio $\beta$,  %
	\item $B_2 \propto \gamma R + (1-\gamma) Q$, where $\gamma \in (0, 1)$, $R$ is a random symmetric permutation matrix, and $Q$ a symmetric matrix with i.i.d. Unif$(0,1)$ entries on and above diagonal.
\end{enumerate}
Here, $ \mathbf{1}$ is the all-ones vector. In both cases, the matrices are normalized to have a given expected average degree $\lambda$. The simple planted partition model $B_1$ generates a very homogeneous assortative network. Model $B_2$ creates a more general model by employing the permutation, allowing a mix of assortative and dissortative communities. Model $B_2$ is in general harder to fit.

Figure~\ref{fig:model_select} illustrates the model selection accuracy of  various methods for the following setup: $n=5000$, true $K=4$ with balanced community sizes, $\beta = 0.2$ and $\gamma = 0.3$.  For the goodness-of-fit tests \cacp, \scacp and AS, we use sequential testing from below to estimate $K$.  In each case, the rejection threshold is set to have a significance level of $10^{-6}$ under null. For tests with bootstrap debiasing, the number of bootstrap simulations is 10. Figure~\ref{fig:model_select} shows the expected model selection accuracy versus the expected average degree $\lambda$ for each method. The accuracy is obtained by averaging over 200 replications. As $\lambda$ increases, the problem gets easier and we expect the performance of consistent methods to improve. 

For both models $B_1$ and $B_2$, the performance of \cacp and BIC are close and they outperform other approaches, except for the BH in the case of the 
$B_1$~model. Note, however, that BH performs extremely poorly under $B_2$, showing that associativity is necessary for its consistency. In fact, as pointed out in~\cite{le2015estimating}, BH requires all the eigenvalues of $\ex[A]$ to be positive, which is violated with positive probability under the  $B_2$ model. The two versions of \scacp perform very close to each other and ranked after the \cacp and BIC pair. That the performance of the bootstrap  \scacp  is very close to that of \scacp with the theoretical threshold, corroborates the accuracy of the null distribution in Theorem~\ref{thm:null:dist:ncac}.  The spectral test (AS) performs reasonably well for model $B_1$, albeit ranked after \scacp, but relatively poorly under $B_2$. The cross-validation approaches generally underperform other approaches for model selection, with ECV significantly outperforming NCV.
 
It is also possible to construct examples where \cacp significantly outperforms BIC. See Figure~\ref{fig:seq:unequal:diag} in Appendix~\ref{app:mod:sel} for one such case.

\subsubsection{ROC curves}\label{sec:roc}

Another way to measure the performance of a test statistic is by means of its Receiver Operating Characteristic (ROC) curve, that is, the power of the test as a function of Type I error; equivalently, the true positive rate (TPR) as a function of the false positive rate (FPR). The ROC curve reveals the best possible performance of a statistic for a given testing problem (one achieved by setting the optimal threshold).  %
Here, we compare the ROC curves of the \cacp and \scacp  tests to the likelihood ratio (LR) and spectral (AS) test, for the problem of testing the null hypothesis of $K=4$ versus the alternative of $K+1 = 5$ communities. This is an example of ``testing from below'' which is encountered in sequential model selection.

\begin{figure}[t!]%
	\centering
	\includegraphics[width=0.47\linewidth]{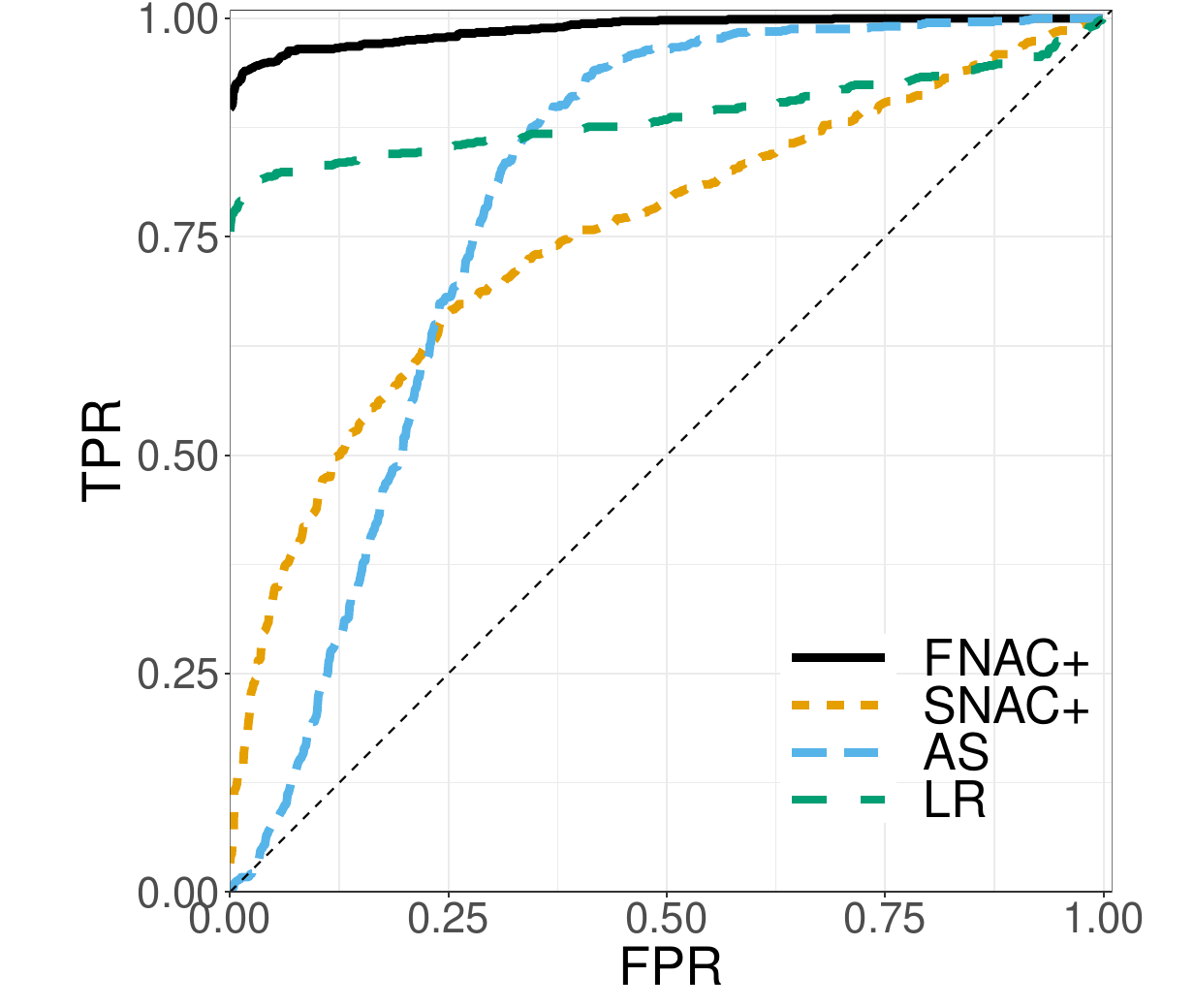}
	\includegraphics[width=.47\linewidth]{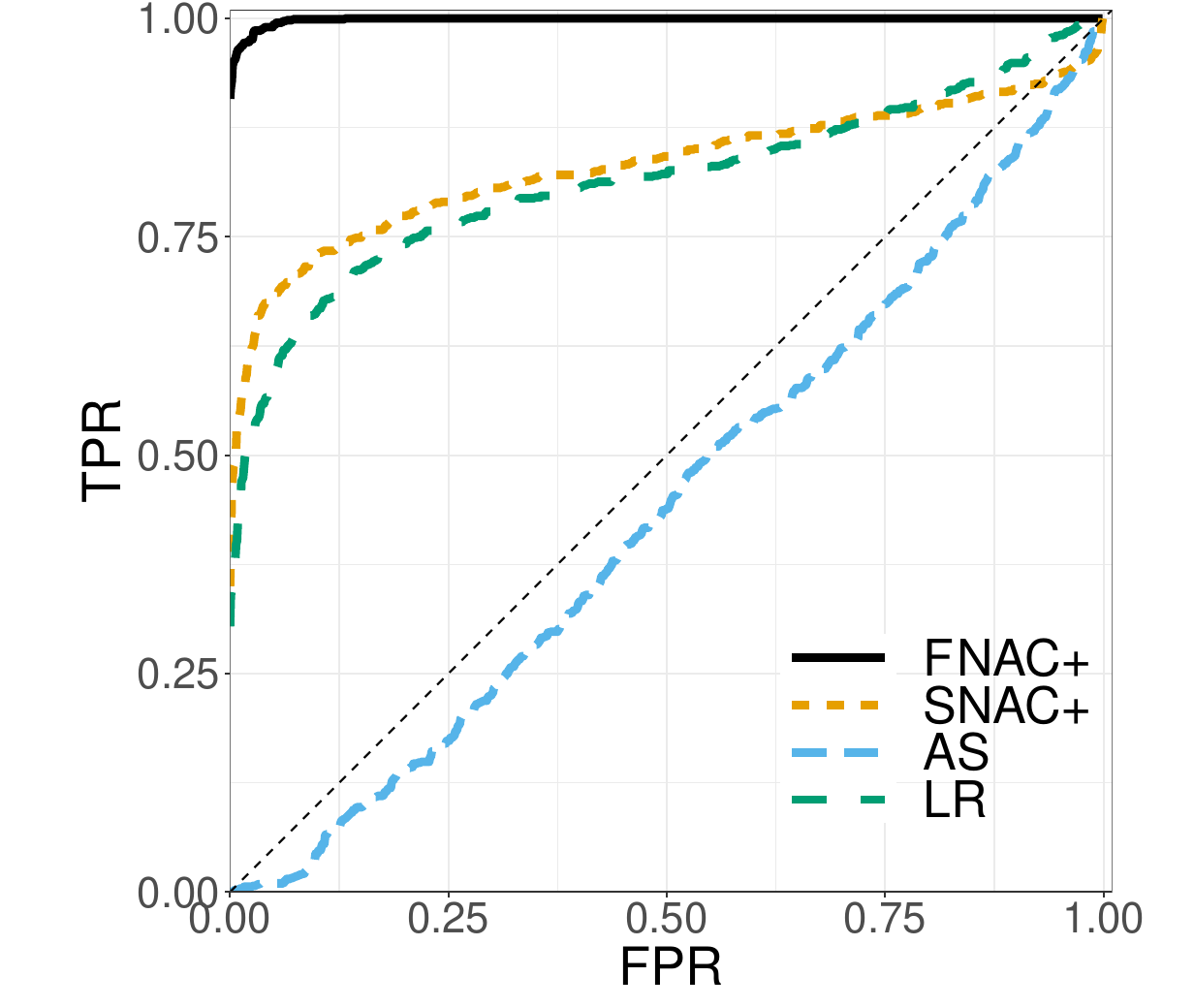}
	\\
	\includegraphics[width=0.47\linewidth]{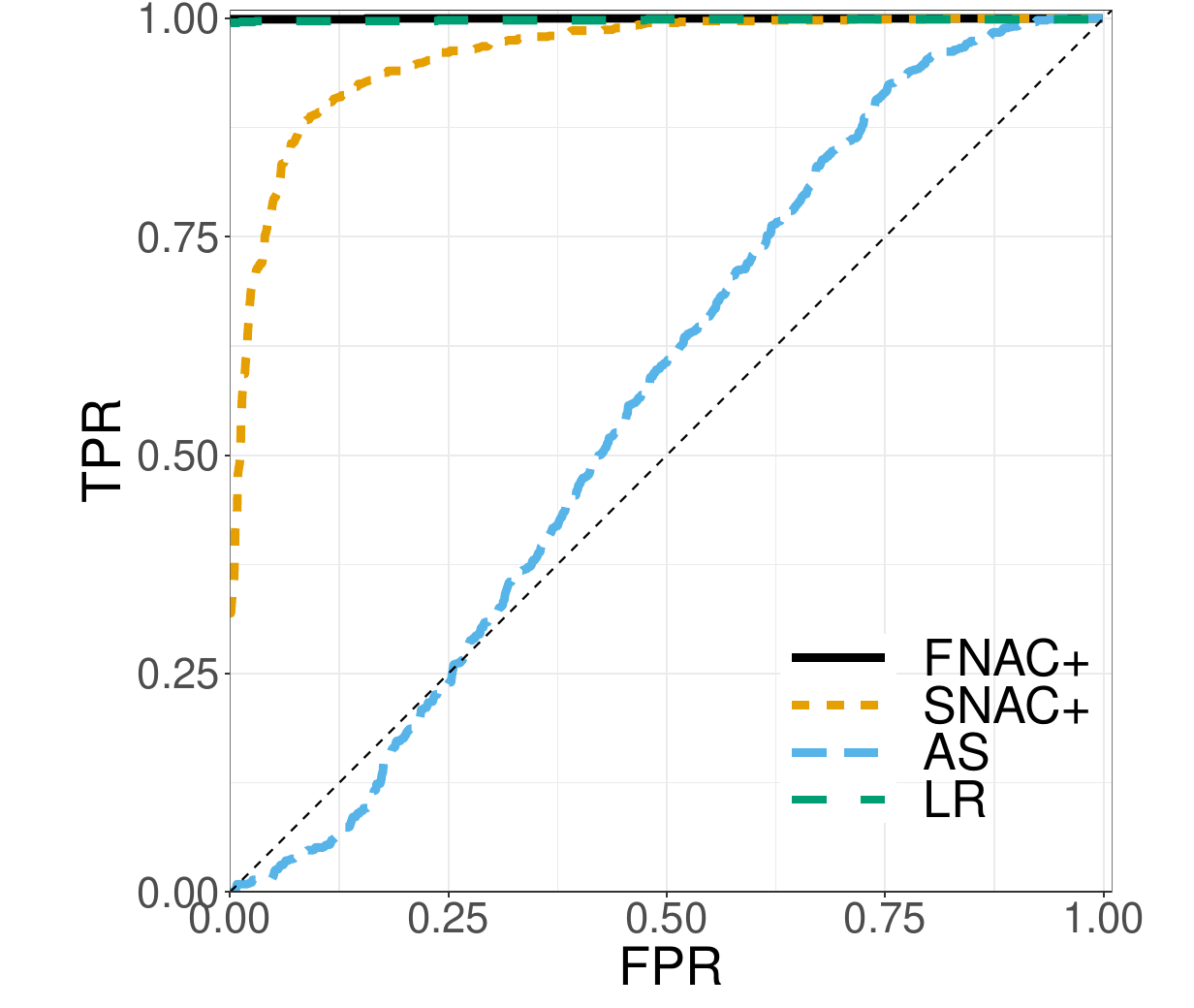}
	\includegraphics[width=.47\linewidth]{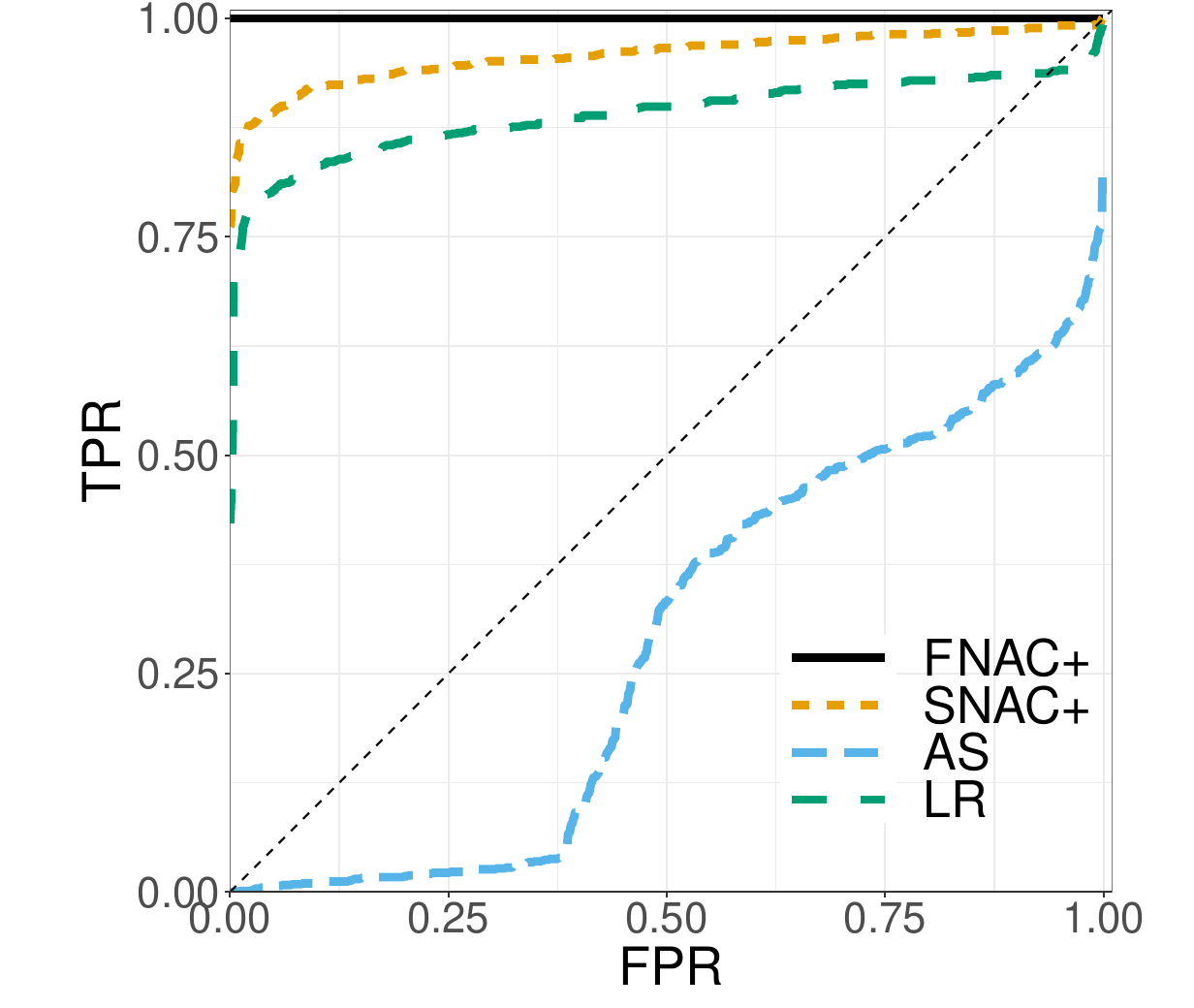}
	\caption{ROC plots for testing 4 versus 5 community models. Top and bottom rows correspond to $n=2000$ and  $n=10000$, respectively. Left and right columns correspond to the DCSBM and DCLVM alternatives, respectively. }
	\label{fig:roc:4vs5}
\end{figure}

For the null hypothesis, we consider a simple DCSBM with $K=4$ communities, having a connectivity matrix of type $B_1$, introduced in Section~\ref{sec:model_select}, with $\beta = 0.1$. For the alternative, we consider two cases: 
\begin{itemize*}
	\item[(a)] a DCSBM with $K+1=5$ and otherwise similar parameters to the null DCSBM, and
	\item[(b)] a degree-corrected latent variable model (DCLVM) with $K+1=5$ communities generated as follows: Given a set of latent node variables $\{x_i\}_{i=1}^n \subset \reals^d$ with $d=K+1$, the adjacency matrix $A = (A_{ij})$ is generated as a symmetric matrix, with independent Bernoulli entries above the diagonal, with %
\end{itemize*}
\begin{align}\label{eq:dclvm:def}
	\ex[\,A_{ij} \mid x, \theta\,] \;\propto\; \theta_i \theta_j e^{- \norm{x_i - x_j}^2} \quad  \text{and} \quad
	x_i = 2 e_{z_i} + w_i
\end{align}
where $e_k$ is the $k$th basis vector of $\reals^d$, $w_i \sim N(0, I_d)$ and $\{z_i\} \subset [K+1]^n$ are multinomial labels.
In other words, the latent positions $\{x_i\}$ are drawn from a Gaussian mixture model with $K+1 = 5$ components, living in $\reals^{K+1}$. The proportionality constant in~\eqref{eq:dclvm:def} is chosen such that the overall network has expected average degree $\lambda$. For all the models, including  the null and the two alternatives, the underlying prior on the labels is taken to be proportional to an arithmetic progression: $\pr(z_i = k) \propto k$ to produce unequal community sizes, and we let $\theta_i \sim \text{Pareto}(3/4, 4)$. 
For the DCSBM and DCLVM, we use average degrees $\lambda = 15$ and $\lambda = 8$, respectively.

Figure~\ref{fig:roc:4vs5} illustrates the resulting ROC curves. As expected,  increasing $n$ generally improves the performance (except for AS). Both \cacp and LR are almost perfect tests for differentiating the two DCSBMs at $n=10^4$. In all cases, the \cacp is more powerful than the sub-sampled version, \scacp. This is expected since \scacp relies on half the data. Note that as $n$ increases, \scacp greatly improves which can be attributed to the label estimation procedure achieving almost exact recovery, even at half the size of the original network. Note that AS generally is much less competitive compared to LR or \cacp. This is  expected since the spectral test relies on a general statistic that is not tailored to the blocked nature of the adjacency matrix of a DCSBM. 

Interestingly, 
\cacp is 
almost perfect 
for DCLVM even at $n = 2000$, whereas LR test underperfroms under the DCLVM. This is also expected, since the LR test incorporates the likelihood of a DCSBM for the alternative, which is mismatched to the actual alternative model. This experiment shows the power of NAC family in rejecting against models outside the family of DCSBM. It highlights the advantage of goodness-of-fit over likelihood-ratio testing  where one does not have to specify explicit alternatives, hence can test against many alternatives simultaneously. 
More ROC results are reported in Appendix~\ref{app:roc}.

\subsection{Goodness-of-fit testing}\label{sec:fb100:got:testing}
The main utility of a goodness-of-fit test is to assess how well real data fits the model. Let us investigate how well a DCSBM fits real networks from the Facebook-100 dataset~\cite{traud2011comparing, TRAUD2012fb}, hereafter referred to as FB-100. This dataset is a collection of 100 social networks, each the entire Facebook network within one university from a date in 2005. The networks vary considerably in size and degree characteristics; some statistics are provided in Table~\ref{tab:fb100}.

\begin{table}[t]
	\centering
	\caption{Statistics on the FB-100 dataset. Qu. is a short-hand for quartile.} %
	\label{tab:fb100}
	\begin{tabular}{r|cccccc}
		\hline
		& Min. & 1st Qu. & Median & Mean & 3rd Qu. & Max. \\ 
		\hline
		$n$ & 769 & 4444 & 9950 & 12083 & 17033 & 41554 \\ 
		Mean deg. & 39 & 65 & 77 & 77 & 88 & 116 \\ 
		3rd Qu. deg. & 54 & 91 & 110 & 108 & 124 & 166 \\ 
		Max. deg. & 248 & 673 & 1202 & 1787 & 2123 & 8246 \\ 
		\hline
	\end{tabular}
\end{table}

Figure~\ref{fig:gof:fb100} shows the violin plots of the \scacp statistic, with degree-filtering threshold $\sigma = 0.2$, versus the number of communities, for the entire FB-100 data. The variation at each $K$ is due to the variability of \scacp over the 100 networks in the dataset. 
For each FB network, we sample a twin network from a synthesized 3-cluster DCSBM that matches the original network in degree distribution.
Violin plots are also shown for these twin networks for comparison. For model parameters, each synthesized DCSBM has its own $\theta$ parameter proportional to the corresponding FB network degree vector, but they all share the same connectivity matrix $B$, which is set to the corresponding MLE based on all the FB networks. To get the shared $B$, we first apply spectral clustering with $K = 3$ to each FB network $A^{(s)}$, $s = 1, \dots, 100$
to get estimated labels $\zh^{(s)}$. Then, for each $\zh^{(s)}$, we compute the corresponding 
block sum and block size matrices, $N^{(s)}$  and $M^{(s)}$, as  in Section~\ref{sec:methods}.
Finally, we set $B = \sum_{s}N^{(s)}/\sum_{s} M^{(s)}$, where the summation and division 
are elementwise. The community sizes for the synthesized networks are taken to be balanced. Kolmogorov–Smirnov test was performed between the degree distributions of each FB network and its twin,
and 84 out of such 100 pairs resulted in $p$-values greater than 0.05, indicating 
close matches.

The results in Figure~\ref{fig:gof:fb100}
show a marked deviation of FB-100 networks from a DCSBM model as measured by \scacp goodness-of-fit test. If the networks were generated from a DCSBM, one would expect the distribution of \scacp to drop to within a narrow band around zero once $K$ surpasses the true number of communities. Only at $K=25$ a small fraction of FB-100 networks have \scacp values within, say, the interval $[-5,5]$, showing that a DCSBM with $K <25$ is not a good model for any of these networks. Even at $K=25$, the majority of FB-100 networks are still ill-fitted. 

On the other hand, 
we observe that \scacp 
is nearly normally distributed for $K=3$, while remaining large for $K=1$ and $K=2$. This corroborates the results of both Theorem~\ref{thm:null:dist:ncac} and Theorem~\ref{thm:consist} that predict exactly this behavior. Note that this conclusion holds despite the variation in the sizes and average degrees of the simulated networks, showing the insensitivity of the null distribution of \scacp to those parameters, as predicted by the theory.

\begin{figure}
	\centering
	\includegraphics[width=.49\textwidth]{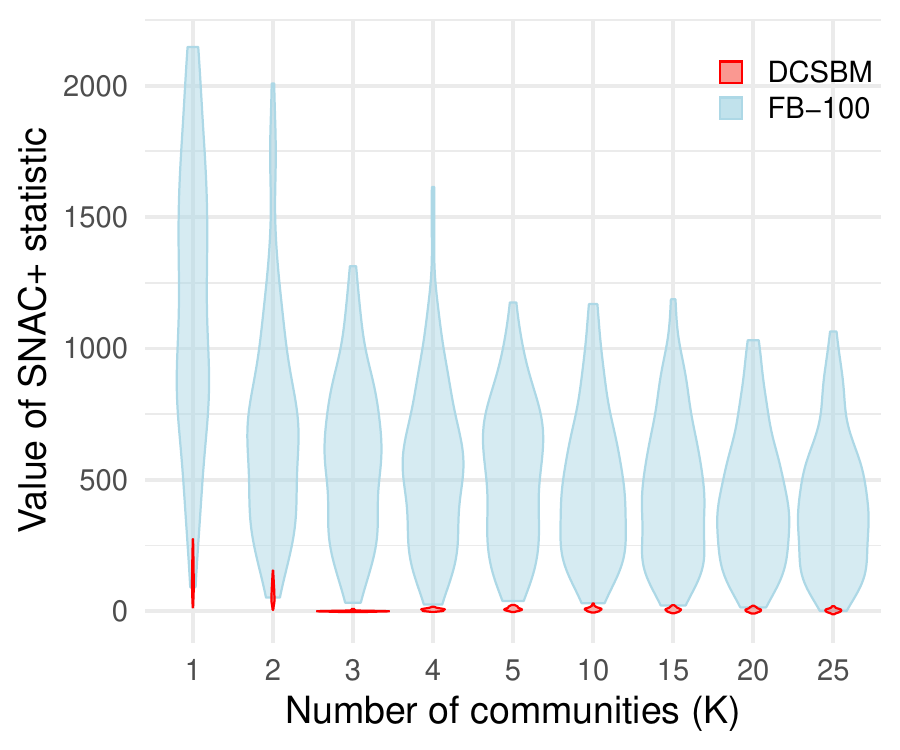}
	\includegraphics[width=.49\textwidth]{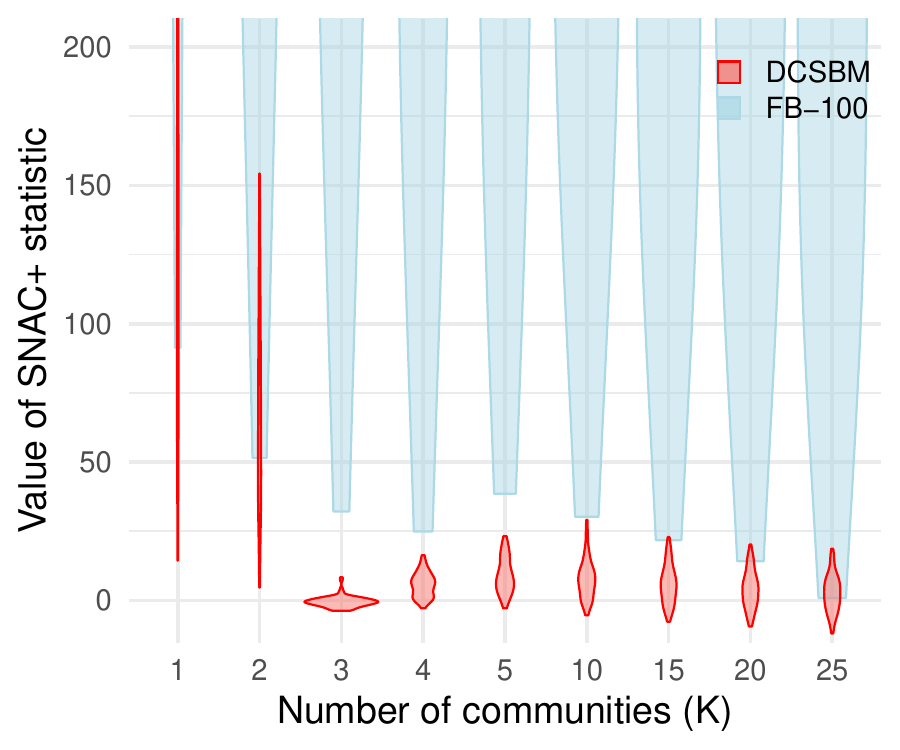}
	\caption{Comparing the goodness-of-fit of DCSBM to the FB-100 dataset versus a dataset simulated from twin DCSBMs with $K=3$ communities, and 
	having the same sizes and 
	degree distributions as those of FB-100. The left plot is the zoomed-in version of the right.}
	\label{fig:gof:fb100}
\end{figure}

Examining the FB-100 data further, one observes that most networks show some very high degree nodes that seem to skew the result of community detection as well as graph drawing algorithms. This can also be inferred from the significant divide between the third quartile and the maximum degree in Table~\ref{tab:fb100}. Le et al. \cite{le2017concentration} have also shown that abnormally high degrees can obstruct community detection. Treating these high-degree nodes as outliers, one could ask what happens if we remove them and refit the model? Figure~\ref{fig:gof:fb100:75pct} shows the result of performing the same experiment, %
but applied to the \emph{reduced} FB-100 networks, obtained  by restricting to the (induced) subnetwork formed by nodes having degrees below the 3rd quartile (i.e., the 75 percentile). Table~\ref{tab:fb100:75pct} shows the statistics on these reduced networks, revealing less skewed degree distributions compared to the original data. Figure~\ref{fig:gof:fb100:75pct} shows that the reduction leads to an overall improvement in the fit: More networks have \scacp values that drop to near zero and this happens for lower values of $K$. This shows the effectiveness of goodness-of-fit testing, in the sense that it allows us to test the hypothesis that removing the high-degree nodes causes a better DCSBM fit. Nevertheless, Figure~\ref{fig:gof:fb100:75pct} shows that the majority of the reduced networks are still far from a DCSBM with few number of communities. 
 
\begin{table}[ht]
	\centering
	\caption{Statistics on the reduced FB-100 dataset.} %
	\label{tab:fb100:75pct}
	\begin{tabular}{r|cccccc}
		\hline
		& Min. & 1st Qu. & Median & Mean & 3rd Qu. & Max. \\ 
		\hline
		$n$ & 544 & 3293 & 7356 & 8930 & 12601 & 30590 \\ 
		Mean deg. & 11 & 20 & 24 & 24 & 28 & 36 \\ 
		3rd Qu. deg. & 16 & 29 & 34 & 34 & 40 & 52 \\ 
		Max. deg. & 38 & 74 & 89 & 87 & 101 & 149 \\ 
		\hline
	\end{tabular}	
\end{table}

\begin{figure}
	\centering
	\includegraphics[width=.49\textwidth]{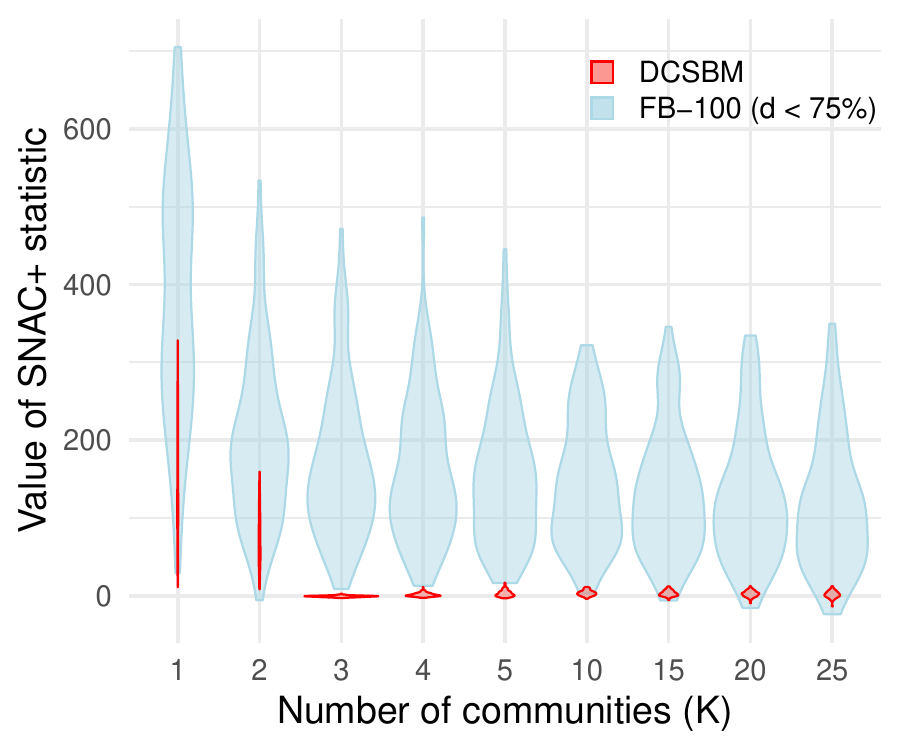}
	\includegraphics[width=.49\textwidth]{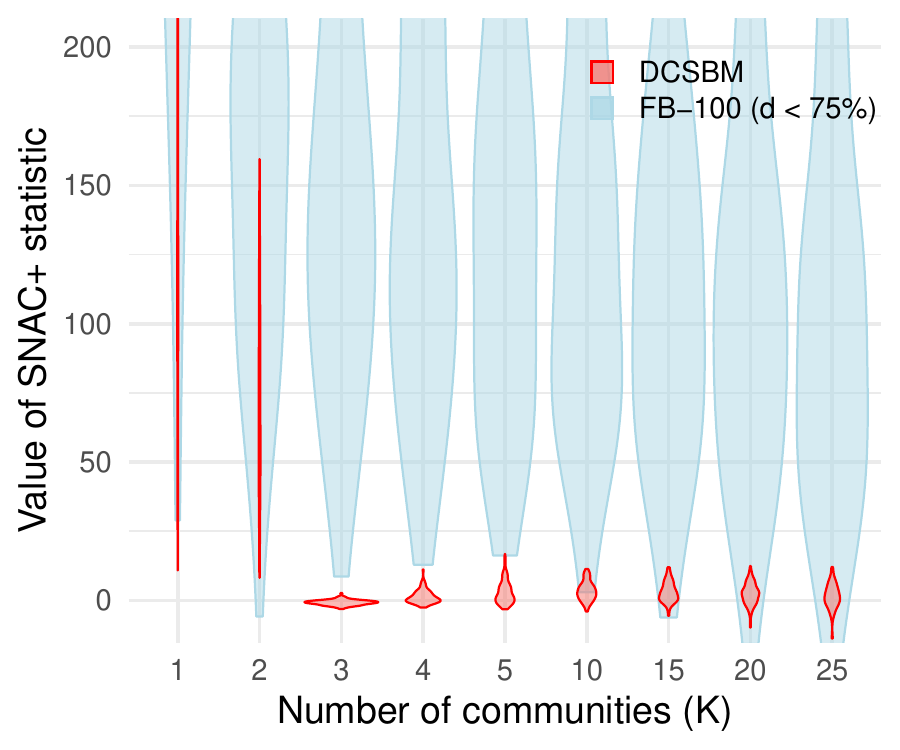}
	\caption{Similar to Figure~\ref{fig:gof:fb100} but with Facebook networks reduced by restricting to nodes with degrees below the 75 percentile.}
	\label{fig:gof:fb100:75pct}
\end{figure}

\subsection{Exploring community structure}\label{sec:comm:profile}

As demonstrated in Section~\ref{sec:fb100:got:testing}, a DCSBM (with small $K$) is not a good fit for most of the networks in FB-100. Even in such cases,
 \scacp has utility beyond testing and can be used to reveal community structure in networks. We demonstrate this by using the reduced FB-100 networks constructed in Section \ref{sec:fb100:got:testing}.
 Recall that \scacp quantifies the similarity within each estimated community, and a smaller value means that the rows in an estimated community share a similar connection pattern to other communities. %
 Therefore, sharp drops in the value of \scacp, as $K$ varies can signal the existence of community structure.
 For a sequence of \scacp statistics with increasing $K$, there could be an \emph{elbow} where continuing to increase $K$ does not bring a significant decrease in the statistic, or a \emph{dip} where \scacp starts to increase. These two types of points signal that it is not worthwhile to continue increasing $K$. Furthermore, these transitions are often much more dramatic for \cacp and \scacp tests than the competing methods and can be easily identified by eyeballing the plots.

\begin{figure}[t]
     \includegraphics[width=0.49\linewidth]{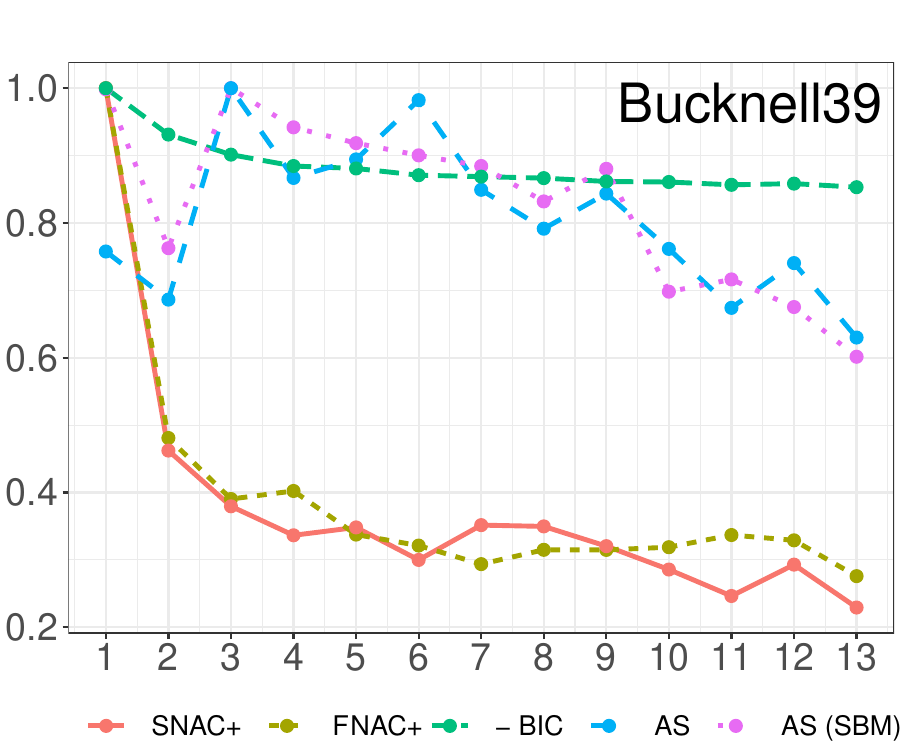}
     \includegraphics[width=0.49\linewidth]{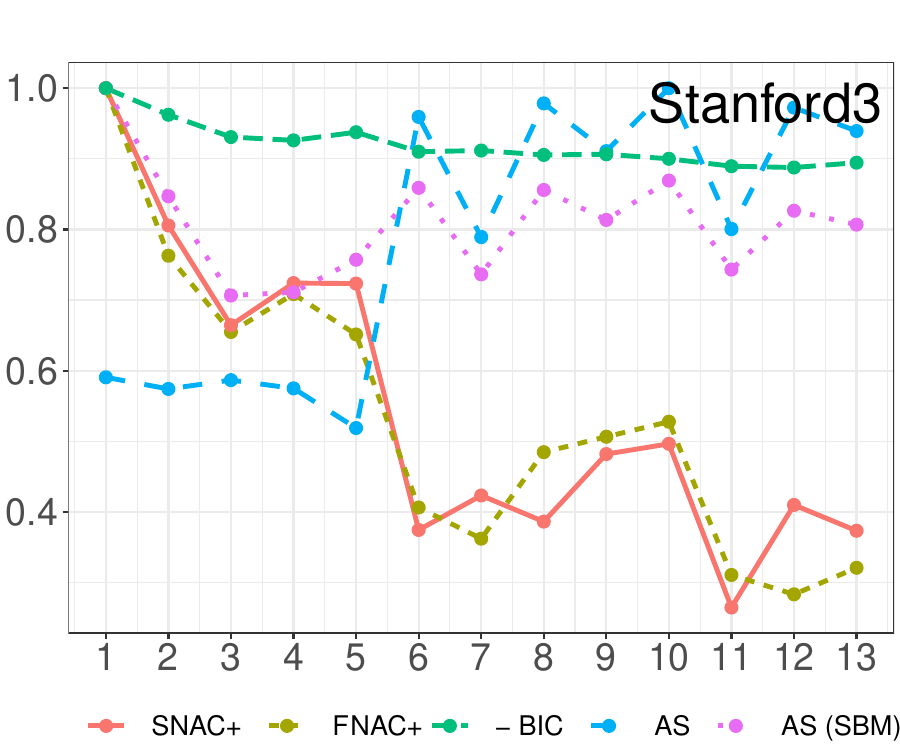}
	\caption{Normalized statistics versus the candidate number of communities ($K$).}
\label{fig:normalized:plots}
\end{figure}

Figure~\ref{fig:normalized:plots} shows the normalized statistic plots for two networks from FB-100. The plots show the normalized value of \scacp, \cacp, AS, AS-SBM and negative BIC statistics for $K = 1,\dots, 13$. The statistics are normalized to fall in the range $[-1,1]$ by dividing by their largest absolute value, for each test, respectively. This allows us to compare the trend of each statistic as $K$ increases among different methods. 

For many of the FB-100 networks, \scacp and \cacp share a similar pattern, with rapid drops followed by the flattening or increase of the statistic, signaling strong community structures. In contrast, AS and AS-SBM generally do not show strong trends, while negative BIC barely fluctuates at all when $K$ increases.
For example, for the Bucknell network (Figure~\ref{fig:normalized:plots}), there is one sharp elbow at $K=2$ for the \cacp tests. The Stanford network shows an elbow/dip at $K = 3$ and a similar elbow/dip at $K= 6$ in \cacp tests. This suggests that the network has two levels of community structure (cf. Figure~\ref{fig:net:plots}), an interesting phenomenon not captured by other statistics. Note that AS-SBM captures the community structure at $K=3$ for the Stanford network (with a dip at $K=3$) while missing the $K=6$ possibility. The AS version (employing degree-correction) behaves contrary to expectation in this case and misses both structures. 

\paragraph{Community profiles} We now consider a more quantitative approach to constructing a community profile based on the value of \scacp. We take advantage of the randomness in \scacp due to subsampling, as a natural measure of the uncertainty of the community structure.
For each $K$, we calculate \scacp several times, each time using a random split of the nodes, and then fit a smooth function to the resulting points, treating the problem as a nonparametric regression. Here, we consider smoothing splines but other approaches such as Gaussian kernel ridge regression can be equally useful. The estimated smooth function provides what we refer to as \emph{a community profile} for the network. This profile can be used for comparing and classifying networks as well as determining possible good choices of the number of communities. The subsampling and smoothing provide a degree of robustness to these profiles as illustrated below.

Instead of eyeballing a plot for its elbows and dips, we can rely on the derivatives of the community profile to guide us. We quantify the elbow as the point where the second derivative has the largest value and the dip as where the first derivative turns positive for the first time. Alternatively, one can use the point with the largest curvature as the elbow point~\cite{hansen1993use}. However, we have found, empirically, that the second derivative, as a proxy for the curvature, is much more accurate in capturing the elbow as determined by a human observer.

\begin{figure}[t!]
	\includegraphics[width=0.325\linewidth]{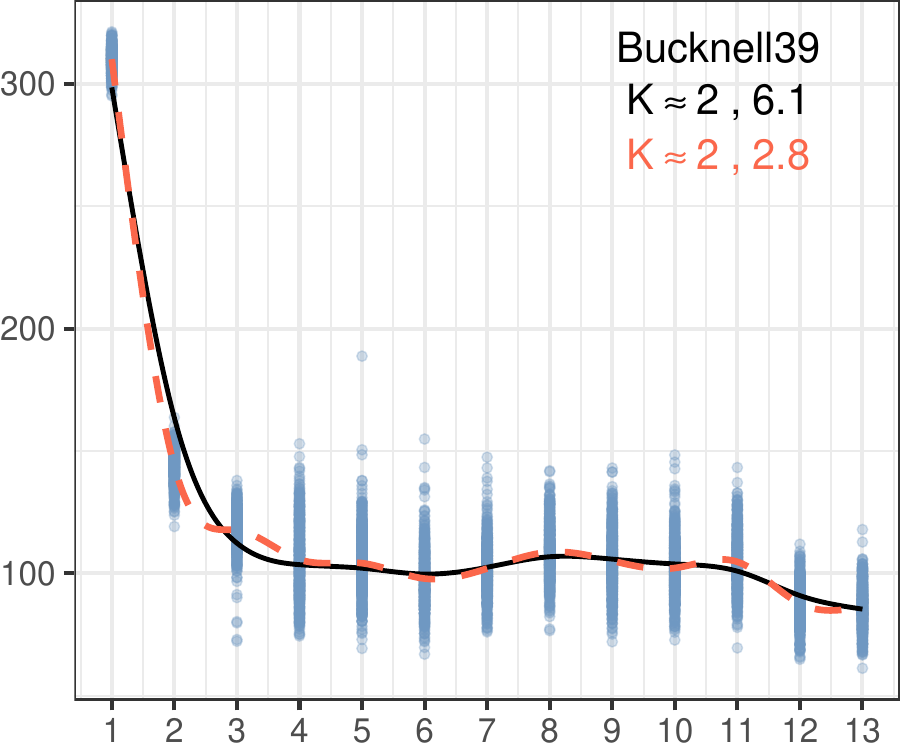}
	\includegraphics[width=0.325\linewidth]{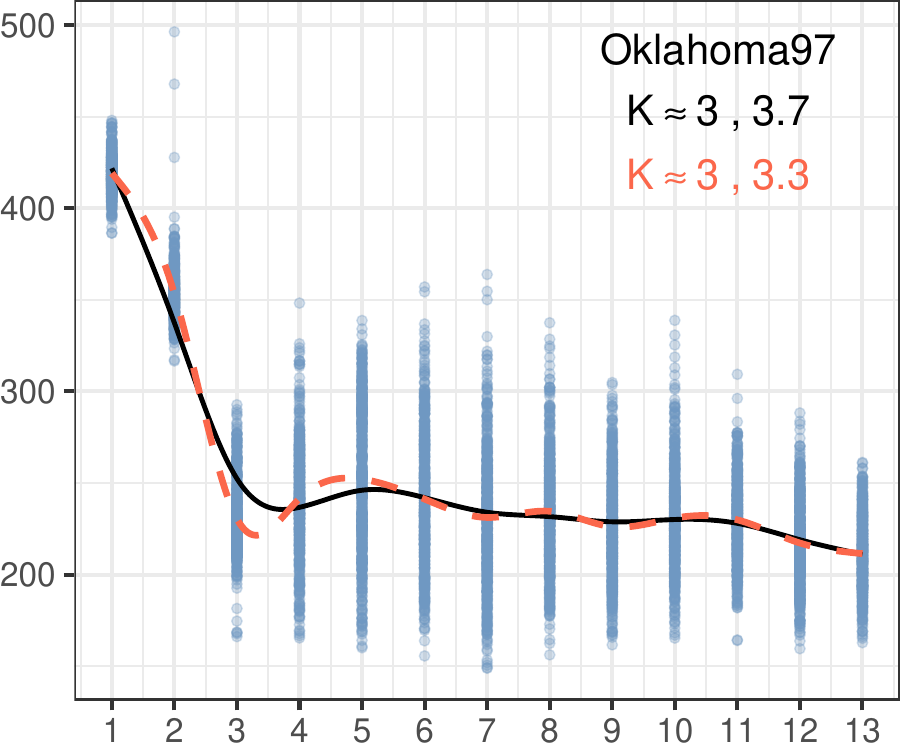}
	\includegraphics[width=0.325\linewidth]{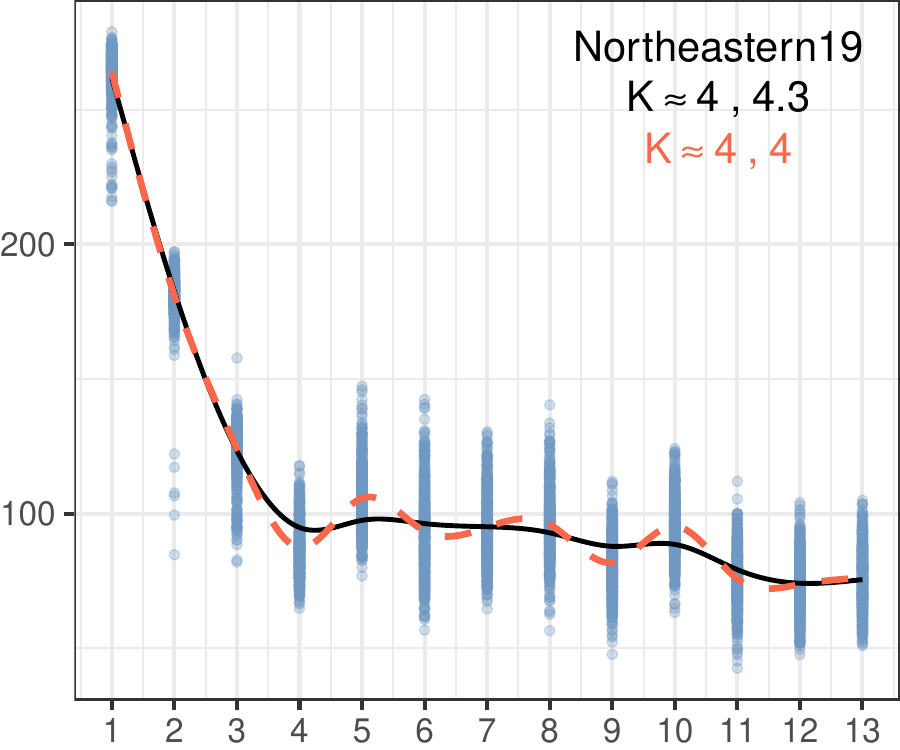}
	\includegraphics[width=0.325\linewidth]{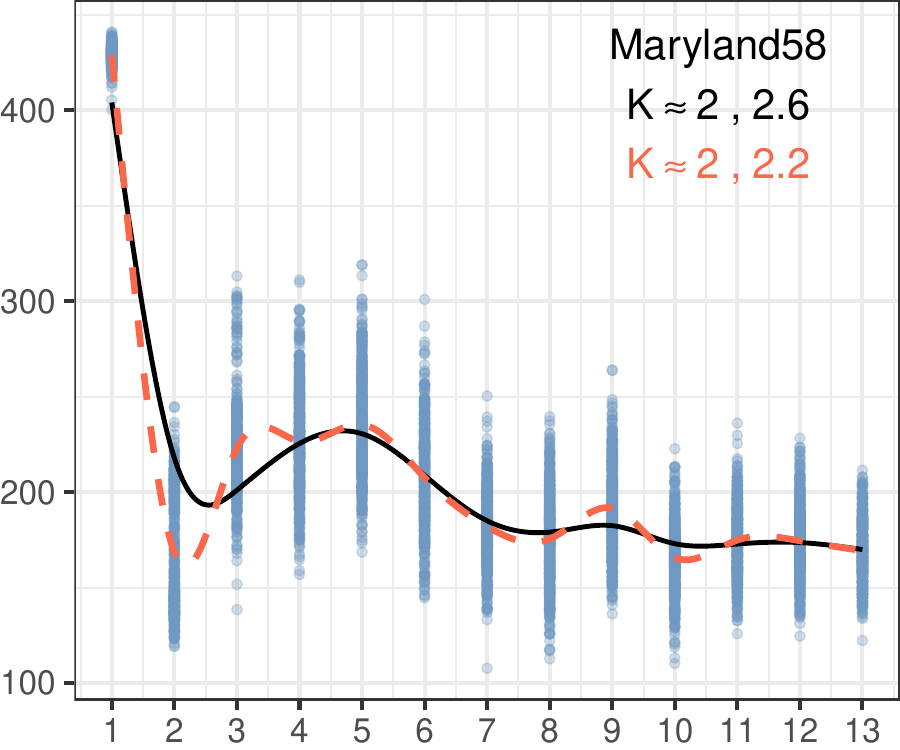}
	\includegraphics[width=0.325\linewidth]{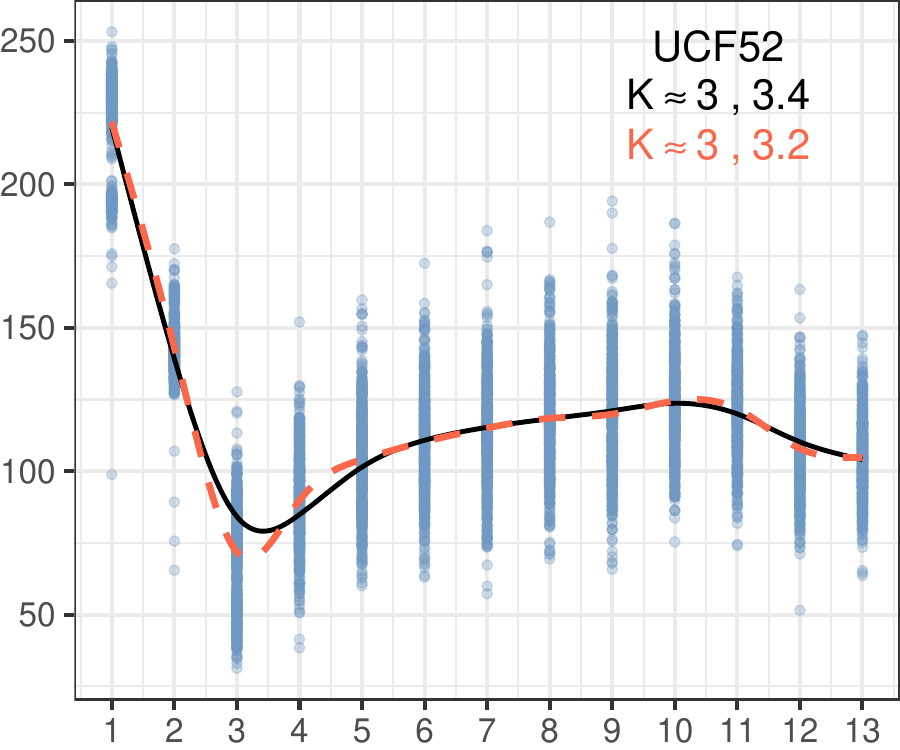}
	\includegraphics[width=0.325\linewidth]{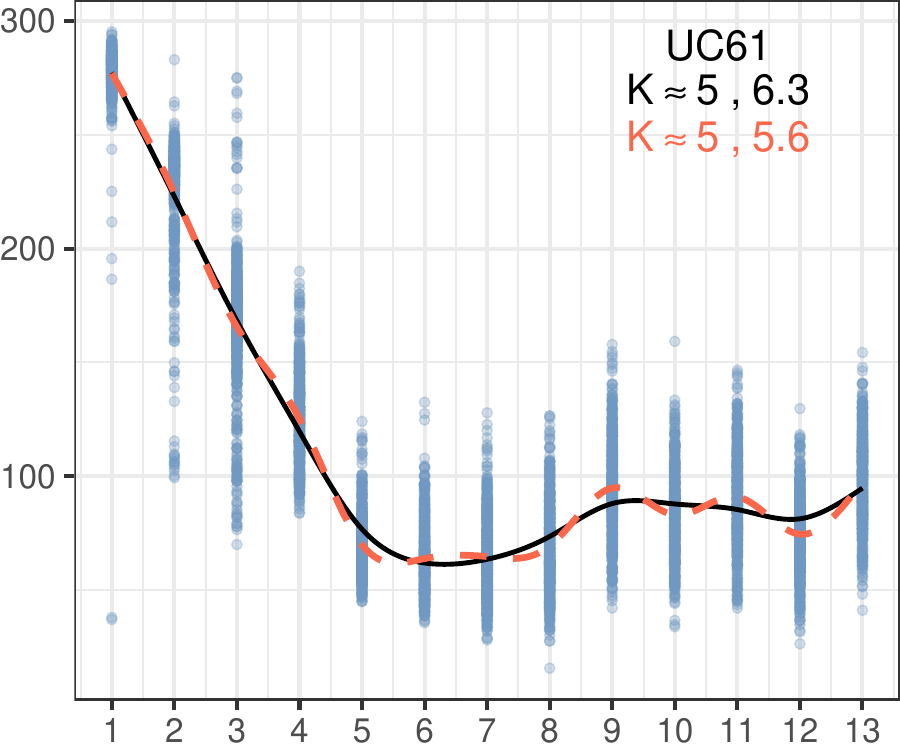}
	\includegraphics[width=0.326\linewidth]{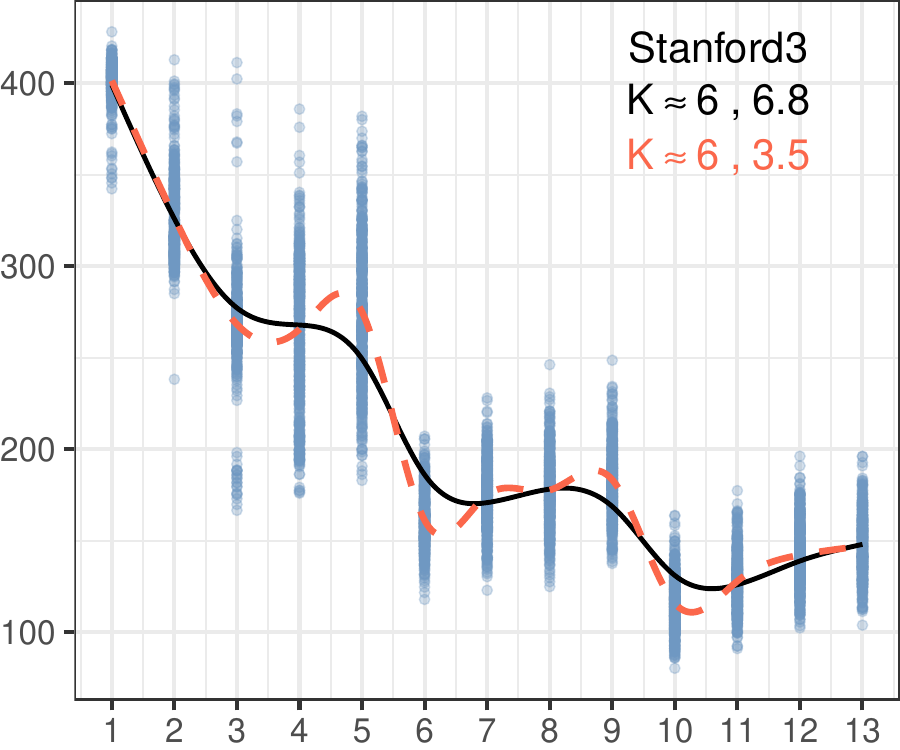}
	\includegraphics[width=0.3275\linewidth]{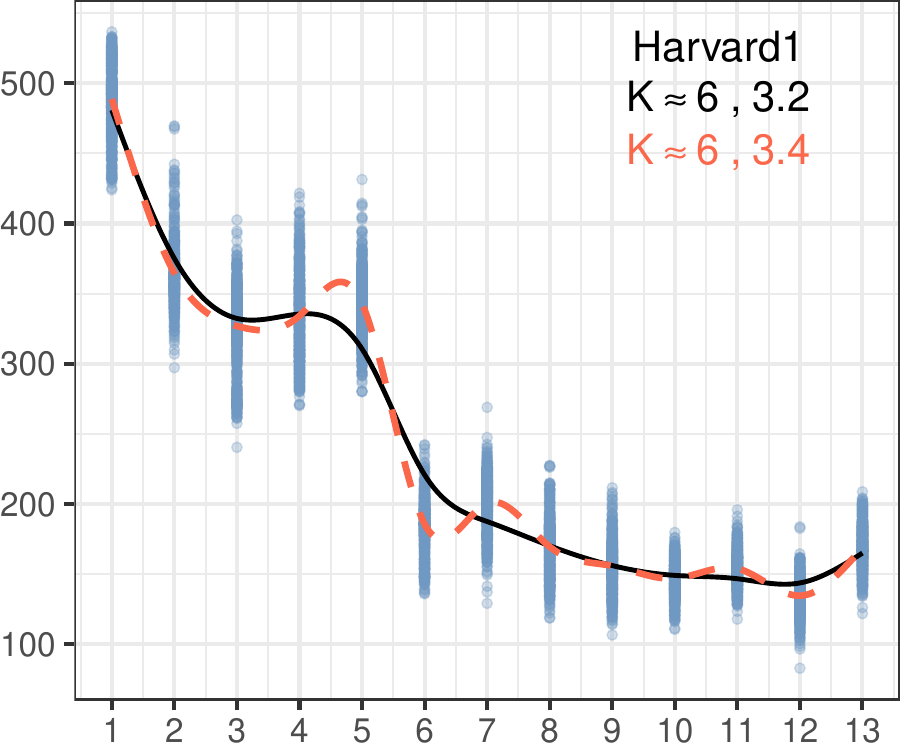}
	\includegraphics[width=0.3275\linewidth]{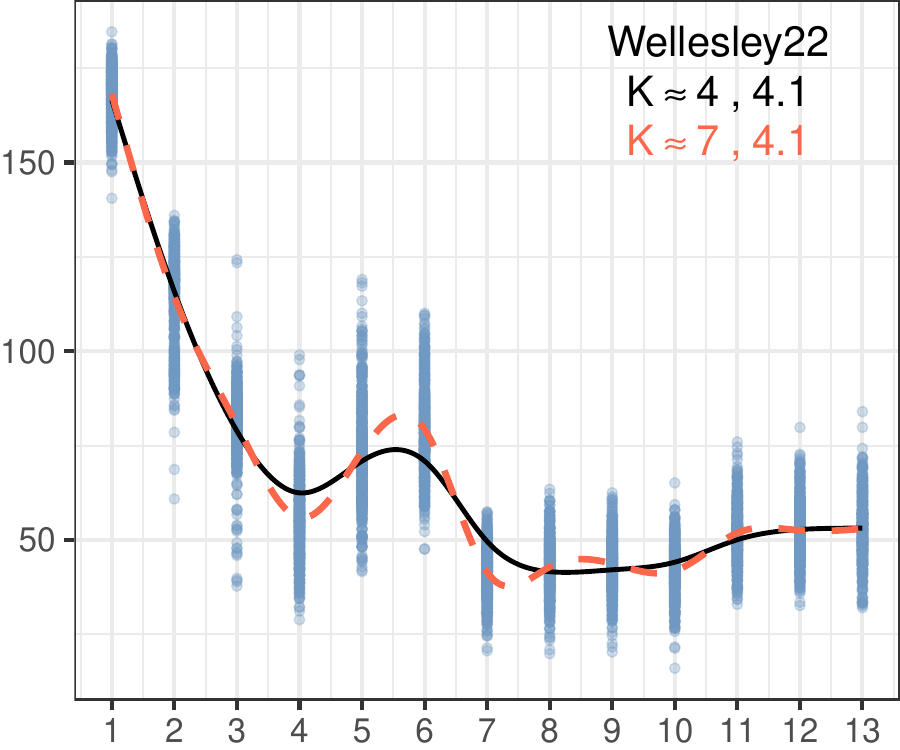}
	\caption{Community profile plots. The solid and dashed lines show the smoothed \scacp statistic versus the candidate number of communities ($K$). The dots each represent the \scacp value for a random split of the network. The difference between the solid and dashed lines is the smoothness level of the fitted smoothing spline.}
	\label{fig:comm:profiles}
\end{figure}

Figure~\ref{fig:comm:profiles} provides instances of three most common patterns of community profiles for the FB-100 networks. For each plot, we show community profiles using two smoothness levels: (1) the dashed red line corresponding to smoothness level set by generalized cross-validation (GCV) \cite{golub1979generalized}, and (2) the solid line providing a smoother fit, corresponding to \texttt{spar} = 0.3, where \texttt{spar} is the smoothness parameter in base \texttt{R}'s implementation of smoothing splines. The GCV version is usually rougher and captures subtle changes, whereas the solid black fit is smoother and more robust.
For each of the two fitted curves, the values of $K$ corresponding to the elbow and dip, as estimated by the derivatives, are given on each plot with the elbow point recorded first. For example, the Harvard network shows an elbow at $K=6$ and a dip at $K \approx 3.2$ according to the smoother profile. Compared with normalized plots (Figure~\ref{fig:normalized:plots}), community profiles show less randomness and the quantified elbows and dips are consistent with those identified by a human observer. It is worth noting that our maximum second derivative criterion for identifying the elbows, surprisingly, almost always returned an integer in these experiments, i.e. no rounding is performed in reporting the elbow points.

The first row in Figure~\ref{fig:comm:profiles} shows a single-elbow pattern, and the second row a single first dip (possibly followed by minor smaller dips later on). The third row illustrates a pattern with more than one significant drop, corresponding to multiple elbow/dips. This interesting multi-stage behavior is exhibited by a few of the FB-100 networks, and suggests the possibility of breaking the networks into communities in multiple (potentially hierarchical) ways. As mentioned earlier, these multi-stage structures are only captured by \scacp among the competing methods. This case illustrates the subtlety of community detection in real networks, showing that insisting on fitting the networks with a single $K$ could lead to missing interesting substructures.
 We also point out that having an elbow/dip at $K=2$ is very common for the FB-100 networks; we refer to the additional profile plots provided in Appendix~\ref{app:more:eg}

Note that in addition to revealing community structure, the absolute value of the profile curves in Figure~\ref{fig:comm:profiles} is also informative and measures the distance of the network form a DCSBM. Since \scacp is guaranteed to be centered around zero under a DCSBM, the networks with a larger absolute value of \scacp are further away from a DCSBM. For example, Figure~\ref{fig:comm:profiles} shows that Wellesley with $K=4$ communities, having an average \scacp value $\approx 60$ is a much better fit to DCSBM than Maryland with $K=2$ communities, showing an average \scacp value $\approx 200$.

\begin{figure}[hbt!]
	\begin{tabular}{ccc}
		\includegraphics[width=0.32\linewidth]{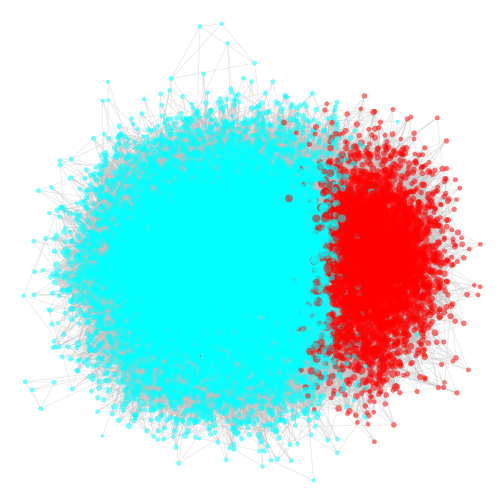}&
		\includegraphics[width=0.32\linewidth]{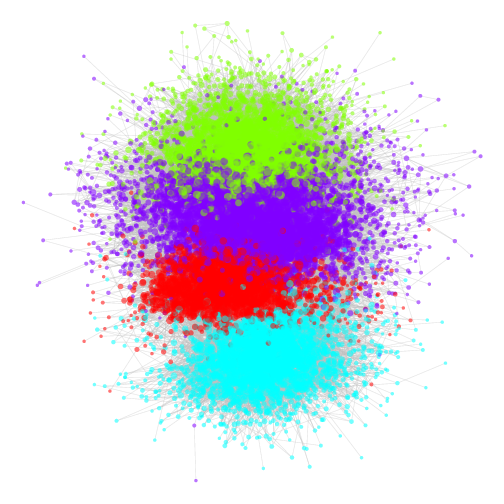} &
		\includegraphics[width=0.32\linewidth]{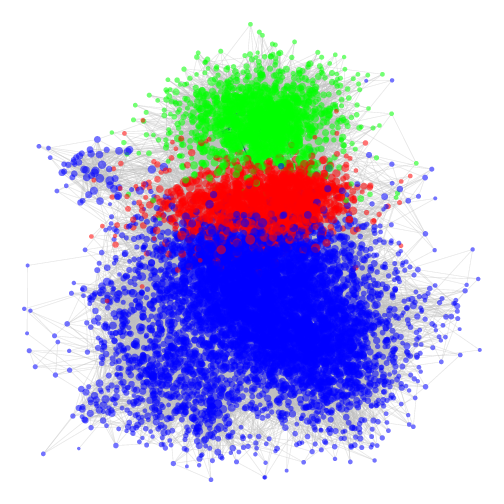} \\[-2ex]
		{\scriptsize Maryland58} & {\scriptsize Northeastern19} & {\scriptsize Stanford3 $(K = 3)$} \\[1ex]
		\includegraphics[width=0.32\linewidth]{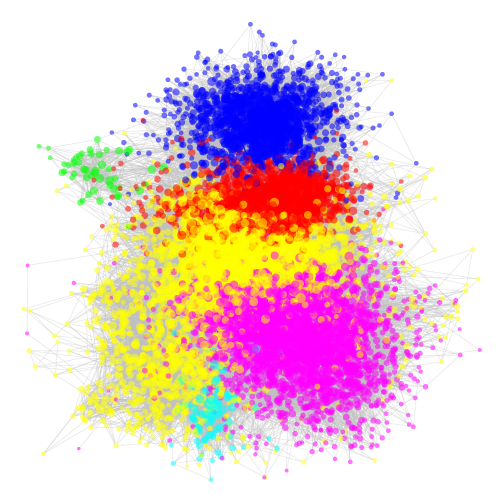}&
		\includegraphics[width=0.32\linewidth]{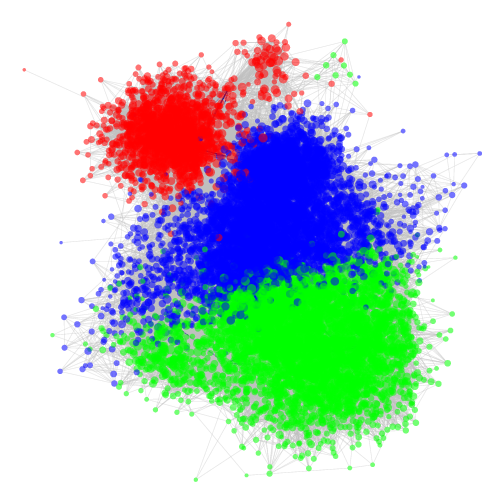} &
		\includegraphics[width=0.32\linewidth]{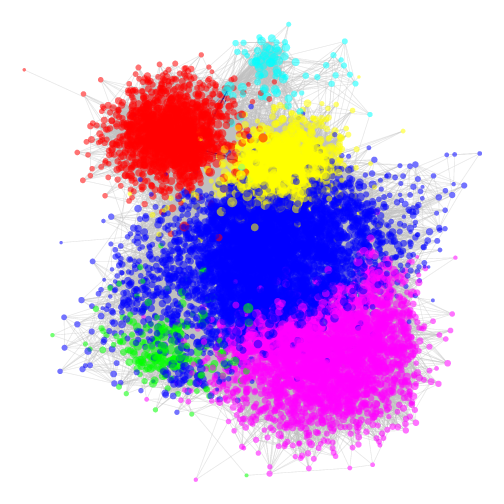} \\[-2ex]
		{\scriptsize Stanford3 $(K = 6)$} & {\scriptsize Harvard1 $(K = 3)$} & {\scriptsize Harvard1 $(K = 6)$} 
	\end{tabular}
	
	\caption{FB-100 network plots. The colors specify the estimated communities. The layouts are generated by the Fruchterman--Reingold algorithm that positions the nodes according to forces exerted along the edges. As a result, spatial proximity in these plots is correlated with network connectivity.}
	\label{fig:net:plots}
\end{figure}

Figure~\ref{fig:net:plots} shows community structure of some of the FB-100 networks with nodes colored according to their estimated community label.
The Stanford and Harvard networks are shown both with $K=3$ and $K=6$ estimated communities, as suggested by the two stages of their community profiles. We note that for both of these networks either of these two divisions into communities is visually sensible, with $K=6$ apparently capturing more refined substructures within the $K=3$ division. It is interesting to note that the $K=6$ partition in each case is not a strict refinement of the $K=3$ partition, but rather close to being a refinement. The community structures shown for Maryland and Northeastern are based on the optimal $K$ predicted by their profile plots, and they too make sense visually.

In Appendix~\ref{app:more:eg}, we also provide normalized and  profile plots (Figure \ref{Fig:polblog_net}) for the political blog network \cite{Adamic2005political_blog} which is widely used as a benchmark for community detection. The profile plot shows an elbow at $K=2$, as identified by the second derivative, matching the expected ground truth of  two communities corresponding to the Democratic and Republican parties.

\section*{Acknowledgement}
This work was supported in part by NSF 
grant DMS-1945667. We thank Mason Porter who provided access to the Facebook-100 dataset.


\printbibliography

\newpage

\appendix

\pagebreak

\begin{center}
    {\LARGE The Supplement to} \\[2ex]
    {\LARGE ``Adjusted  chi-square test for degree-corrected block models''} \\[2ex]
    {\large Linfan Zhang and Arash A. Amini}
    \bigskip
\end{center}

This supplement contains discussion, proofs and additional empirical results.

\section{$K+1$ vs. $K$ column clusters}\label{sec:comp:plus}

We give a simple example to show why the versions of the test with column labels $\yh$ having $L = K+1$ clusters (\cacp and \scacp) are more powerful than the versions with $L = K$ column clusters (\cac and \scac). We compare them using both theoretical calculations and real network simulations.

Consider an SBM with $K_0 = 3$, equal-sized communities and a planted-partition $B$ with $p$ on the diagonal and $q$ on the off-diagonal, that is,
\begin{align*}
    B = \begin{pmatrix}
    p & q & q\\
    q & p & q\\
    q & q & p
    \end{pmatrix}.
\end{align*}
Suppose that we want to test the null hypothesis $K = 2$. For simplicity, let us consider \cacf. The row label $\zh$ is estimated with $K =2$ clusters, potentially merging two of the three clusters. Consider an ideal case where $\zh$ perfectly combines clusters 2 and 3 into one, which we refer to as cluster $2'$, while correctly recovering cluster~1. 

For the \cac, we have column labels $\yh = \zh$, leading to the following confusion matrix $R$---defined in \eqref{eq:R:def}---and $BR$,
\begin{align*}
    R = \begin{pmatrix}
    1/3 & 0\\
    0 & 1/3\\
    0 & 1/3
    \end{pmatrix},\quad
    BR = \begin{pmatrix}
    p/3 & 2q/3 \\
    q/3 & (p+q)/3\\
    q/3 & (p+q)/3
    \end{pmatrix}.
\end{align*}
Recall that the multinomial probability $\rho$---defined in~\eqref{eq:rho:def}---is determined by normalizing rows in $BR$ to make each sum to 1. Therefore $\rho_{2*} = \rho_{3*}$. This means that all the rows in the merged cluster~$2'$ have the same mean vector, and similarly all the rows in true cluster~1. Since  \cac tests the equality of means among rows in 1 and rows in $2'$, it produces a small value and fails to reject the null.


On the other hand, for the \cacp, we fit $\yh$ with $L= K+1 = 3$ clusters, and in the ideal case we recover the true clusters, that is, $\yh = z$. In this case, 
\begin{align*}
    R = \begin{pmatrix}
    1/3 & 0 & 0\\
    0 & 1/3 & 0\\
    0 & 0 & 1/3
    \end{pmatrix},\quad
    BR = \begin{pmatrix}
    p/3 & q/3 & q/3\\
    q/3 & p/3 & q/3\\
    q/3 & q/3 & p/3
    \end{pmatrix}.
\end{align*}
Therefore, the multinomial probability $\rho$ is proportional to $B$, and in particular, $\rho_{2*} \neq \rho_{3*}$. We are still using the same row labels $\zh$ as in the case of \cac,  with the two clusters 1 and $2'$, to compare the equality of means among rows. Since over $2'$, now half the rows have mean $\rho_{2*}$ and half $\rho_{3*}$, and these two are different, the test statistic will be very large. In the notation of Theorem~\ref{thm:consist}, $\omega_2$ in \eqref{eq:omega2:def} is positive and by Theorem~\ref{thm:consist}, we have $\Th_n \gtrsim \nu_n \sqrt n $, and \scacp consistently rejects the null.

\begin{figure}[t]
    \centering
    \begin{tabular}{cc}
    \includegraphics[width=.49\textwidth]{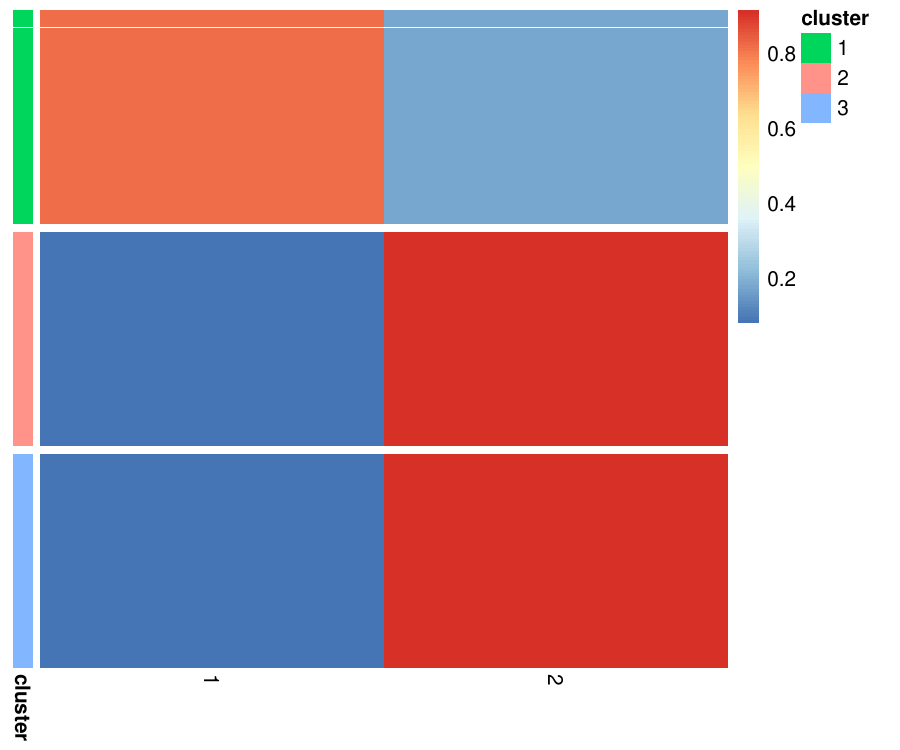} &
    \includegraphics[width=.49\textwidth]{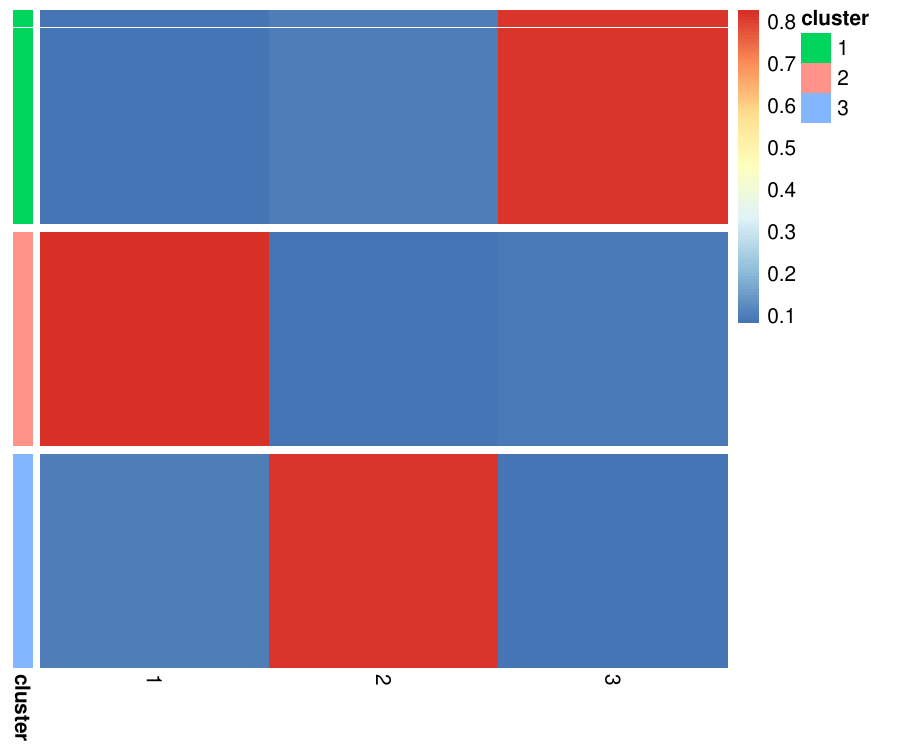}
    \\ 
    (a) \cac with $L = 2$ & (b) \cacp with $L = 3$ 
    \end{tabular}
    \caption{Heatmaps of the multinomial probability matrix, where the $i$-th row equals to $\rho_{z_i*}$. The left panel shows the case for \cac with $L = 2$ and the right, 
    \cacp with $L =3$. The column labels are 
    obtained
    by spectral clustering with $L =2$ and $L =3$, respectively.  
    The rows are ordered and labeled by the true clusters which is indicated by  
    the color bar to the left of the rows.}
    \label{fig:heatmap}
\end{figure}

We can observe the same phenomenon in practice. Consider an SBM on $n = 300$ nodes with $K_0 = 3$ equal-sized clusters, planted partitioned $B$ with out-in-ratio $q/p = 0.1$ and average degree $10$. The null hypothesis is $K = 2$. Applying the spectral clustering with $K=2$ clusters, clusters 2 and 3 are merged as a single cluster, and cluster 1 is mostly correctly recovered. The resulting label vector is set 
as the row label vector $\zh$, for both \cac and \cacp, and also the column label vector $\yh$ in \cac. When applying the spectral clustering with $L=K+1=3$,
the estimated labels are close to the true labels with only one node misclassified, and we set it as the column label vector $\yh$ in \cacp. Figure \ref{fig:heatmap} shows the heatmap of the $n \times L$ matrix $(\rho_{z_i *}, i= 1, \dots, n)$ for the above \cac (left side) and \cacp (right side). Because of the merging in $\yh$ with $L = 2$, the left heatmap shows the same multinomial probabilities for clusters 2 and 3. Whereas, $\yh$ with $L=3$ is close to the true label vector, hence the right heatmap shows distinct multinomial probabilities for the three clusters. This corroborates the discussion above.

\section{Proofs of the main results}

\subsection{Proof of Theorem ~\ref{thm:CCtest:null}}

\label{sec:proof:thm:CCtest:null}

Part~\ref{part:a:thm:CCtest:null} of the theorem
bounds the distance of the AC statistic, computed based on the true clusters and probabilities, to a standard normal. This part is a direct result of Proposition \ref{prop:Sn:grouped} below, whose proof uses the Esseen bound. 
Part~\ref{part:b:thm:CCtest:null} of the theorem follows from Proposition~\ref{prop:Thn:Tn} below by showing that replacing true clusters and probabilities with their estimated counterparts does not change the statistic much.

\begin{prop}\label{prop:Sn:grouped}
	Let $X_i \sim \mult(d_i, p_{k*}),\, i \in \Gc_k, k \in[K]$ be independent $L$-dimensional multinomial variables, with  probability vectors $p_{k*} = (p_{k\ell})$, and let %
	\begin{align*}
		Y_i := \sum_{\ell=1}^L \psi(X_{i\ell}, d_i  p_{g_i,\ell})%
		\quad \text{and}
		\quad S_n = \frac1{v_n} \sum_{i= 1}^{n}(Y_i-\ex[Y_i])
	\end{align*}
	where $v_n^2 := \sum_{i=1}^{n} \var(Y_i)$. Moreover, let $T_n = \frac1{\sqrt 2 \gamma_n} \big(\sum_{i=1}^n Y_i - \gamma_n^2\big)$ where $\gamma_n = \sqrt{n (L-1)}$.
	Let $\pul = \min_{k,\ell} p_{k\ell}$ and assume that $\min\{ h(d), L \} \ge 2$. Then, with $Z\sim N(0,1)$, we have
	\begin{align}
		\dkol\big(S_n, Z\big) &\le %
		\frac{55}{\pul^4\sqrt{L n}}, \label{eq:dkol:Sn} \\
		\dkol(T_n, Z) &\le  \dkol(S_n,Z)  + 
		\frac{\max\{1, \pul^{-1} - L -1\}}{\sqrt{\pi e}} 
		h(d)^{-1}.
		\label{eq:dkol:Tn}
	\end{align}
\end{prop}

\begin{prop}\label{prop:Thn:Tn}
	Recall that $\omega_n = \min_k \pi_k \dav^{(k)}$. Under the assumptions of Theorem~\ref{thm:CCtest:null}, for any nonnegative $u  \le (\pul/8)^2  n \omega_n$, 
	we have
	\begin{align}\label{eq:Thn:Tn:u:bound}
		\begin{split}
		\dkol(\Th_n, Z) &\le d_K(T_n, Z) +  6 K Le^{-u} + 2 \pr \bigl(\miss(\gh, g) \ge  \alphan \bigr)\\
		&\quad  + \frac{\sqrt L}{\pul} \Big[ \sqrt{\frac{8 u}{\omega_n}}  + 12 K \frac{u}{\sqrt n}
		+ 4 C_{3,p}L^{-1}\dmax \sqrt{Kn} \, \alphan\Big],
		\end{split}
	\end{align}

	where $C_{3,p}$ is as defined in Theorem~\ref{thm:CCtest:null}.
\end{prop}
To obtain~\eqref{eq:Thn:Tn:bound} in Theorem~\ref{thm:CCtest:null}, we take $u = \log (K \omega_n)$. To satisfy the condition of Proposition~\ref{prop:Thn:Tn}, we need $\log (K \omega_n) / \omega_n \le (\pul/8)^2 n$.
Since $\omega_n \ge L \ge 2$ by assumption and thus $2\log (K \omega_n) \ge 1$, we have
\begin{align*}
	6K Le^{-u} = 6 L /\omega_n \le \frac{6}{\pul}\sqrt{2\log(K\omega_n)L/\omega_n} =
	3\frac{\sqrt{L}}{\pul} \sqrt{\frac{8u}{\omega_n}},
\end{align*}
and the result follows.

\subsubsection{Proof of Proposition~\ref{prop:Sn:grouped}}

The proof relies on three lemmas.
Lemma~\ref{lem:mult_mean_var} establishes the mean and variance of the chi-square statistic. Lemma~\ref{lem:third_central_mom} is a general result on the growth rate of the third central moment of the empirical variance of a sum of independent variables. Applying this result to a  chi-square statistic, we can bound its third central moment by some constant. 

Plugging the moment estimates into the Esseen bound, we show that the sum of chi-square statistics, normalized by its mean and standard deviation, has a distribution close to standard normal. Finally, in Lemma~\ref{lem:dkol:affine:trans}, we show that by replacing the exact standard deviation in the normalized sum with a simpler form (to get $T_n$) we pay a small price in terms of the distance to the standard normal distribution.

We start by stating the three lemmas, whose proofs can be found in Appendix~\ref{sec:proof:lemmas:thm:CCtest:null}. Recall that $\psi(x,y) := (x-y)^2/y$.
\begin{lem}[Variance of the chi-square statistic]	\label{lem:mult_mean_var}
	Let $X = (X_1,\dots,X_L) \sim \mult(d, p)$, where $p = (p_1, \dots, p_L )$ is a probability vector and let $Y := \sum_{\ell=1}^L \psi(X_\ell, d p_\ell)$.
	Then, for $L \ge 2$,
	\begin{align*}
		\ex[Y] &= L-1,\\
		\var(Y) &= \Big( 1 - \frac1d \Big) 2(L-1) + \frac{1}{d} \Big( \frac{L}{h(p)} -L^2 \Big).
	\end{align*}
	In particular, $\var(Y) \ge ( 1 - 1/d) 2(L-1)$
\end{lem}
Note that we always have $L /h(p) \ge L^2$ since $\sum_\ell p_\ell = 1$. Hence, the variance of $Y$ is a convex combination of two nonnegative terms.  Furthermore, if $d\ge 2$, $\var(Y) \ge L-1$.

\begin{lem}[Central moment growth]\label{lem:mom:growth}
	Let $\{W_1, \dots, W_n\}$ be a sequence of i.i.d. zero mean random variables with finite moments of order 6, and let $X_n = \sum_{i=1}^n W_i$. Then, the third central moment of $X_n^2$ is $O(n^3)$:

	\begin{equation*}
		\ex\big|X_n^2 - \ex X_n^2\big|^3 \le  C_{W_1} n^3,
	\end{equation*}
	where $C_{W_1}$ is a  constant that only depends on the first 6 moments of $W_1$. For the case where $W_1 = \alpha (Z-p)$ with $Z \sim \ber(p)$ and $\alpha \in \reals$, one can take $C_{W_1} = 34.5\, \alpha^6 p (1-p)$.
	\label{lem:third_central_mom}
\end{lem}

\begin{lem}\label{lem:dkol:affine:trans}
   Let $T = \beta S + \alpha$ where $S$ is random variable and $\beta,\alpha \in \reals$ are constants, and let $Z \sim N(0,1)$. Then,
   \begin{align*}
   \dkol \big( T, Z \big)  \le \dkol \big( S, Z \big) + \frac{|\beta - 1|}{\sqrt{2 \pi e} \min\{|\beta|,1\}} + \frac{|\alpha|}{\sqrt{2\pi}}.
   \end{align*}
\end{lem}

\begin{proof}[Proof Proposition~\ref{prop:Sn:grouped}]
	By Esseen's bound for non-identically distributed summands~\cite{Shevtsova2010essen}, 
	\begin{align}\label{eq:Esseen:bound}
		\dkol(S_n, Z) \le \frac{C_0}{(v_n^2)^{3/2}} \sum_{i=1}^n \ex|Y_i-\ex[Y_i]|^3
	\end{align}
	for some constant $C_0 \in [0.41, 0.56]$.
	By Lemma~\ref{lem:mult_mean_var}, $\var(Y_i) \ge ( 1 - d_i^{-1}) 2(L-1)$. Then, using assumption $h(d) \ge 2$,
	\begin{align}\label{eq:vn:lower:bound}
		v_n^2 = \sum_{i=1}^n \var(Y_i) \ge n \big(1-h(d)^{-1}\big)2(L-1) \ge n(L-1).
	\end{align}

	Next, we bound the third central moment of $Y_i$.
	Let $Z_{i\ell} = (X_{i\ell} - d_i p_{g_i\ell}) / p_{g_i\ell}$. We have $Y_i = \sum_\ell p_{g_i\ell} Z_{i\ell}^2 / d_i$. We can write $Z_{i\ell} = \sum_{j=1}^{d_i}(W_j -p_{g_i\ell})/p_{g_i\ell}$, where $W_j \overset{i.i.d.}{\sim} \ber(p_{g_i\ell})$. By Lemma~\ref{lem:mom:growth}, $\ex| Z_{i\ell}^2 - \ex Z_{i\ell}^2|^3 \le  C_{p_{g_i\ell}} d_i^3$,
	for some constant $ C_{p_{g_i\ell}}$ that only depends on $p_{g_i\ell}$.
	Then, 
	\begin{align}\label{eq:3rd:mom:Yi}
		\begin{split}
			\ex|Y_i-\ex[Y_i]|^3 &= \ex\,\Big|\sum_{\ell=1}^{L}p_{g_i\ell} (Z_{i\ell}^2 - \ex Z_{i\ell}^2) / d_i \Big|^3 \\
			&\le \sum_{\ell=1}^{L} \frac{p_{g_i\ell}}{d_i^3} \ex\big| Z_{i\ell}^2 - \ex Z_{i\ell}^2 \big|^3 \le \sum_{\ell=1}^{L} p_{g_i\ell} \Big( \frac{34.5}{p_{g_i\ell}^6}  p_{g_i\ell}(1-p_{g_i\ell})\Big)
		\end{split}
	\end{align}
	where the first inequality is the discrete Jensen's inequality applied to convex function $x \mapsto |x|^3$, that is, $|\sum_\ell q_\ell x_\ell|^3 \le \sum_\ell q_\ell |x_\ell|^3$ for any $\{x_\ell\}$ and probability vector $q = (q_\ell)$. Combining~\eqref{eq:Esseen:bound},~\eqref{eq:vn:lower:bound} and~\eqref{eq:3rd:mom:Yi} gives
	\begin{align*}
		\sqrt{n}\, \dkol(S_n, Z) \le \frac{34.5 C_0}{ (L-1)^{3/2}} \frac1n \sum_{i=1}^n \sum_{\ell=1}^L \frac1{p_{g_i\ell}^{4}} \le 34.5 C_0 2^{3/2} \frac{1}{L^{1/2} \pul^{4}} \le \frac{55}{L^{1/2} \pul^{4}}
	\end{align*}
	using $p_{g_i\ell} \ge \pul$ for all $i$ and $\ell$, $L-1 \ge L/2$ and $C_0 \le 0.56$.

	\medskip
	To prove~\eqref{eq:dkol:Tn}, let $\beta_n = v_n /( \sqrt{2}\gamma_n)$, so that $T_n = \beta_n S_n$. By Lemma~\ref{lem:dkol:affine:trans},
	\begin{align*}
		\dkol \big( T_n, Z \big)  \le \dkol \big( S_n, Z \big) + \frac{\zeta_n}{\sqrt{2 \pi e}}, \quad \zeta_n := \frac{|\beta_n - 1|}{\min\{\beta_n,1\}}.
	\end{align*}
	It remains to bound $\zeta_n$. Let $d_{\Gc_k} = (d_i, i \in \Gc_k)$ and $n_k = |\Gc_k|$. By Lemma~\ref{lem:mult_mean_var},
	\begin{align*}
		v_n^2 &= \sum_{i=1}^n \big( 1 - d_i^{-1}\big) 2(L-1) +  d_i^{-1} \big(L h(p_{g_i*})^{-1} - L^2 \big)
		\\	
		&= \sum_k n_k \Big[ \big( 1 - h(d_{\Gc_k})^{-1}\big) 2(L-1) + h(d_{\Gc_k})^{-1} \big(L h(p_{k*}) - L^2 \big) \Big]
	\end{align*}
	where the second line follows by breaking the sum as $\sum_{i=1}^n(\cdots) = \sum_{k=1}^K \sum_{i \in \Gc_k}(\cdots)$ and using $n_k h(d_{\Gc_k})^{-1} = \sum_{i \in \Gc_k} d_i^{-1}$.
	To simplify, let $\alpha_k = h(d_{\Gc_k})^{-1}$. Then,
	\begin{align*}
		\beta_n = \frac{v_n}{\sqrt2 \gamma_n} = \Big(\sum_k \pi_k( 1 + \alpha_k b_k)\Big)^{1/2}, \quad b_k := \frac{L h(p_{k*})^{-1} - L^2}{2 (L-1)} - 1.
	\end{align*} 
	where $\pi_k = n_k / n$.	
	Since $L \le h(p_{k*})^{-1} \le \pul^{-1}$ and  $L/2 \le L-1$, we have
	\begin{align}\label{eq:bk:bound}
		0 \le b_k + 1 \le \frac{L (\pul^{-1}-  L)}{2 (L-1)} \le  \pul^{-1} - L.
	\end{align}
	Let $u = \sum_k \pi_k \alpha_k b_k$ and note that $\beta_n = \sqrt{ 1 + u}$. We have $ 0 < \sum_k \pi_k \alpha_k = h(d)^{-1} \le 1/2,$ by assumption. Moreover $b_k \ge -1$ for all $k$ from~\eqref{eq:bk:bound}. It follows that $u \ge -1/2$.
	
	If $u \ge 0$, then $\beta_n \ge 1$ and $\zeta_n = \beta_n - 1 \le \frac12 u$, using the inequality $\sqrt{1+x} \le 1 + x/2$ which holds for all $x \ge -1$.
	If $u < 0$, then $\beta_n \in (0,1)$, and
	\begin{align*}
		\zeta_n = \frac{1}{\beta_n} - 1 = \frac{1}{\sqrt{1 - |u|}}-1 \le \sqrt{2} |u|,
	\end{align*}
	using $|u| \le 1/2$ and the inequality $(1-x)^{-1/2} \le1+ \sqrt{2} x$ which holds for $0 \le x \le 0.77$. We have $|b_k| \le \max\{1, \pul^{-1} - L-1\}$, hence $\zeta_n \le  \sqrt{2} \max\{1, \pul^{-1} - L -1\} h(d)^{-1}$. The proof is complete.
\end{proof}

\subsubsection{Proof of Proposition~\ref{prop:Thn:Tn}}
Our strategy for proving Proposition~\ref{prop:Thn:Tn} is to show that $\Th_n$ is close to $T_n$ via a chain of  intermediate counterparts---namely $\Tt_n$ and $\Tt^*_n$---defined by replacing estimated clusters and probabilities with their true versions; see~\eqref{eq:Y:counterparts} and the subsequent paragraph. The fact that the chi-square statistic does not change very much when 
  the probabilities are slightly perturbed
(Lemma \ref{lem:G:expansion}) helps us show that $\Tt_n$ is close to $\Tt^*_n$ and $\Tt^*_n$ is close to $T_n$. 

It remains to show that $\Th_n$ is close to $\Tt_n$. Here, the probabilities defining the underlying chi-square statistics are the same (both estimated), but the clusters are different (estimated versus true). For this step, we use a uniform bound to avoid the dependence of the estimated clusters on the same data used to form the statistic. This is where we need $\dmax \alpha_n \sqrt{n} = o(1)$.

Once we show that $\Th_n$ is close to $T_n$ with high probability, we use the fact that for two random variables close to each other, their Kolmogorv distances to the standard normal distribution are also close (Lemma~\ref{lem:dkol:perturb}).

Throughout the proof, there will be a parameter $u$ and a derived parameter $\delta$ based on $u$. We set $u$ in the end to balance all the terms; see the discussion after the statement of Proposition~\ref{prop:Thn:Tn}. But in reading the proof, it could help to consider the case where all $d_i$ are of the same order say $d_i \asymp d$. Then $u$ will scale like $\log d$ and hence $\delta$ defined in~\eqref{eq:delta:in:prop:2} sclaes as $\delta =  O(\sqrt{\log d /(nd)})$.

We are now ready to give the detailed proof. First, we state the auxiliary lemmas.

\begin{lem}\label{lem:G:expansion}
	Let $x=(x_1,\dots,x_n) \in\reals^d$ and $y, y+v \in \reals \setminus \{0\}$, and  consider the function
	$G(v) = \sum_{i=1}^n d_i \psi(x_i, y+ v)$
	where $\{d_i\}$ are nonnegative and $\psi(s,t) = (s-t)^2/t$. Let $R = \sum_i d_i x_i - d_+ y$ where $d_+ = \sum_{i=1}^n d_i$, and assume further that $|v| \le |y|/2$. Then,
	\begin{align*}
		|G(v) - G(0) | \le \frac{2|v|}{|y|} \big[ G(0) + 2 |R| + |v| d_+ \big].
	\end{align*}
\end{lem}

\begin{lem}\label{lem:dkol:perturb}
	Let $\delta \in [0,1/2]$ and $\eps > 0$. Then, for any two random variables $\Th_n$ and $T_n$, and $Z \sim N(0,1)$
	\begin{align*}
		\dkol(\Th_n, Z) \le d_K(T_n, Z) + \frac12 (\delta + \eps) + \pr\big( |\Th_n - T_n| \ge \delta T_n + \eps\big).
	\end{align*}
\end{lem}

Next, we introduce the intermediaries between $\Th_n$ and $T_n$. Consider 
\begin{align*}
  Y(\{\Gc_k\}, \{p_{k\ell}\}) := \sum_{k=1}^K 
  	\sum_{i \in \Gc_k} 
  	\sum_{\ell = 1}^L \psi( X_{i\ell}, d_i p_{k\ell})
\end{align*}
and let
\begin{align}
\begin{split}\label{eq:Y:counterparts}
	\Yh &= Y(\{\Gh_k\}, \{\ph_{k\ell}\}), \\
	\Yt &= Y(\{\Gc_k\}, \{\ph_{k\ell}\}), \\
	\Yt^* &= Y(\{\Gc_k\}, \{\pt_{k\ell}\}), \\
	Y &= Y(\{\Gc_k\}, \{p_{k\ell}\})
\end{split}
\end{align}
where, for $k \in [K]$ and $\ell \in [L]$,
\begin{align}\label{eq:def:prob}
	\ph_{k\ell} = \frac{\sum_{i \in \Gh_k} X_{i\ell}}{ \sum_{i \in \Gh_k} d_i}, 
    \quad 	
    \pt_{k\ell} = \frac{\sum_{i \in \Gc_k} X_{i\ell}}{ \sum_{i \in \Gc_k} d_i}.
\end{align}
We define the corresponding $T$-statistics based on $Y$-statistics, via the relation $Y = \sqrt 2 \gamma_n T + \gamma_n^2$. For example,
\[
\Yh = \sqrt{2}\gamma_n \Th_n + \gamma_n^2
\]
and similarly for $\Tt_n$, $\Tt_n^*$ and $T_n$. 
The rest of proof is devoted to showing that $\Tt^*_n$ is close to $T_n$, $ \Tt_n$ is close to $\Tt^*_n$ and $\Th_n$ is close to $\Tt_n$. 

\paragraph{Controlling probability estimates}
We first show that the probabilities in~\eqref{eq:def:prob} are close to their true counterparts, $p_{k\ell}$. Let 
\begin{align*}
    X_{+\ell}^{(k)} = \sum_{i \in \Gc_k} X_{i\ell}, \quad d_{+}^{(k)} = \sum_{i \in \Gc_k} d_i,
\end{align*}
and
\begin{align}\label{eq:omega:half}
    \omhalf =  \Big(\sum_{k} (\pi_k \dav^{(k)})^{1/2}\Big)^2, \quad  \omone =  \sum_{k} \pi_k \dav^{(k)},
\end{align}   
and $\Delh_{k\ell} = \ph_{k\ell} - \pt_{k\ell}$ and $\Delt_{k\ell} = \pt_{k\ell} - p_{k\ell}$.

\medskip
First, we control $\Delt_{k\ell}$. Let
\begin{align}\label{eq:delta:in:prop:2}
	\delta_k := 2 (u/ d_{+}^{(k)})^{1/2}, \quad 
	\delta := \max_k \delta_k,
\end{align}
for $u \ge 0$ in the statement of the proposition, and consider the event
\begin{align}\label{eq:max:Delta:event}
	\mathcal B := \Big\{ \max_{\ell} |\Delt_{k\ell}| \le \delta_k, \; \forall k \in [K] \Big\}.
\end{align}
\begin{lem}\label{lem:bound:Delt}
	$\pr(\Bc^c) \le 2K Le^{-u}$ whenever  $u \le \min_k d_+^{(k)}$.
\end{lem}

Recalling the defintion of $\omega_n$ in~\eqref{eq:omegan:dmax:taud}, we note that $\min_k d_+^{(k)} = n\omega_n$. Then, $u \le (\pul/8)^2 n \omega_n \le \min_k d_+^{(k)}$ where the first inequality is by assumption. Hence, the condition of Lemma~\ref{lem:bound:Delt} holds and $\Bc$ is a high probability event. For the rest of the proof, we work on $\Bc$. 
Moreoever, $u \le (\pul/8)^2 n\omega_n \le (\pul/8)^2 d_+^{(k)}$ for all $k$, from which it follows that $\delta \le \pul/4$. Since on $\Bc$, we have $\max_{k,\ell} |\Delt_{k\ell}| \le \delta$, then for all $k, \ell$,
\begin{align}\label{eq:pt:lower:bound}
	\pt_{k\ell} \ge p_{k\ell} - \delta \ge \pul/2. 
\end{align}
Next, we control $\Delh_{k\ell}$. Recall that $\tau_d = \omega_n / \dmax$  as defined in~\eqref{eq:omegan:dmax:taud}. Let
\[
	\Mc_n := \{\miss(\zh, z) \le  \alphan\}.
\] 

\begin{lem}\label{lem:control:Delh}
	Assume that $\alphan \le \tau_d \,\pul /2$ and $\delta \le \pul/2$ and let
	\begin{align}
		 \delh := \frac{6}{\pul\, \tau_d} \alphan.
	\end{align}
	 Then, on $\Bc \cap \Mc_n$, we have $ |\Delh_{k\ell}| \le \delh \cdot \pt_{k\ell}$ for all $k$ and $\ell$.
\end{lem}

Since by assumption in Theorem~\ref{thm:CCtest:null}, $\alphan \le \pul/(8C_{3,p})$, which implies the assumption $\alphan \le \tau_d \,\pul /2$ in  Lemma~\ref{lem:control:Delh}. And since we established $\delta \le \pul/4$ previously, we can apply Lemma~\ref{lem:control:Delh} with $C_{3, p} = 6/(\pul \tau_d)$. Then we have $\delh \le \pul/8$, and furthermore 
\begin{align*}
    \ph_{k\ell}\ge \pt_{k\ell} - \delh \ge \pul/2 - \pul/8 \ge \pul/4.
\end{align*}

\paragraph{Controlling $\Tt_n^*$ in terms of $T_n$} 
Apply Lemma~\ref{lem:G:expansion} with $x_i = X_{i\ell}/ d_i$, $y = p_{k \ell}$ and $v = \pt_{k\ell} - p_{k \ell} = \Delt_{k \ell}$. The condition $|v| \le |y|/2$ of the lemma is satisfied on $\mathcal B$, as long as $\delta \le \pul/2$, which is the case as established earlier. Let
\begin{align*}
	G_{k \ell}(\Delt_{k\ell})= \sum_{i \in \Gc_k} d_i \psi(X_{i\ell} / d_i, p_{k \ell} + \Delt_{k\ell}) .
\end{align*}
We have $\Yt^* = \sum_{k,\ell} G_{k \ell}(\Delt_{k\ell})$ and $Y = \sum_{k,\ell} G_{k\ell}(0)$, hence
\begin{align*}
	|\Yt^* - Y| &\le  \sum_{k,\ell} |G_{k \ell}(\Delt_{k\ell})- G_{k\ell}(0)|  \\
	&\le 2 \sum_{k,\ell} \frac{|\Delt_{k\ell}|}{p_{k\ell}} \Big[ G_{k\ell}(0) + 2 |X_{+\ell}^{(k)} - d_{+}^{(k)} p_{k\ell}| + |\Delt_{k\ell}| d_{+}^{(k)} \Big] \\
	&= 2\sum_{k,\ell} \frac{|\Delt_{k\ell}|}{p_{k\ell}} \Big[ G_{k\ell}(0) + 3 |\Delt_{k\ell}| d_{+}^{(k)} \Big]
\end{align*}
where we have used $X_{+\ell}^{(k)} - d_{+}^{(k)} p_{k\ell} =  d_{+}^{(k)} \Delt_{k\ell}$ since $\pt_{k\ell} = X_{+\ell}^{(k)} / d_{+}^{(k)}$.  
By assumption $p_{k\ell} \ge \pul$ for all $k$ and $\ell$. Hence,
\begin{align*}
	\sqrt{2} \gamma_n |\Tt^*_n - T_n| = |\Yt^* - Y| &\le \frac{2}{\pul} \Big[\delta  \sum_{k,\ell} G_{k\ell}(0) +  3 L  \sum_k \delta_k^2 d_{+}^{(k)}\Big] \\
	&=\frac{2}{\pul} \Big[ \delta(\sqrt{2}\gamma_n T_n + \gamma_n^2) +  12 L K u \Big].
\end{align*}
Then, on $\Bc$, we have
\begin{align*}
	|\Tt^*_n - T_n|
	&\le \frac{2}{\pul} \Big[ \delta (T_n + \sqrt{nL/2}) + 12 K u \sqrt{L/n} \Big]
\end{align*}
using $\sqrt{n L /2} \le \gamma_n \le \sqrt{n L}$ which holds for $L\ge 2$. Since $2 \delta / \pul  \le 1/2$, we can apply Lemma~\ref{lem:dkol:perturb} to get
\begin{align}\label{eq:Tts:T}
	\dkol(\Tt^*_n, Z) &\le d_K(T_n, Z) + \frac1{\pul} \Big[ \delta(1+\sqrt{nL/2}) + 12 K u \sqrt{L/n} \Big] + \pr(\Bc^c). 
\end{align}

\paragraph{Controlling $\Tt_n$ in terms of $\Tt_n^*$}
We consider the event $\Bc \cap \Mc_n$ from now on. We apply Lemma~\ref{lem:G:expansion} with $x_i = X_{i\ell}/ d_i$, $y = \pt_{k \ell}$ and $v = \ph_{k\ell} - \pt_{k \ell} = \Delh_{k \ell}$. Condition $|v| \le |y|/2$ of the lemma is satisfied, as long as $\delh \le \pul/2$, which is the case as established earlier.
 Letting 
 \[F_{k\ell}(\Delta) := \sum_{i \in \Gc_k} d_i\psi(X_{i\ell}/d_i, \pt_{k\ell} + \Delta),\]  
 Lemma~\ref{lem:G:expansion} implies
\begin{align*}
	| F_{k\ell}(\Delh_{k\ell}) - F_{k\ell}(0) | &\le 
	\frac{2|\Delh_{k\ell}|}{ \pt_{k\ell}} \Bigl( F_{k\ell}(0) + 2 | X_{+\ell}^{(k)} - d_{+}^{(k)} \pt_{k\ell}| + |\Delh_{k\ell}| d_{+}^{(k)} \Bigr) \\
	&\le 
	\frac{4}{\pul} |\Delh_{k\ell}| \Bigl( F_{k\ell}(0) + \bigl[2 |\Delt_{k\ell}|  + |\Delh_{k\ell}| \bigr] d_{+}^{(k)} \Bigr), 
\end{align*}
where we have used $d_+^{(k)} \Delt_{k\ell} = X_{+\ell}^{(k)} - d_{+}^{(k)} \pt_{k\ell}$ and $\pt_{k\ell} \ge \pul/2$ on event $\Bc$; see~\eqref{eq:pt:lower:bound}. We have $\Yt = \sum_{k,\ell} F_{k\ell}(\Delh_{k\ell})$ and $\Yt_* =  \sum_{k,\ell} F_{k\ell}(0)$.
It follows that
\begin{align*}
	\sqrt 2 \gamma_n|\Tt_n - \Tt_n^*| = | \Yt - \Yt^*|
	& \le \sum_{k,\ell}  | F_{k\ell}(\Delh_{k\ell}) - F_{k\ell}(0) | \\
	&\le \frac{4}{\pul} \delh  \Bigl( \sum_{k,\ell} F_{k\ell}(0) + L  \sum_k \bigl[2  \delta_k  + \delh \bigr] d_{+}^{(k)} \Bigr) \\
	&= \frac{4}{\pul} \delh \Bigl( \Yt_* +  2L \sum_k  \delta_k d_{+}^{(k)} + L \delh d_+ \Bigr).
\end{align*}
Using $d_+^{(k)} = \dav^{(k)} \pi_k n$, and the defintions~\eqref{eq:omega:half} and~\eqref{eq:delta:in:prop:2}, we obtain
\begin{align*}
	\sum_k \delta_k d_+^{(k)} = 2 \sum_k ( u d_+^{(k)})^{1/2}  = 2 \sqrt{ n u \omhalf} .
\end{align*}
Noting that $d_+ = n\omone$, we have
\begin{align*}
		\sqrt 2 \gamma_n|\Tt_n - \Tt_n^*| \le  \frac{4}{\pul} \delh \Bigl(\sqrt 2 \gamma_n \Tt_n^* + \gamma_n^2 +  4 L\sqrt{ n u \omhalf}+ L \omone \delh  n \Bigr).
\end{align*}
Using $\sqrt{n L /2} \le \gamma_n \le \sqrt{Ln}$, we obtain, on $\Bc \cap \Mc_n$,
\begin{align*}
	|\Tt_n - \Tt_n^*| \le 
	\frac{4\delh}{\pul} 
	 \Bigl(   \Tt_n^* + \sqrt{n L /2} + 4 \sqrt{ L u \omhalf}  +  \omone \delh  \sqrt{L n} \Bigr).
\end{align*}
Recalling that $\delh/\pul \le 1/8$, Lemma~\ref{lem:dkol:perturb} gives
\begin{align}\label{eq:Tt:Tts}
\begin{split}
	\dkol(\Tt_n, Z) &\le \dkol(\Tt_n^*, Z) \;+ \pr(\Bc^c \cup \Mc_n^c)  \\
	&+ \frac{2\delh}{\pul} 
	\Bigl(  1+\sqrt{n L /2} + 4 \sqrt{ L u \omhalf}  +  \omone \delh  \sqrt{L n} \Bigr) .
\end{split}
\end{align}

\paragraph{Controlling $\Th_n$ in terms of $\Tt_n$}
Working $\Bc \cap \Mc_n$ and recalling $\ph_{k\ell}  \ge \pul/4$, 
\begin{align*}
	\sum_{\ell} \psi(X_{i\ell},  d_i \ph_{k\ell}) =  \frac{d_i}{\ph_{k\ell}} \sum_\ell (X_{i\ell}/d_i -  \ph_{k\ell})^2 \le 8 d_i / \pul,
\end{align*}
where we have used the following result:
\begin{lem}\label{lem:simplex:diam}
	$\max_{x,y \in \psim_L} \norm{x-y}^2 = 2$, where $\psim_L$ is the probability simplex in $\reals^L$. 
\end{lem}
Letting $\Hc_k = \Gc_k \Delta \Gh_k := (\Gc_k \setminus \Gh_k) \cup (\Gh_k \setminus \Gc_k)$,
\begin{align*}
	|\Yh  - \Yt| \le \sum_{k, \ell} 
	\sum_{i \in \Hc_k} 
	\psi( X_{i\ell}, d_i \ph_{k\ell}) \le 
	\frac{8}{\pul} \sum_k \sum_{i \in \Hc_k} d_i \le 	\frac{8 \dmax}{\pul} \alphan n
\end{align*}
using $\sum_{k} |\Hc_k| \le \alphan n$. Hence, on $\Bc \cap \Mc_n$, we have
\begin{align*}
	|\Th_n - \Tt| \le \frac{1}{\sqrt{2}\gamma_n } |\Yh - \Yt| \le   
	 \frac{8 \dmax \alphan \sqrt n}{\pul \sqrt{L}} 
\end{align*}
using $\sqrt 2 \gamma_n \ge \sqrt{n L}$. Applying Lemma~\ref{lem:dkol:perturb}, 
\begin{align}\label{eq:Th:Tt}
		\dkol(\Th_n, Z) &\le \dkol(\Tt_n, Z) +  \frac{4 \dmax  \alphan \sqrt n}{\pul \sqrt{L}}  + \pr(\Bc^c \cup \Mc_n^c).
\end{align}

\paragraph{Putting the pieces together}

Combining~\eqref{eq:Tts:T}, ~\eqref{eq:Tt:Tts} and \eqref{eq:Th:Tt}, we have
\begin{align*}
		\dkol(\Th_n, Z) &\le 
		 d_K(T_n, Z) \;+  \\
	 	&\qquad \frac1{\pul} \Big[ \delta(1+\sqrt{nL/2}) + 12 K u \sqrt{L/n} \Big] + \pr(\Bc^c) \;+ \\
	 	&\qquad \frac{2\delh}{\pul} 
		 \Bigl(  1+\sqrt{n L /2} + 4  \sqrt{L u \omhalf}  +  \omone \delh  \sqrt{L n} \Bigr) + \pr(\Bc^c \cup \Mc_n^c) \;+ \\
		&\qquad  \frac{4 \dmax  \alphan \sqrt n}{\pul \sqrt{L}}  + \pr(\Bc^c \cup \Mc_n^c)
		. 
\end{align*}
Using $1+\sqrt{nL/2} \le 2 \sqrt{nL/2}$ and the union bound,
\begin{align*}
	 &\dkol(\Th_n, Z)  - 	d_K(T_n, Z) \le \\
	 &\qquad 
	  \frac{\sqrt L}{\pul} \cdot \Bigl[ 2 \delh
	  \Bigl(  \sqrt{2n} + 4 \sqrt{u \omhalf}  +  \omone \delh  \sqrt{n} \Bigr) + \frac{4 }{L}\dmax  \alphan \sqrt n \,+ \\
		&\qquad \delta\sqrt{2n} + 12 K \frac{u}{\sqrt n} \Bigr] +  3 \pr(\Bc^c) + 2 \pr(\Mc_n^c).
\end{align*}
Substituting $\delh = C_{3,p}\alphan$, where
$C_{3,p} := 6 /(\pul \tau_d)$,
and $\delta \sqrt{n}  = 2 \sqrt{u / \omega_n}$, and the upper bound on $\pr(\Bc^c)$ from Lemma~\ref{lem:bound:Delt}, we obtain after some rearranging,
\begin{align*}
    &\dkol(\Th_n, Z)  - 	d_K(T_n, Z) \le \\
	 &\qquad \frac{\sqrt L}{\pul} \cdot  \Bigl[\sqrt\frac{8u}{\omega_n} + 12K\frac{u}{\sqrt{n}} + 2\alpha_n\zeta_n\Bigr] + 6 K Le^{-u} + 2 \pr \bigl(\miss(\gh, g) \ge  \alphan \bigr)
\end{align*}
where
\begin{align}\label{eq:zetan:1}
	\zeta_n =  \bigr(\sqrt2 C_{3,p} + 2L^{-1} \dmax + \omone C_{3,p}^2 \alphan\bigl) \sqrt n  + 4	C_{3,p}  \sqrt{u\omhalf}.
\end{align}
Note that by the Cauchy--Schwarz inequality, 
\begin{align*}
    \sum_{k =1}^K  (\pi_k \dav^{(k)})^{1/2} \le \Bigl(\sum_{k =1}^K  \pi_k\Bigr)^{1/2}  \Bigl(\sum_{k =1}^K   \dav^{(k)}\Bigr)^{1/2} \le \sqrt{K\dmax}
\end{align*}
Therefore, $\omhalf \le K \dmax$. Moreover, $\omega_n \vee \omega_{n,1} \le \dmax$. By the assumptions 
$\sqrt{2} C_{3,p} \le 2 L^{-1} \dmax$, 
$\alphan \le 2/(C_{3,p}^2L)$ and  $u  \le (\pul/8)^2  n \omega_n$. Plugging these bounds into~\eqref{eq:zetan:1}, we obtain
\begin{align*}
    \zeta_n &\le  6\dmax\sqrt{n}/L + \pul C_{3,p} \dmax \sqrt{nK}/2 \\
    &\le 2 C_{3,p}L^{-1}\dmax \sqrt{Kn}
\end{align*}
where the last inequality is due to $K \ge 1$, $C_{3, p} \ge 6$ and $\pul \le L^{-1}$. The result follows.

\subsection{Proofs of Theorems \ref{thm:null:dist:ncac} and \ref{thm:consist}}

We start by setting up notation and deriving some preliminary results that are common to both proofs. Throughout, $K_0$ denotes the true number of communities. Recall that $S_1 \subset [n]$ is determined by including any element of $[n]$ with probability $1/2$, and $S_2 = [n]\setminus S_1$. Then, 
\begin{align}\label{eq:di:U}
	d_i := \sum_{j \in S_1} A_{ij} = \sum_{j=1}^n A_{ij} U_j
\end{align}
where $U_j = 1\{j \in S_1\}, j \in [n]$ is an independent $ \text{Ber} (1/2)$ sequence. We often work conditioned on $S_1$, $A_{S_1 S_1}$ and $(d_i, i \in S_2)$. Let $\Fc$ be the $\sigma$-field generated by these variables:
 \begin{align}\label{eq:Fc:def}
 	\Fc = \sigma\big(S_1, A_{S_1,S_1}, (d_i, i \in S_2)\big),
 \end{align}
and let $\ex^{\Fc}$ and $\pr^{\Fc}$ denote the expectation and probability operators, conditioned on $\Fc$. We assume without loss of generality that the community detection algorithm is nonrandomized, so that conditioned on $\Fc$,  $\yh$ is fixed. (Otherwise, we add the independent source of randomness used by the algorithm to $\Fc$.) The idea in both proofs is to first condition on $\Fc$ and derive bounds given that the parameters are all fixed. Then, we can leave out $\Fc$ thanks to the fact that the parameters are all bounded with high probability as we show below.

\paragraph{Controlling $\rho_{k\ell}$, $d_i$ and $|\Gc_k|$} 
Recall the definition of $X_{i\ell}(\yh)$ in \eqref{eq:X:def}. Then, conditioned on $\Fc$, for $i \in S_2$, $X_{i*}(\yh) \sim \mult(d_i, \rho_{z_i*})$ independently and thus $\ex^\Fc[X_{i\ell}(\yh)] = d_i\rho_{z_i\ell}$ where $z_i \in [K_0]$.  We can obtain a lower bound on $\rho_{k\ell}$ as follows:
\begin{align}\label{eq:rho:min:bound}
   \rho_{k\ell} = \frac{\sum_{h = 1}^{K_0}B_{kh}^0 R_{h\ell}}{\sum_{\ell'=1}^{L}\sum_{h = 1}^{K_0}B_{kh}^0 R_{h\ell'}} \ge \tau_{B} \tau_\theta \frac{\sum_{j \in S_1}1\{\yh_j =\ell\}}{|S_1|} \ge 
\tau_{B} \tau_\theta\tau_{0} = \tau_\rho,
\end{align}
where the last inequality is due to the stability Assumption~\ref{assum:DCSBM}\ref{assume:stability}. We have
\begin{align}\label{eq:rhou:taur}
	\rhou :=  \min_{k, \ell} \rho_{k\ell} \ge \tau_\rho.
\end{align}
Letting $\ds_i = \ex[d_i]$, we can write
\begin{align}\label{eq:ah:def}
\ds_i = \frac12 \sum_{j=1}^n \ex[A_{ij}] = \frac12 \theta_i a_{z_i}, \quad 
\text{where} \quad a_h := \sum_{k=1}^{K_0} B_{hk} \sum_{j \in \Cc_k} \theta_j
\end{align}
for all $h \in [K_0]$. Let us derive some bounds on $\ds_i$. Recalling that $\thetamax = 1$, 
\begin{align}\label{eq:ah:lower:bound}
a_h \ge \tau_\theta \thetamax n_k \sum_k B_{hk} = \tau_\theta  n_k \frac{\nu_n}{n} \norm{B^0_{h*}}_1 \ge \tau_\theta \nu_n \tau_\Cc \min_{h'}\norm{B^0_{h'*}}_1.
\end{align}
Using the definition of $C_1$ in~\eqref{eq:c1:taua:taurho:def}, we obtain
\begin{align}\label{eq:ds:lower:bound}
\ds_i \ge \frac12 C_1  \nu_n, \quad \forall i.
\end{align}
Let $\amax = \max_h a_h$. Since $|\Cc_k| \le n$, we have $\amax \le \nu_n \cdot\max_h \norm{B^0_{h*}}_1$. 
Combining with assumption \eqref{eq:ah:lower:bound}, we obtain
\begin{align}\label{eq:ah:amax:ineq}
a_h \ge \tau_a \amax, \quad \tau_a :=  \tau_\theta \tau_B \tau_\Cc.
\end{align}
Since $\infnorm{B^0} = 1$ by assumption, we have %
\begin{align}\label{eq:ds:upper:bound}
 \ds_i \le \frac12 K_0 \,\nu_n, \quad \forall i.
\end{align}

Let $\Cc_k = \{i \in [n]:\; z_i = k\}$ be the true community $k$, and $n_k = |\Cc_k|$. We also let $\Gc_k := \{i \in S_2:\; z_i = k\} = \Cc_k \cap S_2$. 
Consider the event:
\begin{align}
\Ac  = \Big\{|\Gc_k| \in [0.4 n_k, 0.6n_k],\ \forall k \in [K_0]\Big\} \cap \Big\{ d_i \in \Big[\frac{\ds_i}2, \frac{3\ds_i}2 \Big], \; \forall i \in [n] \Big\}.
\label{eq:A}
\end{align}
Note that $\Ac$ is deterministic conditioned on $\Fc$. The next lemma guarantees that this event holds with high probability: 
\begin{lem}\label{lem:A:event:prob}
$\pr(\Ac^c) \le 7 n^{-1}$ if $\log n / n \le \frac{3C_1}{400} \wedge \frac{\tauc}{300}$ and $\log n / \nu_n  \le 10^{-3} C_1$. 
\end{lem}
Combining~\eqref{eq:ds:lower:bound},~\eqref{eq:ds:upper:bound} and the definition~\eqref{eq:A}, we have on $\Ac$,
\begin{align}\label{eq:di:interval}
	\frac14 C_1 \nu_n \le d_i \le \frac{3}{4}  K_0 \, \nu_n,\quad \forall i \in [n].
\end{align}

\paragraph{From $S_2$ to $S_2'$}
Recall that in \scacf, we first use random sampling to get $S_2$ and then use quantile filtering 
 in each estimated cluster to get $S_2'$. 
Recall that in each $\Gh_k = \{i\in S_2: \zh_i = k\}$ we keep nodes with degrees at least that of the $\sigma$-th quantile of $\{d_i: i \in \Gh_k\}$ to form $\Gh'_k$.
It follows that $|\Gh'_k| \ge (1-\sigma) |\Gh_k|$ and $| S_2'| \ge (1-\sigma)|S_2|$, where $S'_2 = \bigcup_{k = 1}^ {K}\Gh'_k \subset S_2$.

\smallskip
Now let us get a lower bound on the size of $\Gc_k' := \Cc_k \cap S_2' = \Gc_k \cap S_2'$, which is used in the proofs of Theorem~\ref{thm:null:dist:ncac} and~\ref{thm:consist}. Given an estimated label vector $\zh$ and a true label vector $z$, consider the event
\begin{align*}
    \Mc_n := \{\miss(\zh, z) \le  \alphan\}.
\end{align*}
On $\Mc_n$, $|\Gh_k \,\Delta\, \Gc_k| \le \alphan n$, from which we have
\begin{align*}
	|\Gc_k \cap \Gh_k| \ge |\Gc_k| - \alphan n, \quad |\Gh_k| \le |\Gc_k| + \alphan n.
\end{align*}
Furthermore, on $\Ac$, $|\Gc_k| \ge 0.4 \tau_\Cc n$.
Therefore, on $\Mc_n \cap \Ac$,
\begin{align}
	|\Gc_k'| 
	\ge |\Gc_k \cap \Gh_k'| 
	&= |\Gc_k \cap \Gh_k| - |\Gh_k \setminus \Gh_k'| \notag \\
	&\ge |\Gc_k| - \alphan n  - \sigma|\Gh_k| \notag \\
	&\ge |\Gc_k| - \alphan n  - \sigma(|\Gc_k| + \alphan n) \notag \\
	&\ge (1-\sigma)0.4 \tau_\Cc n - (1+\sigma)\alphan n \notag \\
	&\ge 0.2(1-\sigma)\tau_\Cc n = c_1 K_0 n \label{eq:Gc:prime:lower:bound}
\end{align}
using the assumption 
$\alphan \le \frac{\tau_\Cc }{5}\frac{1-\sigma}{1+\sigma}$  and the definition of $c_1$ from~\eqref{eq:const:defs}.

\subsubsection{Proof of Theorem~\ref{thm:null:dist:ncac}}\label{app:proof:2}

The proof has two parts. In the first part, we get an upper bound on the Kolmogorov distance conditional on the $\sigma$-field $\Fc$ defined in~\eqref{eq:Fc:def},
resulting from combining Lemma~\ref{lem:dkol:cond} and Theorem~\ref{thm:CCtest:null}. The second part is to show that on event $\Ac \cap \Mc_n$, the random quantities are bounded by constants.

For a random variable $Y$, let $\law(Y \mid \Fc)$ be the law of $Y$ conditioned on $\Fc$ and let
\begin{align}
	\dkol\big(\law(Y \mid \Fc), Z\big) = \sup_{t \in \reals} | \pr(Y \le t \mid \Fc) - \pr(Z \le t) |.
\end{align}
\begin{lem}\label{lem:dkol:cond}
	For any random variables, $Y$ and $Z$, any event $\Bc$ and any $\sigma$-field $\Fc$, we have
	\begin{align*}
		\dkol(Y, Z) &\le \ex \big[ \dkol\big(\law(Y  \mid \Fc), Z\big) \big], \\
		 |\dkol(Y,Z) - \dkol(Y 1_\Bc,Z)| &\le \pr(\Bc^c).
	\end{align*}
\end{lem}

Applying Lemma~\ref{lem:dkol:cond} and since the Kolmogorov distance is bounded above by 1, we obtain
\begin{align}
\label{eq:dkol:decom}
\begin{split}
    \dkol(\Th_n, Z) &\le 
	\ex\big[ \dkol\big(\law(\Th_n \mid \Fc), Z\big) \big] \\
    &\le \ex
	[\dkol\big(\law(\Th_n \mid \Fc), Z\big) \cdot 1_{\Ac \cap \Mc_n} ] + \pr(\Ac^c) + \pr(\Mc_n^c).
\end{split}
\end{align}
Conditioned on $\Fc$, $X_{i\ell}(\yh) \sim \mult(d_i, \rho_{z_i*})$  for $i \in S_2$, as discussed in~\eqref{eq:Xis:distn}. Recall that $\Gc_k' = \Gc_k \cap S_2'$. Then, we can apply Theorem~\ref{thm:CCtest:null}, assuming its conditions hold, to the submatrix $(X_{i\ell}(\yh): i \in S'_2, \ell \in [L])$, with estimated and true labels $\gh = (\zh_i, i \in S'_2)$, $g = (z_i, i \in S'_2)$. It follows that the conditional law of $\Th_n$ given $\Fc$ satisfies 
\begin{align*}
    \begin{split}
	\dkol\big(\law(\Th_n \mid \Fc), Z\big) &\le \frac{C_{1,\rho}}{\sqrt{L |S'_2|}} + \frac{ C_{2,\rho}}{h(d)} \\
	&+ 12\frac{\sqrt{L}}{\rhou}\left( \sqrt{\frac{\log (K_0\omega_n)}{\omega_n}} + \frac{K_0 \log (K_0\omega_n) }{\sqrt{|S'_2|}}\right.\\
	&+\left.\frac{C_{3,\rho}}{3L}\dmax \sqrt{K_0|S'_2|\, }\atn \right) + 2\pr\big(\miss(\gh, g) > \atn \given \Fc\big)
	\end{split}
\end{align*}
for any $\atn \in [0,1]$, 
where $h(d)$ is the harmonic mean of $(d_i, i \in S'_2)$, $\omega_n = \min_{k} \pi_k \dav^{(k)}$ with $\dav^{(k)}$ the arithmetic mean of $(d_i, i \in \Gc'_k)$, $\pi_k = |\Gc'_k| / |S'_2|$, $d_{\max} = \max_{i \in S_2'} d_i$  and $\tau_d = \omega_n/\dmax$. The constants $C_{1,\rho}$, $C_{2,\rho}$ and $C_{3,\rho}$ depend on the $\rho$ matrix defined in~\eqref{eq:rho:def}.  Note that $h(d)$, $\omega_n$ and $\dmax$ although in general random, are deterministic conditioned on $\Fc$. 

Now we bound the above distance on $\Ac$, and without further specification the following results are all stated on $\Ac$. 
Recall from~\eqref{eq:Gc:prime:lower:bound} that $|\Gc'_k| \ge c_1K_0n$ on event $\Mc_n \cap \Ac$. We also have $|S'_2| \le 0.6 n$. It follows that 
$\pi_k \in [c_1 K_0/0.6, 1]$. On $\Ac$, $\dav^{(k)}$, $\dmax$ and $h(d)$ satisfy the same upper and lower bounds as $d_i$ in~\eqref{eq:di:interval}. 
Combined with the bounds on $\pi_k$, we have 
\begin{align*}
    \frac{5}{12}c_1C_1K_0 \nu_n &\le \omega_n \le \frac{3}{4}  K_0 \, \nu_n, \quad
    \tau_d \ge  \frac{5}{9}c_1 C_1, \\
	 h(d) &\ge \frac{C_1 \nu_n}{4}, \quad 
	\dmax \le \frac{3}{4}  K_0 \, \nu_n.
\end{align*}
Recalling that $\rhou \ge \tau_\rho$ from~\eqref{eq:rhou:taur}, 
\begin{align}\label{eq:C3rho:C2}
    C_{3, \rho} =\frac{6}{\rhou \, \tau_d}  \le \frac{54}{5c_1 C_1 \tau_\rho}=: C_2.
\end{align}
Next we bound $\miss(\gh, g)$ in terms of $ \miss(\zh, z)$.
To bound the probability of the missclassification rate, we first note that 
\begin{align*}
    |S'_2| \miss(\gh, g) \le  \sum_{i \in S'_2} 1\{g_i \ne \gh_i\} &\le 
	\sum_{i =1}^{n} 1\{z_i \ne \zh_i \} = n \miss(\zh, z).
\end{align*}
Let $\sigb = 1-\sigma$.
Furthermore, $|S'_2|/n \ge 0.4\sigb$ on $\Ac$. Set $\atn = \alphan/ (0.4\sigb )$. Then, 
\begin{align*}
    \pr\big(\miss(\gh, g)  > \atn \given \Fc \big) \cdot 1_{\Ac \cap \Mc_n} &\le \pr\Big(\frac{n}{|S'_2|}\miss(\zh, z)  >  \atn  \given \Fc\Big) \cdot 1_{\Ac}\\
    & = \pr\Big(\Big\{\miss(\zh, z)  > \frac{|S'_2|}{n} \atn \Big\} \cap \Ac \given \Fc\Big)\\
    &\le \pr(\miss(\zh,z) > \alphan \given \Fc).
\end{align*}
Note that the equality is due to $\Ac$ being deterministic on $\Fc$. Using $\log (K_0 \omega_n) \le \log  \big((3/4)K_0^2 \nu_n\big) =: \beta_n$ and $|S'_2| \in [0.4\sigb, 0.6]n$, we obtain
\begin{align*}
	\ex\Big[\dkol\big(\law(\Th_n \mid \Fc), Z\big) 
	\cdot 1_{\Ac \cap \Mc_n} \Big] &\le 
	\frac{55 \taur^{-4}}{\sqrt{L \cdot 0.4\sigb n}} + \frac{ C_{2,\rho}}{C_1 \nu_n /4} \\
	&+ 12\frac{\sqrt{L}}{\rhou}\left( \sqrt{\frac{ \beta_n}{5 c_1C_1K_0 \nu_n /12}} + \frac{K_0 \beta_n }{\sqrt{ 0.4\sigb n}} \right.\\ 
	&\left. +\frac{C_2}{3L} \cdot\frac{3}{4}  K_0 \, \nu_n \cdot \sqrt{K_0 \cdot 0.6n}\cdot \frac{\alpha_n}{0.4 \sigb}\right)\\
	&+ 2\pr(\miss(\zh, z) > \alphan).
\end{align*}
Simplifying the above and plugging into~\eqref{eq:dkol:decom}, we have
\begin{align*}
    \dkol(\Th_n, Z) &\le  \frac{C_3}{\sqrt{\sigb L n}} + \frac{ C_4 }{C_1 \nu_n}\\
    &+ \frac{19 \sqrt{L}}{\tau_\rho}\left(\frac{1}{\sqrt{c_1C_1}} \sqrt{\frac{\beta_n }{K_0 \nu_n}} + \frac{K_0 \beta_n}{\sqrt{\sigb n}} + C_2 \frac{K_0^{3/2}}{\sigb L}\nu_n \sqrt{n}\, \alphan\right) \\
    &+ 3\pr(\miss(\zh, z) > \alphan)
\end{align*}
where $C_3 = 94\taur^{-4}$ and $C_4 = 4(\pi e)^{-1/2} \max\{1, \tau_\rho^{-1} - L -1\}$. Note that we have absorbed $\pr(\Ac^c) \le 7n^{-1} \le 7 \taur^{-4}/\sqrt{\sigb L n}$ into the first term above. The above is the desired bound.

Now let us simplify all assumptions. We  need to consider the assumptions in Lemma~\ref{lem:A:event:prob}, assumption $\alphan \le \frac{\tau_\Cc }{5}\frac{1-\sigma}{1+\sigma}$ and the conditions of Theorem~\ref{thm:CCtest:null} which  hold on $\Ac$ if 
\begin{align}\label{eq:cond:list:thm:null}
\begin{array}{lcl}
    \frac{1}{4} C_1 \nu_n \ge \max\{2, \frac{C_2 L}{\sqrt{2}}\} & \quad &
    \frac{5}{12}c_1C_1K_0\nu_n \ge  L \ge 2
    \\
     \beta_n/ (\frac{5}{12}c_1C_1K_0\nu_n) \le (\taur/8)^2 n & \quad &
     \alphan \le \frac{\tau_\rho}{8C_2} \wedge \frac{2}{LC_2^2}.
\end{array}
\end{align}
The first condition in~\eqref{eq:cond:list:thm:null} can be simplified since $\frac{C_2 L}{\sqrt{2}} \ge 2$. 
The assumptions on $\nu_n$ in~\eqref{eq:cond:list:thm:null} and in Lemma~\ref{lem:A:event:prob} can be summarized as
\begin{align*}
    \nu_n &\ge \frac{1}{C_1} \max\Bigl\{\frac{12L}{5c_1K_0},  \,
		2\sqrt{2}C_2L, \,
		10^3 \log n,\,
		\frac1{(\frac{\tau_\rho}{8})^2 \frac{5}{12}c_1K_0} \frac{\beta_n}{n}\Bigr\}.
\end{align*}
This can be further simplified to~\eqref{assum:scaling:b}, since
\begin{align*}
    2\sqrt{2}C_2L \Big/\frac{12L}{5c_1K_0} =
    9\sqrt{2}\frac{K_0}{C_1\taur} \ge 1.
\end{align*}
The assumptions on $\alphan$ can be simplified to~\eqref{assum:scaling:c}, since
\begin{align*}
    \frac{2}{LC_2^2} \Big/ \frac{\tau_\rho}{8C_2} = \frac{16}{C_2 L \taur} = \frac{8\sigb \tauc C_1}{27LK_0} \le 1. 
\end{align*}

The assumptions on $\log n / n$ in Lemma~\ref{lem:A:event:prob} are satisfied under~\eqref{assum:scaling:a} since $C_1 \le \tauc$. Finally, we note that $C_2$, defined in~\eqref{eq:C3rho:C2}, can be replaced with its upper bound $11/(c_1 C_1 \taur)$.

\subsubsection{Proof of Theorem~\ref{thm:consist}}\label{app:proof:cons}
    The proof has two main components. 
    First, we show that when $K < K_0$, there exists a mixed estimated community, say $\Ch_k$, that contains large pieces of two distinct true communities, say $\Cc_1$ and $\Cc_2$. This holds for any underfitted set of labels (i.e., with $K < K_0$ communities) regardless of what community detection algorithm is used.
    
    Second, we show that the chi-square statistic is large over that mixed cluster ($\Ch_k$). This is done by first showing that the estimated parameter vector $\rhoh_{k*}$ is close to a mixture of the true parameter vectors $\rho_{r*}, r \in [K_0]$, which we refer to as $\rhob_*$. Since $\Ch_k$ contains large pieces of true communities $\Cc_1$ and $\Cc_2$, the weights of $\rho_{1*}$ and $\rho_{2*}$ in the mixture forming $\rhob_{*}$ will be bounded away from zero.
    On the other hand, in forming $\Th_n$ over $\Ch_k$, we effectively compare the row $X_{i*}$ to $d_i \rhob_*$. However, $X_{i*}$ is close to either $d_i \rho_{1*}$ or $d_i \rho_{2*}$ depending on which of the two true chunks, $i$ belongs to. Since $\|\rho_{r*} - \rhob_{*}\|^2$ is bounded away from zero for $r = 1,2$, this leads to the chi-square statistic being large over $\Ch_k$. 
    
    
    \smallskip
    We prove the result first, assuming $\sigma = 0$ in Algorithm~\ref{alg:snac} (i.e., no quantile filtering), so that $S'_2 = S_2.$. At the end, we will show how the result can be extended to include $\sigma > 0$.
    %
    %
    %
    Recall that $\Ch_k = \{i: \zh_i = k\}$ and $\nh_k = |\Ch_k|$. For $r \in [K_0]$, a true community $\Cc_{r}$ is partitioned into $\Ch_{k,r} = \{i:\zh_i = k,  z_i = r\}$, $k \in [K]$. With such partition, for each $r \in [K_0]$, there exists $k_{r} \in [K]$  such that $|\Ch_{k_r,r} \cap S_2| \ge |\Cc_{r} \cap S_2|/K$. Since $K < K_0$, there are $r_1, r_2 \in [K_0]$ such that $r_1 \neq r_2$ and $k_{r_1} = k_{r_2} =: \kh$. Note that $\kh$ is random and potentially dependent on $A$. Without loss of generosity, assume that $r_1 = 1$ and $r_2 = 2$. Therefore, $\Ch_{\kh}$ contains ``large'' pieces $\Ch_{\kh,1}$ and $\Ch_{\kh,2}$ of two different true communities $1$ and $2$. 
    We will show below that this further guarantees that $\Gh_{\kh} = \Ch_{\kh} \cap S_2$ has a substantial size. 
    First, recalling that $\Gc_{r} = \Cc_{r} \cap S_2$, on event $\Ac$, we have
    \begin{align} \label{eq:Gk:bound}
    |\Ch_{\kh,1}\cap S_2| \ge |\Cc_{1} \cap S_2|/K \ge |\Gc_{1}|/K_0 \ge (0.4\tau_\Cc /K_0) n \ge c_1 n , 
    \end{align}
    where $c_1 = (1-\sigma)\frac{\tau_\Cc}{5K_0}$ as in \eqref{eq:const:defs}.  The same bound holds for {$|\Ch_{\kh,2}\cap S_2|$. 
    Therefore, we have $|\Gh_{\kh}| \ge |\Ch_{\kh,1}\cap S_2| + |\Ch_{\kh,2}\cap S_2| \ge 2c_1 n$. We will focus on $\Gh_{\kh}$ in the rest of the argument. 
    
    \medskip
    Let $\bigcup_{r=1}^{K_0} \Tch_{r}$ be the disjoint partition of $\Gh_{\kh}$ into the true communities, 
    with $\Tch_r = \{i \in S_2:\zh_i = k, z_i = r\} = \Ch_{k,r}\cap S_2$. Some $\Tch_r$ might be empty, but we can safely ignore them 
    and focus on the two big pieces $\Tch_1$ and $\Tch_2$, that are guaranteed by the earlier argument. Since $\Tch_r \subset \Cc_r \cap S_2$, for any $i \in \Tch_r$, we have $\ex^\Fc[\xi_{i\ell}] = \rho_{r \ell}$, where $\rho_{r\ell}$ is defined based on \eqref{eq:Xis:distn} and \eqref{eq:rho:def}. Let us define
    \begin{align*}
    	\alh_r := \sum_{i \in \Tch_r} d_i, \quad \beh_{r} := \frac{\alh_r}{ \alh_+ }, \quad \rhob_\ell := \sum_{r=1}^{K_0} \beh_r \rho_{r\ell},
    \end{align*}
    where $\alh_+ = \sum_r \alh_r = \sum_{i \in   \Gh_\kh} d_i$. Note that $\alh_r = 0$ if $\Tch_r$ is empty. We also note that on $\Ac$, we have $\alh_+ > 0$, hence the division by $\alh_+$ is valid. In fact, using $d_i \ge  C_1 \nu_n /4$ from~\eqref{eq:di:interval} and $|\Gh_{\kh}| \ge 2 c_1 n$, we have
    \begin{align}\label{eq:alh:plus:lower:bound}
        \alh_+ \ge  c_1 C_1 n  \nu_n /2.
    \end{align}
    Consider the event
    \begin{align}
    	 \Ec:= \bigl\{ \max_{r, \ell}\, \max_{i \in \Tch_r } |\xi_{i\ell} - \rho_{r\ell} | \le \eps_n
    	    \bigr\}, \quad \eps_n := 4 \sqrt{\frac{\log n}{C_1 \nu_n}}.
    \end{align}
    The following lemma shows that $\Ec$ holds with high probability and we work on $\Ec$ for the rest of the proof. 
   	\begin{lem}\label{lem:E:event:prob}
   		$\pr(\Ec^c \cap \Ac) \le 2 L n^{-1}$ whenever $\frac{\log n}{\nu_n} \le C_1/4$.
    \end{lem}
    The assumption of Lemma~\ref{lem:E:event:prob} holds under the stronger assumption $\frac{\log n}{\nu_n} \le C_1\tau_\rho^2/ 64$ that we made in the statement of the theorem.
    We next show that $\rhoh_{\kh\ell}$ is close to $\rhob_\ell$ for all $\ell \in [L]$. We have
    \begin{align*}
    	\Big| \sum_{i \in \Gh_\kh} X_{i\ell}(\yh) - \sum_r \alh_r \rho_{r\ell}\Big| &=
    	\Big| \sum_r\sum_{i \in \Tch_r} d_i \xi_{i\ell} - \sum_r  \sum_{i \in \Tch_r} d_i \rho_{r\ell}\Big| \\
    	&\le  \sum_r\sum_{i \in \Tch_r} d_i \big| \xi_{i\ell} - \rho_{r\ell}\big| \le \eps_n \sum_r \alh_r = \eps_n \alh_+.
    \end{align*}
    Dividing by $\alh_+$ and recalling the definition of $\rhoh_{k\ell}$ in \eqref{eq:rhoh:def}, we get 
    \begin{align} \label{eq:rho:diff}
    	\big|\rhoh_{\kh\ell} - \rhob_\ell \big| = 
    	\Big| \frac{\sum_{i \in  \Gh_\kh} X_{i\ell}(\yh)}{ \sum_{i \in  \Gh_\kh} d_i} - \frac{\sum_r \alh_r \rho_{r\ell}}{\alh_+}\Big| \le \eps_n,  \quad \forall \ell \in [L].
    \end{align}
    We now apply the following lemma:
    \begin{lem}\label{lem:psi:dev:1}
    	Let $\psi(x,y) = (x-y)^2/y$.
    	Consider $(x,y)$ and $(x',y')$ in $[0,1] \times [1/c_1, 1]$, where $c_1 > 1$, such that $\max\{|x-x'|, |y-y'|\} \le \eps \le 1$. Then,
    	\begin{align}
    	\big|\psi(x',y') - \psi(x,y)\big| \le 12 c_1^3 \,\eps. %
    	\end{align}
    \end{lem}
	 
	 Note that $\rhob_\ell$ is a 
	 convex combination of $(\rho_{r\ell})$ over $r$, hence using \eqref{eq:rhou:taur},
	 \begin{align}\label{eq:rhob:ell:lower:bound}
	     \rhob_\ell \ge \rhou \ge \tau_\rho.
	 \end{align}
	 Furthermore, by assumption $\eps_n \le \tau_\rho/2$, combined with~\eqref{eq:rho:diff}, we have $\min\{\rhoh_{\kh\ell}, \rhob_\ell\} \ge \tau_\rho/2$ for all $\ell \in [L]$. Therefore, we can apply Lemma~\ref{lem:psi:dev:1} with $c_1 = 2/\tau_\rho$ to obtain 
   \begin{align*}
	   \big| \psi(\xi_{i\ell}, \rhoh_{\kh\ell}) -  \psi(\rho_{r\ell}, \rhob_\ell)
	   \big| \le 96 \tau_\rho^{-3} \eps_n,  \quad \forall i \in \Tch_r, \;  \forall \ell \in [L]. 
   \end{align*}
  	For two vectors $x,y \in \reals^L$, let us write $\Psi(x,y) = \sum_\ell \psi(x_\ell, y_\ell)$. Let $\xi_i = (\xi_{i\ell})$,  and  $\rhoh_{u*} = (\rhoh_{u\ell})$,
  	and set $\Yh^{(u)}_+ = \sum_{i \in \Gh_u} d_i \Psi(\xi_i, \rhoh_{u*})$ for any $u \in [K]$.
    By the triangle inequality,
  	\begin{align*}
  		\Yh^{(\kh)}_+ &\ge \sum_r \sum_{i \in \Tch_r} d_i \big( \Psi(\rho_{r*}, \rhob_*) - 96\tau_\rho^{-3}\eps_n L\big)\\ 
  		&= \sum_r \alh_r \big( \Psi(\rho_{r*}, \rhob_*)  - 96\tau_\rho^{-3}\eps_n L\big)
  	\end{align*}
  	where $\rho_{r*} = (\rhob_{r\ell})$ and $\rhob_* = (\rhob_\ell)$.
  	Dividing by $\alh_+$, we have
  	\begin{align}\label{eq:Yh:k:lower:bound}
  		\frac1{\alh_+ } \Yh^{(\kh)}_+ 
  		&\ge \omega_1  - 96\tau_\rho^{-3} \eps_n L 
  	\end{align}
  	where we have defined $\omega_1 := \sum_r \beh_r \Psi(\rho_{r*}, \rhob_*)$.
  	
  	\paragraph{Controlling $\omega_1$}
  	Recall that  $|\Tch_1|, |\Tch_2| \ge c_1n$, as argued in~\eqref{eq:Gk:bound}. We also recall the definition of $a_h$ in~\eqref{eq:ah:def}. 
  	 Then, on the event $\Ac$, for $u=1,2$, we have
  	 \begin{align*}
  	 \beh_u := \frac{\sum_{i \in \Tch_u} d_i}{ 
  	 	\sum_{r} \sum_{i \in \Tch_{r}} d_i} \ge \frac13 \frac{\sum_{i \in \Tch_u} \ds_i}{ 
   		\sum_{r} \sum_{i \in \Tch_{r}} \ds_i} &= \frac13 \frac{\sum_{i \in \Tch_u} \theta_i a_{u}}{\sum_{r} \sum_{i \in \Tch_{r}}  \theta_i a_{r}} \\
   		&\ge \frac13 \frac{ \tau_\theta \thetamax a_{u}|\Tch_u| }{\thetamax \amax  |\Gh_{\kh}|}  \ge \frac13 \tau_\theta \tau_a  c_1
  	\end{align*}
  	using $a_h \ge \tau_a \amax$ from~\eqref{eq:ah:amax:ineq}, $\theta_i \ge \tau_\theta \thetamax$ and $|\Gh_{\kh}| \le n$. We have
  	\begin{align*}
  		\omega_1 = \sum_r  \beh_r  \sum_\ell \frac{(\rho_{r\ell}- \rhob_\ell)^2}{\rhob_\ell} = \sum_\ell \frac1{\rhob_\ell} \sum_r \beh_r (\rho_{r\ell}- \rhob_\ell)^2.
  	\end{align*}
  	The inner summation is the variance of a random variable taking values  $(\rho_{r\ell})$ with probabilities $(\beh_r)$. Applying Lemma~\ref{lem:var:low} in Appendix~\ref{app:tech}
	and recalling the definition of $\omega_2$ from~\eqref{eq:omega2:def}, we have
	\begin{align}\label{eq:omega2}
		\omega_1 \ge \frac1{\max_\ell \rhob_\ell} \frac12 \beh_1 \beh_2 \sum_\ell (\rho_{1\ell} - \rho_{2\ell})^2 \ge \frac1{18} \tau_\theta^2 \tau_a^2 c_1^2 \norm{\rho_{1*} - \rho_{2*}}^2 \ge L \omega_2
	\end{align}
	since $\max_{\ell} \rhob_{\ell} \le \max_{k,\ell} \rho_{k \ell} \le 1$.
	
\paragraph{Putting the pieces together} 
  	On $\Omega_n$, by definition
	$96 \tau_\rho^{-3} \eps_n \le  \frac{1}{2}\omega_2$, which combined with~\eqref{eq:Yh:k:lower:bound} and~\eqref{eq:omega2}, gives $\frac1{\alh_+ }  \Yh^{(\kh)}_+ \ge \frac12 L \omega_2$.
Combined with~\eqref{eq:alh:plus:lower:bound}, on $\Omega_n \cap \Ec \cap \Ac$, we have
  	\begin{align*}
  		 \Yh^{(\kh)}_+ \ge \frac14 c_1C_1 L  \omega_2 \,n \nu_n.
  	\end{align*}
  	Furthermore, $\nt = |S_2| \le 0.6 n$ on $\Ac$, hence $\gamma_{\nt} = \sqrt{\nt(L-1)} \le  \sqrt{0.6 n L}$
    and
  	\begin{align*}
  	    \Th_n= \frac1{\sqrt{2}} \Big(
  	    \frac1{\gamma_{\nt}}\sum_{u=1}^K \Yh_+^{(u)} - \gamma_{\nt}\Big) &\ge 
  	    \frac1{\sqrt{2}} \Big(\frac1{\gamma_{\nt}} \Yh^{(\kh)}_+  - \gamma_{\nt}\Big) \\
  	    &\ge \sqrt{\frac{n}2}\left(\frac{c_1 C_1 L \omega_2 \nu_n /4}{ \sqrt{0.6 L}} -  \sqrt{0.6 L}\right).
  	\end{align*}
  	On $\Omega_n$, we have $\frac12(c_1 C_1 L \omega_2 \nu_n /4) \ge 0.6 L$, hence on $\Omega_n \cap \Ec \cap \Ac$, we obtain
  	\begin{align}\label{eq:Tn:low}
	  	 \Th_n \ge  \sqrt{\frac{n}2}\left(\frac{c_1 C_1 L \omega_2 \nu_n /8}{ \sqrt{0.6 L}}\right) \ge \frac{c_1C_1}{9} \omega_2 \nu_n \sqrt{L n}.
  	\end{align}
  	Furthermore, we note that 
  	\begin{align*}
  	\pr((\Omega_n \cap \Ec \cap \Ac)^c) \le \pr(\Omega_n^c) + 2Ln^{-1} + 2(7n^{-1}) \le  \pr(\Omega_n^c) + 9Ln^{-1} 
  	\end{align*}
  using Lemmas~\ref{lem:A:event:prob} and~\ref{lem:E:event:prob} and $L \ge 2$. The proof for the case $\sigma = 0$ is complete.
  
  \smallskip
  To extend the proof to the case $\sigma > 0$, we replace
   inequality \eqref{eq:Gk:bound} with
    \begin{align*}
        |\Ch_{\kh,1}\cap S'_2|\ge|\Cc_{1} \cap S'_2|/K &\ge |\Gc_1\cap S'_2|/K_0 
        \ge c_1 n,
    \end{align*}
    which holds on 
    on event $\Ac \cap \Mc_n$ according to~\eqref{eq:Gc:prime:lower:bound}. Then, \eqref{eq:Tn:low} is true under event $\Omega_n \cap \Ec \cap \Ac \cap \Mc_n$ with probability at least $1- \pr(\Omega_n^c) - \pr(\Mc_n^c) - 9Ln^{-1}$. 
}The proof of Theorem~\ref{thm:consist} is complete.
  
\subsection{Proof of Theorem \ref{thm:consist:others}} \label{app:proof:4}
The proof has six steps as outlined below:
\begin{enumerate}
    \item 
    Showing that multinomial probabilities of the $i$-th node are close to $H_\ell(x_i)$ defined in \eqref{eq:H:def} with high probability. 
    
    \item Showing that the (partial) degrees are proportional to $\nu_n$.
    
    \item Showing that the estimated probabilities are close to those based on the \emph{limiting} row labels $z$. 
    
    \item Controlling the chi-square statistics by $\fvarmin$.
    
    \item Showing that the chi-square statistic with estimated column labels $\yh$ is close to the one with the limiting column labels $y$.
    
    \item Simplifying the assumptions.
\end{enumerate}
Steps 1--4 are carried out assuming that $\yh = y$ and then the result is extended, in step 5, to $\yh$ approaching $y$ in the limit.

Let $\Cc_k = \{i \in [n]:\; z_i = k\}$ be the  community $k$ defined by label vector $z$, and  $n_k = |\Cc_k|$. We also let $\Gc_k := \{i \in S_2:\; z_i = k\} = \Cc_k \cap S_2$. Consider event 
\begin{align}\label{eq:event:A1}
    \Ac_1 = \bigl\{|\Gc_k| \in [0.4 n_k, 0.6n_k],\  \forall k \in [K] \bigr\}.
\end{align}
Since $|\Cc_k \cap S_1| = n_k - |\Gc_k|$, on $\Ac_1$, we also have
\begin{align}\label{eq:Ck:cap:S1:bound}
    |\Cc_k \cap S_1| \in [0.4 n_k, 0.6 n_k], \quad \forall k \in [K].
\end{align}
Under the assumption $0.4\tau_\Cc n \ge 2$, we have,  
\begin{align}\label{eq:Gsize:lower}
    |\Gc_k| \ge 2, \;\forall k \in [K], \quad \text{on $\Ac_1$}.
\end{align}
From the proof of Lemma~\ref{lem:A:event:prob}, we obtain:
\begin{lem}\label{lem:A1:event:prob}
    $\pr(\Ac_1^c) \le n^{-1}$ if $\frac{\log n}{n} \le \tau_c/300$.
\end{lem}

For two $\sigma$-fields $\Fc$ and $\Hc$, we write $\Fc \vee \Hc = \sigma(\Fc \cup \Hc)$ for the $\sigma$-field generated by their union. Recall that with subsampling, the set $S_1 \subset [n]$ is determined by including any element of $[n]$, indepenently with probability $1/2$, and $S_2 = [n] \setminus S_1$. Let $d_i = \sum_{j \in S_1} A_{ij}$ and consider the following $\sigma$-fields 
\begin{align}
\begin{split}\label{eq:Fc:def2}
    \Fc_0 &= \sigma\bigl(S_1\bigr), \\
    \Fc_1 &= \Fc_0 \vee \sigma(x_{S_2}),\\
    \Fc_2 &= \Fc_1 \vee \sigma(x_{S_1}) = \Fc_0 \vee \sigma(x_{[n]}), \\ 
    \Fc &= \Fc_2 \vee \sigma\bigl( (d_i, i \in S_2) \bigr),
\end{split}
\end{align} 
where $x_{S_2} = (x_i, i \in S_2)$, and similarly for $x_{S_1}$, and $x_{[n]} = (x_1,\dots,x_n)$.  Note that conditioned on $\Fc_0$, $\yh$ is fixed, and conditioned on $\Fc_2$, $(p_{ij})$ is fixed.

We first consider the case where  $\yh_{S_1} = y_{S_1}$. In this case, we drop the dependence of $X_{i\ell}(\yh)$ (defined in~\eqref{eq:X:def}) on $\yh$, and write
\begin{align}\label{eq:X:il}
    X_{i\ell} := \sum_{j \in S_1} A_{ij} 1\{y_j = \ell\}.
\end{align}
Since conditioned on $\Fc_2$, $x_{[n]}$ are fixed, it follows that
\begin{align}\label{eq:qil:def}
    X_{i\ell}
    \,\given\, \Fc_2 \sim \poi (q_{i\ell}), \quad \text{where} \quad q_{i\ell} := \sum_{j \in S_1} p_{ij} 1\{y_j = \ell\},
\end{align}
independently across $\ell$. Since for $i \in S_2$, the sum of $X_{i\ell}$ over $\ell$ is $d_i$, and when we condition on $\Fc$, we are also conditioning on $d_i, i \in S_2$, we obtain 
\begin{align}\label{eq:X:il:mult}
    (X_{i\ell})_{\ell=1}^L \,\given\, \Fc \sim \mult\bigl(d_i, (\rho_{i\ell})_{\ell=1}^L\bigr) 
    \quad \text{where} \quad \rho_{i\ell} :=  \frac{q_{i\ell}}{\sum_{\ell'} q_{i\ell'}}.
\end{align} 
independently across $i \in S_2$.

\medskip
\paragraph{Controlling conditional probabilities}
Let us set
\begin{align}\label{eq:tau:rho:def}
    \tau_\rho := \frac{C_8}{2 L\tau_\theta} = \frac{ \tau_\Cc \tau_h \tau_{\theta}}{2L}.
\end{align}
As the first step in the proof, we show that $\rho_{i\ell}$ is close to $H_{\ell}(x_i)$.
More specifically, the following event
\begin{align}\label{eq:R:event}
    \Rc := \Big\{|\rho_{i\ell} - H_{\ell}(x_i)| \le 
    \frac{4K}{\tau_\rho}\sqrt{\frac{\log n}{n}},
    \ \forall i\in S_2, \ \forall \ell \in [L] \Big\}
\end{align}
holds with high probability:

\begin{lem}
    \label{lem:controlling:Rc}
    There is an event $\Wc$ such that 
    \[
        \Rc \supseteq \Gamma \cap \Wc, \quad \text{and} \quad \pr(\Wc^c) \le
        4LK n^{-1}
    \]
    whenever
    $
        \frac{\log n}{n} < 
        (\frac{\tau_\rho}{4K})^2.
    $
\end{lem}

\paragraph{Controlling the degrees}
Consider the event
\begin{align}\label{eq:event:A2}
    \Ac_2 := \bigl\{d_i \in [0.16C_8\nu_n, 0.96\nu_n], \; \forall i \in S_2 \bigr\}
\end{align}
The next lemma guarantees that $\Ac_2$ holds with high probability: 
\begin{lem}
    \label{lem:A2:event:prob}
    There is an event $\Dc$ such that 
    \[
        \Ac_2 \supseteq \Ac_1 \cap \Gamma \cap \Dc \quad \text{and} \quad \pr(\Dc^c \cap \Ac_1) \le 2.2n^{-1}
    \]
    wherever $\frac{\log n}n \le 0.04 C_8^2$ and $\frac{\log n}{\nu_n} \le 0.001 C_8$.
\end{lem}
From now on, let $\Ac := \Ac_1 \cap \Ac_2$. Let $\dpk = \sum_{i \in \Gc_k} d_i$ and $\omega_n = \min_k  \dpk/|S_2|$. On $\Ac$, we have
\begin{align}
    \min_{k} \dpk &\ge (0.16 C_8 \nu_n) (0.4 n_k) \ge  
    0.064 \tau_\Cc C_8\, n \nu_n, \label{eq:min:dpk} \\
    \quad  \omega_n &\ge  (8/75) \tau_\Cc C_8\nu_n \label{eq:omegan:lower:bound}
\end{align}
using $0.4n \le |S_2| \le 0.6 n$. 

\paragraph{Controlling probability estimates} 


Recall $\Gh_k = \{ i \in S_2: \zh_i = k \}$, and let
\begin{align}\label{eq:three:rhos}
    \rhoh_{k\ell} = \frac{\sum_{i \in \Gh_k} X_{i\ell}}{\sum_{i \in \Gh_k} d_i}, \quad 
    \rhot_{k\ell} = \frac{1}{\dpk} \sum_{i \in  \Gc_k} X_{i\ell}, 
    \quad \rhob_{k\ell} = \frac{1}{\dpk} \sum_{i \in  \Gc_k} d_i \rho_{i\ell},
\end{align}
 $\Delh_{k\ell} = \rhoh_{k\ell} - \rhot_{k\ell}$ and  $\Delt_{k\ell} = \rhot_{k\ell} - \rhob_{k\ell}$. To control these deviations, we first show that $\rhob_{k \ell_k}$ is lower-bounded, where $\ell_k$ is as in~\eqref{eq:r:k:def}:
\begin{lem}\label{lem:rho:lb}
    Assume 
    $\frac{\log n}{n} \le 
    (\frac{\tau_\rho^2}{4K})^2$.
    Then, with $\{\ell_k\}_{k=1}^K$ as defined in~\eqref{eq:ell:k:def},
    on $\Gamma \cap \Rc$,
    we have
    \begin{align*}
        \rho_{i \ell_{z_i}} \ge
        \tau_\rho
       \quad 
        \forall i \in S_2.
    \end{align*}
\end{lem}
Combined with definition of $\rhob_{k\ell}$ in~\eqref{eq:three:rhos}, Lemma~\ref{lem:rho:lb} immediately implies that under the same condition, on $\Gamma \cap \Rc$,
\begin{align}\label{eq:rhob:lower:bound}
    \rhob_{k \ell_k} \ge \tau_\rho, \quad \forall k \in [K].
\end{align}
Note that $2 \tau_\rho = \tau_\Cc \tau_h \tau_{\theta}/L$ hence $2 \tau_\rho \le 1$.

Next, we show that $\Delt_{k\ell}$ is small by considering the following event
%
\begin{align}\label{eq:def:B}
	\mathcal B := \Big\{ \max_{\ell} |\Delt_{k\ell}| \le 
       \frac{8}{\sqrt{\tau_\Cc C_8}} \sqrt{\frac{\log n}{n \nu_n}}   =: \delta, \; \forall k \in [K] \Big\}.
\end{align}

\begin{lem}\label{lem:bound:Delt2}
    $\pr(\Bc^c \cap \Ac) \le 2L n^{-1}$ whenever $\frac{\log n}{n \nu_n} \le 0.064 \tau_\Cc C_8$.
\end{lem}
To simplify the notation, let us define
\begin{align}
    \Nc_n := \Ac \cap \Bc \cap \Mc_n \cap \Rc \cap \Gamma.
\end{align}
The next step is to control $\Delh_{k \ell_k}$:
\begin{lem}\label{lem:Delh:k:ell:k}
    Assume that $\alphan \le \tau_\Cc \tau_\rho C_8 /18$. Then, on 
    $\Nc_n$, 
    \begin{align}\label{eq:delh:bound}
        |\Delh_{k\ell_k}| \le  \frac{54}{\tau_\rho \tau_\Cc C_8}  \, \alpha_n \, \rhot_{k\ell_k}, \quad \forall k \in [K].
    \end{align}
\end{lem}

Combining \eqref{eq:def:B} and \eqref{eq:delh:bound}, on $\Nc_n$, we have
\begin{align} 
    |\rhoh_{k\ell_k} - \rhob_{k\ell_k}|  
    &\le  \frac{54}{\tau_\rho \tau_\Cc C_8} \alphan +  \frac{8}{\sqrt{\tau_\Cc C_8}} \sqrt{\frac{\log n}{n \nu_n}} \notag \\
    &\le  \frac{58}{\tau_\rho \tau_\Cc C_8} \sqrt{\frac{\log n}{\nu_n}}.
    \label{eq:rhoh:to:rohb}
\end{align}
The second inequality uses $\tau_\Cc C_8 \le 1$ and $2 \tau_\rho \le 1$ to replace the prefactor of the second term with $4 /(\tau_\rho \tau_\Cc C_8)$ and then uses the assumption $\alphan \le \sqrt{(\log n) / \nu_n}$ to combine the two terms. We note that the fast rate $\sqrt{(\log n) / (n \nu_n)}$ of the second term is not helpful since it will be dominated later in the argument by the slow rate $\sqrt{(\log n) / \nu_n}$  needed to control~\eqref{eq:def:Ec}.

Let $\xi_{i\ell} = X_{i\ell}/d_i$ for $i \in [S_2]$. Then $\ex^{\Fc}[\xi_{i\ell}] = \rho _{i\ell}$. Consider the event 
\begin{align} \label{eq:def:Ec}
	\Ec := \Big\{ \max_{i \in S_2, \ \ell \in [L]} |\xi_{i\ell} - \rho_{i\ell} | 
    \le 5\sqrt{\frac{\log n}{C_8 \nu _n}} \Big\}.
\end{align}
Then, we have:

\begin{lem}\label{lem:Ec:bound}
    $\pr(\Ec^c \cap \Ac) \le 2Ln^{-1}$ whenever $\frac{\log n}{\nu_n} \le 0.16 C_8$.
\end{lem}

\paragraph{Controlling chi-square statistics}
Now, let us define
\begin{align}\label{eq:epsn}
    \eps_n := \frac{58}{\tau_\rho \tau_\Cc C_8} \sqrt{\frac{\log n}{\nu_n}}.
\end{align}
 Then, on $\Nc_n \cap \Ec$, we have 
 \begin{align}\label{eq:two:epsn:deviations}
    |\rhoh_{k\ell_k} - \rhob_{k\ell_k}| \le \eps_n, \quad 
    |\xi_{i\ell} - \rho_{i\ell}| \le \eps_n
 \end{align}
 for all $k, \ell$ and $i \in S_2$. This follows by recalling that $\tau_{\Cc}, \tau_\rho, C_8 \le 1$. Combining~\eqref{eq:rhob:lower:bound},~\eqref{eq:two:epsn:deviations} and the assumption $\eps_n \le \tau_\rho/2$, we obtain 
\begin{align}\label{eq:rhoh:rhob:lower}
    \min\{\rhoh_{k\ell_k}, \rhob_{k \ell_k}\} \ge \tau_\rho/2, \quad \text{on}\; \Nc_n.
\end{align}
Hence, we can apply Lemma~\ref{lem:psi:dev:1} with $c_1 = 2/\tau_\rho$, using~\eqref{eq:two:epsn:deviations} to obtain that,  on $\Nc_n \cap \Ec$, 
\begin{align*}
   \big| \psi(\xi_{i\ell_k}, \rhoh_{k\ell_k}) - \psi(\rho_{i\ell_k}, \rhob_{k\ell_k}) \big| \le 96 \tau_\rho^{-3} \eps_n,  \quad \forall i \in \Gc_k,\; \forall k \in [K_0]. 
\end{align*}
Define 
$
    \Yt := \sum_{k=1}^{K} \sum_{i \in \Gc_k} d_i\psi(\xi_{i\ell_k}, \rhoh_{k\ell_k}).
$
Then, on $\Nc_n \cap \Ec$,
\begin{align*}
    \Yt &\ge \sum_{k=1}^{K} \sum_{i \in  \Gc_k}  d_i\big[(\rho_{i\ell_k} - \rhob_{k\ell_k})^2 - 96 \tau_\rho^{-3} \eps_n\big]\\
    &=
     \sum_{k=1}^{K} \dpk \Big[\sum_{i \in  \Gc_k} \frac{d_i}{\dpk}(\rho_{i\ell_k} - \rhob_{k\ell_k})^2 - 96 \tau_\rho^{-3} \eps_n \Big]
\end{align*}
where the first inequality also uses $\psi(x,y) \ge (x-y)^2$ for $x, y \in [0,1]$. Let 
\begin{align}
   \varpi_k = \sum_{i \in  \Gc_k}\frac{d_i}{\dpk}(\rho_{i\ell_k} - \rhob_{k\ell_k})^2, \quad k \in [K].
\end{align}

 Note that $\varpi_k$ is the variance of a random variable taking value $\rho_{i\ell_k}$ with probability $d_i/\dpk$ for $i \in \Gc_k$.   Recalling that $\fvar_{k\ell} := \var(H_\ell(x))$ when $x \sim \Qb_k$, we have the following:
 \begin{lem}\label{lem:varpi}
    Assume $\frac{\log n}{n} \le \min\{ 
        \frac{\tau_\rho^2}{4},
        \tau_\Cc\}$. Then, there is an event $\Hc$ on which 
    \begin{align}
        \varpi_k  \ge \frac{C_8^2}{144} \fvar_{k \ell_k} - \frac{C_8 }{8} \tau_\theta L \sqrt{\frac{\log n}{n}}, \quad k \in [K]
    \end{align}
    and we have $\pr(\Hc^c \cap \Ac \cap \Rc) \le K n^{-c}$.
\end{lem}

Let 
$\mu_n :=  \max\{1, L \sqrt{\nu_n / n}\}$ and
\begin{align}
	\tilde\eps_n = \frac{58}{\tau_\rho \tau_\Cc C_8} \mu_n \sqrt{\frac{ \log n}{\nu_n}}.
\end{align}
so that $\eps_n \le \tilde\eps_n$. We have 
\begin{align*}
	\frac{C_8 }{8} \tau_\theta  L \sqrt{\frac{\log n}{n}}  \le \mu_n\sqrt{\frac{\log n}{\nu_n}} \le \tilde\eps_n.
\end{align*}
It follows that on $\Nc_n \cap \Ec \cap \Hc$, we have
\begin{align*}
	\Yt \ge \sum_{k=1}^{K} \dpk( \varpi_k -  96 \tau_\rho^{-3} \eps_n) 
	\ge \sum_{k=1}^{K}   \dpk\Bigl( \frac{C_8^2}{144} \fvar_{k \ell_k} -  97 \tau_\rho^{-3} \tilde\eps_n\Bigr) 
\end{align*}
Recalling $\fvarmin := \min_k \fvar_{k \ell_k} $ and by assumption 
$
	 \frac{C_8^2}{144} \fvarmin \ge 2\cdot 97 \tau_\rho^{-3} \tilde\eps_n,
$
we get
\begin{align}\label{eq:Yt:lower}
	\Yt \ge \frac{C_8^2}{288} \fvarmin \sum_{k=1}^{K} \dpk \ge \frac{C_8^3}{1800} \fvarmin \, n \nu_n
\end{align}
using $d_i \ge 0.16C_8 \nu_n$ on $\Ac_2$; see~\eqref{eq:event:A2}.

\medskip
Let $\Yh = \sum_{k=1}^{K} \sum_{i \in \Gh_k} d_i\psi(\xi_{i\ell_k}, \rhoh_{k\ell_k})$ and $\Hc_k = \Gc_k \Delta \Gh_k := (\Gc_k \setminus \Gh_k) \cup (\Gh_k \setminus \Gc_k)$. Note that $\sum_k |\Hc_k| \le \alphan n$ on event $\Mc_n$. Therefore, on $\Nc_n \cap \Ec$  
\begin{align} 
	|\Yh  - \Yt| &= \sum_{k = 1}^{K} 
	\sum_{i \in \Hc_k} 
	d_i\psi( \xi_{i\ell}, \rhoh_{k\ell_k})  \notag \\ 
	&\le \frac{2}{\tau_\rho} \sum_{k = 1}^{K} \sum_{i \in \Hc_k} d_i
	 \le	\frac{1.92}{\tau_\rho} \alphan n  \nu_n.  \label{eq:Yh:Yt:diff}
\end{align}
The second inequality follows from~\eqref{eq:rhoh:rhob:lower} and noting that $(\xi_{i\ell} - \rhoh_{k\ell_k})^2 \le 1$. The third inequality is by~\eqref{eq:event:A2}.

The assumption $\frac{C_8^3}{1800} \fvarmin \ge 2 \cdot \frac{1.92}{\tau_\rho} \alphan$ combined with~\eqref{eq:Yt:lower} and~\eqref{eq:Yh:Yt:diff} gives
\begin{align}\label{eq:Yh:lower}
	\Yh \ge  \frac{C_8^3}{3600} \fvarmin \, n \nu_n.
\end{align}

\newcommand\rhohp{\widehat \rho'}
\newcommand\Yhp{\widehat Y'}

\paragraph{Estimated column labels}
Finally, we consider the case where the column labels $\yh$ are estimated using the community detection algorithm.
Let 
$X'_{i\ell} = \sum_{j \in S_1} A_{ij} 1\{\yh_j = \ell\}$ and $\xip_{i\ell} = X'_{i\ell}/d_i$, 
\begin{align*}
    \Yhp &= \sum_{k=1}^{K}\sum_{i \in \Gh_k} d_i \psi(\xip_{i\ell}, \rhohp_{k\ell_k}), \quad \text{where}\;\;
        \rhohp_{k\ell} = \frac{\sum_{i \in \Gh_k} X'_{i\ell}}{\sum_{i \in \Gh_k} d_i}.
\end{align*}
On $\Mc_n$, we have $|X_{i\ell} - X'_{i\ell}| \le n\kappa_n$. 
Letting 
\begin{align*}
    \epsnp := \frac{n\kappa_n}{0.16 C_8 \nu_n},
\end{align*}
it follows that on $\Ac \cap \Mc_n$,
\begin{align*}
    |\xi_{i\ell} - \xip_{i\ell}| \le n\kappa_n/ d_i \le \epsnp. 
\end{align*}
Assuming $\alphan \le 0.2 \tauc$, we have $|\Gh_k| \ge (0.4\tauc - \alphan)n \ge 0.2 \tauc n$ on $\Ac$. Then, on $\Ac \cap \Mc_n$, 
\begin{align*}
    |\rhoh_{k\ell_k} - \rhohp_{k\ell_k}| \le \frac{n\kappa_n}{0.16 C_8\nu_n \cdot  0.2 \tauc n } = 
    \frac{\epsnp}{0.2\tauc n} \le \epsnp
\end{align*}
where we have used the assumption $n \ge 5/\tauc$.

Recall that $\rhoh_{k\ell_k} \ge \tau_\rho/2$ on event $\Nc_n$. Then, by the assumption that $\epsnp \le \tau_\rho/4$, we have $\min\{\rhoh_{k\ell_k}, \rhohp_{k\ell_k} \} \ge \tau_\rho/4$. We can, then, apply Lemma~\ref{lem:psi:dev:1} with $c_1 = 4/\tau_\rho$ to obtain 
\begin{align*}
   \big| \psi(\xi_{i\ell_k}, \rhoh_{k\ell_k}) - \psi(\xip_{i\ell_k}, \rhohp_{k\ell_k}) \big| \le 768 \tau_\rho^{-3} \epsnp,  \quad \forall i \in \Gc_k,\; \forall k \in [K]. 
\end{align*}
Furthermore, on $\Nc_n \cap \Ac$,  
\begin{align}\label{eq:Yh:Yp:diff}
    |\Yh' -\Yh| \le 768 \tau_\rho^{-3} \epsnp \cdot (0.96 \nu_n) \cdot (0.6n) \le  \frac{2765}{C_8\tau_\rho^3}\kappa_n n^2
\end{align}
Combining~\eqref{eq:Yh:lower} and \eqref{eq:Yh:Yp:diff}, on event $\Nc_n \cap \Ec \cap \Hc$
\begin{align*}
    \Yh' \ge  \Big(\frac{C_8^3}{3600} \fvarmin - \frac{2765}{C_8\tau_\rho^3}\kappa_n n /\nu_n\Big)n\nu_n \ge  \frac{C_8^3}{7200} \fvarmin n \nu_n,
\end{align*}
assuming $\fvarmin C_8^3/7200 \ge \frac{2765}{C_8\tau_\rho^3}\kappa_n n /\nu_n$. 

Furthermore, $\nt = |S_2| \le n$, hence $\gamma_{\nt} = \sqrt{\nt(L-1)} \le  \sqrt{nL}$ and we have
\begin{align*}
  \Th_n \ge  \frac1{\sqrt{2}} \Big(
  \frac{\Yh'}{\gamma_{\nt}} - \gamma_{\nt}\Big)
  &\ge \sqrt{\frac{Ln}2}\left( \frac{C_8^3}{7200L}\fvarmin\, \nu_n  - 1\right) \ge \frac{C_8^3}{14400\sqrt{2L}}\fvarmin\sqrt{n}\nu_n.
\end{align*}
where the last inequality is by assumption $\fvarmin \ge 14400L/(C_8^3\nu_n)$.

\medskip
Finally, we put together the probabilities. From Lemma~\ref{lem:A1:event:prob} and \ref{lem:A2:event:prob}, 
\begin{align*}
    \pr(\Ac) \ge \pr(\Ac_1 \cap \Dc) = \pr(\Ac_1 ) - \pr(\Ac_1 \cap \Dc^c) \ge 1 - 3.2n^{-1}.  
\end{align*}
Furthermore, with Lemma \ref{lem:controlling:Rc}, \ref{lem:bound:Delt}, \ref{lem:Ec:bound} and \ref{lem:varpi}, 
\begin{align*}
    \pr(\Nc_n \cap \Ec \cap \Hc) &= \pr\Big(\Ac \cap \Bc \cap \Mc_n \cap \Rc \cap \Ec \cap \Hc\Big)\\
    &= \pr(\Ac \cap \Rc \cap \Mc_n) - \pr\Big(\Ac \cap \Rc \cap \Mc_n \cap (\Bc \cap \Ec \cap \Hc)^c\Big)\\
    &\ge \pr(\Ac \cap \Rc \cap \Mc_n)  - \pr(\Ac \cap \Bc^c) - \pr(\Ac \cap \Ec^c) - \pr(\Ac \cap \Rc \cap \Hc^c) \\
    &\ge 1- 3.2 n^{-1} - 4LKn^{-1} - \pr(\Mc_n^c) - 2Ln^{-1} - 2Ln^{-1} - Kn^{-c}\\
    &\ge 1- 12KLn^{-1} - Kn^{-c}- \pr(\Mc_n^c).
\end{align*}

\paragraph{Simplifying the assumptions}

The following is a list of all the assumptions we used in the proof:
\begin{align*}
    \renewcommand{\arraystretch}{1.5}
    \begin{array}[]{lll}
        n \ge 5/\tau_\Cc & 
        \frac{\log n}{n} \le \tau_c/300 &
        \frac{\log n}{n} < 
        (\frac{\taur}{4 K})^2
        \\
        \frac{\log n}n \le 0.04 C_8^2  &
        \frac{\log n}{\nu_n} \le 0.001 C_8 &
        \frac{\log n}{n} \le 
        (\frac{\taur^2}{4K})^2
        \\
        \frac{\log n}{\nu_n} \le 0.064 \tau_\Cc C_8 n &
        \alphan \le \tau_\Cc \taur C_8 /18 &
        \alphan \le \sqrt{(\log n) / \nu_n} \\
        \frac{\log n}{\nu_n} \le 0.16 C_8 &
        \eps_n = \frac{58}{\taur \tau_\Cc C_8} \sqrt{\frac{\log n}{\nu_n}} \le \taur /2 &
        \frac{\log n}{n} \le \min\{ 
            \frac{\taur^2}{4},
            \tau_\Cc\} \\
        \frac{C_8^2}{144} \fvarmin \ge 2\cdot 97 \taur^{-3} 
        \mu_n \eps_n
         &
        \frac{C_8^3}{1800} \fvarmin \ge 2 \cdot \frac{1.92}{\taur} \alphan &
        \epsnp = n\kappa_n/(0.16 C_8 \nu_n) \le \taur/4\\ 
        \frac{C_8^3}{7200 } \fvarmin  \ge \frac{2765}{C_8\taur^3}\kappa_n n /\nu_n &
        \fvarmin \ge 14400L/(C_8^3\nu_n) &
        \alpha_n \le 0.2 \tauc
    \end{array}
\end{align*}
We recall that
\[
    c_2 = \frac{C_8}{100} = 
    \frac{\tauc \tauh \taut^2}{100}, 
    \quad \taur = \frac{C_8}{2\taut L} = 
    \frac{50 c_2}{\taut L} = \frac{\tauc \tauh \taut}{2L}.
\]
The conditions on $\frac{\log n}n$ can summarized as follows:
\begin{align}
    \label{assume:init:logn:over:n}
    \sqrt{\frac{\log n}{n}} \le \min\left\{\sqrt{\frac{\tau_\Cc}{300}}, \
    \frac{\taur}{2}, \
    20 c_2, \
    \frac{\taur^2}{4K}
    \right\}. 
\end{align}
We also note that if $n \ge 2$, then $\frac{\log n}n \le \tau_{\Cc} / 300$ implies $n \ge 5 / \tau_\Cc$. Since $\taur^2 \le \taur$, we can drop 
$\frac{\taur}{2}$
from~\eqref{assume:init:logn:over:n}. Similarly, 
\begin{align*}
    \frac{\taur^2}{4K} /(20 c_2) = 
    \frac{\taur^2}{4K} \frac{50}{20 L\taur \taut} = 
    \frac{5}{8} \frac{\taur}{K L \taut} = 
    \frac{5}{16} \frac{\tauc \tauh}{KL^2} \le 1,
\end{align*}
hence we can also drop $20c_2$ from~\eqref{assume:init:logn:over:n}. Since $\taur^2/(4K) \ge \tau_\Cc^2 \tau_h^2 \tau_\theta^2 / (18KL^2) = \frac{2}{9}\frac{\taur^2}{K}$
and $\sqrt{\tau_\Cc/ 300} \ge \tau_\Cc^2 / 18$,  condition~\eqref{assume:init:logn:over:n} holds under assumption~\eqref{assume:final:logn:over:n}.

The condition $\eps_n \le \taur/2$ is
\[
    \frac{0.58}{\tau_\Cc c_2} \sqrt{\frac{\log n}{\nu_n}} \le \frac{\taur^2}2 = \frac{1250 c_2^2}{\tau_\theta^2 L^2}
\]  
which is equivalent to
\[
    \sqrt{\frac{\log n}{\nu_n}} \le \frac{1250}{0.58} \frac{\tauc c_2^3}{\taut^2L^2}    = 
    \frac{1250}{0.58} \frac{\tauc^2 \tauh c_2^3}{\tauc \tauh \taut^2L^2} =\frac{1250}{58} \frac{\tauc^2 \tauh c_2^3}{c_2L^2}.  
\]
The condition is satisfied if 
\begin{align}\label{eq:logn:nun}
    \sqrt{\frac{\log n}{\nu_n}}  \le 21 \frac{\tauc^2 \tauh c_2^2}{L^2}
\end{align}
which is what is assumed in~\eqref{assume:mis:rate:bound}.

\smallskip
The three upper bounds on $\frac{\log n}{\nu_n}$ can be combined into
\begin{align*}
    \frac{\log n}{\nu_n} \le c_2 \min \{ 0.1,\, 6.4 \tau_\Cc n \}
\end{align*}
Since $n \ge 5 / \tau_\Cc$, we have $6.4 \tau_\Cc n \ge 0.1$, hence it is enough that $\frac{\log n}{\nu_n} \le 0.1 c_2$.
Next, since $c_2 \le 0.01$, we have
\begin{align*}
    21 \frac{\tauc^2 \tauh c_2^2}{L^2} \le 21 \tauc^2 \tauh c_2^2 \le 21 c_2^2 \le 0.21  c_2 \le \sqrt{0.1 c_2}
\end{align*}
showing that~\eqref{eq:logn:nun} is already enough to guarantee this condition.
    
The three assumptions on $\alpha_n$ can be combined into
\begin{align*}
    \alpha_n \le \min\Big\{
        \frac{100}{36} \tauc^2 \tauh \taut c_2,
        \ \sqrt{\frac{\log n}
        {\nu_n}}
        \Big\}
\end{align*}
Since $c_2 \le 0.01 \taut$, 
we have
\begin{align*}
    21 \tauc^2 \tauh c_2^2 \le 0.21 \tauc^2 \tauh \taut c_2 
\end{align*}
showing that $\alphan \le \sqrt{\frac{\log n}{\nu_n}}$ together with~\eqref{eq:logn:nun} is enough to guarantee both upper bounds on $\alphan$.

The conditions involving $\fvarmin$ are implied by
\begin{align}\label{assume:initial:fvarmin:bound}
    \fvarmin \ge \max &\left\{ \frac{3 L^3 \mu_n \eps_n}{c_2^2\taur^3}, \
    \frac{0.007 L}{\taur c_2^3}\alphan, \ 
    \frac{L^3}{5\taur^3 c_2^4} \frac{\kappa_n n}{\nu_n},\ 
    \frac{0.0144 L}{c_2^3 \nu_n}
    \right\}
\end{align}
where we recall
\[
    \mu_n := \max\{1, L \sqrt{\nu_n / n}\}, \quad  \eps_n = \frac{0.58}{\taur \tau_\Cc c_2} \sqrt{\frac{\log n}{\nu_n}}.
\]
We have
\[
   \frac{3 L^3\mu_n \eps_n}{c_2^2\taur^3}  \le 
   \frac{2L^3}{\taur^4 \tau_\Cc c_2^3} \zeta_n \sqrt{\frac{\log n}{\nu_n}}
\]
Similarly, using $\taur, \tau_\Cc \le 1$,
\begin{align*}
    \frac{0.0144L}{c_2^3 \nu_n} \le  
    \frac{0.0144}{c_2^3} \zeta_n \sqrt{\frac{\log n}{\nu_n}}\le  \frac{2L^3}{\taur^4 \tau_\Cc c_2^3} \zeta_n \sqrt{\frac{\log n}{\nu_n}}
\end{align*}
and using assumption  $\alphan \le \sqrt{\frac{\log n}{\nu_n}}$
\begin{align*}
    \frac{0.007}{\taur c_2^3}\alphan \le \frac{0.007}{\taur c_2^3}\sqrt{\frac{\log n}{\nu_n}} \le \frac{2L^3}{\taur^4 \tau_\Cc c_2^3} \zeta_n \sqrt{\frac{\log n}{\nu_n}}.
\end{align*}
It follows that the assumption~\eqref{assume:final:fvarmin:bound} in the statement of theorem is enough to guarantee~\eqref{assume:initial:fvarmin:bound}.

Finally, condition $n\kappa_n/(16c_2 \nu_n) \le \taur/4$ is equivalent to what is stated in~\eqref{assume:mis:rate:bound}. The proof is complete.

\section{Proofs of Auxiliary Lemmas}\label{sec:aux:proofs}

\subsection{Lemmas in the proof of Theorem~\ref{thm:CCtest:null}}\label{sec:proof:lemmas:thm:CCtest:null}

\subsubsection{Lemmas in the proof of Proposition~\ref{prop:Sn:grouped}}
We first derive some useful relations between the moments and cumulants of a random variable that are used in the proofs of Lemma~\ref{lem:mult_mean_var} and \ref{lem:mom:growth}.
In particular, for Lemma \ref{lem:mom:growth}, we use the following observation: The central moments of sums of i.i.d. random variables grow ``slowly''. To develop an intuition for this observation, recall that
\begin{align}\label{eq:4th:mom:cumul}
	\mu_4(X) =  \kappa_4(X) + 3 \kappa_2^2(X)
\end{align}
where $X$ is any random variable, $\mu_r(X)$ is its $r$th order central moment, and $\kappa_r(X)$ is the corresponding $r$th order cumulant. Assume that $X$ can be written as a sum of i.i.d. variables $\{Y_1,\dots,Y_n\}$, that is, $X = \sum_{i=1}^n Y_i$. Cumulants are additive over independent sums, hence $\kappa_r(X) = \sum_{i=1}^n \kappa_r(Y_i) = n \kappa_r(Y_1)$. It follows that
\begin{align}\label{eq:4th:mom:cumul:v2}
	\mu_4(X) = n \kappa_4(Y_1) + 3 n^2 \kappa_2^2(Y_1) = O(n^2)
\end{align}
assuming $\kappa_r(Y_1) = O(1)$. In other words, $\mu_4(X)$ scales at half the rate of the worst-case scaling of the 4th power of a sum of $n$ deterministic terms (i.e., $O(n^2)$ instead of $O(n^4)$). 
By using $\kappa_4(Y_1) = \mu_4(Y_1) - 3 \kappa_2^2(Y_1)$ and $\kappa_2(Y_1) = \mu_2(Y_1)$, we can express the constants in~\eqref{eq:4th:mom:cumul:v2} in terms of the central moments of $Y_1$,
\begin{align}\label{eq:4th:mom:2}
\mu_4(X) = n \mu_4(Y_1) +3n (n - 1) \mu_2^2(Y_1) \sim 3  \mu_2^2(Y_1) n^2 .
\end{align}
A similar idea holds for higher-order central moments, an example of which is Lemma~\ref{lem:mom:growth}.

\begin{proof}[Proof of Lemma~\ref{lem:mult_mean_var}]
	For the expectation, we note that $\ex(X_{\ell} - d p_\ell)^2 = p_\ell ( 1-p_\ell)$, hence $\ex \psi(X_i, dp_\ell) = 1-p_\ell$ and the result follows since $\sum_\ell(1-p_\ell) = L-1$. We now turn to the variance.
	Let $\Xt  = X - d p = \sum_{i=1}^d \Ut_{i}$, where $\Ut_i = U_i -p$ and $U_i \sim \mult(1, p)$, independently. We have 
	\begin{equation*}
	d^2 \ex Y^2 = \sum_{\ell=1}^{L}\frac{\ex \Xt_{\ell}^4}{p_\ell^2} + 
		\sum_{\ell\ne \ell'}^{L}\frac{\ex \Xt_{\ell}^2\Xt_{\ell'}^2}{p_\ell p_{\ell'}}.
	\end{equation*}
	Noting that $\Xt_\ell = \sum_i \Ut_{i\ell}$, we obtain
	\begin{equation*}
	\ex ( \Xt_{\ell}^2\Xt_{\ell'}^2 ) = 
	\ex \Big(\sum_{i_1, i_2} \Ut_{i_1 \ell} \Ut_{i_2 \ell'} \Big)^2  = 
	\sum_{i_1, \,i_2, \, i_3, \,i_4} \ex[\Ut_{i_1\ell}\Ut_{i_2\ell}\Ut_{i_3\ell'}\Ut_{i_4\ell'}],
	\end{equation*}
	where all four indices running from 1 to $d$. 
	We can categorize the general term $\ex[\Ut_{i_1\ell}\Ut_{i_2\ell}\Ut_{i_3\ell'}\Ut_{i_4\ell'}]$ based on how many different values $i_1, i_2, i_3$ and $ i_4$ take.
	If $i_1, i_2, i_3$ and $ i_4$ take 3 or 4 different values, the term is zero by independence. 
	The remaining three cases are summarized below:
	\begin{align*}
	\ex[\Ut_{i_1\ell}\Ut_{i_2\ell}\Ut_{i_3\ell'}\Ut_{i_4\ell'}] = 
	\begin{cases}
	\ex[\Ut_{1\ell}^2] \cdot \ex[\Ut_{1\ell'}^2], & i_1 = i_2 \ne i_3 = i_4, \\
	\big(\ex[\Ut_{1\ell}\Ut_{1\ell'}] \big)^2, & i_1 = i_3 \ne i_2 = i_4 \;\; \\ & \quad \text{or}\;\; i_1 = i_4 \ne i_2 = i_3, \\
	\ex[\Ut_{1\ell}^2\Ut_{1\ell'}^2], & i_1 = i_3 = i_2 = i_4,
	\end{cases}
	\end{align*}
	which simplifies to
	\begin{align*}
		\ex[\Ut_{i_1\ell}\Ut_{i_2\ell}\Ut_{i_3\ell'}\Ut_{i_4\ell'}] = 
		\begin{cases}
		 p_\ell(1-p_\ell)p_{\ell'}(1-p_{\ell'}), & i_1 = i_2 \ne i_3 = i_4, \\
		 p_\ell^2 p_{\ell'}^2, & i_1 = i_3 \ne i_2 = i_4\;\;  \\ & \quad \text{or}\;\; i_1 = i_4 \ne i_2 = i_3, \\
		 p_{\ell}p_{\ell'}(p_{\ell}+p_{\ell'}- 3p_{\ell}p_{\ell'}), & i_1 = i_3 = i_2 = i_4.
		\end{cases}
	\end{align*}
	The first two cases follow easily from independence. $\ex[\Ut_{1\ell}^2]  = \var(U_{1\ell}) = p_\ell (1-p_\ell)$ and $\ex[\Ut_{1\ell}\Ut_{1\ell'}] = \cov(U_{1\ell}, U_{1,\ell'}) = - p_\ell p_{\ell'}$.
	The third case follows, after some algebra, from the following observation:
	\begin{equation*}
	(\Ut_{1 \ell}, \Ut_{1 \ell'}) =
	\begin{cases}
	(-p_{\ell}, -p_{\ell'}) & \text{w.p.}\ 1-(p_{\ell}+p_{\ell'})\\
	(-p_{\ell}, 1-p_{\ell'}) & \text{w.p.}\ p_{\ell'}\\
	(1-p_{\ell}, -p_{\ell'}) & \text{w.p.}\ p_{\ell}\\
	\end{cases}.
	\end{equation*}
	To sum up, for $\ell \ne \ell'$, we have
	\begin{align*}
		\ex[\Xt_{\ell}^2\Xt_{\ell'}^2] &= (d^2-d)p_{\ell}p_{\ell'}[(1-p_\ell)(1-p_{\ell'}) + 2 p_\ell p_{\ell'}] + dp_{\ell}p_{\ell'}(p_{\ell}+p_{\ell'}- 3p_{\ell}p_{\ell'})\\
		&= d p_{\ell}p_{\ell'} \big[ (d-1) + (2-d)(p_\ell+p_{\ell'}) + (3d-6)p_\ell p_{\ell'}) \big].
	\end{align*}
	Let  $\alpha := \sum_\ell p_\ell^2$.
	Using $\sum_{\ell \neq \ell'} p_{\ell} p_{\ell'} = 1 - \alpha$ and $\sum_{\ell\neq \ell'} p_\ell = \sum_{\ell\neq \ell'} p_{\ell'} = L-1$, we have
	\begin{align*}
		\frac1 d \sum_{\ell \neq \ell'}\frac{\ex[\Xt_{\ell}^2\Xt_{\ell'}^2]}{p_\ell p_{\ell'}} = 
		(d-1)(L^2-L) + 2(2-d)(L-1) + (3d-6)(1-\alpha),
	\end{align*}
	for $\ell \neq \ell'$. 	Next, we consider the case $\ell = \ell'$.
	Let $\kappa_n$ and $\mu_n$ denote $n$th order cumulants and central moments of  $\Ut_{1\ell}$. 
	By~\eqref{eq:4th:mom:2},
	\begin{align*}
	\ex[\Xt_\ell^4] 
	&= d\mu_4 + 3d(d-1)\mu_2^2\\
	&= d[p_\ell(1-p_\ell)^4 + p_\ell^4(1-p_\ell)]+ 3d(d-1)p_\ell^2(1-p_\ell)^2\\
	&= dp_{\ell}^2[1/p_\ell +(3d -7) +(12-6d)p_{\ell} + (3d-6)p_{\ell}^2].
	\end{align*}
	We obtain
	\begin{align*}
		\frac1d \sum_{\ell=1}^{L}\frac{\ex \Xt_{\ell}^4}{p_\ell^2} = \frac{L}{h(p)} + L(3d-7) + (12-6d) +(3d-6)\alpha.
	\end{align*}
	Putting the pieces together, we have
	\begin{align*}
		d \,\ex Y^2 = d(L^2 -1) + \frac{L}{h(p)} - L(L+2) + 2.
	\end{align*}
	Combining with $\var(Y) = \ex Y^2 - (L-1)^2$ and some algebra finishes the proof.	
\end{proof}

\begin{proof}[Proof of Lemma~\ref{lem:mom:growth}]
	Let $\{W_i'\}$ be an independent copy of $\{W_i\}$, and let $X_n' = \sum_{i=1}^n W_i'$. The function $x \mapsto |x|^3$ is convex on $\mathbb{R}$. Applying Jensen's inequality with respect to $X'$ and the Cauchy-Schwartz inequality in probability,
	\begin{align*}
	\ex\big|X_n^2 - \ex X_n^2\big|^3 \le \ex\big|X_n^2 - (X_n')^2\big|^3
	\le \big[\ex \big| X_n+X'_n\big|^6\big]^{1/2} \big[ \ex \big|X_n-X_n'\big|^6\big]^{1/2}.
	\end{align*}
	For a random variable $U$, write $\kappa_i(U)$ for its $i$th cumulant Then,
	\begin{align*}
	\kappa_i({X_n + X_n'}) &= \kappa_i({X_n}) + \kappa_i({X_n'}) = 2n\kappa_i(W_1),\\
	\kappa_i({X_n - X_n'}) &= \kappa_i({X_n}) + (-1)^{i}\kappa_i({X_n'}) = 2n\kappa_i(W_1) \cdot 1\{i \ \text{is even}\}.
	\end{align*}
	
	Recall that the 6th central moment $\mu_6$ of any random variable can be written in terms of its cumulants $\{\kappa_i\}$ as follows: $\mu_6 = \kappa_6 + 15 \kappa_4 \kappa_2 + 10 \kappa_3^2 + 15\kappa_2^3$. Writing $\kapt_i = \kappa_i(W_1)$, and applying this relation to $X_n+X_n'$ and $X_n - X_n'$, we have
	\begin{align*}
	\ex \big| X_n+X'_n\big|^6 &= \mu_6(X_n+X'_n\big)
	= 2n\kapt_6 + 60n^2\kapt_4\kapt_2 + 40n^2\kapt_3^2 + 120n^3\kapt_2^3, \\
	\ex \big| X_n-X'_n\big|^6 &= \mu_6(X_n - X_n') 
	= 2n\kapt_6 + 60n^2\kapt_4\kapt_2 + 120n^3\kapt_2^3.
	\end{align*}
	Let $C_{W_1} = 2|\kapt_6| + 60|\kapt_4|\kapt_2 + 40\kapt_3^2 + 120\kapt_2^3$. Then, $\ex \big| X_n \pm X'_n\big|^6 \le C_{W_1} n^3$ and the result follows.
	
	\medskip
	For the case of where $W_1 = \alpha(Z-p)$ where $Z \sim \ber(p)$, let $\kappa_i = \kappa_i(Z)$ and note that $\kapt_i = \alpha^i \kappa_i$. It follows that 
	\begin{align*}
		C_{W_1} = \alpha^6 \big(2|\kappa_6| + 60|\kappa_4|\kappa_2 + 40\kappa_3^2 + 120\kappa_2^3 \big).
	\end{align*}
	 Next, we have $\kappa_2 = p(1-p)$, $\kappa_3 = \kappa_2(1-2p)$, $\kappa_{4} =\kappa_{2}(1-6 \kappa_{2})$, $\kappa_{6} =\kappa_{2}\big(1-30 \kappa_{2}(1-4 \kappa_{2}) \big)$. We have $\kappa_2 \in [0,1/4]$, hence $\kappa_3 /\kappa_2 \in [-1,1]$, $\kappa_4 / \kappa_2 \in [-\frac12, 1]$ and $\kappa_6 / \kappa_2 \in [-\frac{7}{8}, 1]$. It follows that $|\kappa_r| \le \kappa_2 \le 1/4$ for all $r=3,4,6$. Then,
	$$C_{W_1} / \alpha^6 \le 2 \kappa_2 + 15 \kappa_2 + 10 \kappa_2 + 7.5 \kappa_2 = 34.5 \kappa_2$$ and the proof is complete.
\end{proof}

\begin{proof}[Proof of Lemma~\ref{lem:dkol:affine:trans}]
	We have 
	\begin{align*}
	d_K \big( T, Z ) \big) 
	&= \sup_{t \in \reals} | \pr( \beta S +\alpha \le t) - \Phi(t) | \\
	&= \sup_{t \in \reals} | \pr( S \le t) - \Phi(\beta t + \alpha) | \\
	&\le \sup_{t \in \reals} \big( | \pr( S \le t) - \Phi(t) | +   |\Phi(t) - \Phi(\beta t+\alpha) | \big) \\
	&= d_K \big(S, Z \big) + \sup_{t \in \reals} |\Phi(t) - \Phi(\beta t + \alpha) |.
	\end{align*}
	Then,
	\begin{align*}
	\big|\Phi(t) - \Phi(\beta t)\big| 
	&=  \Big|\int_{\beta t}^t \frac1{\sqrt{2\pi}} e^{-x^2/2}dx \Big| \\
	&\le | \beta t - t | \frac1{\sqrt{2\pi}} e^{-\min(t,\beta t)^2/2}
	= \frac1{\sqrt{2\pi}} |\beta - 1| \cdot |t| e^{-a t^2/2},
	\end{align*}
	where $a = \min(\beta^2,1)$.
	Note that $t \mapsto t e^{-a t^2/2}$ achieves its maximum of $1/\sqrt{a e}$ over $[0,\infty)$ at $t = 1/\sqrt{a}$. We also have $\big|\Phi(s) - \Phi(s+\alpha)\big| \le |\alpha| \sup_{\tilde s} \Phi'(\tilde s)   = \frac1{\sqrt{2\pi}}|\alpha|$. Putting the pieces together finishes the proof.
\end{proof}

\subsubsection{Lemmas in the proof of Proposition~\ref{prop:Thn:Tn}}
\begin{proof}[Proof of Lemma~\ref{lem:G:expansion}]
	We have
	$\sum_i d_i(x_i - y - v)^2 = \sum_i d_i(x_i-y)^2 - 2 v R + d_+ v^2$.
	Hence,
	\begin{align*}
	\sum_i d_i \psi(x_i, y+v) = \frac{\sum_i d_i(x_i-y)^2}{y+v} - \frac{2v}{y+v}R + \frac{v^2}{y+v} d_+.
	\end{align*}
	It follows, after some algebra, that
	\begin{align*}
	\sum_i d_i [\psi(x_i, y+v)  - \psi(x_i, y)] = -\frac{v}{y+v} \Big[ 
	\sum_i d_i \psi(x_i, y)
	+2 R - v d_+\Big].
	\end{align*}
	We obtain
	\begin{align*}
	|G(v) - G(0) | \le \frac{|v|}{|y+v|} \big[ G(0) + 2 |R| + |v| d_+ \big].
	\end{align*}
	Applying the inequality $|a|/|1+a| \le 2|a|$ which holds for any $|a| \le 1/2$, with $a = v / y$ finishes the proof.
\end{proof}


\begin{proof}[Proof of Lemma~\ref{lem:dkol:perturb}]
	Let $\Ac = \{ |\Th_n - T_n| \ge \delta T_n + \eps\}$ and $q = \pr(\Ac)$. For any $t \in \reals$, we have
	\begin{align*}
	\pr(\Th_n \le t) &\le \pr(\{\Th_n \le t\} \cap \Ac^c) + \pr(\Ac) \\
	&\le \pr((1-\delta)T_n -\eps \le t) +  q.
	\end{align*}
	Subtracting $\Phi(t) = \pr(Z \le t)$ from both sides, we get
	\begin{align*}
	\pr(\Th_n \le t) - \Phi(t) &\le \dkol\big((1-\delta)T_n -\eps,Z\big) + q \\
	&\le d_K(T_n, Z) + \frac{2 \delta}{\sqrt{2 \pi e}}  + \frac{\eps}{\sqrt{2\pi}} + q \\ &\le  d_K(T_n, Z) + \frac12 (\delta + \eps) + q,
	\end{align*}
	by Lemma~\ref{lem:dkol:affine:trans} and noting that $\min\{|1-\delta|,1\} \ge 1/2$ by assumption. Similarly, for any $s \in \reals$,
	\begin{align*}
	\pr(T_n \le s) &\le \pr(\{T_n \le s\} \cap \Ac^c) + \pr(\Ac) \\
	&\le \pr(\Th_n\le (1+\delta) s + \eps) +  q.
	\end{align*}
	Applying the change of variable $t = (1+\delta)s +\eps$, adding $\Phi$ and rearranging, we obtain
	\begin{align*}
	\Phi(t)-\pr(\Th_n\le t) 
	&\le \Phi(t) -\pr((1+\delta) T_n +\eps \le t)+  q,
	\end{align*}
	and the rest of the argument follows as in the previous case. Putting the pieces together finishes the proof. 
\end{proof}

\begin{proof}[Proof of Lemma~\ref{lem:bound:Delt}]
    Note that $d_{+}^{(k)} \Delh_{k\ell}$ is a centered $\bin(d_{+}^{(k)}, p_{k\ell})$ variable. Applying Proposition~\ref{prop:gen:tail:bound} (Section~\ref{app:tech}), we have
    \begin{align*}
    	\pr\Big(| \Delh_{k\ell} | \ge \sqrt{\frac{2 u}{d_{+}^{(k)}}} + \frac{u}{3d_{+}^{(k)}}\Big) \le 2e^{-u}.
    \end{align*}
    Then the result follows by using union bound when $u \le \min_k d_+^{(k)}$.
\end{proof}

\begin{proof}[Proof of Lemma~\ref{lem:control:Delh}]
	Fix $k$ and $\ell$ and consider $i \in \Gc_k$. Define
	\begin{align*}
		a := \frac{\sum_{i \in \Gh_k} X_{i\ell}}{\sum_{i \in \Gc_k} X_{i\ell} } - 1 = 
		\frac{ \sum_{i \in \Gh_k \setminus \Gc_k} X_{i\ell} 
			- \sum_{i \in \Gc_k \setminus \Gh_k} X_{i\ell} }{\sum_{i \in \Gc_k} X_{i\ell} }.
	\end{align*}
	On event $\Mc_n$, we have
	\begin{align*}
		\Big|\sum_{i \in \Gh_k \setminus \Gc_k} X_{i\ell} - \sum_{i \in \Gc_k \setminus \Gh_k} X_{i\ell}\Big| \le d_{\max} (|\Gh_k \setminus \Gc_k| + |\Gc_k \setminus \Gh_k|) \le d_{\max}(\alphan n).
	\end{align*}
	 Recall that we have $|X_{+\ell}^{(k)} - d_{+}^{(k)} p_{k\ell} | \le  \delta d_{+}^{(k)}$ on event $\Bc$. Furthermore, by assumption $\delta \le \pul/2$, we obtain 
	\begin{align*}
		X_{+\ell}^{(k)}  \ge d_+^{(k)}\bigl(  p_{k\ell} - \delta \bigr) \ge d_+^{(k)}\pul/2.
	\end{align*}
	It follows that
	\begin{align*}
		|a| \le \frac{2(\alphan n)d_{\max}}{d_+^{(k)}\pul} \le \frac{2 d_{\max}}{\omega_n \pul}\alphan = \frac{2 \alphan}{\tau_d \,\pul}.
	\end{align*}    
	Similarly, letting $b := (\sum_{i \in \Gh_k} d_i)/(\sum_{i \in \Gc_k} d_i) - 1$, we have
	\begin{align*}
	    |b| \le \frac{d_{\max}(\alphan n)}{d_+^{(k)}} \le \frac{d_{\max}}{\omega_n}\alphan =  \frac{\alphan}{\tau_d}.
	\end{align*}
	Then 
	\begin{align*}
		\ph_{k\ell} = \frac{\sum_{i \in \Gh_k} X_{i\ell}}{\sum_{i \in \Gh_k} d_i } =  
		\frac{(1+a)\sum_{i \in \Gc_k} X_{i\ell} }{(1+b)\sum_{i \in \Gc_k} d_{i}} 
		= \frac{1+a}{1+b} \cdot \pt_{k\ell}.
	\end{align*}
	By assumption~$\alpha_n \le \tau_d \pul /2$, we have $|a|\le1$ and $b\le 1/2$. Hence,
	\begin{align*}
		|\ph_{k\ell} - \pt_{k\ell}| = \frac{|a - b|}{|1+b|} \cdot \pt_{k\ell} \le \frac{|a| + |b|}{1-|b|} \cdot\pt_{k\ell} \le 2(|a| + |b|) \cdot \pt_{k\ell}.
	\end{align*}
	Note that $|a| + |b| = (2\pul^{-1} + 1)\frac{\alphan}{\tau_d} \le (3\alphan)/(\tau_d \pul)$. Then the result follows.
\end{proof}

\begin{proof}[Proof of Lemma~\ref{lem:simplex:diam}]
	Let $E = \{e_\ell, \ell \in [L]\}$ be the standard basis of $\reals^L$. Then, $E$ is the set of extreme points of $\psim_L$ and $\psim_L$ is the (closed) convex hull of $E$. The function $x \mapsto \norm{x-y}$ is a continuous convex function, hence achieves its maximum over $\psim_L$ at the set of extreme points. Then,
	\begin{align*}
		\max_{y \in \psim_L} \max_{x \in \psim_L} \norm{x-y} &= 
			\max_{y \in \psim_L} \max_{x \in E} \norm{x-y} = \max_{y \in E} \max_{x \in E} \norm{x-y} 
	\end{align*}
	where the last equality applies the same idea to the function $y \mapsto \norm{x-y}$. The result follows since $\norm{e_\ell - e_k} = \sqrt{2}$ for any $k \neq \ell$.
\end{proof}

\subsection{Lemmas in the proofs of Theorems \ref{thm:null:dist:ncac} and \ref{thm:consist}}
	The following proposition, controlling the tail probability of a randomly-selected Poisson sum, is used in the proof of Lemma~\ref{lem:A:event:prob}:
	\begin{prop}\label{prop:poi:bern:tail}
	Let $A_j \sim \poi(\lambda_j)$ and $U_j \sim \ber(1/2)$ for $ j=1,\dots,n$, and assume that $\{A_j, U_j, j=1,\dots,n\}$ are independent. Let $d = \sum_{j=1}^n A_j U_j$ and $\ds = \ex[d]$. Then,
	\begin{align*}
	\pr \big( |d - \ds| \ge \ds/2\big) \le 2 e^{- 0.008\, \ds}  + 4 e^{-0.03\, \ds/\lamax}
	\end{align*}
	where $\lamax = \max_j \lambda_j$. %
	\end{prop}
	\begin{proof}[Proof of Propoisition~\ref{prop:poi:bern:tail}]
	Let $\dt = \sum_j \lambda_j U_j$ and $\ds = \frac12 \sum_j \lambda_j$, so that $\ds = \ex[\dt]$. Conditioned on $U = (U_1,\dots,U_n)$, $d$ is a Poisson variable with mean $\dt$. If $X \sim \poi(\lambda)$, then for any $t \in (0,1]$, we have $\pr(|X - \lambda | \ge t \lambda) \le 2 \exp( - \lambda t^2 / 4)$; see Lemma~\ref{lem:poi:tail} (Section~\ref{app:tech}).
	Then,
	\begin{align*}
	\pr\big( |d - \dt| \ge 0.2 \dt \mid U \big) \le 2 \exp (-  0.01\dt).
	\end{align*}
	Next, we apply Proposition~\ref{prop:gen:tail:bound} (Section~\ref{app:tech}) to $\dt - \ds = \sum_j \lambda_j(U_j - 1/2)$. Since
	$|\lambda_j(U_j - 1/2)|\le \lamax$ and $\var(\dt - \ds) =\sum_j \lambda_j^2 /4 \le \lamax (\ds / 2)$, we have
	\begin{align*}
	\pr\big(\, |\dt - \ds| \ge \sqrt{\lamax \ds u} + \lamax u /3\, \big) \le 2 e^{-u}.
	\end{align*}
	Taking $u = 0.03 \ds / \lamax$, we obtain $\pr(|\dt - \ds| \ge 0.2 \ds ) \le 2 \exp(-0.03\ds/\lamax)$.
	
	\medskip
	Let $\Ac =  \{ |d - \dt| \ge 0.2 \dt\}$ and $\Bc = \{|\dt - \ds| \ge 0.2\ds\}$. Note that $\Bc$ is completely determined by $U$. On $\Ac^c \cap \Bc^c$, we have $(0.8)^2 \ds < d < (1.2)^2 \ds$, implying $|d - \ds| < \ds/2$. It follows that
	\begin{align*}
	\pr \big( |d - \ds| \ge \ds/2\big) \le \pr(\Ac \cup \Bc) \le \pr(\Ac) + \pr(\Bc).
	\end{align*}
	We have $\pr(\Ac) = \ex[\pr(\Ac \mid U) 1_{\Bc^c} + \pr(\Ac \mid U)  1_{\Bc}]$, hence
	\begin{align*}
	\pr(\Ac) &\le \ex[\pr(\Ac \mid U) 1_{\Bc^c} +   1_{\Bc}] \\
	&\le 2 \ex[e^{- 0.01 \dt} 1_{\Bc^c}] + \pr(\Bc) \\
	&\le 2 e^{- 0.008 \ds} \ex[ 1_{\Bc^c}] + \pr(\Bc)
	\end{align*}
	using $\dt \ge 0.8 \ds$ on $\Bc^c$. We further bound $\ex[ 1_{\Bc^c}] \le 1$. Putting the pieces together finishes the proof.
\end{proof}

\begin{proof}[Proof of Lemma~\ref{lem:A:event:prob}]
Recall that $d_i  = \sum_{j=1}^n A_{ij} U_j$ where $\{U_j = 1\{j \in S_1\}\}$ is an independent $ \ber(1/2)$ sequence, and $\ds_i = \ex[d_i]$. We also recall from~\eqref{eq:ds:lower:bound} that $\ds_i \ge \frac12 C_1 \nu_n$ for all $i \in [n]$. Fix $i \in [n]$.
	We apply Proposition~\ref{prop:poi:bern:tail} to $d_i$ with $\lambda_j = \ex[A_{ij}] = (\nu_n / n)\theta_i \theta_j B^{0}_{z_i z_j}$. Since $\infnorm{B^0} = 1$ and $\thetamax = 1$, we have $\max_j \lambda_j \le \nu_n / n$, and thus
	\begin{align*}
	\frac{\ds_i}{ \max_j \lambda_j} \ge \frac{C_1}{2} n \ge \frac{200}{3} \log n,
	\end{align*}
	where the first inequality is by~\eqref{eq:ds:lower:bound} and the second by the assumption that $\log n / n \le (3/400)C_1$. Proposition~\ref{prop:poi:bern:tail} gives
	\begin{align*}
	\pr \big( |d_i - \ds_i| \ge \ds_i/2\big) \le 2 e^{- 0.004 C_1 \nu_n /2}  + 4 e^{-2\log n} \le 6 n^{-2}
	\end{align*}
	since $0.004 C_1 \nu_n /2 \ge 2\log n$ by assumption. By union bound, 
	\begin{align}
	\pr\big( d_i \notin [\tfrac12 \ds_i, \tfrac32 \ds_i]\;\; \text{for some} \; i \in [n] \big) \le 6 n^{-1}.
	\label{eq:bound_d}
	\end{align}
	
	Furthermore, $\nt_k := |\Gc_k| = n_k - |\Cc_k \cap S_1| = n_k  - \sum_{i \in \Cc_k} U_i$ for all $k$. %
	Applying Proposition~\ref{prop:gen:tail:bound}  (Section~\ref{app:tech}) with $u= 0.01n_k$, we obtain
	\begin{equation*}
    \Big|	\frac{\nt_k}{n_k}  - \frac12 \Big|= \Big|\frac1{n_k}\sum_{i \in \Cc_k} U_i- \frac12 \Big| \ge \sqrt{0.01} + \frac{0.01}{3} \ge 0.1
	\end{equation*}
	with probability $\le 2e^{-0.01n_k}$.
	By union bound 
	\begin{align}
	\pr\big(\nt_k \notin[0.4 n_k, 0.6n_k]\;\;\text{for some} \;k \in [K_0] \big) 
	&\le 2\sum_{k=1}^{K_0}e^{-0.01n_k}\\ 
	&\le 2K_0e^{-0.01\tau_{\Cc}n} \le n^{-1}.
	\label{eq:bound_n}
	\end{align}
	The last inequality is implied by $0.01 \tau_\Cc n \ge\log(n^3) \ge \log (2K_0 n)$ and it holds under  the assumption $\log n / n \le \tauc/300$.
	The result follows by combining \eqref{eq:bound_d} and \eqref{eq:bound_n}.
\end{proof}

\subsubsection{Lemmas in the proof of Theorem~\ref{thm:null:dist:ncac}}

\begin{proof}[Proof of Lemma~\ref{lem:dkol:cond}]
	Let $U_t := \pr(Y \le t\mid \Fc)$ and set $U = (U_t, t \in \reals)$ and $b_t = \pr(Z \le t)$. The function $f(U) = \sup_{t \in \reals} |U_t - b_t|$ is convex, hence by Jensen's inequality
	\begin{align*}
	\dkol(Y,Z) = f(\ex U) \le \ex f(U) = \ex \big[ \dkol\big(\law(Y \mid \Fc), Z\big)\big].
	\end{align*}
	Next, letting $Y' := Y 1_\Bc$, we have 
	\begin{align*}
	\pr(Y' \le t) &\le \pr(\{Y' \le t\} \cap \Bc) + \pr(\Bc^c)  \\
	&= \pr(\{Y \le t\} \cap \Bc) + \pr(\Bc^c) \le \pr(Y \le t) +  \pr(\Bc^c)
	\end{align*}
	and
	\begin{align*}
	\pr(Y' \le t) &\ge \pr(\{Y' \le t\} \cap \Bc) = \pr(\{Y \le t\} \cap \Bc) \ge \pr(Y \le t) -  \pr(\Bc^c).
	\end{align*}
	It follows that $|\pr(Y' \le t) - \pr(Y \le t)| \le \pr(\Bc^c)$ for all $t \in \reals$. An application of the triangle inequality gives $|\dkol(Y,Z) - \dkol(Y',Z) | \le \pr(\Bc^c)$ finishing the proof.
\end{proof}

\subsubsection{Lemmas in the proof of Theorem~\ref{thm:consist}}

\begin{proof}[Proof of Lemma~\ref{lem:E:event:prob}]
	Recall that $\Tch_r \subset \Cc_r \cap S_2 = \Gc_r$ and for any $i \in \Gc_r$, we have $d_i \xi_{i\ell} \sim \bin(d_i,q_{r\ell})$, conditioned on $\Fc$. Thus, we can write $d_i (\xi_{i\ell} - q_{r\ell}) = \sum_{j=1}^{d_i} Z_j$ where $Z_j$ are centered Bernoulli variables with parameter $q_{r\ell}$. Applying Proposition~\ref{prop:gen:tail:bound} (Appendix~\ref{app:tech}),
	we have
	\begin{align*}
	\pr^\Fc \Big( |\sum_j Z_j | \ge \sqrt{2 v u} + \frac{u}{3}\Big) \le2 e^{-u}, \quad u \ge 0,
	\end{align*}
	where $v = \sum_j \var(Z_j)$. Since, $v = d_i q_{r\ell}(1-q_{r\ell}) \le  d_i/4$, taking $u = 2 \log n$, we have
	\begin{align*}
	\pr^{\Fc} \Big(  |\xi_{i\ell} - q_{r\ell} | \ge \sqrt{\frac{\log n}{d_i}} + \frac{2\log n}{3d_i}\Big) \le2 n^{-2}.
	\end{align*}
	On event $\Ac$, we have $d_i \ge \ds_i/ 2 \ge C_1 \nu_n /4$ for all $i$, by~\eqref{eq:ds:lower:bound}. By assumption, $4\log n \le C_1 \nu_n$, hence on $\Ac$,
	\begin{align*}
	\sqrt{\frac{\log n}{d_i}} + \frac{2\log n}{3d_i} %
	\le  4 \sqrt{\frac{\log n}{C_1 \nu_n}} = \eps_n.
	\end{align*}
	We have
	\begin{align*}
		\Ec^c = \Bigl\{ \max_{r, \,\ell}\max_{i \in \Tch_r } |\xi_{i\ell} - q_{r\ell} | \ge \eps_n \Bigr\} 
		&\subset \Bigl\{ \max_{r, \,\ell}\max_{i \in \Gc_r } |\xi_{i\ell} - q_{r\ell} | \ge \eps_n \Bigr\}.
	\end{align*}
	Using  
	$|\bigcup_r \Cc_{h_r}| = n$
	 and the union bound, we obtain $\pr^{\Fc} ( \Ec^c \cap \Ac ) \le 2 (nL) \cdot n^{-2} = 2 Ln^{-1}$. The lemma follows by taking the expectation of both sides and using the smoothing property of conditional expectation.
\end{proof}

    Lemma~\ref{lem:psi:dev:1} follows from the following more refined result:
\begin{lem}\label{lem:psi:dev:2}
	Let $\psi(x,y) = (x-y)^2/y$.
	For all $(x,y)$ and $(x',y')$ in $[0,1] \times [1/c_1,1]$, where $c_1 > 1$, we have
	\begin{align}
	\big|\psi(x',y') - \psi(x,y)\big| \le c_2 |x-y| \cdot \norm{\delta} + c_3 \norm{\delta}^2
	\end{align}
	where $\delta = (x-x',y-y')$, $c_2 = c_1 \sqrt{4 + (1+c_1)^2}$ and $c_3 = 4c_1^3$. 
\end{lem}
Assuming that $|x-x'|\le \eps$ and $|y-y'| \le \eps$, so that $\norm{\delta}\le \sqrt2 \eps$, and using $|x-y|\le 1$,  
\begin{align}
\big|\psi(x',y') - \psi(x,y)\big| \le \sqrt{2}c_2 \eps + 2 c_3 \eps^2 \le c_4 \max(\eps, \eps^2)
\end{align}
where $c_4 = \sqrt 2 c_2 + 2 c_3$. Since $c_2 \le \sqrt{8} c_1^2$, we have $c_4 \le 12 c_1^3$ and Lemma~\ref{lem:psi:dev:1} follows.
\begin{proof}[Proof of Lemma~\ref{lem:psi:dev:2}]
	The function $\psi$ is continuously differentiable of all orders, on $\reals \times \reals_{++}$, with the gradient and Hessian given by
	\begin{align*}
	\grad \psi(x,y) = \big(x/y - 1\big)\begin{bmatrix}
	2 \\ - (1+x/y)
	\end{bmatrix}, 
	\quad 
	\hess \psi(x,y) = (2/y)
	\begin{bmatrix}
	1 & -x/y \\ -x/y & x^2/y^2
	\end{bmatrix}.
	\end{align*}
	The Hessian has eigenvalues $0$ and $2(x^2+y^2)/y^3$. 	%
	By Taylor expansion,
	\begin{align*}
	\psi(x',y') - \psi(x,y) = \ip{\grad \psi(x,y), \delta} + \frac12 \ip{\delta, \hess \psi(\tilde x, \tilde y), \delta}
	\end{align*}
	where $(\tilde x, \tilde y)$ is a point between $(x,y)$ and $(x',y')$. Since $0 \preceq \hess \psi(\tilde x, \tilde y) \preceq 2(\tilde x^2+ \tilde y^2)/\tilde y^3 I_2$ and $\tilde y \ge \min\{y,y'\} \ge 1/c_1$, we obtain
	\begin{align*}
	\big|\ip{\delta, \hess \psi(\tilde x, \tilde y), \delta}\big| \le \frac{2(\tilde x^2+ \tilde y^2)}{\tilde y^3}\norm{\delta}^2 \le 4 c_1^3 \norm{\delta}^2.
	\end{align*}
	We also have
	\begin{align*}
	\big|\ip{\grad \psi(x,y), \delta}\big| & \le |x/y - 1| \sqrt{4 + (1+x/y)^2} \norm{\delta} \le c_2 \norm{\delta}
	\end{align*}
	using the assumption on the ranges of $x$ and $y$. The result follows.
\end{proof}

\subsection{Lemmas in the Proof of Theorem \ref{thm:consist:others}}

\begin{proof}[Proof of Lemma~\ref{lem:controlling:Rc} ]
    For any $x \in \reals^d$, let
\[
    W_\ell(x) := \sum_{j \in S_1} \theta_j g(x,x_j) 1\{y_j = \ell\}.
\]
From the definition of $q_{i\ell}$ in~\eqref{eq:qil:def}, we have
\[
    q_{i\ell} = \frac{\nu_n}{n} \theta_i W_\ell(x_i).
\]
To control $q_{i\ell}$, it is enough to control $W_\ell(x_i)$.
Recall that 
$\Fc_1 = \Fc_0 \vee \sigma(x_{S_2}) = \sigma(S_1, x_{S_2})$.
Note that on $\Fc_1$, both $S_1$ and $S_2$ are fixed.
Then, for $i \in S_2$ and $j \in S_1$, we have
\begin{align*}
    \ex[g(x_i, x_j) \given \Fc_1] &= \ex[g(x_i, x_j) \given x_{S_2}, S_2]  \\
    &=\ex[g(x_i, x_j) \given  x_i]  \\ &= h_{z_j}(x_i)  
\end{align*}
where we have used the independence of $x_i$ and $x_j$. It follows that 
\begin{align}
    \ex[W_\ell(x_i) \given \Fc_1] &=  \sum_{j \in S_1} \theta_j h_{z_j}(x_i) 1\{y_j = \ell\} \notag \\
    &= \sum_{k = 1}^K h_k(x_i) \Rt_{k\ell}
    \label{eq:ex:W:cond:F1}
\end{align}
where $\Rt_{k\ell} := \sum_{j \in S_1}\theta_j 1\{z_j = k, \, y_j =\ell\}$. Furthermore, let $\mt_\ell := \sum_{k=1}^{K} \Rt_{k\ell} = \sum_{j \in S_1} \theta_j 1\{y_j = \ell\}$. 
The next lemma shows that 
$W_\ell(x_i)$ concentrates near its conditional mean.
\begin{lem}\label{lem:McDiarmid:W}
Assume that $\max_{j \in S_1} \theta_j \le 1$ and $g(\cdot,\cdot)$ is bounded above by $1$. Then, for any fixed $x \in \reals^d$, with $\Fc_1$-probability at least $1-2e^{-t}$, 
\begin{align}
     |W_\ell(x) - \ex[W_{\ell}(x) \given \Fc_1]| \le \sqrt{\mt_\ell t / 2}
\end{align}
\end{lem}
\begin{proof}[Proof of Lemma~\ref{lem:McDiarmid:W}]
Conditional on $\Fc_1$, $S_1$ is fixed. We note that $W_\ell(x) = F(x_{S_1})$ where $F(\cdot)$ is a function with the bounded difference property, that is, if $x_{S_1}$ and $x'_{S_1}$ differ only in their $j$th coordinate, then 
$|F(x_{S_1}) - F(x'_{S_1})| \le \theta_j 1\{y_j = \ell\}$ 
since the range of $g$ is in $[0,1]$. By the McDiarmid's inequality, with $\Fc_1$-probability at least $1 - 2 e^{-2u^2/L^2}$, we have $|W_\ell(x) - \ex[W_{\ell}(x)\given \Fc_1]| \le u$, where 
$L^2 :=  \sum_{j\in S_1}\theta_j^2 1\{y_j = \ell\} \le \mt_\ell$. Taking $u^2 = t L^2/2 \le t \mt_\ell/2$ finishes the proof.
\end{proof}

Applying the union bound over $(i,\ell) \in S_2 \times [L]$, we have with $\Fc_1$-conditional probability at least $1- 2 n L e^{-t}$,
\begin{align*}
    \Big|W_\ell(x_i) - \sum_{k = 1}^K h_k(x_i) \Rt_{k\ell}\Big| \le \sqrt{\mt_\ell} t/2, \quad \forall i \in S_2, \; \ell \in [L] .
\end{align*}
where we have used $x_{S_2}$ being fixed given $\Fc_1$. Taking $t = 2 \log n$ and noting that $\mt_\ell \le n$, we can integrate out the conditional probability to get
\begin{align} \label{eq:W:ell:i:concent}
    \pr\Big(\Big|W_\ell(x_i) - \sum_{k = 1}^K h_k(x_i) \Rt_{k\ell} \Big| \le \sqrt{n\log n}, \quad \forall i \in S_2, \; \ell \in [L] \Big) \ge 1- 2Ln^{-1}
\end{align}

We can write $\Rt_{k\ell} = \sum_{j =1}^n \theta_j U_j 1\{z_j = k, \, y_j = \ell\}$, for some i.i.d. $\text{Ber}(1/2)$ sequence $\{U_j\}_{j=1}^n$. 
Recalling the definition of $R_{k\ell}$ from~\eqref{eq:H:def}, we have 
\[
\ex[\Rt_{k\ell}] = \frac12 \sum_{j =1}^n 
\theta_i 1\{z_j = k, \, y_j = \ell\} = 
R_{k\ell}.
\]
 Applying Proposition~\ref{prop:gen:tail:bound} with $v  = n/4 \ge \var(\Rt_{k\ell})$ and $u = \log n$, we have 
\begin{align*}
    \pr\Big(|\Rt_{k\ell} - R_{k\ell}| \ge \sqrt{\frac{n \log n}{2}} + \frac{\log n}{3} \Big) \le 2n^{-1}.
\end{align*}
By union bound, with probability at least $1 - 2LKn^{-1}$,
\begin{align} \label{eq:R:bound}
    |\Rt_{k\ell} - R_{k\ell}| \le \sqrt{n\log n}, \quad \forall k \in [K], \; \forall \ell \in [L].
\end{align}

Let $\Delta_{i\ell} =  W_\ell(x_i) -  \sum_{k = 1}^K h_k(x_i) R_{k\ell}$, and consider the event,
\begin{align*}
    \Wc = \Big\{|\Delta_{i\ell}| \le 2K\sqrt{n \log n}, \quad \forall i \in S_2, \; \ell \in [L] \Big\}.
\end{align*}
Combining~\eqref{eq:W:ell:i:concent} and~\eqref{eq:R:bound}, using $h_k(x)\le 1$, the triangle inequality, and $K+1 \le 2K$, we have $\pr(\Wc^c) \le 4KL n^{-1}$.

Next we note that 
\begin{align*}
    \rho_{i\ell} = 
    \frac{W_{\ell}(x_i)}{ \sum_\ell' W_{\ell'}(x_i)} =  
    \frac{\sum_k h_k(x_i) R_{k\ell} +  \Delta_{i\ell}}{\sum_{\ell'} \bigl(\sum_k h_k(x_i) R_{k\ell'} + \Delta_{i\ell'} \bigr)}.
\end{align*}
Furthermore, on $\Gamma$,
\begin{align*}
    \sum_{\ell'} \sum_k h_k(x_i) R_{k\ell'} &= \frac12 \sum_k h_k(x_i) \sum_{j \in S_1} \theta_j 1\{z_j = k\} \\
    & \ge \frac12\tau_\theta h_{r_{z_i}}(x_i) n_{r_{z_i}} \\
    &\ge \frac12\tau_\theta \tau_\Cc \tau_h n =  \tau_\rho Ln
\end{align*}
where we have used~\eqref{eq:r:k:def} and the definition of $\taur$ in~\eqref{eq:tau:rho:def}.
By the assumption that $\tau_\rho L n > 4K L \sqrt{n \log n}$, on event $ \Gamma \cap \Wc$, applying Lemma \ref{lem:norm:diff} below, we have for all $i \in S_2$ and $\ell \in [L]$,
\begin{align*}
    \Bigl| \rho_{i\ell} - \frac{\sum_k h_k(x_i) R_{k\ell}}{\sum_{\ell'} \sum_k h_k(x_i) R_{k\ell'} } \Bigr| \le \frac{ 4K \sqrt{n \log n }}{\tau_\rho n} = \frac{4K}{\tau_\rho} \sqrt{\frac{\log n}{n}}
\end{align*}
which is the event $\Rc$. That is, we have shown $\Rc \supseteq \Gamma \cap \Wc$, and the claim follows.

\begin{lem}\label{lem:norm:diff}
	For $a = (a_\ell) \in \reals_{+}^L \setminus \{0\}$, let $a_+ = \sum_{\ell=1}^L a_\ell$ and consider the function $U(a) = a_1 / a_+$.
     Let $\delta \in \reals^L$ and  $\infnorm{\delta} = \max_{\ell} |\delta_\ell|$. If $a_+ > L\infnorm{\delta}$, then 
    \begin{align*}
    	|U(a + \delta) - U(a)| \le  \frac{(L-1)\infnorm{\delta}.}{a_+ - L\infnorm{\delta}} 
    \end{align*}    
	In particular, $	|U(a + \delta) - U(a)| \le (2L/a_+) \infnorm{\delta}$ if $a_+ > 2L\infnorm{\delta}$.
\end{lem}
\begin{proof}[Proof of Lemma~\ref{lem:norm:diff}]
	The gradient of $U$ at $c \in \reals_{++}^L$ is given by 
	\[
	\nabla U(c) = \frac{1}{c_+^2} (c_+ - c_1, -c_1,\dots,-c_1).
	\]
	For $a, a+\delta \in  \reals_{+}^L$, there exist $c$ in the line-segment connecting $a$ and $a+\delta$ such that
	$
	U(a+\delta) - U(a) = \ip{\nabla U(c), \delta}.
	$
	From H\"{o}lder's inequality, we have
	\begin{align*}
		|U(a+\delta) - U(a)| \le \norm{\nabla U(c)}_1 \infnorm{\delta}
	\end{align*}
	where
	\[
	\norm{\nabla U(c)}_1 = \frac{1}{c_+^2} \bigl( c_+ - c_1 + (L-1)c_1\bigr) \le \frac{L-1}{c_+}.
	\]
	Noting that $c_+ \ge a_+ - L\infnorm{\delta}$ finishes the proof.
\end{proof}
\end{proof}

\begin{proof}[Proof of Lemma~\ref{lem:A2:event:prob}]
    For $x \in \reals^d$, let $V(x) := \sum_{j \in S_1} \theta_j g(x, x_j)$.
Recall that $d_i = \sum_{j \in S_1} A_{ij}$ and for $i \in S_2$, consider
\begin{align*}
    \dt_i := \ex\big[d_i\given\Fc_2\big] = 
    \sum_{j \in S_1} p_{ij} = \theta_i \frac{\nu_n}{n} V(x_i). 
\end{align*}
We refer to~\eqref{eq:Fc:def2} for the defintion of $\Fc_1, \Fc_2$, etc. Note that $\Fc_1 \subseteq \Fc_2$.
Let $m = \sum_{j \in S_1} \theta_j$. 
Applying the  same idea as in Lemma~\ref{lem:McDiarmid:W}, we have with $\Fc_1$-conditional probability at least $1 - 2e^{-t}$,
\begin{align*}
    \bigl|V(x) - \ex[V(x) \given \Fc_1]\bigr| \le \sqrt{t m /2} \le \sqrt{t n /2}
\end{align*}
where the second inequality uses $\theta_i \le 1$ and $|S_1| \le n$.
Since conditional on $\Fc_1$, $x_i, i \in S_2$ are fixed, it follows that $\Fc_1$-conditional probability at least $1-2ne^{-t}$,
\begin{align*}
    \bigl|V(x_i) - \ex\big[V(x_i) \given \Fc_1\big]\bigr| \le \sqrt{t n /2}, \quad \forall i \in S_2,
\end{align*}
from which we get, multiplying both sides by $\theta_i \nu_n / n$ and using $\theta_i \le 1$,
\begin{align}\label{eq:dti:concent}
    \bigl| \dt_i - \ex[\dt_i \given \Fc_1] \bigr| \le \nu_n \sqrt{t / (2n)}, \quad \forall i \in S_2.
\end{align}
Consider the event
\begin{align}\label{eq:dt:bound}
    \Dc_1 = \Bigl\{\bigl|\dt_i - \ex\big[\dt_i \given \Fc_1\big]\bigr| 
        \le  \nu_n \sqrt{\frac{\log n}{n}}, \ \forall\,  i \in S_2\Bigr\}.
\end{align}
Taking $t = 2\log n$ in~\eqref{eq:dt:bound}, we obtain $\pr(\Dc_1^c \given \Fc_1) \le n^{-1}$, hence $\pr(\Dc_1^{c}) \le n^{-1}$ by taking the expecation of both sides. 

\medskip
Now let us control $\ex\big[\dt_i \given \Fc_1\big]$.
For $i \in S_2$, we have
\[
    \ex[V(x_i) \given \Fc_1] = \ex[V(x_i) \given x_i] =\sum_{j \in S_1} \theta_j h_{z_j}(x_i) =
        \sum_{r} \sum_{j \in S_1 \,\cap \,\Cc_r}  \theta_j h_{r}(x_i)
\] 
On $\Gamma$, by~\eqref{eq:r:k:def}, we have $h_{r_{z_i}}(x_i) \ge \tau_h$ for all $i \in [n]$.
This gives, 
\begin{align*}
    \tau_\theta \tau_h |S_1 \cap \Cc_{r_{z_i}} | \;\le\;
     \ex[V(x_i) \given \Fc_1]    \; \le\; |S_1|
\end{align*}
where we have also used $0 \le h_k(\cdot) \le 1$ and $\tau_\theta \le \theta_j \le 1$. On $\Ac_1$, we have 
$|S_1 \cap \Cc_{r_{z_i}}| \ge 0.4 n_{r_{z_i}} \ge 0.4 \tau_{\Cc} n$ and $|S_1| \le 0.6n$. It follows that on $\Gamma \cap \Ac_1$,
\begin{align*}
    0.4 \tau_\theta^2 \tau_h \tau_\Cc  \nu_n \;\le\;
     \ex[\dt_i \given \Fc_1] \;\le\; 0.6\nu_n 
\end{align*}
for all $i \in S_2$. Recall that $C_8 = \tau_\theta^2 \tau_h \tau_\Cc$. Since by assumption $\sqrt{(\log n) / n} \le 0.2 C_8 \le 0.2$, it follows that on $\Gamma \cap \Ac_1 \cap \Dc_1$, we have
\begin{align}\label{eq:dti:two:sided}
    \dt_i / \nu_n   \in [0.2 C_8, 0.8], \quad \forall i \in S_2.
\end{align}
Next we show that $d_i$ has the same growth rate as $\dt_i$. We have
$
    d_i \given \Fc_2 \sim \poi(\dt_i)
$ for  all $i \in S_2$.
Consider the event
\begin{align}
    \Dc_2 := \{ |d_i - \dt_i| \le 0.2\dt_i,\; \forall i \in S_2\}.
\end{align}
Applying Lemma~\ref{lem:poi:tail}, we have 
$
    \pr(\Dc_2^c \given \Fc_2) \le 2\sum_{i\in S_2}\exp(-0.01\dt_i), 
$
hence
\begin{align*}
    \pr(\Dc_2^c \cap \Ac_1 \cap \Dc_1) &= 
    \ex\bigl[ \pr(\Dc_2^c \cap \Ac_1 \cap \Dc_1 \given \Fc_2) \bigr] \\
    &= \ex[ \pr(\Dc_2^c \given \Fc_2 )1_{\Ac_1 \cap \Dc_1}] \\
    &=\ex[ \pr(\Dc_2^c \given \Fc_2 )1_{\Ac_1 \cap \Dc_1 \cap \Gamma}] \\
    &\le 2\ex \Bigl[ \sum_{i\in S_2} e^{-0.01\dt_i} 1_{\Ac_1 \cap \Dc_1 \cap \Gamma}\Bigr] 
    \le 1.2 n e^{-0.002C_8 \nu_n}.
\end{align*}
where the second equality is since $\Ac_1 \cap\Dc_1$ is deterministic given $\Fc_2$, the third equality is by $\pr(\Gamma) = 1$, and the final inequality uses~\eqref{eq:dti:two:sided} and that $|S_1| \le 0.6 n$ on $\Ac_1$; see~\eqref{eq:event:A1}. The LHS above is also equal to $ \pr(\Dc_2^c \cap \Ac_1 \cap \Dc_1 \cap \Gamma)$. Hence,
\begin{align*}
    \pr(\Dc_2^c \cap \Ac_1 \cap \Dc_1 \cap \Gamma) \le 1.2 n^{-1}
\end{align*}
using the assumption $(\log n) / \nu_n \le C_8 /1000$. We note that on $\Gamma \cap \Ac_1 \cap \Dc_1 \cap \Dc_2$, we have~\eqref{eq:dti:two:sided} and $d_i / \dt_i \in [0.8,1.2]$, which imply 
$d_i  / \nu_n \in [0.16C_8, 0.96]$, that is, $\Ac_2$ hold. Let $\Dc = \Dc_1 \cap \Dc_2$. 
    We have $\Dc^c = \Dc_1^c \uplus ( \Dc_2^c \cap \Dc_1)$ where $\uplus$ denotes the disjoint union. Then,
\begin{align*}
    \pr(\Dc^c \cap \Ac_1) &= \pr(\Dc_1^c \cap \Ac_1) + \pr(\Dc_2^c \cap \Dc_1 \cap \Ac_1)  \\ 
    &\le \pr(\Dc_1^c) + \pr(\Dc_2^c \cap \Ac_1 \cap \Dc_1 \cap \Gamma) 
    \le 2.2 n^{-1}
\end{align*}
and the result follows.
\end{proof}

\begin{proof}[Proof of Lemma~\ref{lem:rho:lb}]
    We first develop a lower bound for $H_{\ell_{z_i}}(x_i)$. Using the $\ell_k$ defined in~\eqref{eq:ell:k:def}, 
\begin{align*}
    R_{k\ell_k} \ge \frac{1}{L}\sum_\ell R_{k \ell} \ge \frac{\tau_\theta n_k}{2L} = \frac{\tau_\theta \tau_\Cc}{2L} n
\end{align*}
Recall that on event $\Gamma$, $h_{r_{z_i}}(x_i) \ge \tau_h$. Then we can control the numerator of $H_{\ell_{z_i}}(x_i)$ by 
\begin{align}\label{eq:H:num:lower:bound}
    \sum_k h_k(x_i)R_{k\ell_{z_i}} \ge h_{r_{z_i}}(x_i) R_{r_{z_i} \ell_{z_i}} \ge \frac{\tau_\theta \tau_\Cc \tau_h }{2L} n
\end{align}
To control its denominator, using $\theta_j \le 1$ and $h_k(x_i) \le 1$, we have 
\begin{align}\label{eq:H:denum:upper:bound}
    \sum_{\ell'} \sum_k h_k(x_i) R_{k\ell'} \le \sum_{\ell'} \sum_k  R_{k\ell} = \frac12 \sum_{j =1}^n \theta_j \le \frac12 n.
\end{align}

Combining~\eqref{eq:H:num:lower:bound} and~\eqref{eq:H:denum:upper:bound} and the definition of $H_\ell(x_i)$, we obtain
$H_{\ell_{z_i}}(x_i) \ge  2\tau_\rho$.
Finally, by definition~\eqref{eq:R:event}, on $\Rc$, we have
\[
\rho_{i\ell_{z_i}} \ge  H_{\ell_{z_i}}(x_i) - \frac{4K }{\tau_\rho} \sqrt{\frac{\log n}{n}},
\] 
which together with the assumption on $\frac{\log n }n$ gives the desired result.
\end{proof}

\begin{proof}[Proof of Lemma~\ref{lem:bound:Delt2}]
    Conditioning on $\Fc$, the quantities $\Gc_k$, $\rhob_{k\ell}$, $(d_i, i \in S_2)$ and $\dpk$ are fixed. Moreover,  by~\eqref{eq:X:il:mult} we have $X_{i\ell} \given \Fc \sim \text{Bin}(d_i, \rho_{i\ell})$. Then, by Proposition 4, 
    \begin{align*}
        \pr^{\Fc}\Bigl(\Big|\sum_{i\in \Gc_k} (X_{i\ell} - d_i\rho_{i\ell})\Big| \ge \sqrt{2vu} + \frac{u}{3}\Bigr) \le 2e^{-u}
    \end{align*}
    for any $v \ge \var(\sum_{i\in \Gc_k} X_{i\ell})$  and $\pr^\Fc$ denote the probability conditional on $\Fc$.  We have $\var(\sum_{i\in \Gc_k} X_{i\ell}) = \sum_{i\in \Gc_k} d_i\rho_{i\ell}(1-\rho_{i\ell})\le \dpk/4$ .
    Taking $v = \dpk/4$, $u = 2\log n$, we have 
    \begin{align*}
        \pr^{\Fc}\Biggl(|\Delt_{k\ell}| \ge \sqrt{\frac{\log n}{\dpk}} + \frac{2\log n}{3\dpk}\Biggr) \le 2n^{-2}.
    \end{align*}
    From~\eqref{eq:min:dpk}, on event $\Ac$, we have $\dpk \ge 0.064 \tau_\Cc C_8 n \nu_n \ge \log n$ for all $k \in [K]$, where the second inequality is by assumption. It follows that
    \begin{align*}
        \sqrt{\frac{\log n}{\dpk}} + \frac{2\log n}{3\dpk}\le 2 \sqrt{\frac{\log n}{\dpk}} \le
        \frac{8}{\sqrt{\tau_\Cc C_8}}\sqrt{\frac{\log n}{n \nu_n}}.
    \end{align*}
    Therefore, $\pr^\Fc(\Bc^c \cap \Ac) = \pr^{\Fc}(\Bc^c) 1_{\Ac} \le 2KLn^{-2} \le 2Ln^{-1}$ by union bound and since on $\Fc$, the event $\Ac$ is deterministic. The result follows by taking the expectation to both sides. 
\end{proof}

\begin{proof}[Proof of Lemma~\ref{lem:Delh:k:ell:k}]
    We use an idea similar to the one used in~Lemma~\ref{lem:control:Delh} in the proof of Proposition~\ref{prop:Thn:Tn}. We note that 
$\delta$ plays a similar role in both proofs, and $\rhob_{k\ell}$ and $\rhot_{k\ell}$ here play the role of $p_{k\ell}$ and $\pt_{k\ell}$ there.
Let
\begin{align}
    \tau_d := \frac{\omega_n}{\max_{i \in S_2}d_i}.
\end{align}
Combining~\eqref{eq:event:A2} and~\eqref{eq:omegan:lower:bound}, on $\Ac$, we have
\begin{align}\label{eq:tau:d:lower:bound}
    \tau_d \ge \tau_\Cc C_8/9.
\end{align}
By Lemma \ref{lem:rho:lb}, on $\Gamma \cap \Rc$, we have $\rhob_{k \ell_k} \ge \tau_\rho$ for all $k \in [K]$. Hence, $\tau_\rho$ plays the role of $\pul$ in  Proposition~\ref{prop:Thn:Tn}. 

Then, to apply Lemma~\ref{lem:control:Delh}, we need $\alphan \le  \tau_d \tau_\rho /2$ and $\delta \le \tau_\rho/2$, with $\tau_\rho$. By~\eqref{eq:tau:d:lower:bound}, the first condition is satisfied on $\Ac$, if $\alphan \le \tau_\Cc \tau_\rho C_8 /18$, which holds by assumption. Then, the equivalent of Lemma~\ref{lem:control:Delh} in this proof implies that on $\Bc \cap \Mc_n \cap (\Gamma \cap \Rc \cap \Ac)$, we have $|\Delh_{k\ell_k}| \le \delh \cdot \rhot_{k\ell_k}$ for all $k \in [K]$, where

\begin{align*} 
	 \delh := \frac{6}{\rhou\, \tau_d} \alphan  \le \frac{54}{\tau_\rho \tau_\Cc C_8} \alphan.
\end{align*}
\end{proof}

\begin{proof}[Proof of Lemma~\ref{lem:Ec:bound}]
        The proof is similar to that of Lemma~\ref{lem:bound:Delt2} to which we refer for more details. We have $X_{i\ell} \given \Fc \sim \bin(d_i, \rho_{i\ell})$. Hence, by Proposition~\ref{prop:gen:tail:bound},
    \begin{align*}
        \pr^{\Fc}\Big(|\xi_{i\ell} - \rho_{i\ell}| \ge \sqrt{\frac{\log n }{d_i}} + \frac{2\log n}{3d_i}\Big) \le 2n^{-2}.
    \end{align*}
    Recalling~\eqref{eq:event:A2}, on $\Ac$, we have $d_i \ge 0.16 C_8\nu_n$ for $i \in S_2$ and by assumption $\log n \le 0.16 C_8\nu_n$. Hence, on $\Ac$, 
    \begin{align*}
        \sqrt{\frac{\log n }{d_i}} + \frac{2\log n}{3d_i} \le 5 \sqrt{\frac{\log n}{C_8 \nu_n}}
    \end{align*}
    By union bound over $\ell \in [L]$, we obtain $\pr^{\Fc}(\Ec^c \cap \Ac) \le 2L n^{-1}$. The result then follows by taking the expectation to both sides.
\end{proof}

\begin{proof}[Proof of Lemma~\ref{lem:varpi}]
Let $\evar((a_i),(p_i))$ be the variance of a random variable that takes values $a_i$ with probability $p_i$, that is,
\begin{align*}
	\evar((a_i),(p_i)) &= \sum_{i} p_i\Bigr(a_i - \sum_j p_j a_j\Bigl)^2 = \frac12 \sum_{i,j} p_i p_j (a_i - a_j)^2.
\end{align*}
Note that $\varpi_k = \evar((\rho_{i\ell_k}),(d_i/\dpk))$. 
The next step is to show that in the definition of $\varpi_k$, we can replace $\rho_{i \ell_k}$ with $H_{\ell_k}(x_i)$ and $d_i/\dpk$ with deterministic quantities. 

\begin{lem}\label{lem:a:b:diff}
	For $a_i, b_i \in \reals$, with $\max_i |a_i-b_i| \le \eps$, we have
	$|a_i-a_j| \ge |b_i - b_j| - 2\eps$. 
\end{lem}
\begin{proof}[Proof of Lemma~\ref{lem:a:b:diff}]
Assume $|a_1 - b_1| \le \eps$ and $|a_2 - b_2| \le \eps$. Then,
\begin{align*}
    |a_1 -  a_2| &= |(b_1 - b_2) + (a_1 - b_1) - (a_2 - b_2)|\\
    &\ge |b_1 - b_2| - |(a_1 - b_1) - (a_2 - b_2)|\\
    &\ge |b_1 - b_2| - |a_1 - b_1| -|a_2 - b_2|\\
    &\ge |b_1 - b_2| - 2\eps.
\end{align*}
\end{proof}

Let us define
\begin{align}\label{eq:zetan}
	\zeta_n := \frac{4\tau_\theta L}{C_8} \sqrt{\frac{\log n}{n}} =
	\frac{4 L}{\tau_\Cc \tau_h \tau_\theta} \sqrt{\frac{\log n}{n}}.
\end{align}
so that on $\Rc$, we have (see~\eqref{eq:R:event})
\[
|\rho_{i\ell} - H_{\ell}(x_i)| \le \zeta_n \quad \forall i \in S_2, \; \forall\ell \in [L].
\]
Then, by Lemma~\ref{lem:a:b:diff}, for all $k \in [K]$,
\begin{align*}
	|\rho_{i\ell_k} - \rho_{j\ell_k}| \ge |H_{\ell_k}(x_i) - H_{\ell_k}(x_j)| - 2\zeta_n.
\end{align*}
Using the fact that $a \ge b -c$ implies $a^2 \ge \frac12 b^2 - c^2$ for $b \ge 0$, we have
\begin{align*}
	(\rho_{i\ell_k} - \rho_{j\ell_k})^2 \ge \frac12 [H_{\ell_k}(x_i) - H_{\ell_k}(x_j)]^2 - 4 \zeta_n^2
\end{align*}
Let 
$s_k := |\Gc_k|$
Recall that $\dpk = \sum_{i \in \Gc_k} d_i$. Then, on $\Ac$, we have 

\begin{align*}
	\frac{d_i}{\dpk} \ge \frac{0.16 C_8 \nu_n}{0.96 \nu_n s_k} = \frac{C_8}{6s_k}.
\end{align*}
It follows that on $\Ac \cap \Rc$, we have
\begin{align}
	\varpi_k &= \frac12 \sum_{i, j \in \Gc_k}\frac{d_i}{\dpk}\frac{d_j}{\dpk}(\rho_{i\ell_k} - \rho_{j\ell_k})^2 \notag \\
	&\ge 
	\frac12 \frac{C_8^2}{36 s_k^2}  \sum_{i, j \in \Gc_k}\Bigl(  \frac12 [H_{\ell_k}(x_i) - H_{\ell_k}(x_j)]^2 -  
	4\zeta_n^2 1\{i \neq j\} \Bigr) \notag \\
	&\ge \frac14 \frac{C_8^2}{36} \Big[\frac{1}{4\binom{s_k}{2}}\sum_{i, j \in \Gc_k} [H_{\ell_k}(x_i) - H_{\ell_k}(x_j)]^2 - 8\zeta_n^21\{i \neq j\}\Big] \label{eq:varpi:lower:1}
\end{align}
since by~\eqref{eq:Gsize:lower} $s_k \ge 2$ and hence $4 \binom{s_k}{2} \ge s_k^2$. The first term above is proportinal to a $U$-statistic providing an estimate of the variance of $H_{\ell_k}(x), x \sim \Qb_k$ based on an i.i.d. sample $x_i \sim \Qb_k, i \in \Gc_k$ (assuming that $S_2$ is fixed).  An argument using the Hansen--Wright inequality shows that such a quantity is concentrated around its mean, which is the population variance. We use the following result from~\cite{kazemitabar2017variable}, with  slight modifications:
\begin{prop}[Corollary 3 in~\cite{kazemitabar2017variable}]\label{prop:var:concent}
	Let $w = (w_1,\dots,w_m) \in \reals^m$ be a random vector with independent components $w_i$ which satisfy $\norm{w_i-\ex w_i}_{\psi_2} \le K$. Let \[\impu(w) := 
	\frac1{\binom{m}{2}} \sum_{1 \le i , j \le m}  \frac14(w_i - w_j)^2\] be the empirical variance of $w$. Then, there is a universal constant $c > 0$ such that for $u \ge 0$,
	\begin{align}
		\pr \Big(  \impu(w) - \ex \impu(w) <- K^2 u \Big) \le 
		\exp\big\{ {-c}\, (m-1) \min(u,u^2) \big\}.
	\end{align}
\end{prop}
We note the alternative expression $\impu(w) =  \binom{m}{2}^{-1} \sum_{1 \le i < j \le m}  \frac12(w_i - w_j)^2$. In the context of Proposition~\ref{prop:var:concent}, if $w_1,\dots,w_m$ are i.i.d., then 
\[
\ex \impu(w) = \frac12 \ex(w_1 - w_2)^2 = \var(w_1).
\]
Since $H_{\ell_k}(\cdot)$ is bounded in $[0,1]$, we have $\norm{H_{\ell_k}(x_i)-\ex H_{\ell_k}(x_i)}_{\psi_2} \le 1$. Recall that $\fvar_{k\ell} = \var(H_\ell(x))$ when $x \sim \Qb_k$. Then, conditional on $\Fc_0 = \sigma(S_1)$ so that $\Gc_k$ is fixed, we have for $i,j \in \Gc_k$ and $i\neq j$,
\begin{align*}
	\frac12 \ex[H_{\ell_k}(x_i) - H_{\ell_k}(x_j)]^2 = \fvar_{k \ell_k}
\end{align*}
Applying the Proposition~\ref{prop:var:concent}, we obtain, for $u \in [0,1]$,
\begin{align*}
	\pr^{\Fc_0} \Bigl( 
	\frac{1}{4 \binom{s_k}{2}} \sum_{i, j \in \Gc_k} [H_{\ell_k}(x_i) - H_{\ell_k}(x_j)]^2 < \fvar_{k \ell_k} - u
	\Bigr) \le e^{-c(s_k-1) u^2}
\end{align*}
On $\Ac$, $s_k -1 \ge s_k/2 \ge 0.2 \tau_\Cc n$. Take $u = u_n := \sqrt{\log n / (\tau_\Cc n)}$. By the scaling assumption $ \log n / n \le \tau_\Cc$, we have $ u _n \le 1$, hence
\begin{align*}
	\pr^{\Fc_0} \Bigl( 
	\frac{1}{4 \binom{s_k}{2}} \sum_{i, j \in \Gc_k} [H_{\ell_k}(x_i) - H_{\ell_k}(x_j)]^2 < \fvar_{k \ell_k} -  u_n
	\Bigr) 1_{\Ac \cap \Rc} \le n^{- c_1}
\end{align*}
where $c_1 = 0.2 c$. Combining with~\eqref{eq:varpi:lower:1}
\begin{align*}
	\pr^{\Fc_0} \bigl( \frac{144}{C_8^2} \varpi_k + 4\zeta_n^2 <\fvar_{k \ell_k} -  u_n \bigr)  1_{\Ac \cap \Rc} \le n^{- c_1}
\end{align*}
Taking the union bound and removing the conditioning, we get
\begin{align*}
	\pr \bigl(\bigl\{ \exists k \in [K],\; \frac{144}{C_8^2} \varpi_k + 4\zeta_n^2 <\fvar_{k \ell_k} -  u_n \bigr\} \cap \Ac \cap \Rc \bigr)   \le K n^{- c_1}
\end{align*}
Let us call the first event above $\Hc^c$. Then, on $\Hc$, we have
\begin{align}
	\frac{144}{C_8^2}\varpi_k \ge \fvar_{k \ell_k} - (4\zeta_n^2 + u_n), \quad \forall k \in [K].
\end{align}
We have, using assumption $\zeta_n \le 1$,
\begin{align}
	4\zeta_n^2 + u_n \le 4\zeta_n + u_n \le 
	\frac{16 L}{\tau_\Cc \tau_h \tau_\theta} \sqrt{\frac{\log n}{n}} + 
	\sqrt{\frac{\log n}{\tau_\Cc n}} \le   
	\frac{18 L}{\tau_\Cc \tau_h \tau_\theta} \sqrt{\frac{\log n}{n}} 
\end{align}
since $\tau_\Cc \le 1$. 
\end{proof}

\section{Other technical results}\label{app:tech}

\begin{lem}\label{lem:var:low}
  	Assume that $Z$ is a random variable taking values $z_1,\dots,z_R$ with probabilities $\beh_1,\dots,\beh_R$ respectively. Then, $\var(Z) \ge \frac12 \beh_1 \beh_2 (z_1 - z_2)^2$.
\end{lem}

\begin{proof}
  Let $Z'$ be an independent copy of $Z$. Then $\var(Z) = \frac12 \ex(Z-Z')^2$, and $(Z,Z')$ takes value  $(z_1,z_2)$ with probability $\beh_1\beh_2$. The result follows.
\end{proof}
  
	\begin{lem}\label{lem:poi:tail}
	Let $X \sim \poi(\lambda)$. Then, for any $t \in (0,1]$, 
	\[\pr(|X - \lambda | \ge t \lambda) \le 2 \exp( - \lambda t^2 / 4).\]
\end{lem}
\begin{proof}[Proof of Lemma~\ref{lem:poi:tail}]
	Fix $t \in (0,1]$. For $\theta \in (0,1.79]$, by the Chernoff bound,
	\begin{align*}
	\pr(X-\lambda \ge t\lambda) \le  e^{-\theta t \lambda } \ex[e^{(X-\lambda)\theta}] = e^{-\theta t \lambda}\exp(\lambda(e^\theta -1 -\theta)) \le e^{\lambda \theta^2 - \theta t \lambda}
	\end{align*}
	using $e^\theta -1 -\theta \le \theta^2$ when $\theta \le 1.79$. Since $\lambda\theta^2 - \theta t \lambda$ attains its minimum at  $\theta = t/2\le 1$, we obtain $ \pr(X-\lambda \ge t\lambda) \le \exp(-\lambda t^2/4)$.
	On the other hand, 
	\begin{align*}
	\pr(\lambda -X \ge t\lambda) 
	\le e^{-\theta t \lambda} \ex[e^{(\lambda-X)\theta}]
	=e^{-\theta t \lambda}\exp(\lambda(e^{-\theta} -1 +\theta)) \le e^{\lambda\theta^2/2 - \theta t \lambda}
	\end{align*}
	using $e^{-\theta} -1 +\theta \le \theta^2/2$ for $\theta \ge 0$. Since $\lambda\theta^2/2 - \theta t \lambda$ attains its smallest value at  $\theta = t\le 1$, we get
	$
	\pr(\lambda -X \ge t\lambda)  \le \exp(-\lambda t^2/2),
	$ finishing the proof.
\end{proof}

\begin{prop}[Giné and Nickl~\cite{gine_nickl_2015}~Theorem~3.1.7]
\label{prop:gen:tail:bound}
	Let $S = \sum_{i=1}^n X_i$ where $\{X_i\}$ are independent random variables with $|X_i - \ex X_i| \le c$ for all $i$. Let $v \ge \var(S)$. Then, for all $u \ge 0$,
	\begin{align*}
	\pr \Big( |S - \ex S| \ge \sqrt{2 v u} + \frac{c u}{3} \Big)  \le 2 e^{-u}.
	\end{align*}
\end{prop}
In particular, if $S \sim \bin(n,p)$, then we can take $v = \ex[S] \ge \var(S)$. Letting $\ph = S/n$, the result gives 
$
\pr\big(| \ph - p| \ge \sqrt{\frac{2 p u}{n}} + \frac{u}{3n}\big) \le 2e^{-u}.
$

\section{Extra simulations}

\subsection{Degree growth rate}\label{app:deg:grow}
Figure~\ref{fig:deg:growth} shows that $h(d)$ grows linearly with $n$ in FB networks. To generate the plot, we first selected networks of similar size ($\sim$ 9k nodes) from the FB-100 data set. From each of those original networks, we randomly sampled 100 nodes as the initial seed set. We then kept adding more nodes randomly to the seed sets, recoding the induced subnetwork within the original network, to create networks of increasing size $n$. 
The plot shows the $h(d)$ of the seed set as the network size grows. 
	
\begin{figure}[t!]
    \centering
    \includegraphics[width=.6\textwidth]{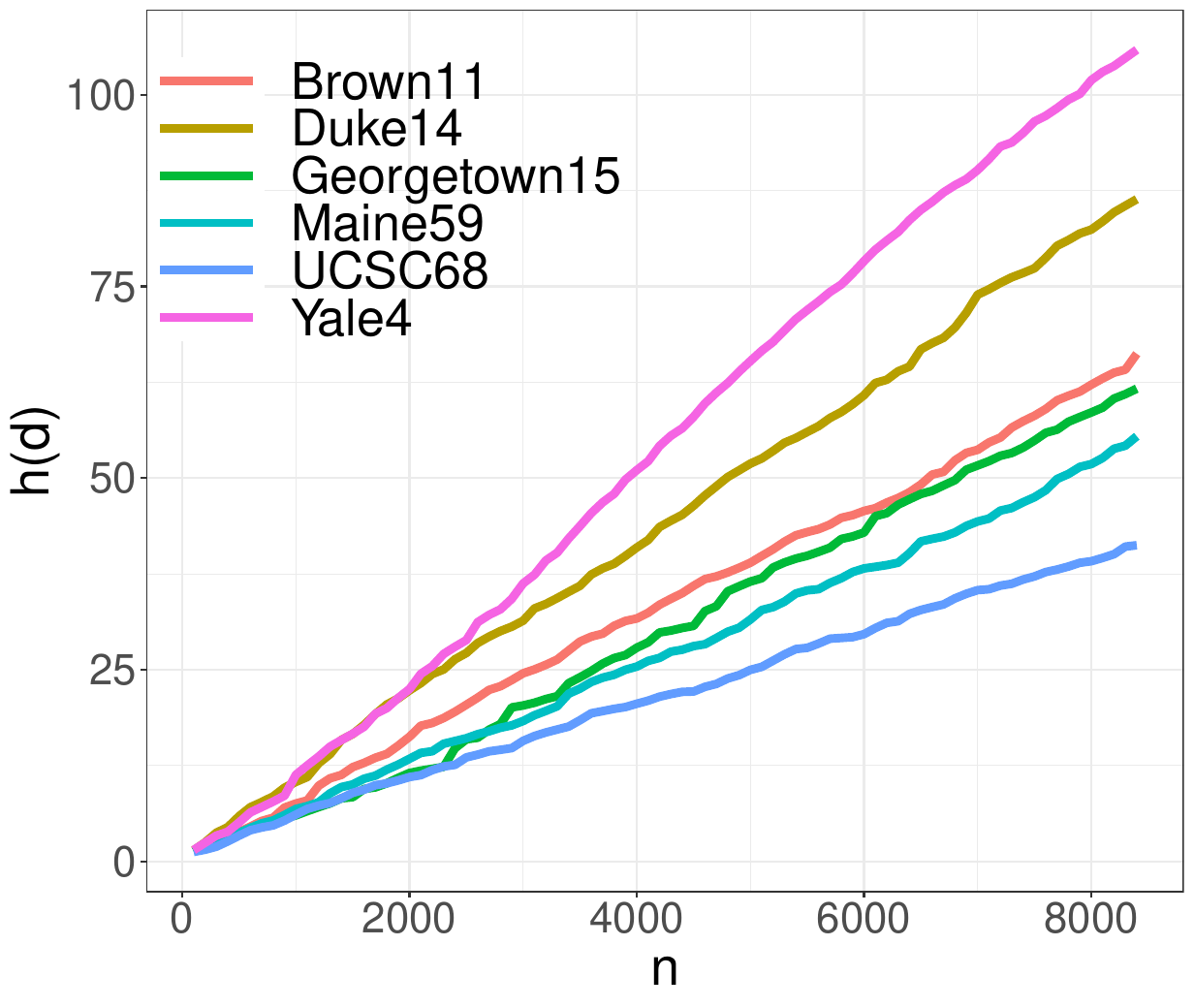}
    \caption{For a fixed seed set of nodes, its $h(d)$ ($y$ axis) grows as the subnetwork, which the seed set belongs to, increases in size ($x$ axis).}
    \label{fig:deg:growth}
\end{figure}

\subsection{Bootstrap comparison}
\label{sec:boot:comp}
 
\begin{figure}[h!]
    \centering
    \begin{tabular}{cc}
    \includegraphics[width=.49\textwidth]{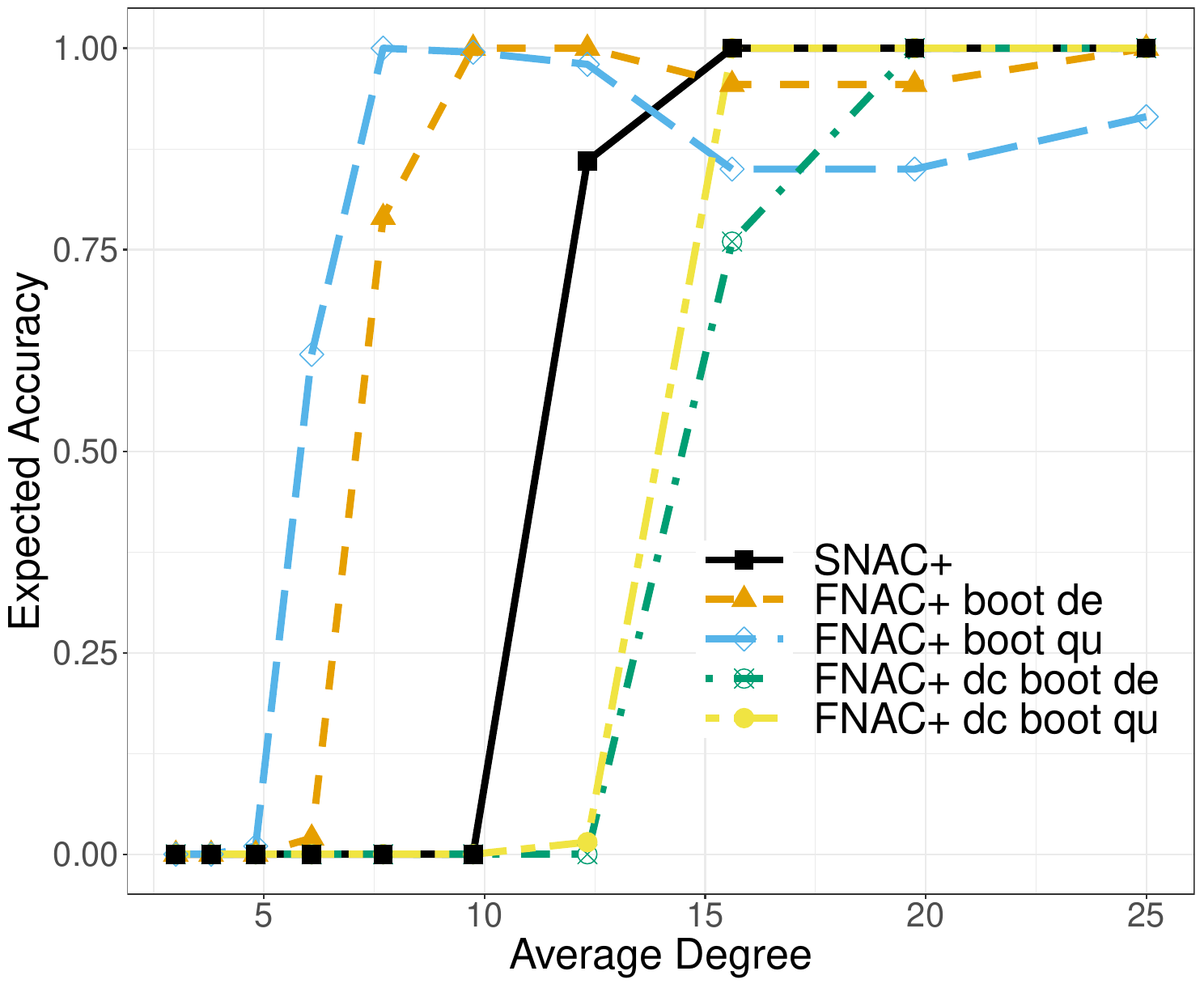} &
    \includegraphics[width=.49\textwidth]{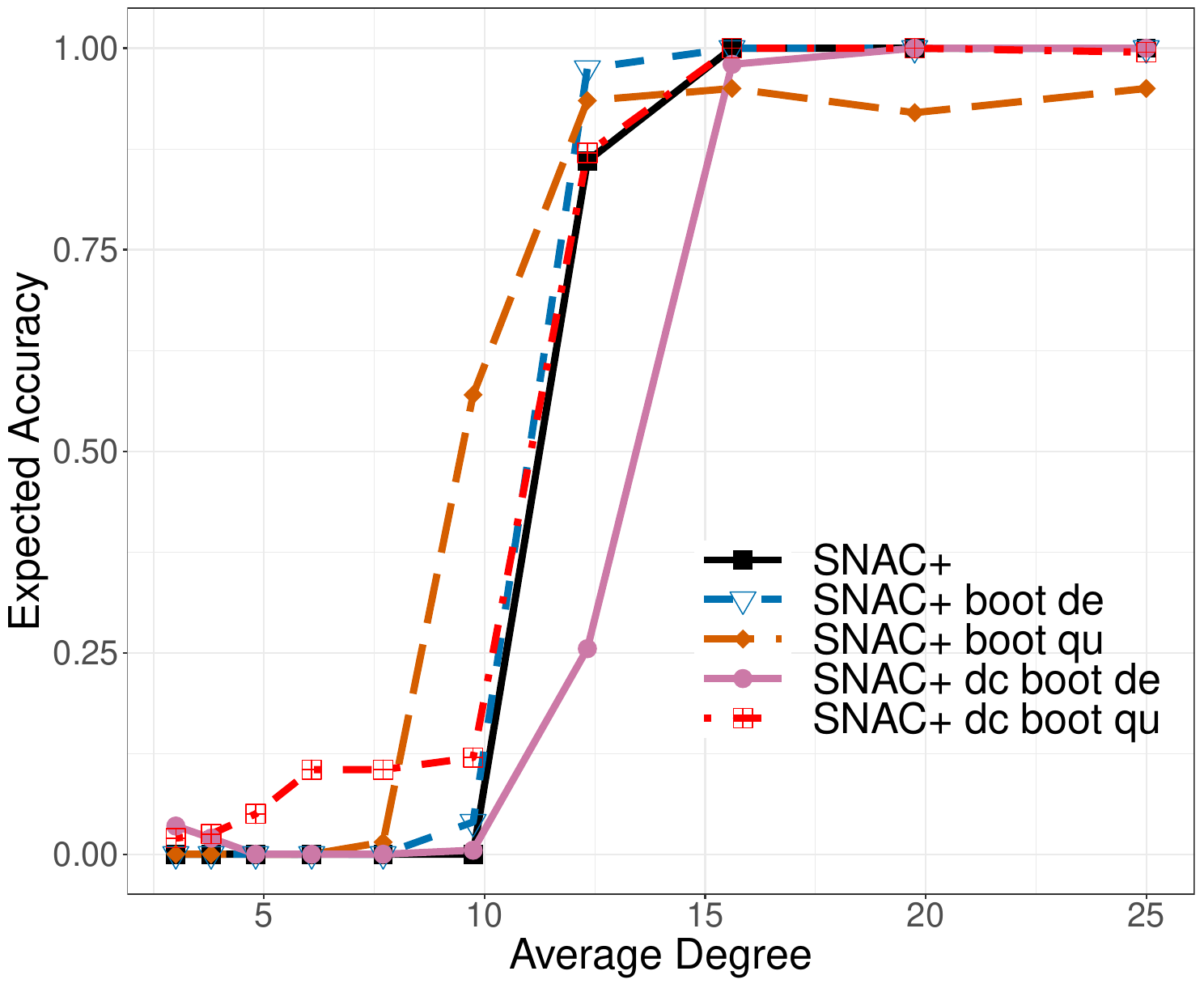}
    \\ 
    (a) Variants of \cacp 
    & (b) Variants of \scacp 
    \end{tabular}
     \caption{Comparing different  bootstrap approaches using expected accuracy of selecting the true number of communities versus expected average degree of the network. \scacp is shown in both plots as a benchmark. Details of each method in the legend is explained in the text. }
    \label{fig:boot:comparison}
\end{figure}
Based on Section~\ref{sec:boot},
we consider four versions of the bootstrap approach to determine the optimal rejection rules for \cacf and \scacf 
empirically:
\begin{enumerate}
    \item \textbf{boot de}: Generate SBM bootstrap samples and obtain their mean and standard deviation to standardize the original statistic.  Reject the null hypothesis with level-$\alpha$ critical threshold from the standard normal.
    \item \textbf{boot qu}: Generate SBM bootstrap samples and use their $\alpha$-quantile as the rejection threshold.
    \item \textbf{dc boot de}: Same as \textbf{boot de} except for generating DCSBM bootstrap samples instead.
    \item \textbf{dc boot qu}: Same as \textbf{boot qu} except for generating DCSBM bootstrap samples instead.
\end{enumerate}

In Figure \ref{fig:boot:comparison}, we compare the above four approaches for bootstrapping 
\scacp and \cacp.
In both cases, we include \scacp without bootstrap as the comparison baseline. The simulation data follows a DCSBM with $n = 5000$, $K = 4$, $\theta_i \sim \text{Pareto}(3/4, 4)$, connectivity matrix as $B_1$ defined in the paper and balanced community sizes. For both \scacp and \cacp, the \textbf{boot~de} approach has the most stable performance and that is why we use it in simulations of Section~\ref{sec:model_select}. However, there is no absolutely superior choice among all, and in practice, one can try different bootstrap approaches and compare results to make a conclusion about hypothesis testing.  


\subsection{Model selection}\label{app:mod:sel}

\begin{figure}[t]
	\centering
	\includegraphics[width=0.49\linewidth]{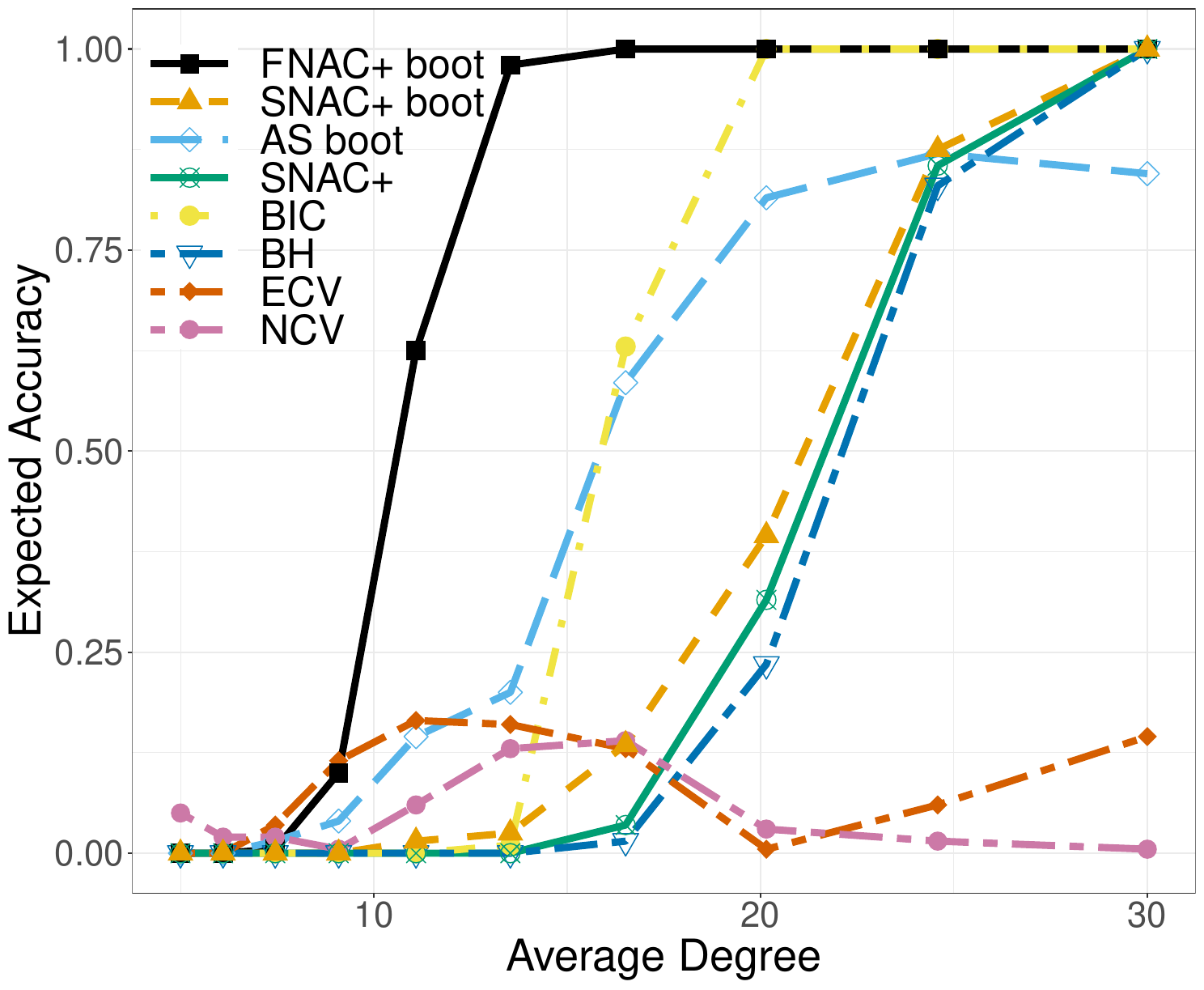}
	\includegraphics[width=0.49\linewidth]{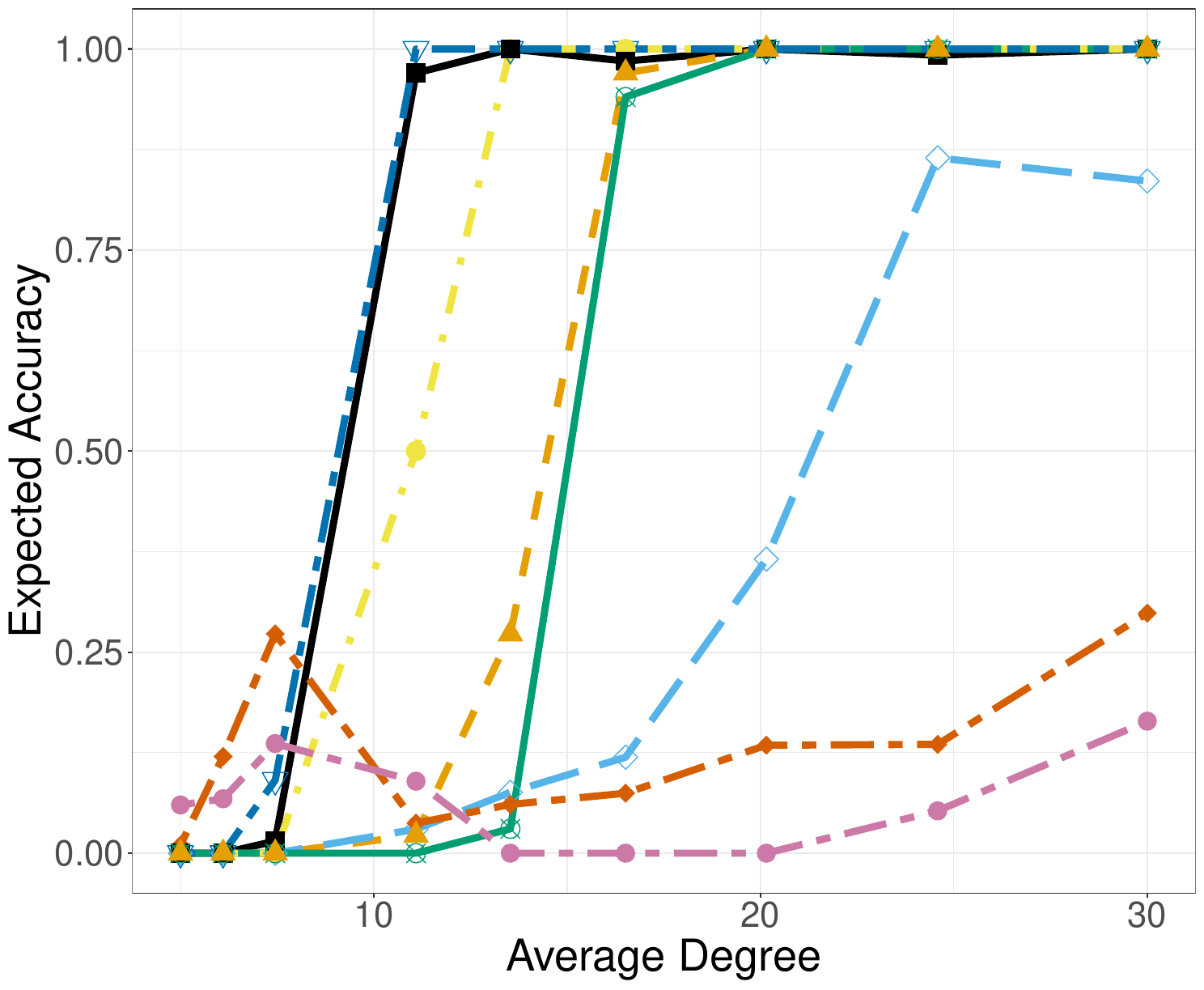}
	\includegraphics[width=.49\textwidth]{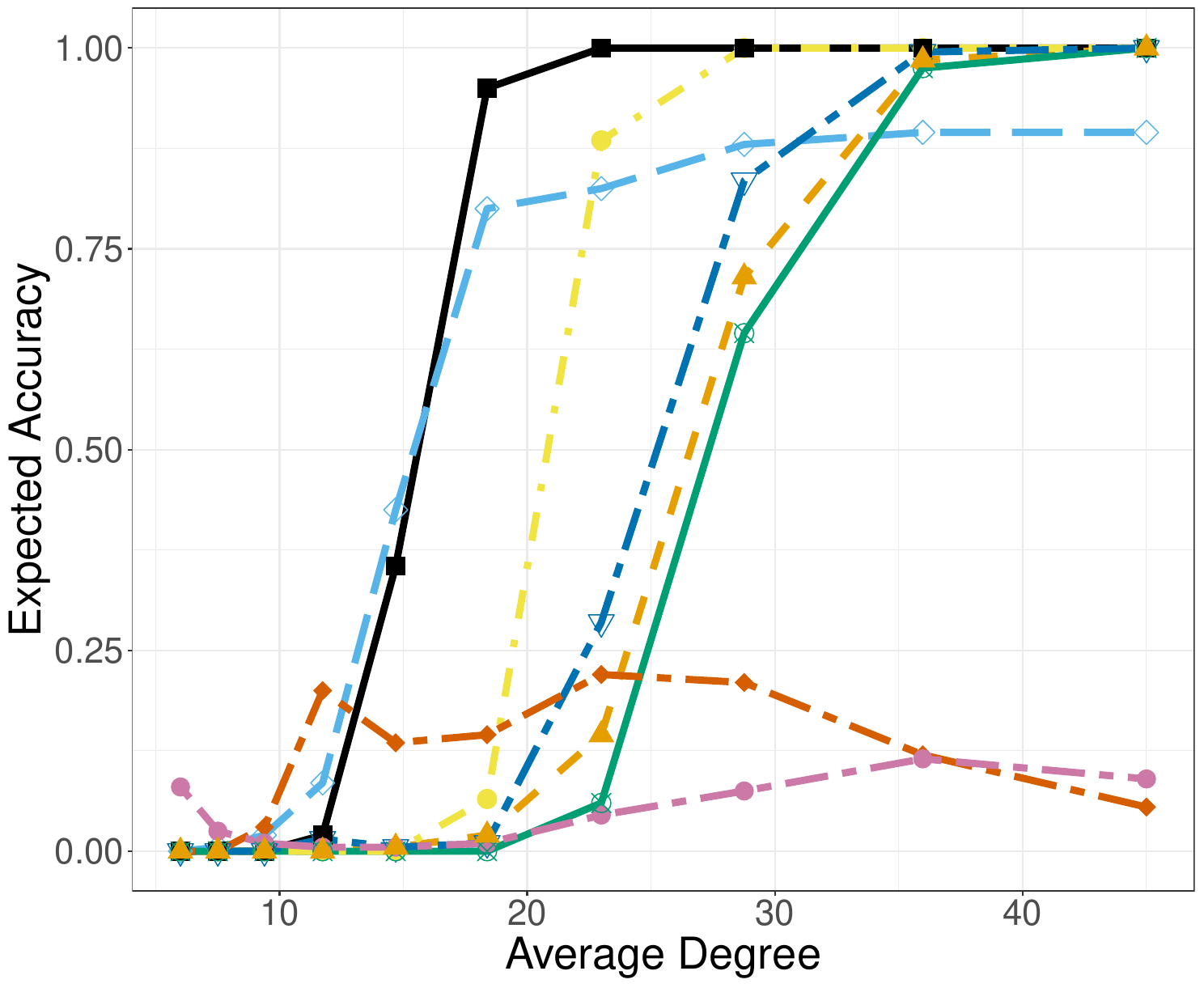}
	\includegraphics[width=.49\textwidth]{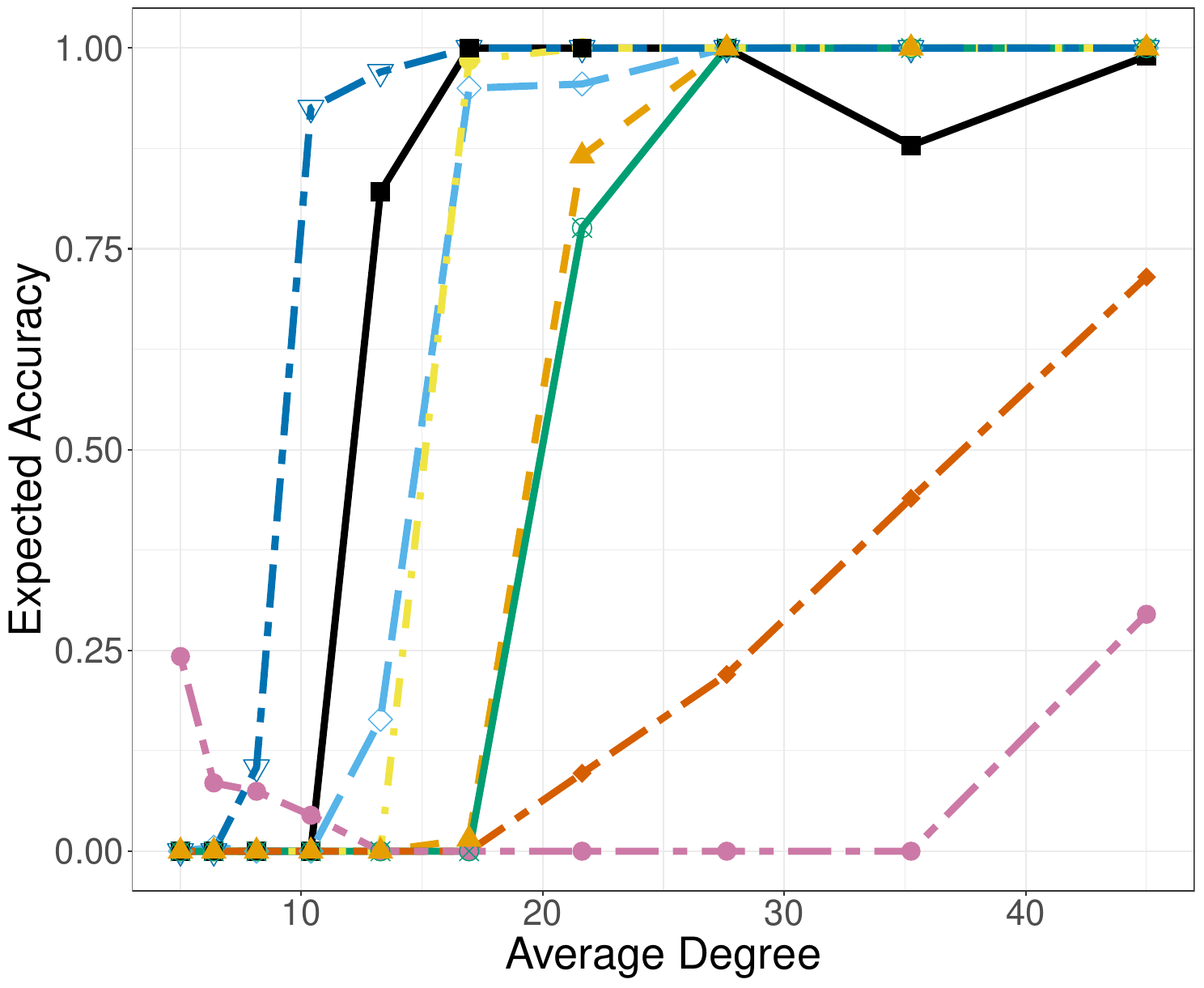}
	\caption{Expected accuracy of selecting the true number of communities versus expected average degree of the network. The data follows a DCSBM with $n = 5000$, $\theta_i \sim \text{Pareto}(3/4,4)$. The first row is generated with $\beta=0.2$, connectivity matrix $B_3$. The top left plot has unbalanced community sizes proportional to $(1,1,2,3)$ and the right plot has balanced community sizes. The second row is generated with connectivity matrix $B_1$. The bottom left plot has community sizes proportional to $(1,2,3,4)$ and out-in-ratio $\beta = 0.2$. The bottom right plot has balanced community sizes and out-in-ratio $\beta = 0.3$.}
	\label{fig:seq:unequal:diag}
\end{figure}

Figure~\ref{fig:seq:unequal:diag} shows model selection accuracy with four variants of DCSBM parameters. All plots have DCSBM with parameters $n = 5000$, $\theta_i \sim \text{Pareto}(3/4, 4)$. The top row is generated with a generalized version of $B_1$ as the connectivity matrix, given by 
\begin{align*}
	B_3 \propto (1-\beta)\diag(w) + \beta \mathbf{1} \mathbf{1}^{T}.
\end{align*}
Note that $B_1$ is a special case of $B_3$ where $w$ is the an all-ones vector. Here, we set $w=(1,2,3,1)$ under $K = 4$. The top left plot shows the case where the DCSBM has unbalanced community sizes proportional to $(1,1,2,3)$ and the right plot shows balanced community sizes.

The bottom row is generated based on the planted partition model, but with different community sizes and out-in-ratio than that in Figure~\ref{fig:model_select}. The bottom left plot has unbalanced community sizes proportional to $(1,2,3,4)$ and out-in-ratio $\beta = 0.2$ and the bottom right plot has balanced community sizes and out-in-ratio $\beta = 0.3$.
All methods have lower accuracy in the unbalanced setting except for the AS. BH is affected the most while \cacp the least. The robustness of \cacp  could be because its performance mainly relies on the full version of $\rho$  and unbalanced sizes retain its rows' distinction. However, the \scacp is still affected by the unbalanced community sizes because of the increased difficulty in recovering the correct labels and the increased variance in $\rho$ due to subsampling.

\subsection{ROC curves}\label{app:roc}

\begin{figure}[t]
	\centering
	\includegraphics[width=0.47\linewidth]{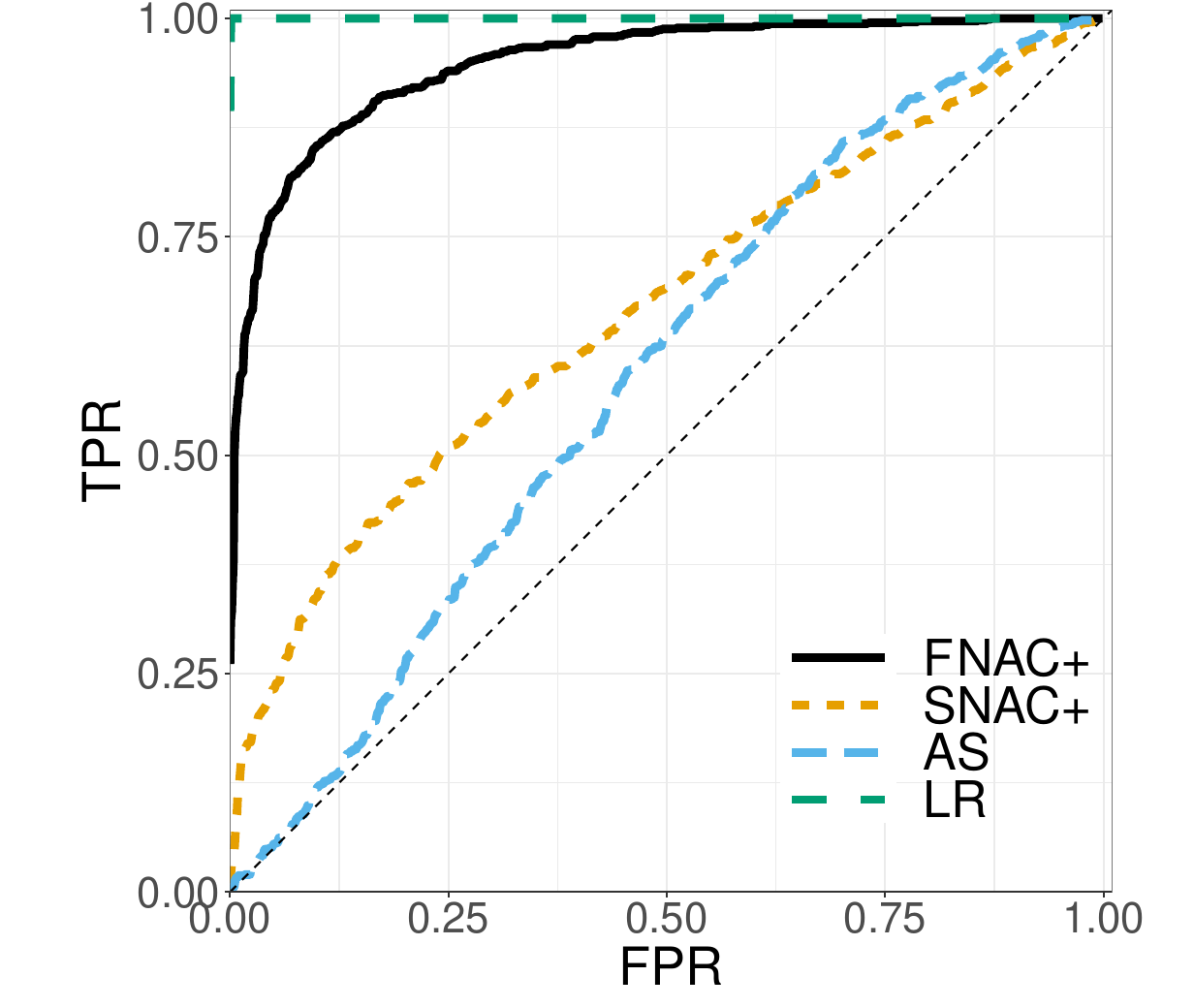} 
	\includegraphics[width=0.47\linewidth]{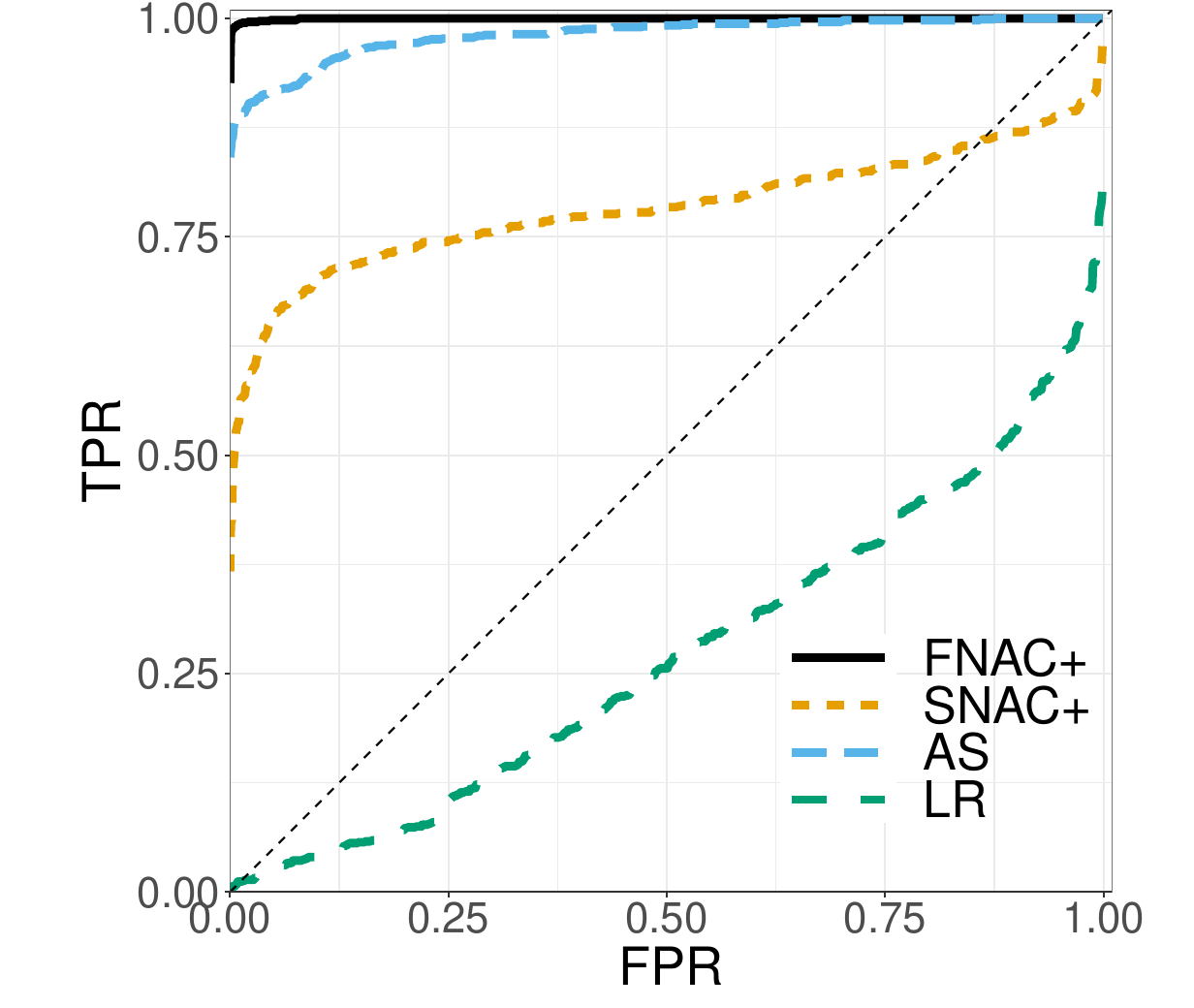} 
	\\
	\includegraphics[width=0.47\linewidth]{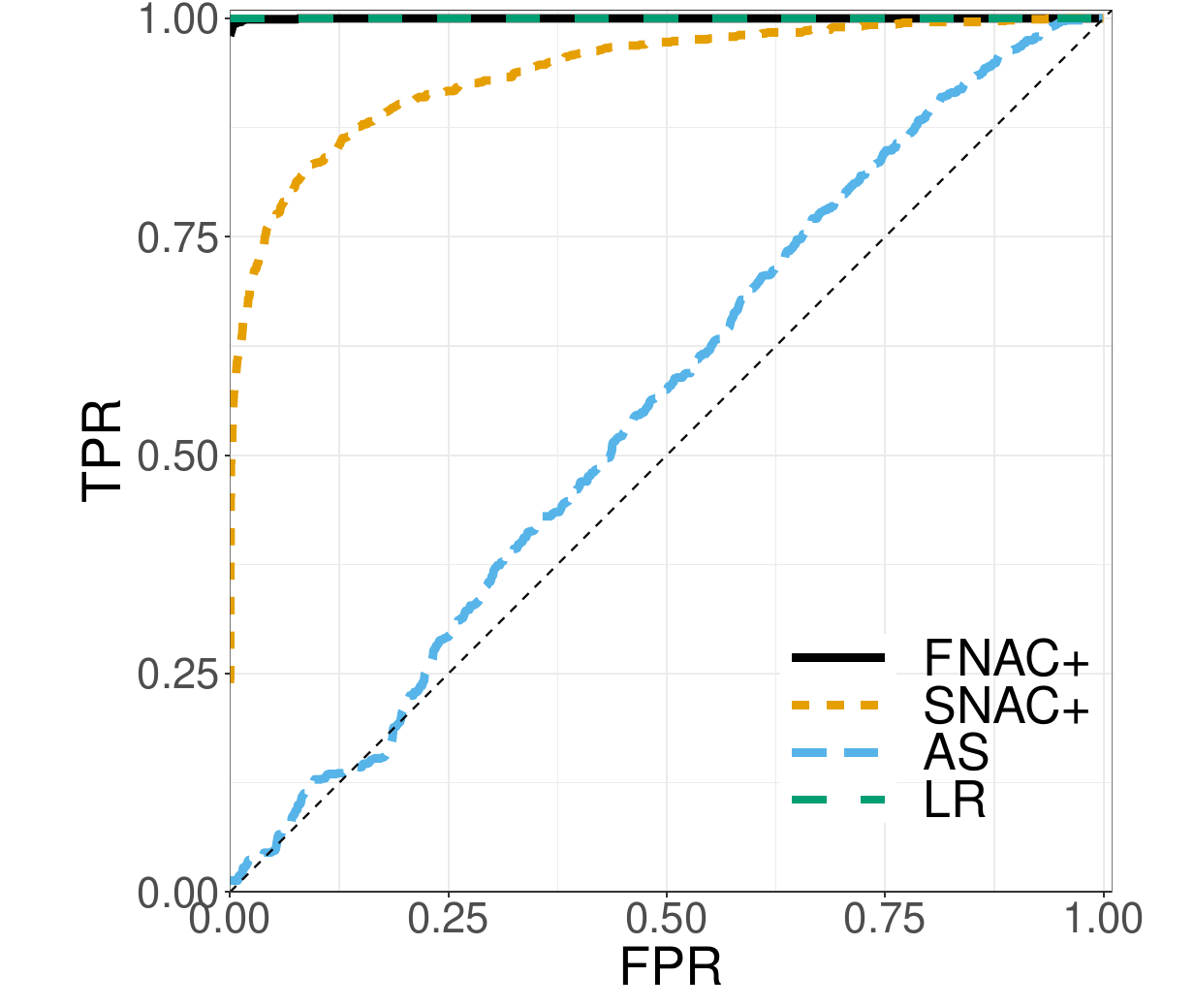}
\includegraphics[width=0.47\linewidth]{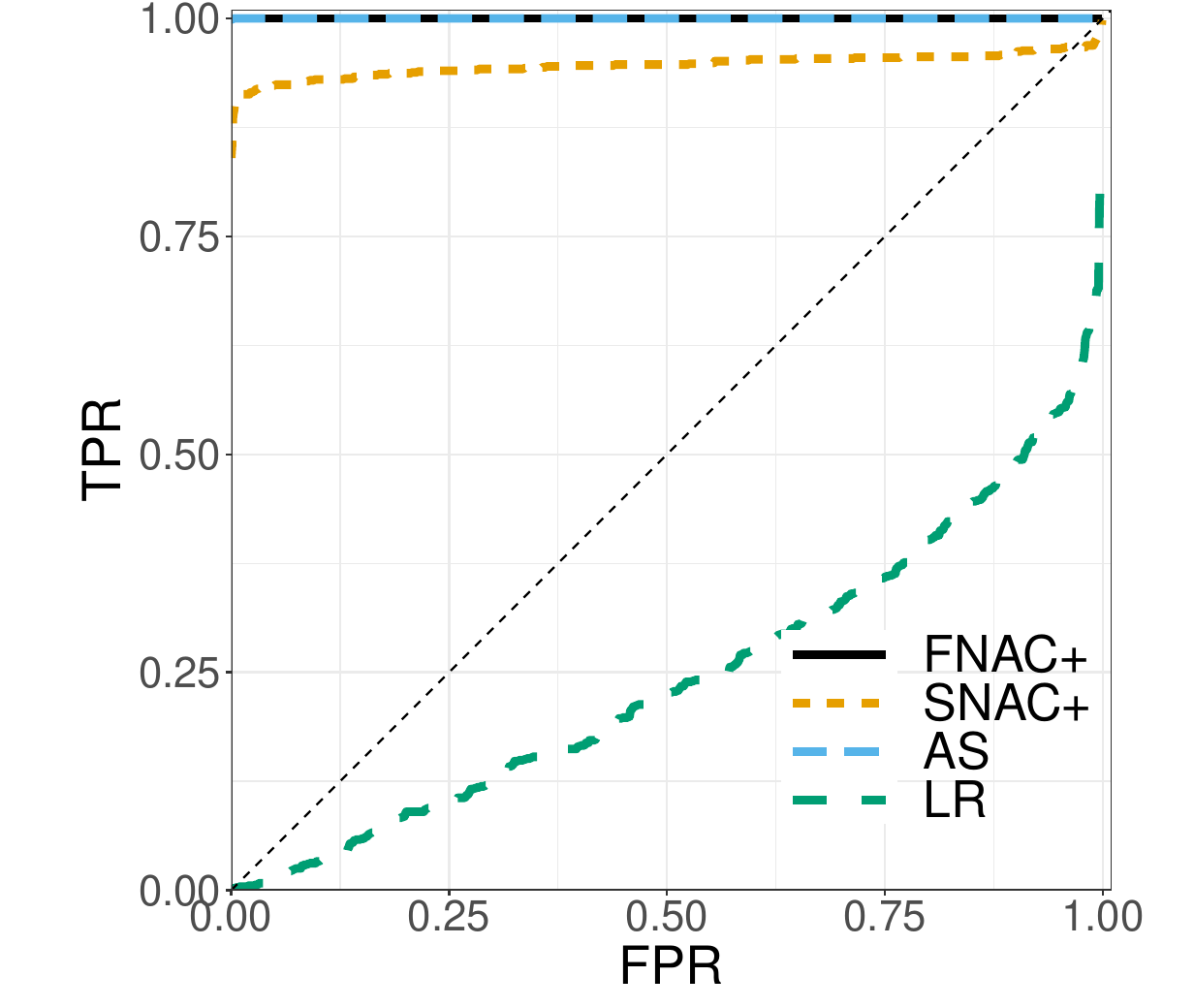}
	\caption{ROC plots for testing 4- versus 3-community models. Top and bottom rows correspond to $n=2000$ and  $n=10000$, respectively. Left and right columns correspond to the DCSBM and DCLVM alternatives, respectively. }
	\label{fig:roc:4vs3}
\end{figure}

We consider additional testing with $H_0: K =4$ vs. $H_a: K = 3$. Other DCSBM simulating parameters are the same as in Section \ref{sec:roc}. Figure \ref{fig:roc:4vs3} shows ROC curves for the null being DCSBM with $K=4$ and two alternatives: a DCSBM with $K = 3$ (left) and a DCLVM with $K = 3$ (right). In addition, we also have $n = 2000$ for the upper row and $n = 10000$ for the lower. Similar to Figure~\ref{fig:roc:4vs5}, the performance of the tests get better as $n$ increases. 
\cacp
and AS tests are nearly perfect  (achieve $100\%$ recovery for very small type I error) when the alternative is DCLVM. The LR test is almost perfect in distinguishing two DCSBMs but has very poor power when the alternative is a DCLVM.

\begin{figure}[t]
	\centering
    \includegraphics[width=0.47\linewidth]{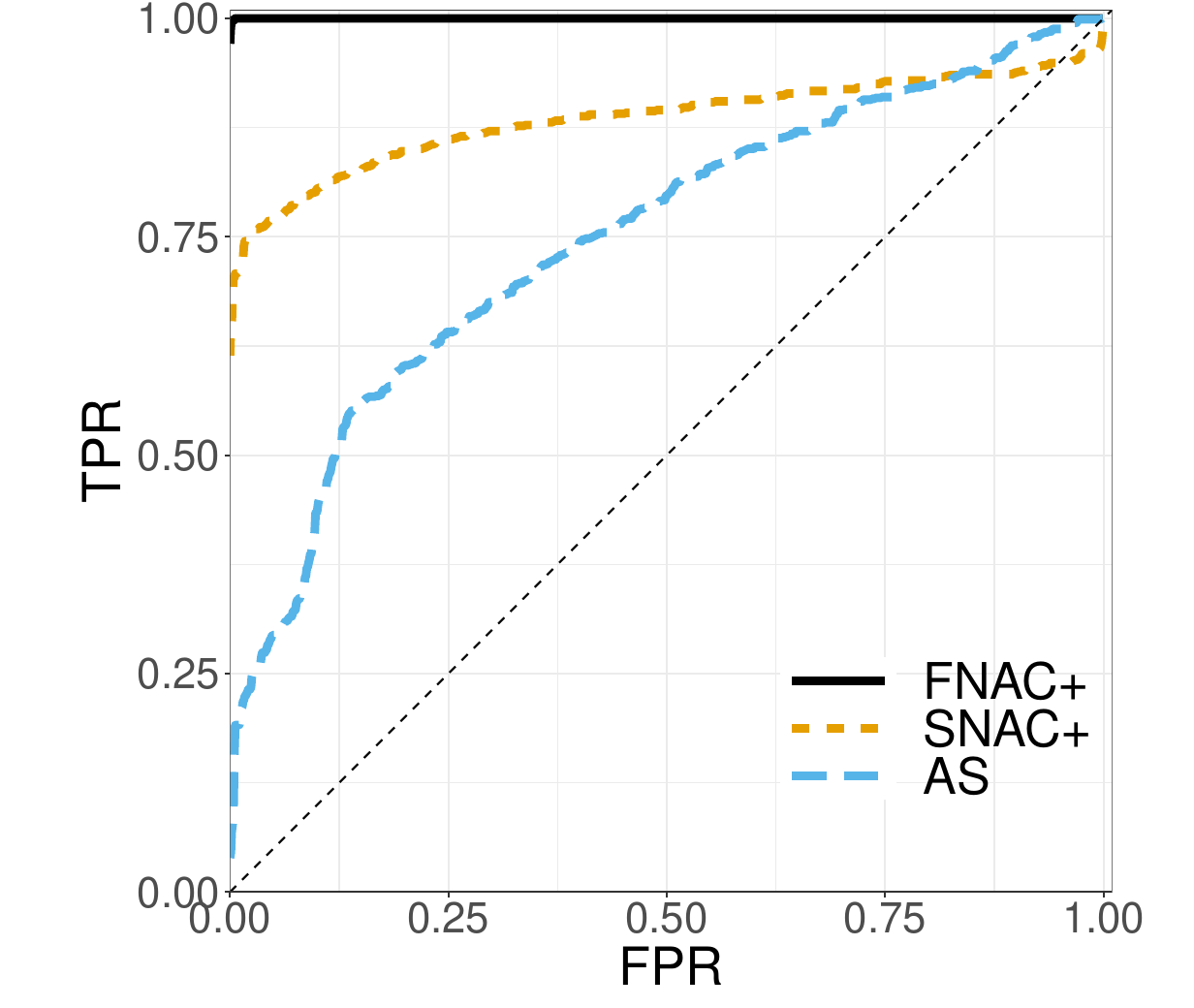} 
	\includegraphics[width=0.47\linewidth]{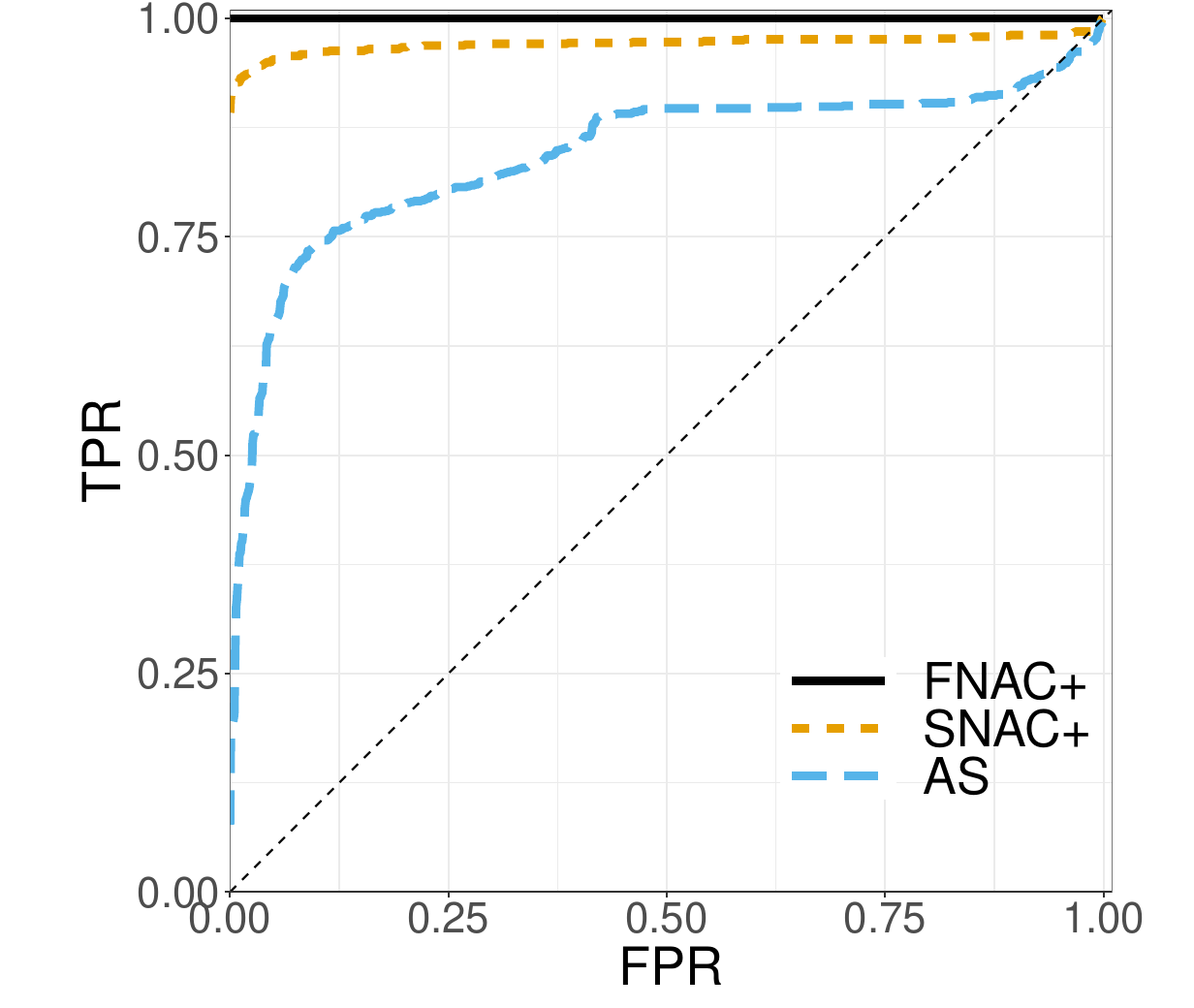}
	\caption{ROC plots for testing $H_0: K =4$ DCSBM vs. $H_a: K = 4$ DCLVM. Left has $n=2000$ and right $n=10000$.}
	\label{fig:roc:4vs4}
\end{figure}

We also include the test $H_0: K =4$, DCSBM vs. $H_a: K = 4$, DCLVM with similar parameters as in Figure~\ref{fig:roc:4vs4}. It shows that 
\cacp tests are still able to reject when the true model is a DCLVM with the same number of communties as the DCSBM. Note that we have excluded the LR test in this case, since it is the likelihood ratio of two fitted DCSBMs with different number of communities, but here we have models with the same number of communities. 

\subsection{More real network examples}\label{app:more:eg}

\begin{figure}[t!]
    \includegraphics[width=0.329\linewidth]{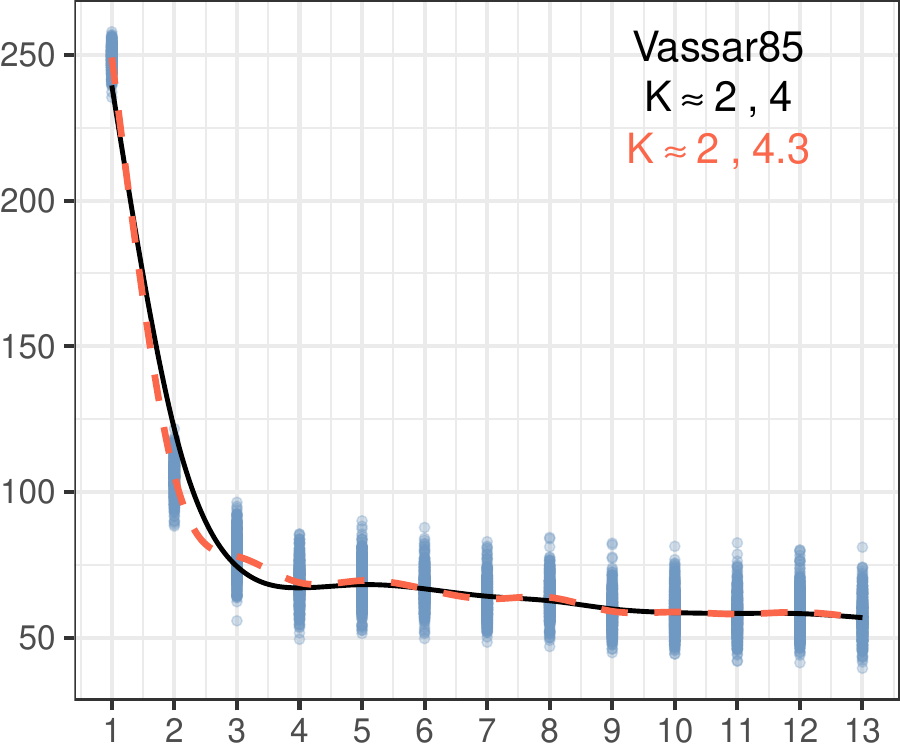}
    \includegraphics[width=0.329\linewidth]{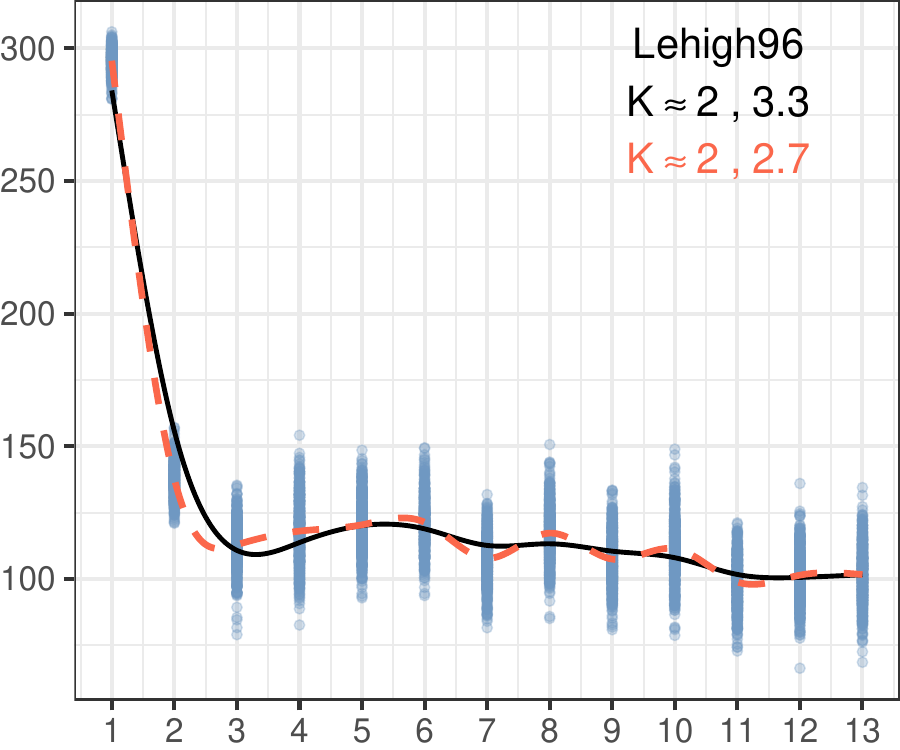}
    \includegraphics[width=0.329\linewidth]{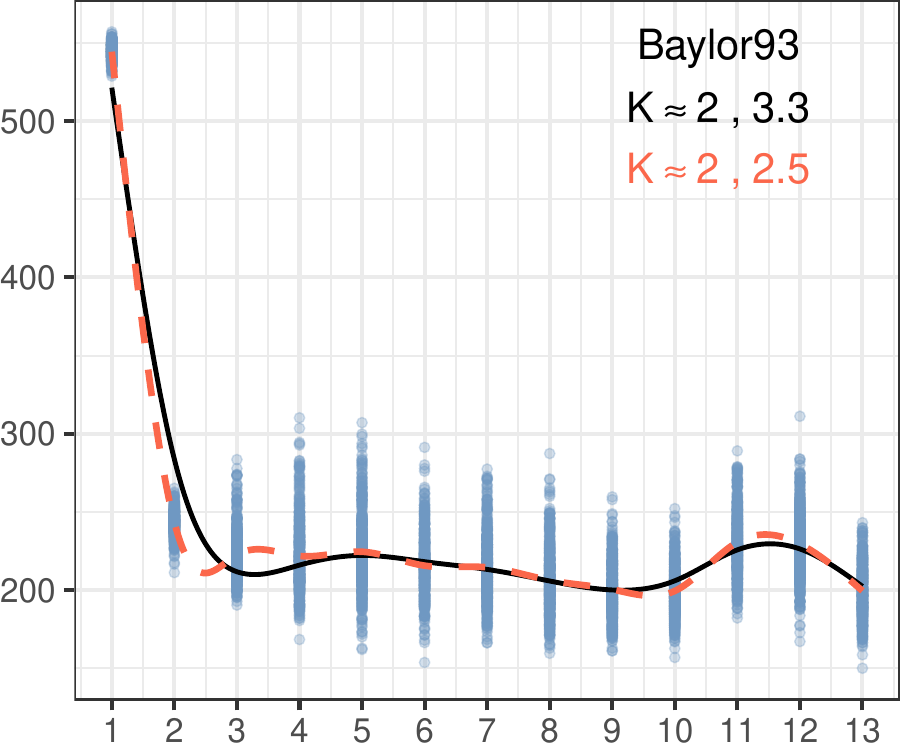}
    \includegraphics[width=0.329\linewidth]{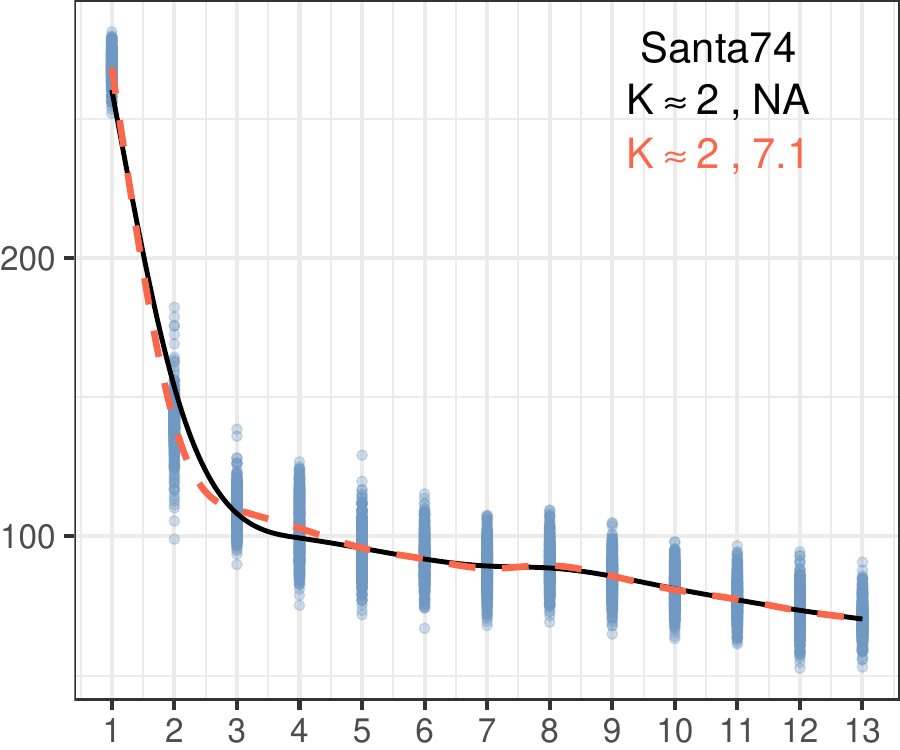}
    \includegraphics[width=0.329\linewidth]{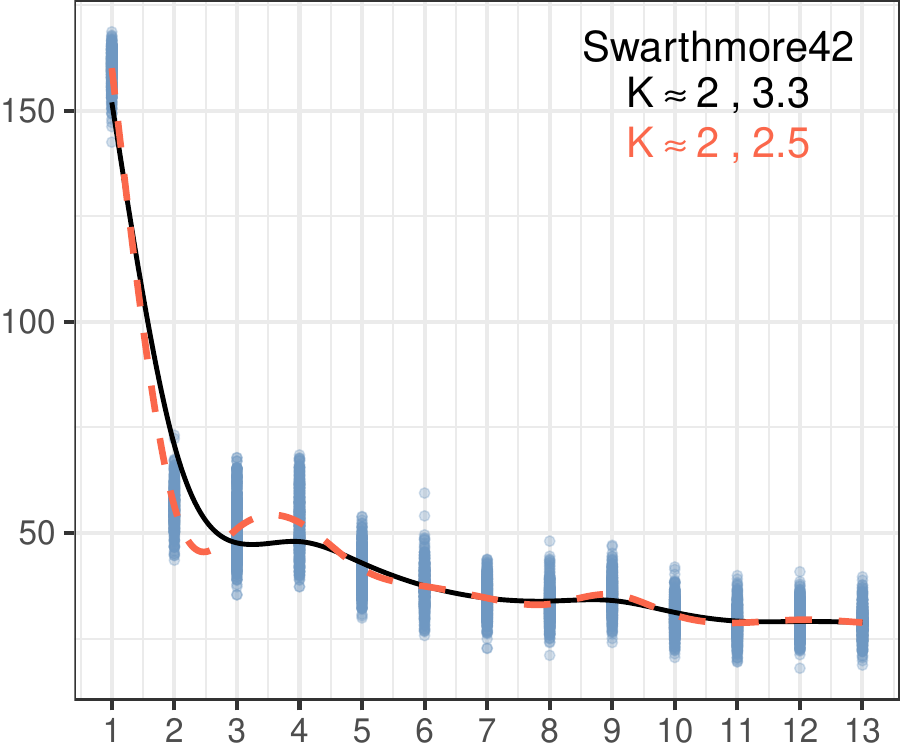}
    \includegraphics[width=0.329\linewidth]{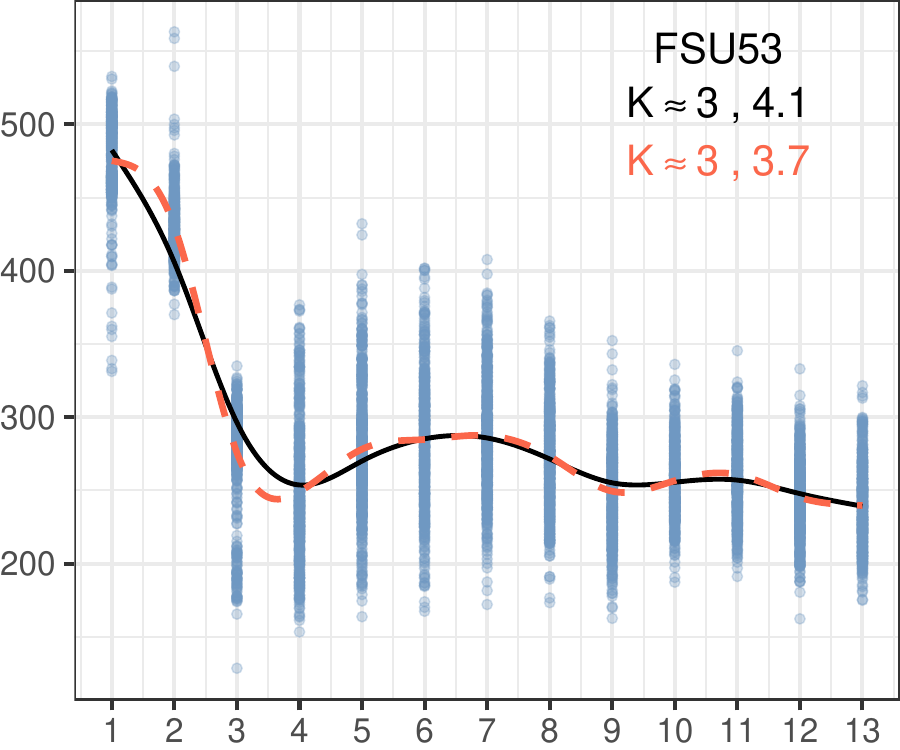}
	\caption{More examples on community profile plots from FB-100. They show a single elbow/dip pattern.}
\label{Fig:fb_profile_2}
\end{figure}

Figures~\ref{Fig:fb_profile_2} and \ref{Fig:fb_profile_3} provide more profile plots for the networks in the FB-100 dataset. The former collection shows profile plots with one-elbow pattern and the latter shows higher variability of \scacp statistics with multi-stage elbows/dips. We also point out that the Caltech network in  Figure~\ref{Fig:fb_profile_3} is the only FB-100 network for which \scacp drops to nearly zero (at $K=10$) within the range of candidate $K$. However, the statistic continues to decrease afterwards and does not show any dips/elbows like others. This suggests that although we cannot reject the null hypothesis of a DCSBM (with $K=10$) in this case, a DCSBM still might not be a good model for the network. That we cannot reject the null is most likely due to the small community sizes we get with $K=10$, leading to an insufficient signal.

\begin{figure}[t!]
    \includegraphics[width=0.329\linewidth]{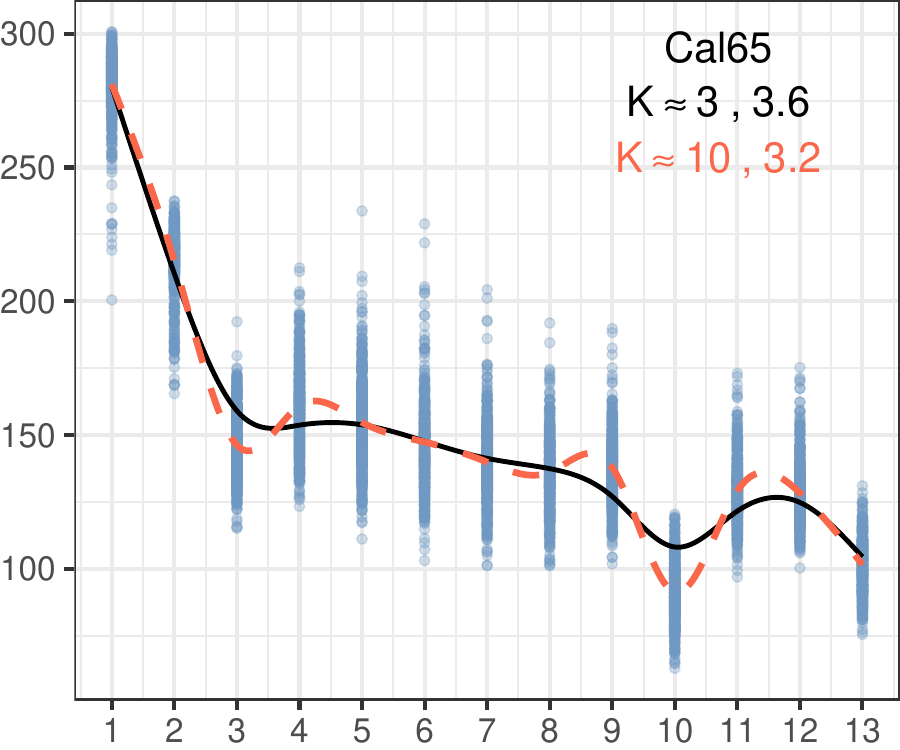}
	\includegraphics[width=0.329\linewidth]{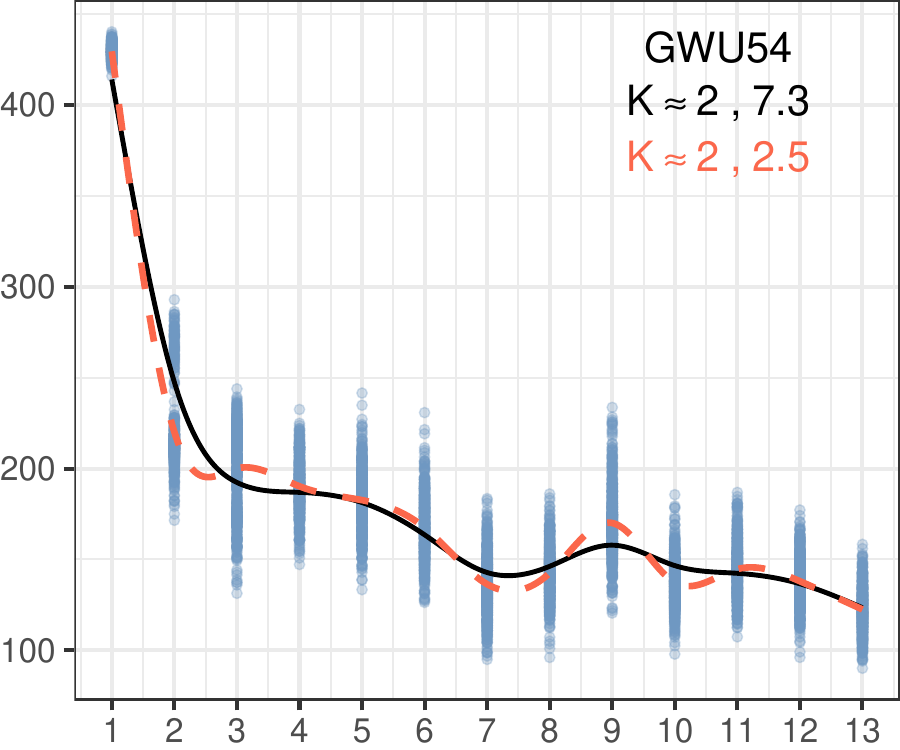}
    \includegraphics[width=0.329\linewidth]{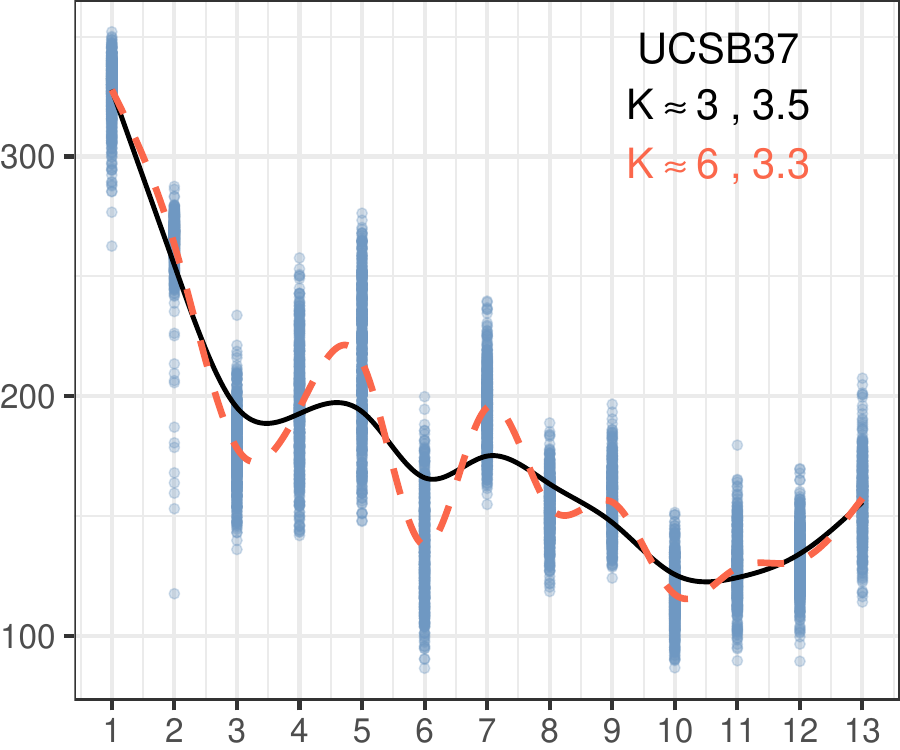}
    \includegraphics[width=0.329\linewidth]{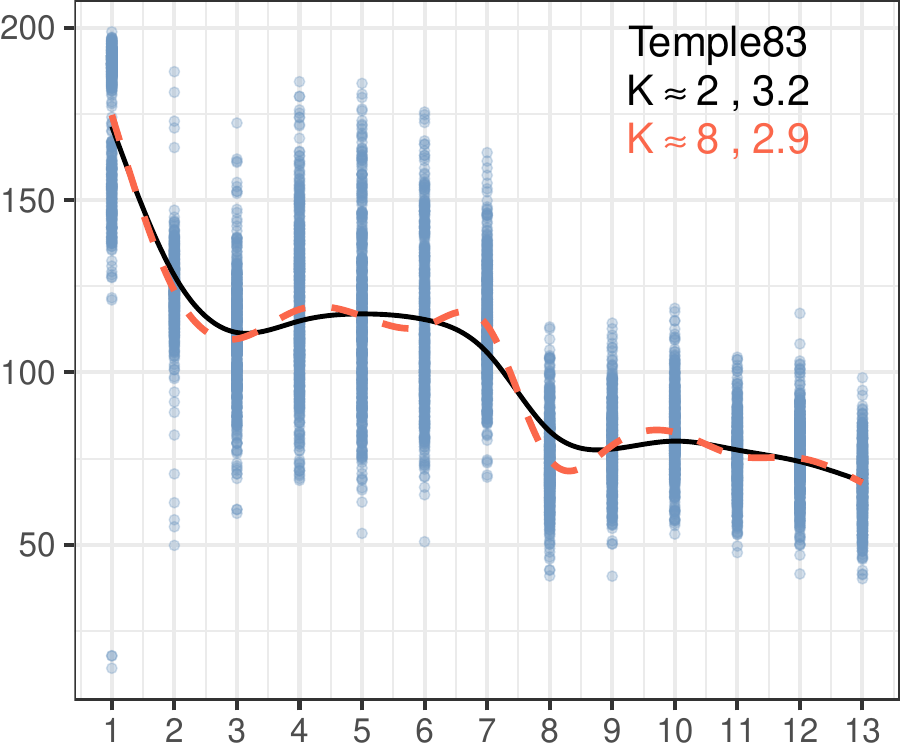}
    \includegraphics[width=0.329\linewidth]{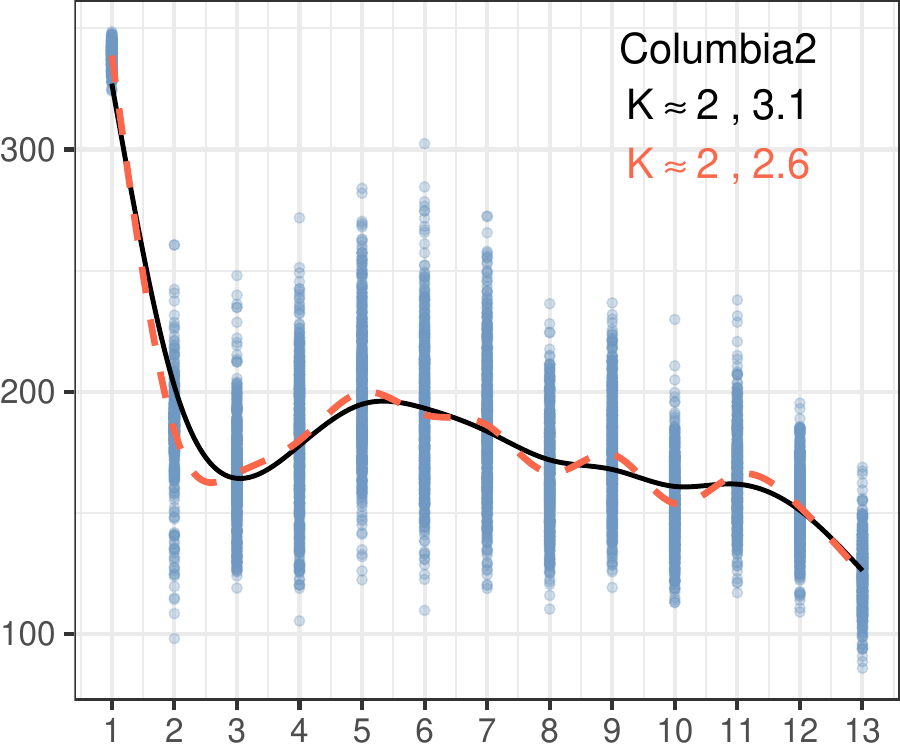}
	\includegraphics[width=0.329\linewidth]{figs_fb2/FSU53_SNAC+_smooth_500.pdf}
	\includegraphics[width=0.329\linewidth]{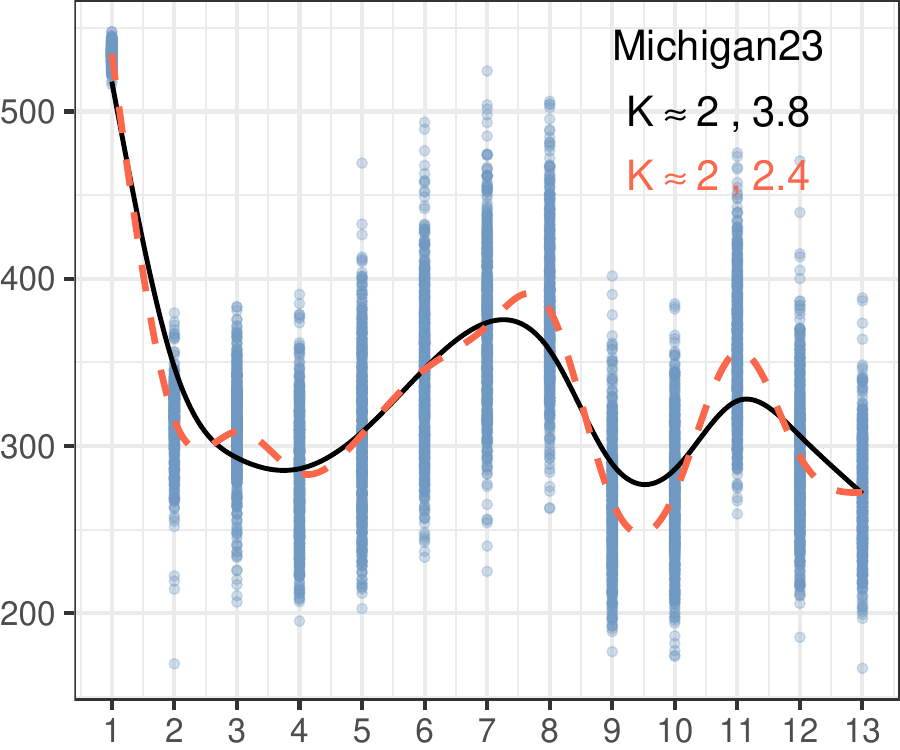}
    \includegraphics[width=0.329\linewidth]{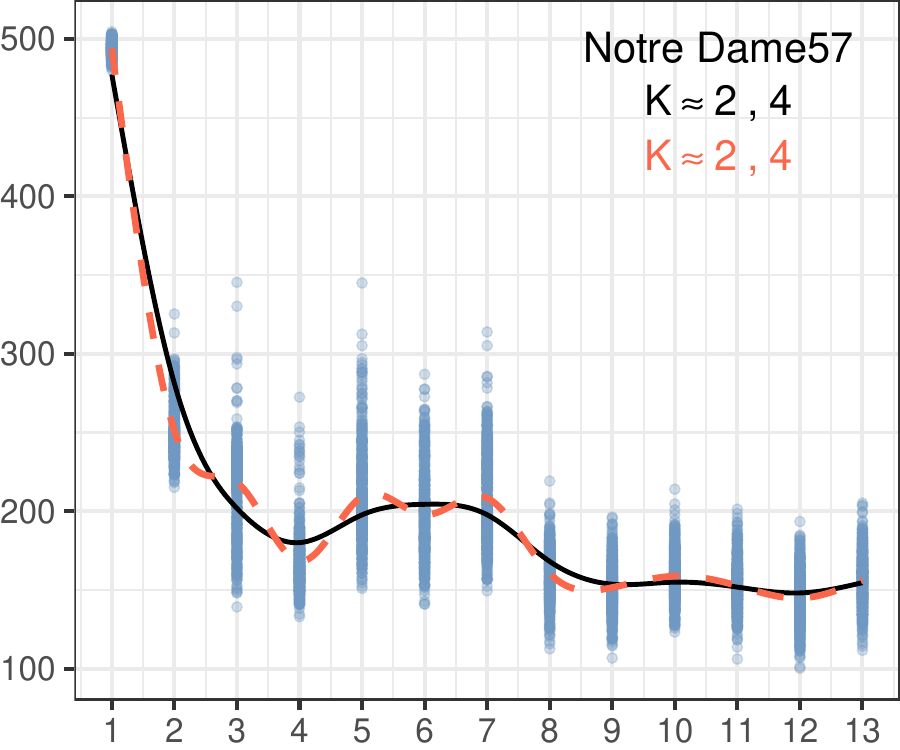}
    \includegraphics[width=0.329\linewidth]{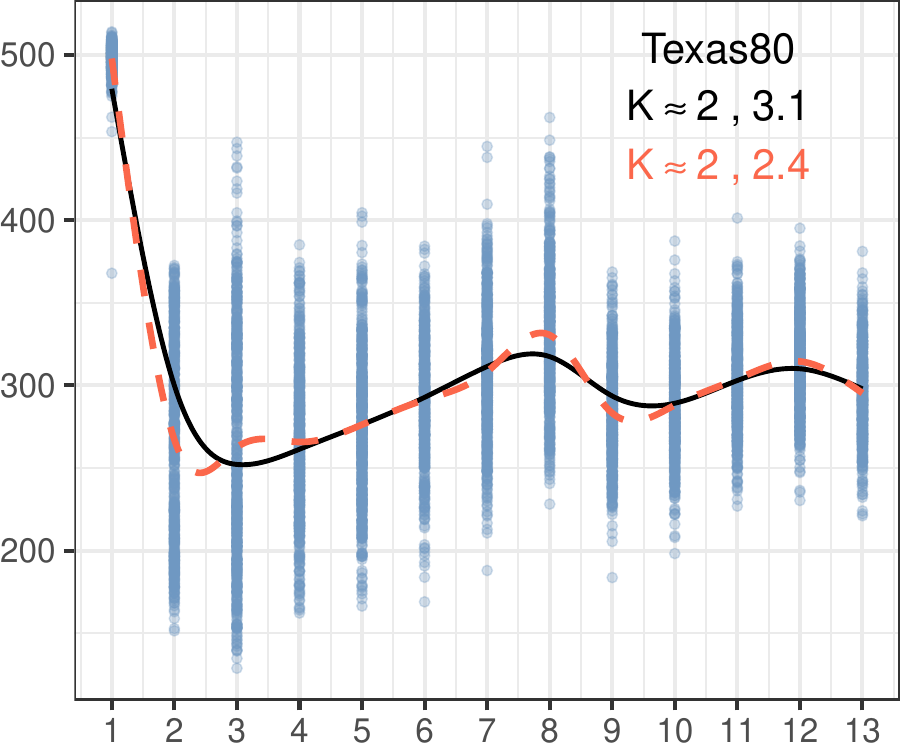}
    \includegraphics[width=0.329\linewidth]{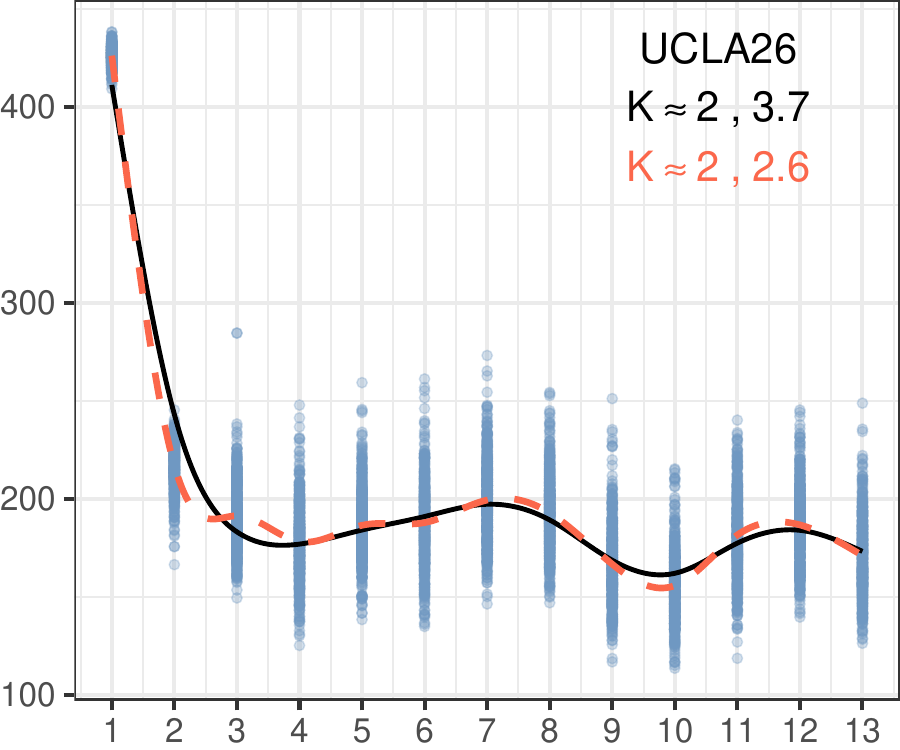}
    \includegraphics[width=0.329\linewidth]{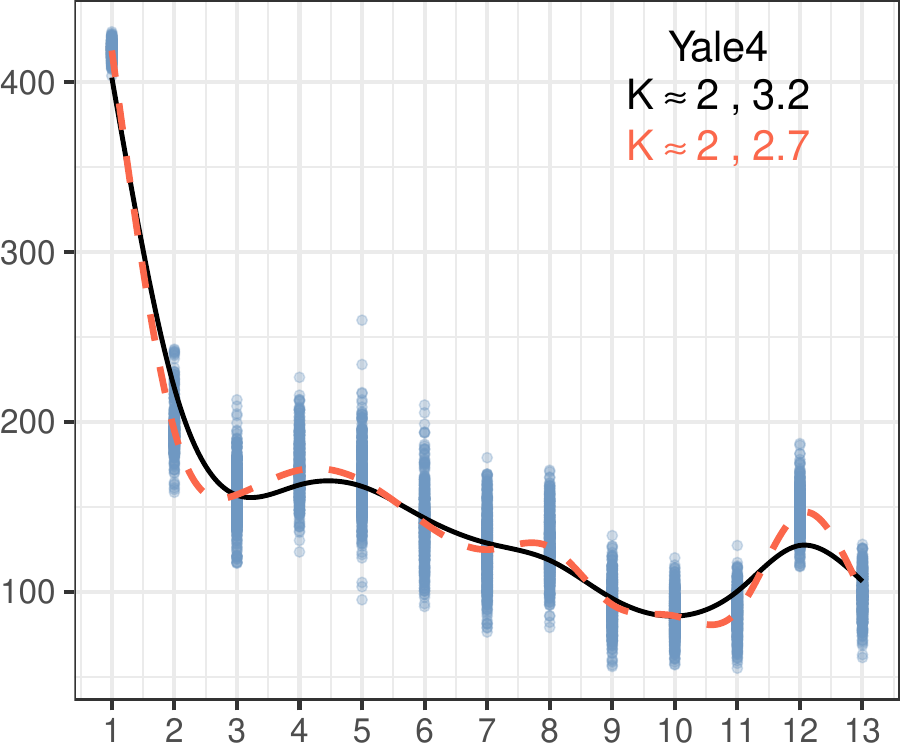}
    \includegraphics[width=0.329\linewidth]{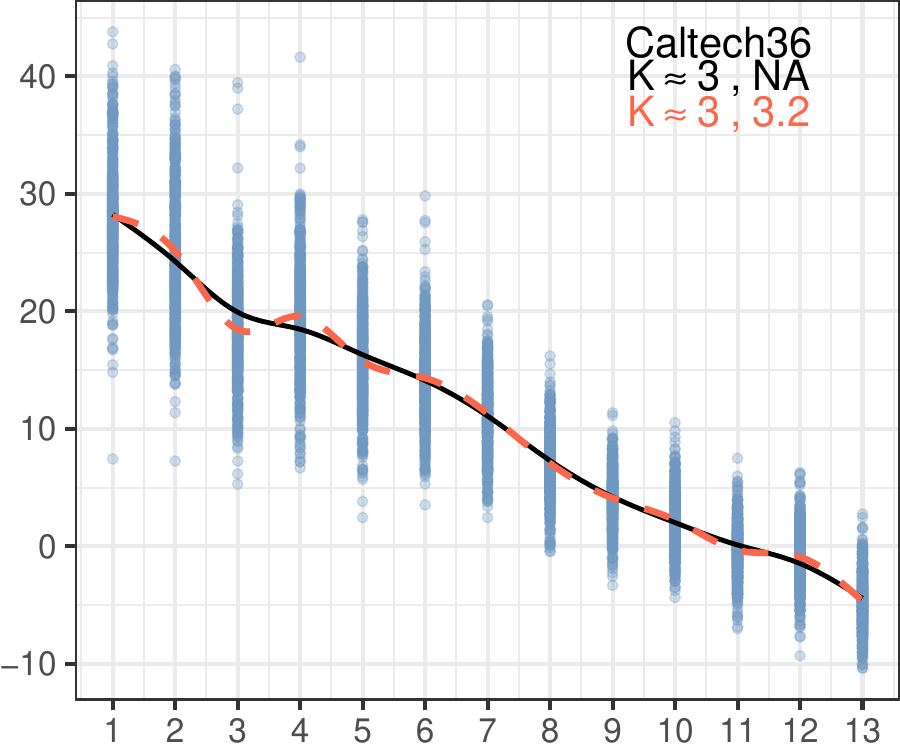}
	\caption{More examples on community profile plots from FB-100. They show a multiple elbows/dips pattern. }
\label{Fig:fb_profile_3}
\end{figure}

Figure~\ref{Fig:polblog_net} shows the profile plot for the political blog network and its community structure. In the profile plot, the elbow point identified by the largest second derivative is at $K=2$, matching the presumed ground truth number of communities in this case. The colored community structure also shows that the fitted two-community model gives a reasonable split of the nodes.

\begin{figure}[th!]
\includegraphics[width=0.49\linewidth]{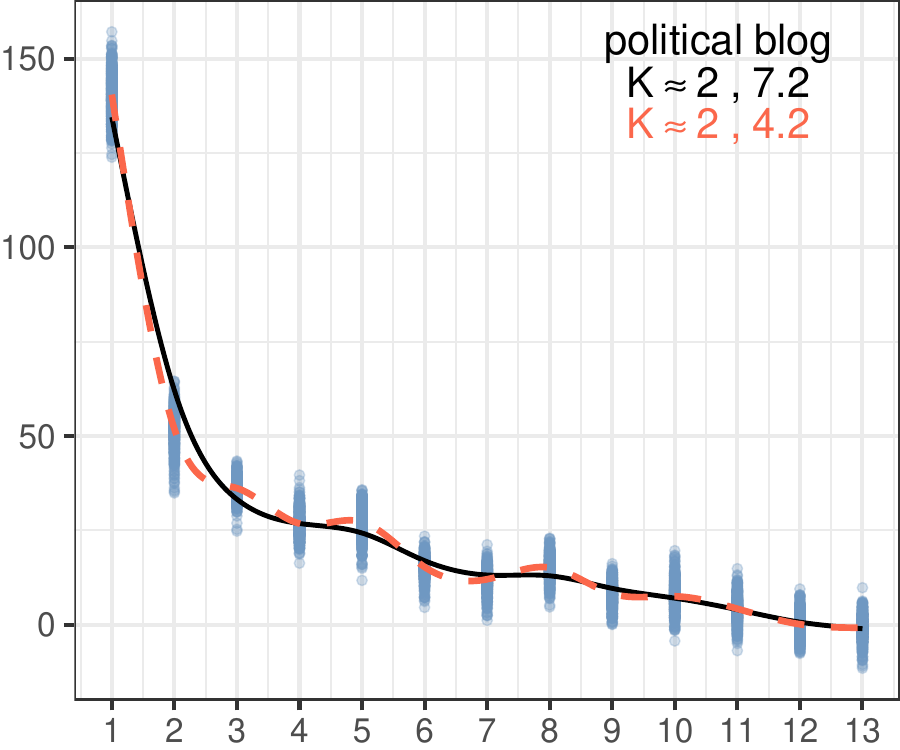}
    \includegraphics[width=0.49\linewidth]{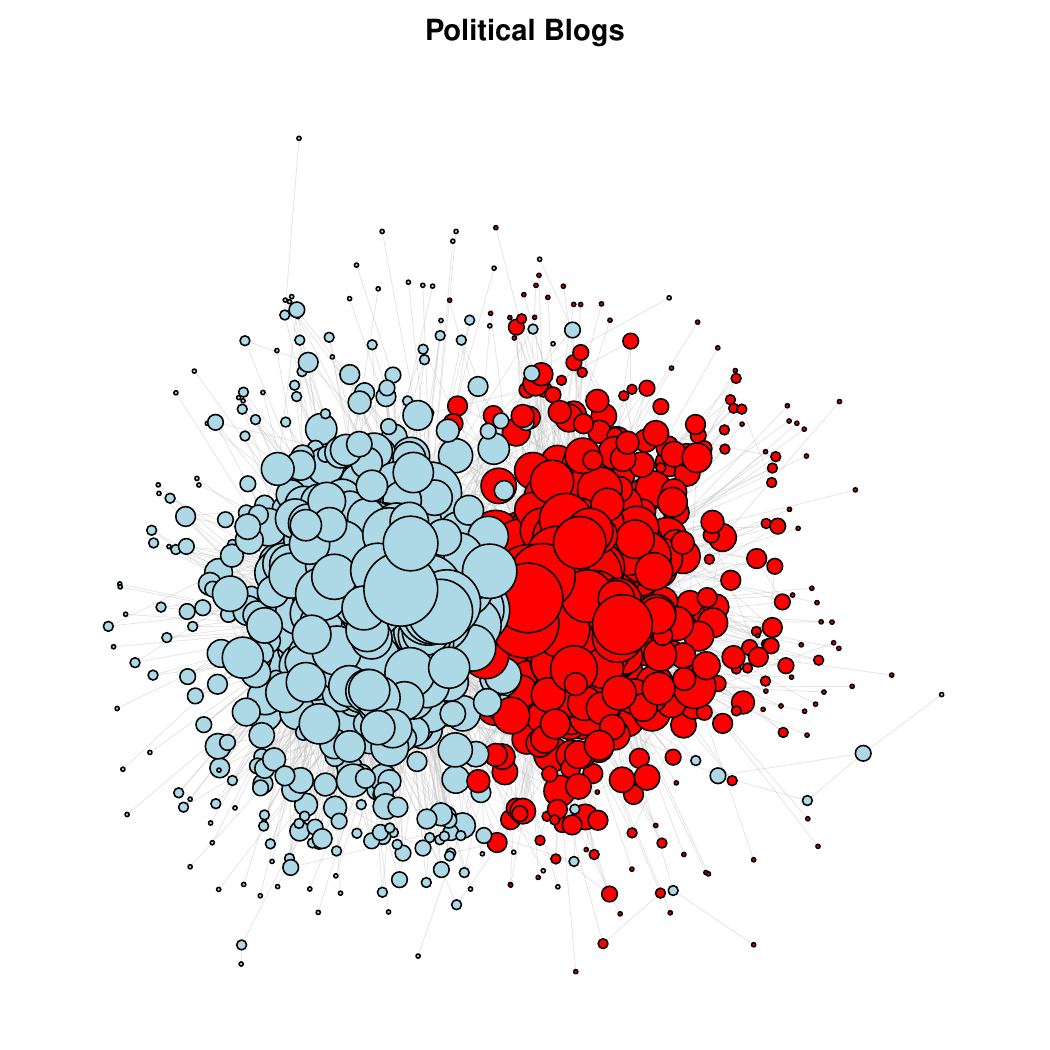}
    \caption{Political blog network: profile plot (left) and community structure (right)}
\label{Fig:polblog_net}
\end{figure}

\end{document}